%% file: shardsumotopes.tex
\documentclass{amsart}

\usepackage[T1]{fontenc}
\usepackage{enumerate, amsmath, amsfonts, amssymb, amsthm, mathrsfs, wasysym, graphics, graphicx, xcolor, url, hyperref, hypcap,  shuffle, xargs, multicol, overpic, pdflscape, multirow, hvfloat, minibox, accents, array, multido, xifthen, a4wide, ae, aecompl, blkarray, pifont, mathtools, etoolbox}
\usepackage{marginnote}
\hypersetup{colorlinks=true, citecolor=darkblue, linkcolor=darkblue}
\usepackage[all]{xy}
\usepackage[bottom]{footmisc}
\usepackage{tikz}
\usetikzlibrary{trees, decorations, decorations.markings, shapes, arrows, matrix, calc, fit, intersections, patterns, angles}
\graphicspath{{figures/}{figures/nodes/}}
\makeatletter\def\input@path{{figures/}}\makeatother
\usepackage{caption}
\usepackage[noabbrev,capitalise]{cleveref}
\usepackage[export]{adjustbox}
\usepackage{ulem}

\newtheorem{theorem}{Theorem}%[section]
\newtheorem{corollary}[theorem]{Corollary}
\newtheorem{proposition}[theorem]{Proposition}
\newtheorem{lemma}[theorem]{Lemma}
\newtheorem{conjecture}[theorem]{Conjecture}
\newtheorem*{theorem*}{Theorem}%[section]

\theoremstyle{definition}
\newtheorem{definition}[theorem]{Definition}
\newtheorem{example}[theorem]{Example}
\newtheorem{remark}[theorem]{Remark}

\newcommand{\R}{\mathbb{R}} % reals
\newcommand{\N}{\mathbb{N}} % naturals
\newcommand{\fS}{\mathfrak{S}} % symmetric group
\newcommand{\fB}{\mathfrak{S}^\textsc{b}} % signed symmetric group
\renewcommand{\c}[1]{\mathcal{#1}} % caligraphic letters
\renewcommand{\b}[1]{\boldsymbol{#1}} % bold letters
\newcommand{\set}[2]{\left\{ #1 \;\middle|\; #2 \right\}} % set notation
\newcommand{\bigset}[2]{\big\{ #1 \;\big|\; #2 \big\}} % big set notation
\newcommand{\Bigset}[2]{\Big\{ #1 \;\Big|\; #2 \Big\}} % Big set notation
\newcommand{\ssm}{\smallsetminus} % small set minus
\newcommand{\dotprod}[2]{\left\langle \, #1 \; \middle| \; #2 \, \right\rangle} % dot product
\newcommand{\symdif}{\,\Delta\,} % symmetric difference
\newcommand{\one}{\b{1}} % the all one vector
\newcommand{\eqdef}{\mbox{\,\raisebox{0.2ex}{\scriptsize\ensuremath{\mathrm:}}\ensuremath{=}\,}} % :=
\newcommand{\defeq}{\mbox{~\ensuremath{=}\raisebox{0.2ex}{\scriptsize\ensuremath{\mathrm:}} }} % =:
\newcommand{\simplex}{\triangle} % simplex
\DeclareMathOperator{\conv}{conv} % convex hull
\DeclareMathOperator{\inv}{inv} % inversions
\DeclareMathOperator{\Binv}{inv^\textsc{b}} % inversions
\DeclareMathOperator{\Vol}{Vol} % (mixed) volume

\newcommand{\ie}{\textit{i.e.}~} % id est
\newcommand{\eg}{\textit{e.g.}~} % exempli gratia
\newcommand{\viceversa}{\textit{vice versa}} % vice versa
\newcommand{\aka}{\textit{a.k.a.}~} % also known as
\definecolor{darkblue}{rgb}{0,0,0.7} % darkblue color
\definecolor{green}{RGB}{57,181,74} % darkblue color
\definecolor{violet}{RGB}{147,39,143} % darkblue color
\newcommand{\darkblue}{\color{darkblue}} % darkblue command
\newcommand{\defn}[1]{\textsl{\darkblue #1}} % emphasis of a definition
\newcommand{\para}[1]{\bigskip\noindent\textbf{\uline{#1.}}} % paragraph
\newcommand{\paraspace}[1]{\bigskip\noindent\textbf{#1.}\medskip} % paragraph more space
\usepackage{todonotes}

\newcommand{\julian}[1]{}%\todo[color=red!30]{\rm #1 \\ \hfill --- J.}}

\makeatletter
\def\part{\@startsection{part}{1}%
\z@{.7\linespacing\@plus\linespacing}{.8\linespacing}%
{\LARGE\sffamily\centering}}
\@addtoreset{section}{part}
\makeatother

\makeatletter
\def\l@part{\@tocline{1}{8pt}{0pc}{}{}}
\def\l@section{\@tocline{1}{3pt}{0pc}{}{}}
\makeatother
\let\oldtocpart=\tocpart
\renewcommand{\tocpart}[2]{\sc\large\oldtocpart{#1}{#2}}
\let\oldtocsection=\tocsection
\renewcommand{\tocsection}[2]{\bf\oldtocsection{#1}{#2}}
\let\oldtocsubsubsection=\tocsubsubsection
\renewcommand{\tocsubsubsection}[2]{\quad\oldtocsubsubsection{#1}{#2}}

\newcommand{\meet}{\wedge} % meet
\newcommand{\join}{\vee} % join
\newcommand{\bigMeet}{\bigwedge} % meet
\newcommand{\bigJoin}{\bigvee} % join
\newcommand{\projDown}{\pi_\downarrow} % down projection map
\newcommand{\projUp}{\pi^\uparrow} % up projection map
\newcommand{\PR}{\mathsf{PR}} % poset of regions
\newcommand{\Bequiv}{{\equiv^\textsc{b}}} % B congruence

\newcommand{\shard}{\polytope{S}}
\newcommand{\shards}{\boldsymbol{S}\!}
\newcommand{\Bshards}{\shards^\textsc{\;b}}
\newcommand{\arc}{\alpha} % arc
\newcommand{\arcs}{{\mathcal{A}}} % arcs
\newcommand{\Barc}{\beta} % barc
\newcommand{\Barcs}{\arcs^\textsc{b}} % barcs
\newcommand{\decoration}{{\b{\delta}}} % decoration
\newcommand{\includeSymbol}[1]{\ensuremath{%
	\mathchoice
		{\raisebox{-.7mm}{\includegraphics[height=2.2ex]{#1}}}	
		{\raisebox{-.7mm}{\includegraphics[height=2.2ex]{#1}}}
		{\raisebox{-.6mm}{\includegraphics[height=1.6ex]{#1}}}
		{\raisebox{-.5mm}{\includegraphics[height=1ex]{#1}}}
}}
\robustify{\includeSymbol}
\newcommand{\noneCirc}{\includeSymbol{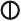}}
\newcommand{\upCirc}{\includeSymbol{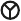}}
\newcommand{\downCirc}{\includeSymbol{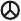}}
\newcommand{\upDownCirc}{\includeSymbol{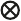}}
\newcommand{\Decorations}{\{\noneCirc{}, \downCirc{}, \upCirc{}, \upDownCirc{}\}} % all decorations

\newcommand{\mat}{M} % matroid
\newcommandx{\asize}[2][2=I]{|#1|_{#2}} % asize

\newcommandx{\ray}[1][1=r]{\boldsymbol{#1}} % ray
\newcommandx{\rays}[1][1=R]{\boldsymbol{#1}} % rays
\newcommand{\hyp}{\mathbb{H}} % hyperplane
\newcommand{\Bhyp}{\mathbb{H}^\textsc{b}} % hyperplane
\newcommand{\Hyp}{\mathbf{H}} % hyperplane
\newcommand{\HA}{\mathcal{H}} % hyperplane arrangement
\newcommand{\HB}{\mathcal{H}^\textsc{b}} % hyperplane arrangement
\newcommand{\rhoB}{\rho^\textsc{b}} % other projection map to type B
\newcommand{\HL}{\textrm{HL}} % Hohlweg-Lange

\newcommand{\polytope}[1]{\mathsf{#1}} % font polytopes
\newcommandx{\Perm}[1][1=n]{\polytope{Perm}_{#1}} % permutahedron
\newcommandx{\BPerm}[1][1=n]{\polytope{Perm}^\textsc{b}_{#1}} % permutahedron
\newcommandx{\Asso}[1][1=n]{\polytope{Asso}_{#1}} % associahedron
\newcommandx{\Zono}[1][1=n]{\polytope{Zono}(#1)} % zonotope
\newcommandx{\quotientope}[1][1=\equiv]{\polytope{Quot}(#1)} % quotientope
\newcommandx{\shardPolytope}[1][1=\arc]{\polytope{SP}(#1)} % shard polytope
\newcommandx{\translation}[2][1=\arc, 2=\arc']{\smash{\b{t}_{#1}^{#2}}} % translation vector
\newcommandx{\translatedShardPolytope}[1][1=\arc]{\overrightarrow{\polytope{SP}}(#1)} % translated shard polytope
\newcommandx{\restShardPolytope}[2][1=\arc]{\polytope{SP}_{#2}(#1)} % restricted shard polytope
\newcommandx{\fan}[1][1=F]{\mathcal{#1}} % fan
\newcommandx{\Fan}[1][1=F]{\fan\!} % fan
\newcommandx{\BFan}{\Fan^{\;\,\textsc{b}}}
\newcommand{\deformedPermutahedronSP}{\polytope{DP\!_s}(\coeffSP)} % shard polytopes deformed permutahedron
\newcommand{\deformedPermutahedronFS}{\polytope{DP\!_y}(\coeffFS)} % faces simplex deformed permutahedron
\newcommand{\deformedPermutahedronRHS}{\polytope{DP\!_z}(\coeffRHS)} % right hand sides deformed permutahedron
\newcommand{\coeffSP}{{\b{s}}} % coefficients in the shard polytope basis
\newcommand{\coeffFS}{{\b{y}}} % coefficients in the faces of simplex basis
\newcommand{\coeffRHS}{{\b{z}}} % right hand sides

\newcommand{\typeCone}{\mathbb{TC}} % type cone
\newcommand{\ctypeCone}{\smash{\overline{\mathbb{TC}}}} % type cone
\newcommandx{\virtualPolytopes}[1][1=n]{\mathbb{V}^{#1}} % virtual polytopes
\newcommandx{\VDP}[1][1=n]{\mathbb{VDP}^{#1}} % virtual deformed permutahedra
\newcommandx{\CVDP}[1][1=n]{\overrightarrow{\mathbb{VDP}}^{#1}} % caged virtual deformed permutahedra
\newcommand{\conn}{\rhd} %$\alpha$-connected
\newcommandx{\nonsing}[1][1=n]{\binom{[#1]}{{\ge} 2}}
\title{Shard polytopes}

\thanks{Partially supported by the French ANR grants CAPPS~17\,CE40\,0018 and CHARMS~19\,CE40\,0017}

\author{Arnau Padrol}
\address[AP]{Institut de Math\'ematiques de Jussieu - Paris Rive Gauche, Sorbonne Universit\'e, Paris}
\email{arnau.padrol@imj-prg.fr}
\urladdr{\url{https://webusers.imj-prg.fr/~arnau.padrol/}}

\author{Vincent Pilaud}
\address[VP]{CNRS \& LIX, \'Ecole Polytechnique, Palaiseau}
\email{vincent.pilaud@lix.polytechnique.fr}
\urladdr{\url{http://www.lix.polytechnique.fr/~pilaud/}}

\author{Julian Ritter}
\address[JR]{LIX, \'Ecole Polytechnique, Palaiseau}
\email{julian.ritter@polytechnique.edu}
\urladdr{\url{http://www.nailuj.de}}

\begin{document}

\begin{abstract}
For any lattice congruence of the weak order on permutations, N.~Reading proved that gluing together the cones of the braid fan that belong to the same congruence class defines a complete fan, called a quotient fan, and V.~Pilaud and F.~Santos showed that it is the normal fan of a polytope, called a quotientope. In this paper, we provide a simpler approach to realize quotient fans based on Minkowski sums of elementary polytopes, called shard polytopes, which have remarkable combinatorial and geometric properties. In contrast to the original construction of quotientopes, this Minkowski sum approach extends to type $B$.
\end{abstract}

\maketitle

\vspace{1.5cm}
\begin{figure}
	\centerline{\includegraphics[scale=.7]{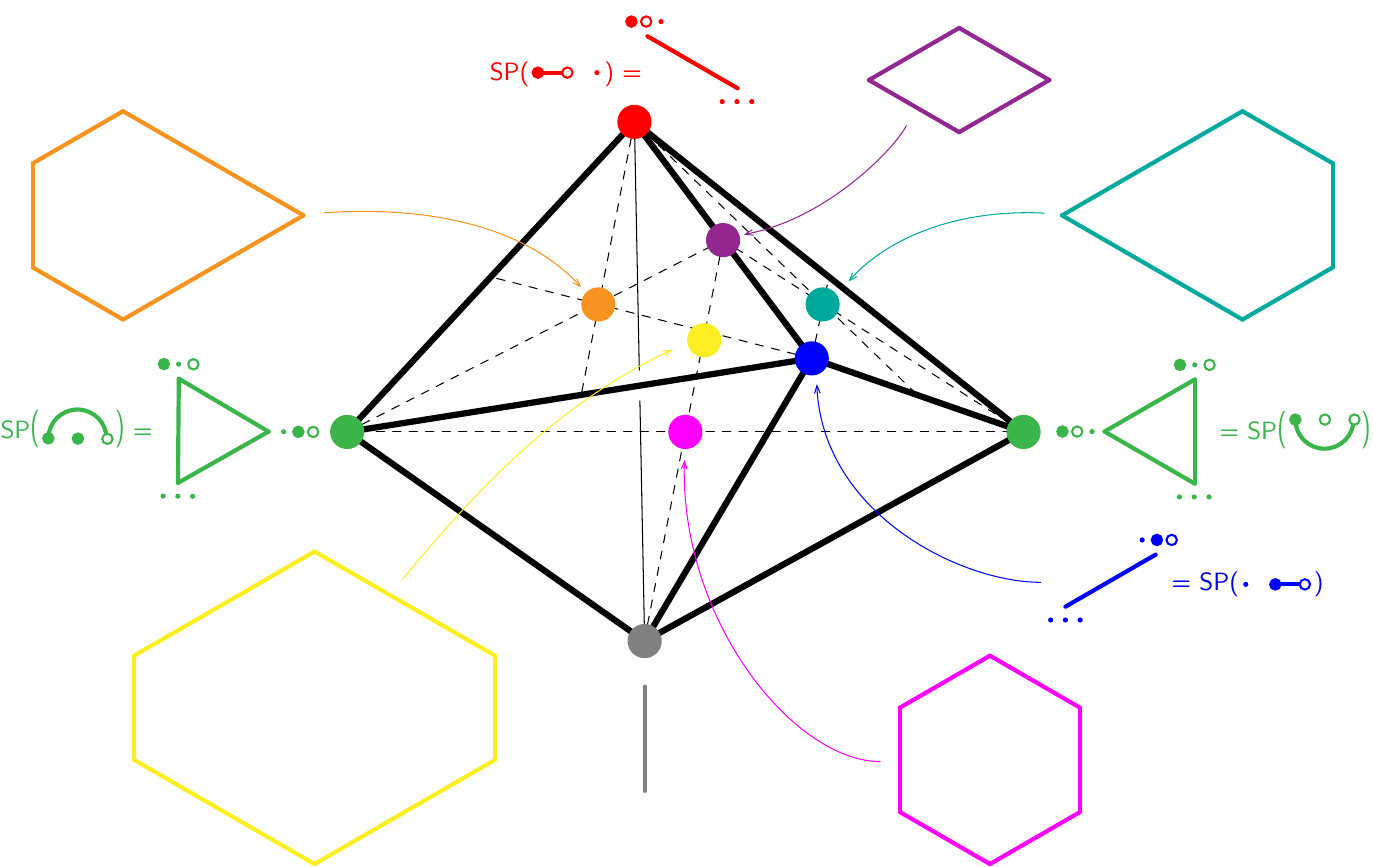}}
\end{figure}

\newpage
\enlargethispage{1.5cm}
\tableofcontents
\newpage

%%%%%%%%%%%%%%%%%%%%%%%%%%%%%%%%%%%%%%

\section*{Introduction}

\para{Context}
This paper deals with polytopal realizations of lattice quotients of the weak order of the symmetric group~$\fS_n$ and of the hyperoctahedral group~$\fB_n$.
The prototype is the classical Tamari lattice on binary trees with $n$ nodes~\cite{Tamari, Stasheff}, seen as the quotient of the weak order  on permutations of~$[n] \eqdef \{1, \dots, n\}$ by the sylvester congruence~\cite{LodayRonco, HivertNovelliThibon-algebraBinarySearchTrees}.
Its Hasse diagram is the graph of the classical associahedron, which can be defined either as the convex hull of well-chosen points associated to all binary trees on~$n$ nodes~\cite{Loday}, or by deleting some well-chosen inequalities from the facet description of the permutahedron~\cite{ShniderSternberg}, or as the Minkowski sum of the faces of the standard simplex corresponding to all intervals of~$[n]$~\cite{Postnikov}.
Other relevant examples of polytopal realizations of lattice quotients of the weak order on~$\fS_n$ include the cube for the boolean lattice, C.~Hohlweg and C.~Lange's associahedra~\cite{HohlwegLange, LangePilaud} for N.~Reading's (type~$A$) Cambrian lattices~\cite{Reading-CambrianLattices, ChatelPilaud}, the permutreehedra for the permutree lattices~\cite{PilaudPons-permutrees}, the brick polytopes~\cite{PilaudSantos-brickPolytope} for the increasing flip lattice on acyclic twists~\cite{Pilaud-brickAlgebra}, Minkowski sums of opposite associahedra for the rotation lattice on diagonal rectangulations~\cite{LawReading, Giraudo}, etc.
In~\cite{PilaudSantos-quotientopes}, V.~Pilaud and F.~Santos constructed polytopes realizing the quotient fan defined by N.~Reading~\cite{Reading-HopfAlgebras} for an arbitrary lattice congruence of the weak order on~$\fS_n$.
In this paper, we propose a simpler approach to construct polytopal realizations of this quotient fan, using Minkowski sums of elementary polytopes called \defn{shard polytopes}.
Besides containing and explaining the construction of~\cite{PilaudSantos-quotientopes}, the motivation for this new construction is the possibility to extend it to lattice quotients of the poset of regions of hyperplane arrangements beyond the braid arrangement.
Indeed, the combinatorics and the geometry of the poset of regions of a hyperplane arrangement are strongly tied, giving rise to a natural geometric realization of lattice quotients of tight arrangements via polyhedral fans.
We refer to the recent surveys of N.~Reading~\cite{Reading-PosetRegionsChapter, Reading-FiniteCoxeterGroupsChapter} for an introduction to the topic.
However, no general polytopal realization is known.
We achieve the first step in this direction by constructing quotientopes for any lattice quotient of the weak order of the type~$B$ Coxeter group~$\fB_n$.

\para{Lattice congruences and arc ideals}
Recall that a congruence of a finite lattice~$L$ is an equivalence relation compatible with the meet and join operations on~$L$.
A congruence of~$L$ is completely determined by the join-irreducibles of~$L$ it contracts (\ie that are not minimal in their congruence classes).
One says that $j$ forces $j'$ if every congruence that contracts~$j$ also contracts~$j'$, and this relation is acyclic when $L$ is congruence uniform.
The set of congruences of~$L$ ordered by refinement is a lattice isomorphic to the lattice of upper ideals of the forcing relation on join-irreducibles of~$L$.

Specializing these observations provides a convenient and powerful combinatorial model to manipulate lattice congruences of the weak order on~$\fS_n$~\cite{Reading-arcDiagrams}.
The join-irreducible elements of the weak order on~$\fS_n$ are the permutations with a single descent.
Such a permutation~$\sigma$ can be naturally encoded by a quadruple~$(a, b, A, B)$ with~$a < b$ and~$A \sqcup B = {]a,b[}$, where~$A$ and~$B$ respectively record the values of the open interval~$]a,b[ \eqdef \{a+1, \dots, b-1\}$ that appear before or after the unique descent~$ba$ in~$\sigma$.
This quadruple can be represented by a curve, called an \defn{arc}, wiggling around points on the horizontal axis, joining~$a$ to~$b$ while passing above the points of~$A$ and below the points of~$B$.
The forcing relation on join-irreducible permutations translates to a forcing relation on arcs with a simple graphical description.
Each lattice congruence~$\equiv$ of the weak order on~$\fS_n$ thus corresponds to an upper ideal~$\arcs_\equiv$ of the forcing order among arcs, called the \defn{arc ideal} of~$\equiv$.

\para{Shards, quotient fans, and quotientopes}
Geometrically, the arcs correspond to pieces of hyperplanes, called the \defn{shards}, that partition the hyperplanes in the braid arrangement.
Namely, the arc~$\arc \eqdef (a, b, A, B)$ corresponds to the shard~$\shard(\arc)$ defined as the piece of the hyperplane~$\b{x}_a = \b{x}_b$ defined by the inequalities~${\b{x}_{a'} \le \b{x}_a = \b{x}_b \le \b{x}_{b'}}$ for all~$a' \in A$ and~$b' \in B$.
In~\cite{Reading-HopfAlgebras}, N.~Reading proved that each lattice congruence~$\equiv$ of the weak order on~$\fS_n$ defines a complete fan~$\Fan_\equiv$, called the \defn{quotient fan}, whose dual graph is the Hasse diagram of the lattice quotient~$\fS_n / {\equiv}$.
The chambers of~$\Fan_\equiv$ can be seen either by gluing together the chambers of the braid fan that belong to the same congruence class, or as the connected components of the complement of the union of the shards~$\shard(\arc)$ for all arcs~$\arc$ in the ideal~$\arcs_\equiv$.
In~\cite{PilaudSantos-quotientopes}, V.~Pilaud and F.~Santos showed that this quotient fan is the normal fan of a polytope, called a \defn{quotientope}.
These realizations were obtained by a careful but slightly obscure choice of the heights defining inequalities normal to each ray of the braid fan.
In this paper, we propose an alternative approach to construct polytopal realizations of the quotient fan, with several advantages discussed below.

\para{Minkowski sums of associahedra}
Our realizations are obtained as Minkowski sums of elementary polytopes.
We illustrate the idea with a simple construction.
For any arc~$\arc$, denote by~$\arcs_\arc$ the arc ideal generated by~$\arc$.
The corresponding congruence~$\equiv_\arc$ is a Cambrian congruence~\cite{Reading-CambrianLattices} and the corresponding quotient fan~$\Fan_\arc$ is a (type~$A$) Cambrian fan of~\cite{ReadingSpeyer}.
It is the normal fan of the $\arc$-associahedron~$\Asso[\arc]$ of~\cite{HohlwegLange}.
Strictly speaking, the setting of~\cite{Reading-CambrianLattices,ReadingSpeyer,HohlwegLange,HohlwegLangeThomas} is slightly different: it is valid for arbitrary finite Coxeter groups, but it is defined in full dimension, so that we should assume that the endpoints of~$\arc$ are~$1$ and~$n$. Their constructions however extend straightforward to arbitrary arcs and we will do so without further notice.
The motivating observation of this paper is the following statement.

\begin{theorem}[see \cref{coro:commonRefinementQuotientFan,coro:MinkowskiSumAssociahedra}]
\label{thm:main1}
Consider an arbitrary congruence~$\equiv$ of the weak order on~$\fS_n$, and let~$\arc_1, \dots, \arc_p$ denote the arcs generating the ideal~$\arcs_\equiv$. Then the quotient fan~$\Fan_\equiv$ is
\begin{itemize}
\item the common refinement of the Cambrian fans~$\Fan_{\arc_1}, \dots, \Fan_{\arc_p}$, and
\item the normal fan of the Minkowski sum of the associahedra~$\Asso[\arc_1], \dots, \Asso[\arc_p]$.
\end{itemize}
\end{theorem}

Note that this observation was already made for certain specific quotients (\eg for the Baxter congruence corresponding to diagonal rectangulations~\cite{LawReading} or for intersections of essential Cambrian congruences~\cite{ChatelPilaud}) but it was never exploited to realize quotient fans of arbitrary lattice congruences.
In contrast to the intricate construction of~\cite{PilaudSantos-quotientopes}, the simple construction of \cref{thm:main1} has the advantage that it transfers all the geometric difficulty into the construction of the $\arc$-associahedra, which was already done in~\cite{HohlwegLange}.
Intuitively, each $\arc_i$-associahedron $\Asso[\arc_i]$ is responsible for the shards of the ideal~$\arcs_{\arc_i}$ to appear in the normal fan of the Minkowski sum.

The main idea of this paper is to push this idea further.
Just as the classical associahedron decomposes into the Minkowski sum of the faces of the standard simplex corresponding to the intervals of~$[n]$ \cite{Postnikov}, the $\arc$-associahedron can be decomposed further into Minkowski sums of indecomposable polytopes.
These polytopes are the central topic of this paper.

\para{Shard polytopes}
For an arc~$\arc \eqdef (a, b, A, B)$, we define the \defn{shard polytope}~$\shardPolytope$ as the convex hull of all vectors with coordinates in~$\{-1,0,1\}$ where the non-zero entries alternate (starting with a $1$ and ending with a $-1$), and the $1$'s occur at a subset of the positions of~$\{a\} \cup A$ while the $-1$'s occur at a subset of the positions of~$B \cup \{b\}$ (see \cref{def:alternatingMatchings,prop:shardPolytope}).
The family of shard polytopes is interesting by itself: for instance, similarly to the families of permutahedra or associahedra, any face of a shard polytope is a Cartesian product of shard polytopes.
But the crucial property of shard polytopes is~the~following.

\begin{proposition}[see \cref{prop:shardPolytopeFan}]
\label{prop:main2}
For any arc~$\arc$, the union of the walls of the normal fan of the shard polytope~$\shardPolytope$ contains the shard~$\shard(\arc)$ and is contained in the union of the shards~$\shard(\arc')$ for all arcs~$\arc'$ forcing~$\arc$.
\end{proposition}

This property enables us to construct quotientopes as Minkowski sums of shard polytopes.
The idea now is that each shard polytope~$\shardPolytope$ will be responsible for the shard~$\shard(\arc)$ to appear in the normal fan of the Minkowski sum, without introducing unwanted walls.

\begin{corollary}[see \cref{coro:MinkowskiSumShardPolytopes}]
\label{coro:main3}
For any lattice congruence~$\equiv$ of the weak order on~$\fS_n$ and any positive coefficients~$\coeffSP_\arc > 0$ for~$\arc \in \arcs_\equiv$, the quotient fan~$\Fan_\equiv$ is the normal fan of the Minkowski sum~$\shardPolytope[\arcs_\equiv] \eqdef \sum_{\arc \in \arcs_\equiv} \coeffSP_\arc \, \shardPolytope$ of the shard polytopes~$\shardPolytope$ of all arcs~${\arc \in \arcs_\equiv}$.
\end{corollary}

Already setting the coefficients~$\coeffSP_\arc = 1$, this construction recovers relevant realizations of specific quotient fans mentioned above.
For the sylvester congruence, all shard polytopes are faces of the standard simplex and the Minkowski sum~$\shardPolytope[\arcs_\equiv]$ is the classical associahedron of~\cite{ShniderSternberg, Loday, Postnikov}.
More generally, for the $\arc$-Cambrian congruence, the Minkowski sum~$\shardPolytope[\arcs_\arc]$ is the $\arc$-associahedron of~\cite{HohlwegLange}.
In contrast, we show that the standard permutahedron is not a Minkowski sum of dilated shard polytopes, although other realizations of the braid fan are.

\para{Type cones and shard polytopes}
Pursuing the quest for elementary Minkowski summands, one could ask whether shard polytopes can be further decomposed into even simpler polytopes.
However, we show that shard polytopes are elementary geometric and combinatorial objects.

\begin{proposition}[see \cref{prop:SPindecomposable}]
\label{prop:main4}
For any arc~$\arc$, the shard polytope~$\shardPolytope$ is Minkowski indecomposable.
\end{proposition}

This statement can be rephrased in the realization spaces of the quotient fans.
For a fan~$\fan$, the space of all polytopes whose normal fan coarsens the fan~$\fan$ is a cone under Minkowski addition, called the (closed) \defn{type cone} by P.~McMullen~\cite{McMullen-typeCone} or the \defn{deformation cone} by A.~Postnikov~\cite{Postnikov, PostnikovReinerWilliams}.
For a rational fan, the type cone is equivalent to the~\defn{nef (numerically effective) cone} of the toric variety associated to the fan~\cite[Sect.~6.3]{CoxLittleSchenckToric}. Nef cones are central objects in the toric minimal model program, see~\cite[Sect.~15]{CoxLittleSchenckToric}.

For instance, the type cone of the braid fan is isomorphic to the classical space of submodular functions.
\cref{coro:main3,prop:main4} affirm that for each arc~$\arc \in \arcs_\equiv$, the shard polytope~$\shardPolytope$ is a representative of a ray of the type cone of the quotient fan~$\Fan_\equiv$.
More specifically, we show that shard polytopes are the rays of the submodular cone belonging to the type cones of the Cambrian fans of~\cite{ReadingSpeyer}.
Type cones of Cambrian fans have recently received particular attention with the works of~\cite{BazierMatteDouvilleMousavandThomasYildirim, PadrolPaluPilaudPlamondon, BrodskyStump, JahnLoweStump}.
Combining the results of these papers with \cref{coro:main3,prop:main4} imply that shard polytopes can be interpreted as Newton polytopes of $F$-polynomials of cluster variables of acyclic type~$A$ cluster algebras~\cite{FominZelevinsky-ClusterAlgebrasI, FominZelevinsky-ClusterAlgebrasII, FominZelevinsky-ClusterAlgebrasIV, BazierMatteDouvilleMousavandThomasYildirim}, and brick polytope summands of certain sorting networks~\cite{PilaudSantos-brickPolytope, BrodskyStump, JahnLoweStump}.
We are not aware that our elementary vertex description in terms of alternating vectors had been observed earlier for these polytopes.

\para{Matroid polytopes and shard polytopes}
It is not difficult to derive from the definition that (up to a simple translation) shard polytopes are \defn{matroid polytopes}, \ie the convex hull of the characteristic vectors of all bases of a matroid.
It turns out that the resulting matroids actually belong to the relevant specific class of \defn{series-parallel matroids}, defined as the graphical matroids of series-parallel graphs.
More precisely, we explicitly define a graph~$\Gamma_\arc$ for each arc~$\arc$ (see \cref{def:shardgraph}) which yields the following statement.

\begin{proposition}[see \cref{prop:2connected,prop:SPisMP}]
\label{prop:main5}
For any arc~$\arc \eqdef (a, b, A, B)$, the matroid polytope of the series-parallel graph~$\Gamma_\arc$ is the translated shard polytope~${\translatedShardPolytope \eqdef \shardPolytope + \one_{B \cup \{b\}}}$.
\end{proposition}

Series-parallel matroids are an important well known family~\cite[Sect.~5.4]{Oxley}, and their matroid polytopes have been studied because of their extremal properties in the context of subdivisions arising from tropical linear spaces~\cite{Speyer2008,Speyer2009}.

\para{Signed Minkowski sums of simplices and of shard polytopes}
As their normal fans coarsen the braid fan, shard polytopes belong to the class of \defn{deformed permutahedra} studied in~\cite{Postnikov, PostnikovReinerWilliams} (we prefer the name ``deformed permutahedra'' rather than ``generalized permutahedra'' as there are many generalizations of permutahedra).
It thus follows from~\cite{ArdilaBenedettiDoker} that they decompose uniquely as a signed Minkowski sum of faces of the standard simplex.
As a consequence of \cref{prop:main5}, the coefficients in this decomposition can be expressed as signed beta invariants of the graphical matroid of~$\Gamma_\arc$, which can be rewritten as follows.
We denote by~$\triangle_J$ the face of the standard simplex corresponding to a subset~$J \subseteq [n]$.
For~$I, J \subseteq [n]$ of cardinality at least~$2$, we write~$I \conn J$ when~$\{\min J, \max J\} \subseteq {]\min I, \max I[} \symdif I$ and ${]\min J, \max J[} \cap I \subseteq J$.

\begin{proposition}[see \cref{prop:FStoSPa,prop:FStoSPb}]
\label{prop:main6}
For any arc~$\arc \eqdef (a, b, A, B)$, the translated shard polytope~$\translatedShardPolytope \eqdef \shardPolytope + \one_{B \cup \{b\}}$ decomposes as
\[
\translatedShardPolytope = \sum_J (-1)^{|J\cap (B\cup\{a,b\})|} \, \triangle_J
\]
where the sum is indexed by all subsets~$J \subseteq [a,b]$ such that~$|J| \ge 2$ and $(A \cup \{a,b\}) \conn J$.
\end{proposition}

Conversely, we show that the faces of the standard simplex also decompose uniquely as a signed Minkowski sum of translated shard polytopes as follows.

\begin{proposition}[see \cref{prop:SPtoFSa,prop:SPtoFSb}]
\label{prop:main8}
For any subset~$J \subseteq [n]$ such that~$|J| \ge 2$, the face~$\triangle_J$ of the standard simplex decomposes as
\[
\triangle_J = \sum_{\arc \eqdef (a, b, A, B)} (-1)^{|\{a,b\} \cap \{\min J, \max J\}|} \, \translatedShardPolytope
\]
where the sum is indexed by all arcs $\arc \eqdef (a, b, A, B)$ such that $J \conn (A \cup \{a, b\})$.
\end{proposition}

In fact, we prove that all deformed permutahedra decompose in terms of shard polytopes.

\begin{proposition}[see \cref{prop:shardPolytopeBasis}]
\label{prop:main9}
Any deformed permutahedron has a unique decomposition as a Minkowski sum and difference of dilated shard polytopes (up to translation).
\end{proposition}

In other words, just as the faces of the standard simplex form a linear basis of the type cone of the braid arrangement that is adapted to its subarrangements, the translated shard polytopes form a linear basis of the type cone of the braid arrangement that is adapted to its quotient fans.
Moreover, \cref{prop:main6,prop:main8} give the exchange matrices between these two bases.
From \cref{prop:main6,prop:main8}, one can also deduce the matrices that transform the heights of deformed permutahedra into their shard polytope coefficients, and \viceversa.

\para{PS-quotientopes revisited}
These parametrizations of deformed permutahedra enable us in particular to prove that our construction recovers and explains that of V.~Pilaud and F.~Santos~\cite{PilaudSantos-quotientopes}. 
Indeed, the quotientopes of~\cite{PilaudSantos-quotientopes} are parametrized by so-called {forcing dominant functions}~$f : 2^{[n]} \to \R_{>0}$. We show that, regardless of the forcing dominant function used in the construction, the resulting quotientope can be obtained with the construction of \cref{coro:main3}.

\begin{proposition}[see \cref{prop:quotientopesMinkowskiSumsShardPolytopes}]
\label{prop:main10}
Any quotientope of~\cite{PilaudSantos-quotientopes} is a Minkowski sum of dilated shard polytopes (up to translation).
\end{proposition}

\para{Mixed volumes of shard polytopes}
Applying the machinery of A.~Postnikov~\cite{Postnikov}, \cref{prop:main6} yields summation formulas for the volume and mixed volumes~\cite{SangwineYager} of shard polytopes.
These formulas are based on the dragon marriage condition of~\cite{Postnikov} to express mixed volumes of faces of the standard simplex.

\begin{theorem}[see \cref{thm:mixedVolumesShardPolytopes}]
\label{thm:main11}
For~$\arc_1, \dots, \arc_{n-1} \in \arcs_n$, the mixed volume of~$\shardPolytope[\arc_1], \dots, \shardPolytope[\arc_{n-1}]$~is
\[
\Vol(\shardPolytope[\arc_1], \dots, \shardPolytope[\arc_{n-1}]) = \frac{1}{(n-1)!}\sum_{J_1, \dots, J_{n-1}} (-1)^{\asize{J_1}[A_1] + \dots +\asize{J_{n-1}}[A_{n-1}]},
\]
summing over all collections $(J_1, \dots, J_{n-1})\in \nonsing^{n-1}$ verifying $(A_i \cup \{a_i,b_i\}) \conn J_i$ for all $i \in [n-1]$, and such that $J_1, \dots ,J_{n-1}$ satisfies the dragon marriage condition of~\cite{Postnikov}.
\end{theorem}

\para{Type $B$ quotientopes}
The second part of this paper is devoted to polytopal realizations of lattice quotients of the weak order of the type~$B$ Coxeter group~$\fB_n$ (\aka hyperoctahedral group).
The prototype example here is the cyclohedron~\cite{BottTaubes, Simion, HohlwegLange}.
In contrast to type~$A$, no systematic construction of type~$B$ quotientopes was previously known. 

The classical tool to manipulate type~$B$ objects is folding type~$A$ objects by central symmetry: the elements of the type~$B_n$ Coxeter group are commonly represented by signed permutations of size~$n$ or centrally symmetric permutations of~$[\pm n] \eqdef \{-n, \dots, -1, 1, \dots, n\}$; the type~$B_n$ Coxeter arrangement is the section of the type~$A_{2n-1}$ Coxeter arrangement by the centrally symmetric space; and the $n$-dimensional cyclohedron is obtained from a $(2n-1)$-dimensional associahedron~$\Asso[\arc]$ for some centrally symmetric arc~$\arc$ by a suitable projection~$\rhoB$.

This folding procedure gives combinatorial and geometric models for lattice congruences of the weak order on~$\fB_n$.
Namely, the join-irreducible elements of~$B_n$ are in bijection with \defn{$B$-arcs} (\ie centrally symmetric $A$-arcs or centrally symmetric pairs of non-crossing $A$-arcs on~$[\pm n]$) and with \defn{$B$-shards} (\ie intersections of $A$-shards with the centrally symmetric space).
Each lattice congruence~$\Bequiv$ of type~$B$ then corresponds to an upper ideal~$\Barcs_\Bequiv$ of the forcing order among $B$-arcs.
It was also proved in~\cite{Reading-HopfAlgebras} that any lattice congruence~$\Bequiv$ of the weak order on~$\fB_n$ defines a quotient fan~$\BFan_\Bequiv$, whose chambers can be constructed either by gluing together the chambers of the type~$B$ Coxeter fan that belong to the same congruence class, or as the connected components of the complement of the union of the $B$-shards for all $B$-arcs in the ideal~$\Barcs_\Bequiv$.

Consider a type~$B$ congruence~$\Bequiv$ whose $B$-arc ideal~$\Barcs_\Bequiv$ forms the $A$-arc ideal~$\arcs_\equiv$ of a type~$A$ congruence~$\equiv$.
Then the quotient fan~$\BFan_\Bequiv$ is just the section of the quotient fan~$\Fan_\equiv$ with the centrally symmetric space.
It can thus be realized by the image of any quotientope realizing~$\Fan_\equiv$ under the projection map~$\rhoB$, for instance by a Minkowski sum of cyclohedra of~\cite{HohlwegLange}.
This is the case, for instance, for the cyclohedron or more generally for type~$B$ permutreehedra~\cite{PilaudPons-permutrees}.

However, the main difficulty at this point is that, since the forcing order among $B$-arcs slightly differs from the forcing order among the corresponding $A$-arcs, some $B$-arc ideals do not form $A$-arc ideals.
Already in type~$B_2$, some quotient fans cannot be realized as projections of type~$A$ quotientopes, or as Minkowski sums of cyclohedra.

At this point, we need the fine granularity of shard polytopes.
We define the \defn{shard polytope}~$\shardPolytope[\Barc]$ of a $B$-arc~$\Barc$ as the image of the shard polytope~$\shardPolytope$ of the corresponding $A$-arc~$\arc$ under the projection~$\rhoB$.
Again, the crucial property of type $B$ shard polytopes is the following.

\begin{proposition}[see \cref{prop:shardPolytopeFanB}]
\label{prop:main12}
For any $B$-arc~$\Barc$, the union of the walls of the normal fan of the shard polytope~$\shardPolytope[\Barc]$ contains the shard~$\shard(\Barc)$ and is contained in the union of the shards~$\shard(\Barc')$~for~$\Barc \prec \Barc'$.
\end{proposition}

This property provides the first proof that all type~$B$ quotient fans are polytopal.

\begin{corollary}[see \cref{coro:MinkowskiSumShardPolytopesB}]
\label{coro:main13}
For any lattice congruence~$\Bequiv$ of the weak order on~$\fB_n$ and any positive coefficients~$\coeffSP_\Barc > 0$ for~$\Barc \in \Barcs_\Bequiv$, the quotient fan~$\BFan_{\Bequiv}$ is the normal fan of the Minkowski sum $\shardPolytope[\Barcs_\Bequiv] \eqdef \sum_{\Barc \in \Barcs} \coeffSP_\Barc \shardPolytope[\Barc]$ of the shard polytopes~$\shardPolytope[\Barc]$ of all $B$-arcs~${\Barc \in \Barcs_\Bequiv}$.
\end{corollary}

Again, specializing to coefficients~$\coeffSP_\Barc = 1$, we recover the cyclohedra of~\cite{HohlwegLange}.
As in type~$A$, we believe that type~$B$ shard polytopes are elementary polytopes.

\begin{conjecture}[see \cref{conj:BshardPolytopesIndecomposable}]
\label{conj:main14}
For any $B$-arc~$\Barc$, the shard polytope~$\shardPolytope[\Barc]$ is Minkowski indecomposable.
\end{conjecture}

We have checked this conjecture experimentally for all type~$B_n$ shard polytopes for small values of~$n$.
However, the simple Minkowski indecomposability criterion that we use to prove \cref{prop:main4} fails in type~$B$ and the proof of \cref{conj:main14} would require a much finer understanding of the facets of the type~$B$ shard polytopes.

If \cref{conj:main14} holds, then the work of~\cite{BazierMatteDouvilleMousavandThomasYildirim, PadrolPaluPilaudPlamondon, BrodskyStump, JahnLoweStump} implies that for any type~$B$ arc~$\Barc$ which appears in the $B$-arc ideal of some type~$B$ Cambrian congruence, the shard polytope~$\shardPolytope[\Barc]$ is the Newton polytope of the $F$-polynomial of some type~$B$ cluster variable~\cite{FominZelevinsky-ClusterAlgebrasI, FominZelevinsky-ClusterAlgebrasII, FominZelevinsky-ClusterAlgebrasIV}, or equivalently some brick polytope summand of some type~$B$ subword complex~\cite{PilaudStump-brickPolytope, BrodskyStump, JahnLoweStump}.
However, in contrast to the type~$A$ situation, there are $B$-arcs which do not belong to the $B$-arc ideal of any type~$B$ Cambrian congruence.
We call them \defn{unsortable} $B$-arcs as the corresponding join-irreducible elements are not $c$-sortable for any Coxeter element~$c$.
We are not aware that the shard polytopes of these unsortable $B$-arcs appear in any other related construction.
In particular, it is not clear whether they have a relevant algebraic interpretation.

Finally, as in type~$A$, the type~$B$ shard polytopes forms a linear basis of the type cone of the type~$B$ Coxeter arrangement.

\begin{proposition}[see \cref{coro:typeBShardPolytopesBasis}]
\label{prop:main15}
Any type~$B$ deformed permutahedron has a unique decomposition as a Minkowski sum and difference of dilated type~$B$ shard polytopes (up to translation).
\end{proposition}

Note that this solves the problem posed in~\cite[Qu.~9.3]{ArdilaCastilloEurPostnikov} to find an explicit natural basis for the type cone of type~$B$ permutahedra.
An interesting question that deserves further study is to describe the matrices that transform the heights of type~$B$ deformed permutahedra into their type~$B$ shard polytope coefficients, and \viceversa.

\para{Further directions}
We conclude the paper with a discussion on the existence of (weak) shard polytopes for arbitrary hyperplane arrangements (see \cref{def:shardPolytopesHyperplaneArrangements}).
We show, in particular, that they exist in type~$I_2(n)$ for any integer~$n$, and discuss some progress on supersolvable arrangements.
Let us conclude this introduction by mentioning that the existence of shard polytopes would have the following two consequences.

\begin{proposition}[see \cref{prop:quotientopesHyperplaneArrangements}]
\label{prop:main16}
Consider a hyperplane arrangement~$\HA$ with a base region~$\polytope{B}$ such that the poset of regions~$\PR(\HA, \polytope{B})$ is a lattice, and a lattice congruence~$\equiv$ of~$\PR(\HA, \polytope{B})$ with shard ideal~$\shards_\equiv$.
If each shard~$\shard$ of~$\shards_\equiv$ admits a weak shard polytope~$\polytope{P}_\shard$, then the quotient fan~$\Fan_\equiv$ is the normal fan of the Minkowski sum~$\sum_{\shard \in \shards_\equiv} \polytope{P}_\shard$.
\end{proposition}

\begin{proposition}[see \cref{prop:basisTypeConeHyperplaneArrangements}]
\label{prop:main17}
Consider a simplicial arrangement~$\HA$ and a base region~$\polytope{B}$ such that any shard~$\shard$ admits a weak shard polytope~$\polytope{P}_\shard$.
Then any polytope in the deformation cone of the zonotope of~$\HA$ has a unique decomposition as a Minkowski sum and difference of dilated shard polytopes~$\polytope{P}_\shard$ for all shards~$\shard$ (up to translation).
\end{proposition}

\para{Acknoledgements}
We are grateful to the participants of the reading seminar at the Discrete Geometry group at FU Berlin (Matthias Beck, Andrei Com\v{a}neci, Jean-Philippe Labb\'e, Claudia Mitukiewicz and Sophie Rehberg) and to an anonymous referee for numerous comments and suggestions that improved the presentation of this paper.

%%%%%%%%%%%%%%%%%%%%%%%%%%%%%%%%%%%%%%
%%%%%%%%%%%%%%%%%%%%%%%%%%%%%%%%%%%%%%

\newpage
\part{Type~$A$ shard polytopes}
\label{part:typeA}

%%%%%%%%%%%%%%%%%%%%%%%%%%%%%%%%%%%%%%

\section{Preliminaries}
\label{sec:preliminaries}

We start with preliminaries on combinatorial and geometric properties of lattice quotients of the weak order on permutations.
The presentation borrows from the papers~\cite{PilaudSantos-quotientopes, Pilaud-arcDiagramAlgebra, ChatelPilaud} and we reproduce here some of their pictures.
Some combinatorial and lattice theoretic aspects presented here (\cref{subsec:noncrossingArcDiagrams,subsec:canonicalJoinRepresentations}) are not strictly needed for the geometric discussion of \cref{part:typeA}, but will be useful to introduce analogous properties in type~$B$ in \cref{part:typeB}.
Throughout the paper, $n$ is a fixed integer, and we use the notations~$[n] \eqdef \{1, \dots, n\}$ as well as~$[a,b] \eqdef \{a, \dots, b\}$ and~${{]a,b[} \eqdef \{a+1, \dots, b-1\}}$ for two integers~${1 \le a < b \le n}$.

%%%%%%%%

\subsection{Fans and polytopes}
\label{subsec:fansPolytopes}

We start with basic notions of polyhedral geometry.
We refer to G.~Ziegler's classic textbook~\cite{Ziegler-polytopes} for a detailed presentation.

A hyperplane~$\polytope{H} \subset \R^n$ is a \defn{supporting hyperplane} of a set~$\polytope{X} \subset \R^n$ if~$\polytope{H} \cap \polytope{X} \ne \varnothing$ and~$\polytope{X}$ is contained in one of the two closed half-spaces of~$\R^n$ defined by~$\polytope{H}$.

A (polyhedral) \defn{cone} is a subset of~$\R^n$ defined equivalently as the positive span of finitely many vectors or as the intersection of finitely many closed linear halfspaces.
Its \defn{faces} are its intersections with its supporting linear hyperplanes, and its \defn{rays} (resp.~\defn{facets}) are its dimension~$1$ (resp.~codimension~$1$) faces.
A (polyhedral) \defn{fan}~$\fan$ is a collection of cones which are closed under faces (if~$\polytope{C} \in \fan$ and~$\polytope{F}$ is a face of~$\polytope{C}$, then~$\polytope{F} \in \fan$) and intersect properly (if~$\polytope{C}, \polytope{C}' \in \fan$, then~$\polytope{C} \cap \polytope{C}'$ is a face of both~$\polytope{C}$ and~$\polytope{C}'$).
The \defn{chambers} (resp.~\defn{walls}, resp.~\defn{rays}) of~$\fan$ are its codimension~$0$ (resp.~codimension~$1$, resp.~dimension~$1$) cones.

A \defn{polytope} is a subset of~$\R^n$ defined equivalently as the convex hull of finitely many points or as a bounded intersection of finitely many closed affine halfspaces.
Its \defn{faces} are its intersections with its supporting affine hyperplanes, and its \defn{vertices} (resp.~\defn{edges}, resp.~\defn{facets}) are its dimension~$0$ (resp.~dimension~$1$, codimension~$1$) faces.
The \defn{normal cone} of a face~$\polytope{F}$ of a polytope~$\polytope{P}$ is the cone generated by the outer normal vectors of the facets of~$\polytope{P}$ containing~$\polytope{F}$.
The \defn{normal fan} of~$\polytope{P}$ is the fan formed by the normal cones of all faces of~$\polytope{P}$.

The \defn{Minkowski sum} of two polytopes~$\polytope{P}, \polytope{Q} \subset \R^n$ is the polytope~$\polytope{P} + \polytope{Q} \eqdef \set{\b{p}+\b{q}}{\b{p} \in \polytope{P}, \, \b{q} \in \polytope{Q}}$.
For any~$\ray \in \R^n$, the face maximizing the direction~$\ray$ on~$\polytope{P} + \polytope{Q}$ is the Minkowski sum of the faces maximizing the direction~$\ray$ on~$\polytope{P}$ and~$\polytope{Q}$.
Therefore, 
\begin{itemize}
\item the normal fan of~$\polytope{P} + \polytope{Q}$ is the common refinement of the normal fans of~$\polytope{P}$ and~$\polytope{Q}$,
\item the vertex of~$\polytope{P} + \polytope{Q}$ maximizing a generic~$\ray$ is the sum of vertices of~$\polytope{P}$ and~$\polytope{Q}$ maximizing~$\ray$,
\item the facet of~$\polytope{P} + \polytope{Q}$ maximizing a ray~$\ray$ is defined by~$\dotprod{\ray}{\b{x}} = \max\limits_{\b{p} \in \polytope{P}} \dotprod{\ray}{\b{p}} + \max\limits_{\b{q} \in \polytope{Q}} \dotprod{\ray}{\b{q}}$.
\end{itemize}

Further notions on polyhedral geometry and in particular on properties of Minkowski sums will be recalled along the text when needed, in particular weak Minkowski decompositions and type cones in \cref{subsec:typeCone}, virtual polytopes in \cref{subsec:virtualPolytopes}, and mixed volumes in \cref{subsec:matroidPolytopes}.

%%%%%%%%

\subsection{Permutations and noncrossing arc diagrams}
\label{subsec:noncrossingArcDiagrams}

We now briefly recall the bijection between permutations and noncrossing arc diagrams developed by N.~Reading in~\cite{Reading-arcDiagrams}.
The lattice theoretic interpretation of this bijection is presented in \cref{subsec:canonicalJoinRepresentations}

An \defn{arc} is a quadruple~$(a, b, A, B)$ consisting of two integers $1 \le a < b \le n$ and a partition~$A \sqcup B = {]a,b[}$.
This arc can be visually represented as an $x$-monotone continuous curve wiggling around the horizontal axis, with endpoints~$a$ and~$b$, and passing above the points of~$A$ and below the points of~$B$.
For instance, \smash{\raisebox{-.16cm}{\includegraphics[scale=.8]{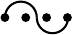}}} represents the arc~$(1,4,\{2\},\{3\})$.
Note that the same information could be recorded in a more compact way, for instance by forgetting either~${B = {]a,b[} \ssm A}$, or~${a = \min(A \sqcup B)-1}$ and~${b = \max(A \sqcup B)+1}$.
Nevertheless, it is convenient in this paper to keep the complete notation~$(a, b, A, B)$.
We denote the set of all arcs by~$\arcs_n \eqdef \set{(a, b, A, B)}{1 \le a < b \le n \text{ and } A \sqcup B = {]a,b[}}$.
Note that~$|\arcs_n| = 2^n-n-1$.

We say that two arcs~$(a, b, A, B)$ and~$(a', b', A', B')$ \defn{cross} if the interior of the two curves representing these arcs intersect in their interior, that is, if both~$(A \cap B') \cup (\{a,b\} \cap B') \cup (A \cap \{a',b'\})$ and~$(B \cap A') \cup (\{a,b\} \cap A') \cup (B \cap \{a',b'\})$ are non-empty.
A \defn{noncrossing arc diagram} is a collection of arcs of~$\arcs_n$ where any two arcs do not cross and have distinct left endpoints and distinct right endpoints (but the right endpoint of an arc can be the left endpoint of an other arc).
See \cref{fig:noncrossingArcDiagrams} for examples of noncrossing arc diagrams.

\begin{figure}
	\capstart
	\centerline{\includegraphics[scale=.8]{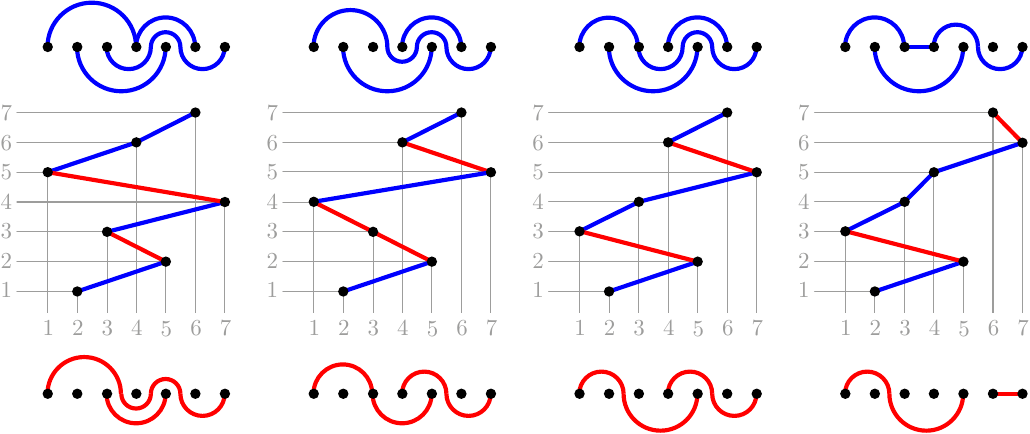}}
	\caption{The noncrossing arc diagrams~$\underline{\delta}(\sigma)$ (bottom) and~$\overline{\delta}(\sigma)$ (top) for the four permutations $\sigma = 2537146$, $2531746$, $2513746$, and $2513476$. \cite[Fig.~2]{Pilaud-arcDiagramAlgebra}}
	\label{fig:noncrossingArcDiagrams}
\end{figure}

Consider now the set~$\fS_n$ of permutations of~$[n]$.
Represent a permutation~$\sigma \in \fS_n$ by its \defn{permutation table} formed by dots at coordinates~$(\sigma_i, i)$ for~$i \in [n]$.
(This unusual choice of orientation fits with the existing constructions~\cite{LodayRonco, HivertNovelliThibon-algebraBinarySearchTrees, ChatelPilaud, PilaudPons-permutrees, Pilaud-arcDiagramAlgebra}.)
Draw a segment between any two consecutive dots~$(\sigma_i, i)$ and~$(\sigma_{i+1}, i+1)$, colored blue for an \defn{ascent}~$\sigma_i < \sigma_{i+1}$ and red for a \defn{descent}~$\sigma_i > \sigma_{i+1}$.
Then move all dots down to the horizontal axis, allowing the segments to curve, but not to cross each other nor to pass through any dot, as illustrated in \cref{fig:noncrossingArcDiagrams}.
The resulting set of blue (resp.~red) arcs is a noncrossing arc diagram~$\overline{\delta}(\sigma)$ (resp.~$\underline{\delta}(\sigma)$).
Less visually and more formally, the blue diagram is~$\overline{\delta}(\sigma) = \set{\overline{\arc}(\sigma, i)}{\sigma_i < \sigma_{i+1}}$ where~$\overline{\arc}(\sigma, i)$ is the arc~$(\sigma_i, \sigma_{i+1}, \set{\sigma_j}{j < i, \,  \sigma_i < \sigma_j < \sigma_{i+1}}, \set{\sigma_j}{j > i+1, \, \sigma_i < \sigma_j < \sigma_{i+1}})$ (resp.~the red diagram~is~$\underline{\delta}(\sigma) = \set{\underline{\arc}(\sigma, i)}{\sigma_i > \sigma_{i+1}}$ where~$\underline{\arc}(\sigma, i)$ is the arc~$(\sigma_{i+1}, \sigma_i, \set{\sigma_j}{j < i, \,  \sigma_{i+1} < \sigma_j < \sigma_i},\newline \set{\sigma_j}{j > i+1, \, \sigma_{i+1} < \sigma_j < \sigma_i})$.
These maps were introduced by N.~Reading in~\cite{Reading-arcDiagrams}, where he proved the following statement.

\begin{theorem}[{\cite[Thm.~3.1]{Reading-arcDiagrams}}]
\label{thm:bijectionNoncrossingArcDiagrams}
The map~$\underline{\delta}$ (resp.~$\overline{\delta}$) is a bijection from the permutations of~$\fS_n$ to the noncrossing arc diagrams on~$\arcs_n$.
\end{theorem}

The reverse bijections~$\underline{\delta}^{-1}$ and~$\overline{\delta}^{-1}$ are explicitly described in~\cite[Prop.~3.2]{Reading-arcDiagrams}.
Briefly speaking, consider the poset of connected components of~$\c{D}$ ordered by (the transitive closure of) the priority $X < Y$ if there is an arc~$(a, b, A, B) \in \c{D}$ with~$A \cap X \ne \varnothing$ and~$a, b \in Y$, or with~$a, b \in X$ and~$B \cap Y \ne \varnothing$.
To obtain~$\underline{\delta}{}^{-1}(\c{D})$ (resp.~$\overline{\delta}{}^{-1}(\c{D})$), choose the linear extension of this priority poset where ties are resolved by choosing first the leftmost (resp.~rightmost) connected component, and order decreasingly (resp.~increasingly) the values in each connected component.
See \cref{fig:noncrossingArcDiagrams}.

%%%%%%%%

\subsection{Weak order and canonical join and meet representations}
\label{subsec:canonicalJoinRepresentations}

Consider a finite lattice~${(L, \le, \meet, \join)}$, \ie a finite set~$L$ partially ordered by~$\le$ where each subset~$X$ of elements admits a \defn{meet}~$\bigMeet X$ (greatest lower bound) and a \defn{join}~$\bigJoin X$ (least upper bound).
A \defn{join representation} of~$x \in L$ is a subset~$J \subseteq L$ such that~${x = \bigJoin J}$.
Such a representation is \defn{irredundant} if~$x \ne \bigJoin J'$ for every strict subset~$J' \subsetneq J$.
The irredundant join representations of an element~$x \in L$ are ordered by containment of the lower ideals of their elements, \ie~$J \le J'$ if and only if for any~$y \in J$ there exists~$y' \in J'$ such that~$y \le y'$ in~$L$.
When this order has a minimal element, it is called the \defn{canonical join representation} of~$x$.
All elements of the canonical join representation~$x = \bigJoin J$ are then \defn{join-irreducible}, \ie cover a single element.
A lattice is \defn{join-semidistributive} when every element has a canonical join representation.
Equivalently~\cite[Thm.~2.24]{FreeseNation}, ${x \join z = y \join z \Longrightarrow x \join z = (x \meet y) \join z}$ for any~${x, y, z \in L}$.
\defn{Canonical meet representations}, \defn{meet-irreducible elements} and \defn{meet-semidistributive lattices} are defined dually.
A lattice is \defn{semidistributive} if it is both join- and meet-semidistributive.

We consider the \defn{weak order} on permutations of~$\fS_n$ defined by ${\sigma \le \tau \iff \inv(\sigma) \subseteq \inv(\tau)}$ where~$\inv(\sigma) \eqdef \set{(\sigma_a, \sigma_b)}{1 \le a < b \le n \text{ and } \sigma_a > \sigma_b}$ is the \defn{inversion set} of the permutation~$\sigma$.
See \cref{fig:sylvesterCongruence}\,(left) for the Hasse diagram of the weak order on~$\fS_4$.
A cover relation in the weak order corresponds to a swap of two letters at consecutive positions.
The permutations covered by (resp.~covering) a permutation~$\sigma \in \fS_n$ in the weak order correspond to the descents (resp.~ascents) of~$\sigma$.
Hence, a permutation~$\sigma \in \fS_n$ is join-irreducible (resp.~meet-irreducible) in the weak order if and only if it has a unique descent (resp.~ascent).

The weak order on~$\fS_n$ is a semidistributive lattice, and the canonical join and meet representations of a permutation were described in~\cite{Reading-arcDiagrams} as follows.
Consider an arc~$\arc \eqdef (a, b, A, B) \in \arcs_n$, where~${A \eqdef \{a_1 < \dots < a_p\}}$ and~${B \eqdef \{b_1 < \dots < b_q\}}$.
We associate to the arc~$\arc$ the join-irreducible permutation $\underline{\lambda}(\arc) \eqdef [1, \dots, a-1, a_1, \dots, a_p, b, a, b_1, \dots, b_q, b+1, \dots, n]$ (resp.~the meet-irreducible permutation $\overline{\lambda}(\arc) \eqdef [n, \dots, b+1, b_q, \dots, b_1, a, b, a_p, \dots, a_1, a-1, \dots, 1]$).
The canonical meet and join representations of~$\sigma$ are then given by their red and blue noncrossing arc diagrams, which provides a lattice theoretic interpretation of the bijections of \cref{thm:bijectionNoncrossingArcDiagrams}.

\begin{theorem}[{\cite[Thm.~2.4]{Reading-arcDiagrams}}]
\label{thm:joinMeetRepresentationsPermutations}
The canonical join and meet representations of a permutation~$\sigma$ are given by~$\bigJoin \set{\underline{\lambda}(\underline{\arc})}{\underline{\arc} \in \underline{\delta}(\sigma)}$ and~$\bigMeet \set{\overline{\lambda}(\overline{\arc})}{\overline{\arc} \in \overline{\delta}(\sigma)}$.
\end{theorem}

%%%%%%%%

\subsection{Lattice quotients}
\label{subsec:latticeQuotients}

We now consider lattice congruences and lattice quotients of the weak order on~$\fS_n$.
For details, we refer to the thorough work of N.~Reading, in particular the articles~\cite{Reading-latticeCongruences, Reading-HopfAlgebras, Reading-CambrianLattices, Reading-arcDiagrams} and the surveys~\cite{Reading-survey, Reading-PosetRegionsChapter, Reading-FiniteCoxeterGroupsChapter}.

A \defn{lattice congruence} of a lattice~$(L,\le,\meet,\join)$ is an equivalence relation on~$L$ that respects the meet and the join operations, \ie such that $x \equiv x'$ and~$y \equiv y'$ implies $x \meet y \, \equiv \, x' \meet y'$ and~$x \join y \, \equiv \, x' \join y'$.
Equivalently, the equivalence classes of~$\equiv$ are intervals of~$L$, and the up and down maps~$\projUp$ and~$\projDown$, respectively sending an element of~$L$ to the top and bottom elements of its $\equiv$-equivalence class, are order-preserving.
A lattice congruence~$\equiv$ defines a \defn{lattice quotient}~$L/{\equiv}$ on the congruence classes of~$\equiv$ where~$X \le Y$ if and only if there exist~$x \in X$ and~$y \in Y$ such that~$x \le y$, and~$X \meet Y$ (resp.~$X \join Y$) is the congruence class of~$x \meet y$ (resp.~$x \join y$) for any~$x \in X$~and~$y \in Y$.
Intuitively, the quotient~$L/{\equiv}$ is obtained by contracting the equivalence classes of~$\equiv$ in the lattice~$L$.
More precisely, we say that an element~$x$ is \defn{contracted} by~$\equiv$ if it is not minimal in its equivalence class of~$\equiv$, \ie if~$x \ne \projDown(x)$.
As each class of~$\equiv$ is an interval of~$L$, it contains a unique uncontracted element, and the quotient~$L/{\equiv}$ is isomorphic to the subposet of~$L$ induced by its uncontracted elements~$\projDown(L)$.

\begin{example}[Tamari]
\label{exm:sylvesterCongruence}
The prototype lattice congruence of the weak order on~$\fS_n$ is the \defn{sylvester congruence}~$\equiv_\textrm{sylv}$ \cite{LodayRonco, HivertNovelliThibon-algebraBinarySearchTrees}.
Its congruence classes are the fibers of the binary search tree insertion algorithm, or equivalently the sets of linear extensions of binary trees (labeled in inorder and considered as posets oriented from bottom to top).
It can also be seen as the transitive closure of the rewriting rule~$U i k V j W \equiv_\textrm{sylv} U k i V j W$ where~$i < j < k$ are letters and~$U,V,W$ are words on~$[n]$.
In other words, the uncontracted permutations in the sylvester congruence are those avoiding the pattern~$312$.
The quotient of the weak order by the sylvester congruence is (isomorphic to) the classical \defn{Tamari lattice}~\cite{Tamari}, whose elements are the binary trees on~$n$ nodes and whose cover relations are rotations in binary trees.
The sylvester congruence and the Tamari lattice are illustrated in \cref{fig:sylvesterCongruence} for~$n = 4$.
We will find the sylvester congruence and the Tamari lattice again in \cref{exm:sylvesterCongruence,exm:noncrossingPartitions,exm:LodayAsso,exm:shardsAsso,exm:LodayAssoMinkowskiSum,exm:LodayShardPolytope}.

\begin{figure}
	\capstart
	\centerline{\includegraphics[scale=.6]{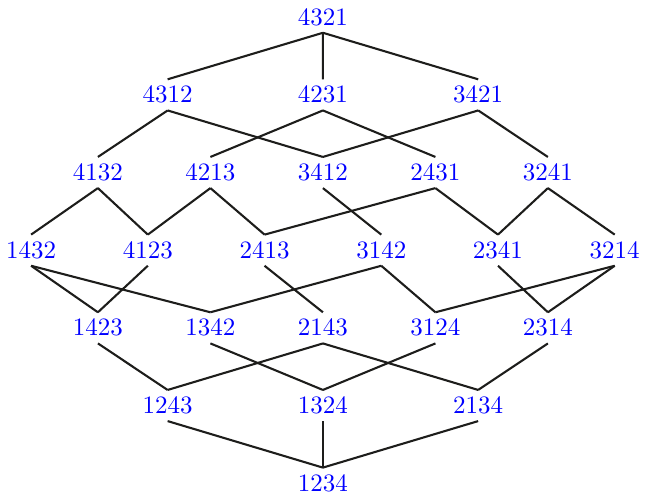} \; \includegraphics[scale=.6]{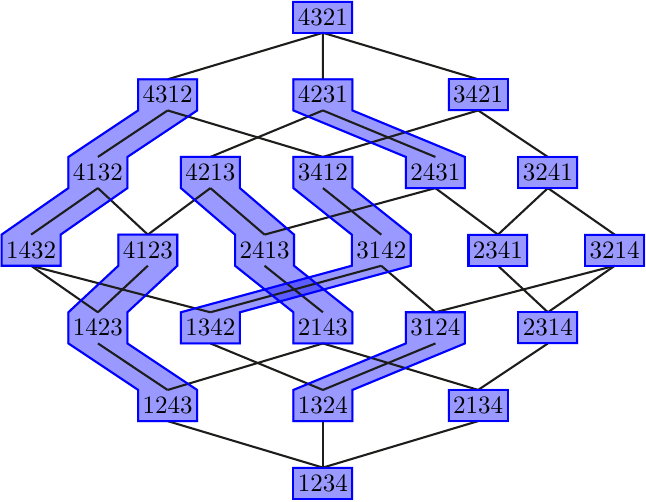} \; \includegraphics[scale=.48]{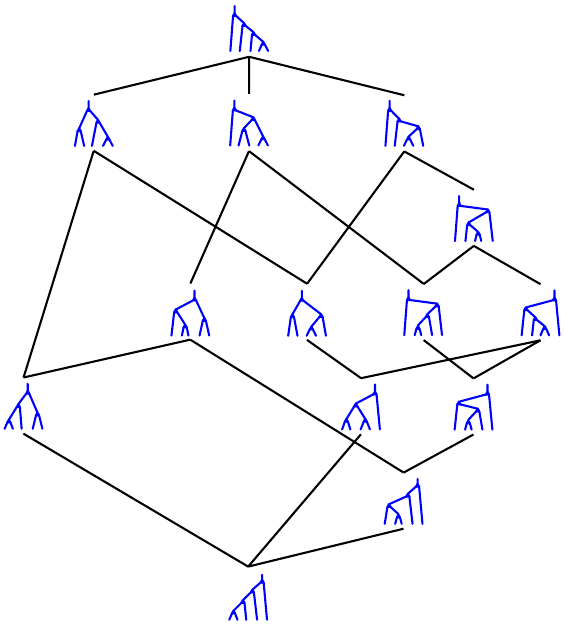}}
	\caption{The weak order on~$\fS_4$ (left), the sylvester congruence~$\equiv_\textrm{sylv}$~(middle), and the Tamari lattice (right). \cite[Fig.~1 \& 2]{PilaudSantos-quotientopes}}
	\label{fig:sylvesterCongruence}
\end{figure}
\end{example}

If a lattice~$L$ is semidistributive, then any lattice quotient~$L/{\equiv}$ is also semidistributive.
Moreover, via the identification between $\equiv$-classes and their minimal elements, the canonical join representations in the quotient~$L/{\equiv} \simeq \projDown(L)$ are precisely the canonical join representations of~$L$ that only involve join-irreducibles of~$L$ uncontracted by~$\equiv$.
We have seen in~\cref{subsec:canonicalJoinRepresentations} that the weak order on~$\fS_n$ is semidistributive, that its join-irreducibles correspond to arcs of~$\arcs_n$, and that the canonical join representations of permutations correspond to noncrossing arc diagrams.
This yields the following statement.

\begin{theorem}[{\cite[Thm.~4.1]{Reading-arcDiagrams}}]
\label{thm:joinMeetRepresentationsQuotient}
For any lattice congruence~$\equiv$ of the weak order on~$\fS_n$, the set of join-irreducibles of~$\fS_n$ uncontracted by~$\equiv$ corresponds to a set of arcs~$\arcs_\equiv$, and the canonical join representations in the lattice quotient~$\fS_n/{\equiv}$ correspond to noncrossing arc diagrams using only arcs of~$\arcs_\equiv$.
\end{theorem}

\begin{example}[Tamari]
\label{exm:noncrossingPartitions}
For the sylvester congruence~$\equiv_\textrm{sylv}$ of \cref{exm:sylvesterCongruence}, the uncontracted join-irreducibles are given by the set~$\arcs_\textrm{sylv} = \set{(a, b, {]a,b[}, \varnothing)}{1 \le a < b \le n}$ of up arcs, \ie those which pass above all dots in between their endpoints.
Therefore, the sylvester congruence classes are in bijection with noncrossing arc diagrams with arcs in~$\arcs_\textrm{sylv}$, also known as noncrossing partitions.
\end{example}

The set of all lattice congruences of a lattice~$L$ ordered by refinement is a lattice whose meet is the intersection of congruences and join is the transitive closure of union of congruences.
In particular, for any join-irreducible element~$j$ of~$L$, there is a unique minimal lattice congruence~$\equiv_j$ contracting~$j$, and any lattice congruence~$\equiv$ of~$L$ is the join~$\bigJoin_j {\equiv_j}$ over all join-irreducible elements~$j$ of~$L$ contracted by~$\equiv$.
For two join-irreducible elements~$j, j'$ of~$L$, we say that~$j$ \defn{forces} $j'$, and write~$j \succ j'$, if every congruence that contracts~$j$ also contracts~$j'$.
The forcing relation is a preposet and thus defines a poset on its equivalence classes, called the \defn{forcing poset} of~$L$.
It follows that the lattice of congruences of~$L$ is isomorphic to the lattice of upper ideals on the forcing poset of~$L$.
Let us finally mention that a lattice is called \defn{congruence uniform} when the forcing relation is a poset, so that the map~$j \mapsto {\equiv_j}$ is a bijection between the join-irreducible elements of~$L$ and that of the lattice of congruences of~$L$.

The weak order on~$\fS_n$ is a congruence uniform lattice, and the forcing order on join-irreducibles can be described visually on arcs as follows.
We say that an arc~$\arc \eqdef (a, b, A, B) \in \arcs_n$ \defn{forces} an arc~$\arc' \eqdef (a', b', A', B') \in \arcs_n$, and we write~$\arc \succ \arc'$, if~$a' \le a < b \le b'$ and~${A \subseteq A'}$ and~${B \subseteq B'}$.
Visually, $\arc$ forces~$\arc'$ if the endpoints of~$\arc$ are located in between those of~$\arc'$ and~$\arc$ agrees with~$\arc'$ in between its endpoints.
The \defn{arc poset} is the poset~$(\arcs_n, \prec)$ of all arcs ordered by inverse forcing (small elements are forced by big elements in this poset).
The forcing relation and the arc poset on~$\arcs_4$ are illustrated in \cref{fig:forcingOrder}.
We thus obtain the following description of the lattice congruences of the weak order on~$\fS_n$.

\begin{figure}
	\capstart
	\centerline{\includegraphics[scale=.7,valign=c]{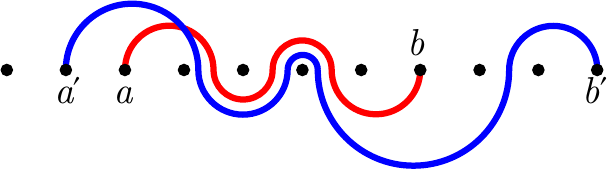} \hspace{1cm} \includegraphics[scale=.6,valign=c]{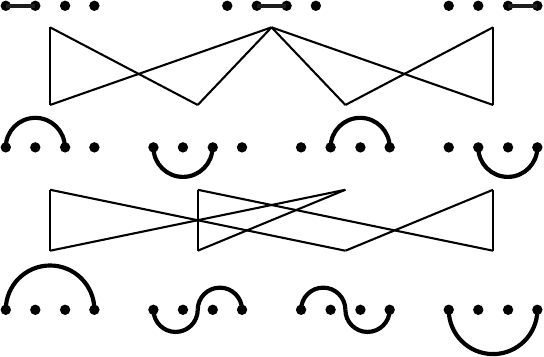}}
	\caption{The forcing relation among arcs (left) and the arc poset for~$n = 4$ (right). The red arc~$(a,b,A,B)$ forces the blue arc~$(a',b',A',B')$. \mbox{\cite[Fig.~5]{PilaudSantos-quotientopes}}}
	\label{fig:forcingOrder}
\end{figure}

\begin{theorem}[{\cite[Thm.~4.4 \& Coro.~4.5]{Reading-arcDiagrams}}]
\label{thm:arcIdeals}
The map~${\equiv} \mapsto \arcs_\equiv$ is a bijection between the lattice congruences of the weak order on~$\fS_n$ and the upper ideals of the arc poset~$(\arcs_n, \prec)$.
\end{theorem}

An \defn{arc ideal} is an upper ideal of the arc poset~$(\arcs_n, \prec)$.
In view of this statement, we make no distinction between arc ideals and lattice congruences.
When needed, we write~$\equiv_\arcs$ for the lattice congruence of the weak order on~$\fS_n$ corresponding to an arc ideal~$\arcs \subseteq \arcs_n$.

\begin{example}[Cambrian]
\label{exm:CambrianCongruences}
For an arc~$\arc \eqdef (a, b, A, B) \in \arcs_n$, we denote by~$\arcs_\arc \eqdef \set{\arc' \in \arcs_n}{\arc \prec \arc'}$ the upper ideal of~${(\arcs_n, \prec)}$ generated by~$\arc$.
The corresponding lattice congruence~$\equiv_\arc$ of the weak order on~$\fS_n$ is called the \defn{$\arc$-Cambrian congruence}, and the lattice quotient~$\fS_n/{\equiv_\arc}$ is the \defn{$\arc$-Cambrian lattice}.
It was introduced and extensively studied by N.~Reading in~\cite{Reading-CambrianLattices}.
For instance, the sylvester congruence of \cref{exm:sylvesterCongruence} is the $(1, n, {]1,n[}, \varnothing)$-Cambrian congruence, and the Tamari lattice is the $(1, n, {]1,n[}, \varnothing)$-Cambrian lattice.
The $\arc$-Cambrian congruence classes are fibers of the $\arc$-Cambrian tree insertion, or equivalently linear extensions of $\arc$-Cambrian trees, see~\cite{LangePilaud, ChatelPilaud, PilaudPons-permutrees}.
Let us just say that an \mbox{$\arc$-Cambrian} tree is a tree on~$[a,b]$ such that the node~$j \in \{a\} \cup A$ (resp.~$j \in B \cup \{b\}$) has one ancestor (resp.~descendant) subtree and two descendant (resp.~ancestor) subtrees, and~$i < j < k$ for any nodes $i$ in the left descendant (resp.~ancestor) subtree of~$j$ and~$k$ in the right descendant (resp.~ancestor) subtree of~$j$.
The $\arc$-Cambrian congruence can also be seen as the transitive closure of the three rewriting rules~${U i j V \equiv_\arc U j i V}$ for~$i < a$ or~$j > b$, $U i k V j W \equiv_\arc U k i V j W$ for~$i < j < k$ with~$j \in A$, and~$U j V i k W \equiv_\arc U j V k i W$ for~$i < j < k$ with~$j \in B$.
In particular, the uncontracted permutations in the $\arc$-Cambrian congruence are those avoiding the consecutive patterns~$ji$ with~$i < j$ and~$i < a$ or~$j > b$, and the patterns~$kij$ for~$i < j < k$ with~$j \in A$ and~$jki$ for~$i < j < k$ with~$j \in B$.
We will find the Cambrian congruences in \cref{exm:CambrianCongruences,exm:HohlwegLangeAsso,exm:HohlwegLangeAssoMinkowskiSum,exm:HohlwegLangeAssoVertexFacetDescription} and use them as a fundamental tool in \cref{subsec:MinkowskiSumAssociahedra}.
\end{example}

\begin{example}[Other relevant congruences]
\label{exm:otherCongruences}
\enlargethispage{.2cm}
Let us gather some other relevant examples of lattice congruences of the weak order on~$\fS_n$, as some of them will appear along this paper:
\begin{enumerate}
\item the \defn{recoil congruence}~$\equiv_\textrm{rec}$ is defined by the ideal~$\arcs_\textrm{rec} = \set{(i, i+1, \varnothing, \varnothing)}{i \in [n-1]}$ of basic arcs. It has a congruence class for each subset~$I \subseteq [n-1]$ given by the permutations whose recoils (descents of the inverse) are at positions in~$I$. It can also be seen as the transitive closure of the rewriting rule~$U i j V \equiv_{\textrm{rec}} U j i V$ for~$|i - j| > 1$. The quotient~$\fS_n/{\equiv_\textrm{rec}}$ is the boolean lattice.
\item for~$\decoration \in \Decorations^n$, the \defn{$\decoration$-permutree congruence}~$\equiv_\decoration$ is defined by the ideal~$\arcs_\decoration$ of arcs which do not pass above the points~$j$ with~$\decoration_j \in \{\upCirc, \upDownCirc\}$ nor below the points~$j$ with~$\decoration_j \in \{\downCirc, \upDownCirc\}$. Its congruence classes correspond to $\decoration$-permutrees~\cite{PilaudPons-permutrees}. It can also be seen as the transitive closure of the rewriting rules~$U i k V j W \equiv_\decoration U k i V j W$ for~$i < j < k$ with~$\decoration_j \in \{\downCirc, \upDownCirc\}$ and~$U j V i k W \equiv_\decoration U j V k i W$ for~$i < j < k$ with~$\decoration_j \in \{\upCirc, \upDownCirc\}$.
\label{item:permutreeCongruence}
\item the \defn{Baxter congruence}~$\equiv_\textrm{Bax}$ is defined by the ideal of arcs that do not cross the horizontal axis, \ie~$\arcs_\textrm{Bax} = \set{(a, b, A, B) \in \arcs_n}{A = \varnothing \text{ or } B = \varnothing}$. Its congruence classes correspond to diagonal rectangulations~\cite{LawReading} or equivalently pairs of twin binary trees~\cite{Giraudo}, which are counted by the Baxter numbers. It can also be seen as the transitive closure of the rewriting rule~$U j V i \ell W k X \equiv_{\textrm{Bax}} U j V \ell i W k X$ for~$i < j, k < \ell$.
\label{item:BaxterCongruence}
\item for~$p \ge 1$, the \defn{$p$-recoil congruence}~$\equiv_{p\textrm{-rec}}$ is defined by the ideal of arcs of length at most~$p$, \ie~$\arcs_{p\textrm{-rec}} = \set{(a, b, A, B) \in \arcs_n}{b-a \le p}$. Its congruence classes correspond to acyclic orientations of the graph on~$[n]$ with edges~$(a,b)$ for~$|a-b| \le p$. It can also be seen as the transitive closure of the rewriting rule~$U i j V \equiv_{p\textrm{-rec}} U j i V$ for~$|i - j| > p$. See \cite{Reading-HopfAlgebras, Pilaud-brickAlgebra}.
\item for~$p \ge 1$, the \defn{$p$-twist congruence}~$\equiv_{p\textrm{-twist}}$ is defined by the ideal of arcs passing below at most~$p$ points, \ie~$\arcs_{p\textrm{-twist}} = \set{(a, b, A, B) \in \arcs_n}{|B| \le p}$. Its congruence classes correspond to certain acyclic pipe dreams~\cite{Pilaud-brickAlgebra}. It can also be seen as the transitive closure of the rewriting rule~$U i k V_1 j_1 \dots V_p j_p W \equiv_{p\textrm{-twist}} U k i V_1 j_1 \dots V_p j_p W$ for~$i < j_1, \dots, j_p < k$.
\end{enumerate}
\end{example}

\begin{remark}
\label{rem:regularQuotients}
The Hasse diagram of a lattice quotient~$\fS_n/{\equiv}$ is not always regular (\ie of constant degree).
H.~Hoang and T.~M\"utze proved in~\cite{HoangMutze} that it is regular if and only if all maximal arcs of~$\arcs_n \ssm \arcs_\equiv$ are of the form~$(a, b, {]a,b[}, \varnothing)$ or~$(a, b, \varnothing, {]a,b[})$.
This holds for instance for all $\arc$-Cambrian congruences of~\cite{Reading-CambrianLattices} and more generally for all permutree congruences of~\cite{PilaudPons-permutrees}
\end{remark}

%%%%%%%%

\subsection{Braid fan and permutahedron}
\label{subsec:braidFanPermutahedron}

We now switch to some geometric considerations on~$\fS_n$.

The \defn{braid arrangement} is the set~$\HA_n$ of hyperplanes~$\set{\b{x} \in \R^n}{\b{x}_a = \b{x}_b}$ for ${1 \le a < b \le n}$.
As all hyperplanes of~$\HA_n$ contain the line~$\R \one \eqdef \R(1,1,\dots,1)$, we restrict to the hyperplane ${\hyp \eqdef \bigset{\b{x} \in \R^n}{\sum_{i \in [n]} \b{x}_i = 0}}$.
The hyperplanes of~$\HA_n$ divide~$\hyp$ into chambers, which are the maximal cones of a complete simplicial fan~$\Fan_n$, called the \defn{braid fan}.
It has 
\begin{itemize}
\item a chamber~$\polytope{C}(\sigma) \eqdef \set{\b{x} \in \hyp}{\b{x}_{\sigma_1} \le \b{x}_{\sigma_2} \le \dots \le \b{x}_{\sigma_n} }$ for each permutation~$\sigma$ of~$\fS_n$, 
\item a ray~$\polytope{C}(R) \eqdef \set{\b{x} \in \hyp}{\b{x}_{r_1} = \dots = \b{x}_{r_p} \le \b{x}_{s_1} = \dots = \b{x}_{s_{n-p}}}$ for each subset~${\varnothing \ne R \subsetneq [n]}$, where~$R = \{r_1, \dots, r_p\}$ and~$[n] \ssm R = \{s_1, \dots, s_{n-p}\}$. When needed, we use the representative vector~$\ray(R) \eqdef |R| \one - n \one_R$ in~$\polytope{C}(R)$, where~$\one \eqdef \sum_{i \in [n]} \b{e}_i$ and~$\one_R \eqdef \sum_{r \in R} \b{e}_r$.
\end{itemize}
The chamber~$\polytope{C}(\sigma)$ has rays~$\polytope{C}(\sigma([k]))$ for~$k \in [n]$.
Figures~\ref{fig:weakOrder4}\,(middle), \ref{fig:shards3}\,(left) and \ref{fig:shards4}\,(left) illustrate the braid fans~$\Fan_n$ for~$n = 3$ and~$n = 4$ (some chambers are labeled in blue and some rays are labeled in red).
Note that~$\Fan_n$ has~$n!$ chambers, $n!(n-1)/2$ walls supported by~$\binom{n}{2}$ hyperplanes, and~$2^n-2$ rays.

The \defn{permutahedron} is the polytope~$\Perm$ defined equivalently as
\begin{itemize}
\item the convex hull of the points~$\sum_{i \in [n]} i \, \b{e}_{\sigma_i}$ for all permutations~$\sigma \in \fS_n$,
\item the intersection of the hyperplane~$\Hyp \eqdef \bigset{\b{x} \in \R^n}{\sum_{i \in [n]} \b{x}_i = \binom{n+1}{2}}$ with the halfspaces $\bigset{\b{x} \in \R^n}{\sum_{r \in R} \b{x}_r \ge \binom{|R|+1}{2}}$ for all proper subsets~${\varnothing \ne R \subsetneq [n]}$,
\item (a translate of) the Minkowski sum of all segments~$[\b{e}_a, \b{e}_b]$ for all~$1 \le a < b \le n$.
\end{itemize}
\cref{fig:weakOrder4}\,(right) shows the permutahedron~$\Perm[4]$.
Note that~$\Perm$ has~$n!$ vertices, $n!(n-1)/2$ edges, and $2^n-2$ facets.
The normal fan of the permutahedron~$\Perm$ is the braid fan~$\Fan_n$.

The Hasse diagram of the weak order on~$\fS_n$ can be seen geometrically as the dual graph of the braid fan~$\Fan_n$, or as the graph of the permutahedron~$\Perm$, oriented in the linear direction~$\b{\gamma} \eqdef \sum_{i \in [n]} (2i-n-1) \, \b{e}_i = (-n+1, -n+3, \dots, n-3, n-1)$.

\begin{figure}
	\capstart
	\centerline{\includegraphics[scale=.6]{weakOrderLeft4} \; \includegraphics[scale=.6]{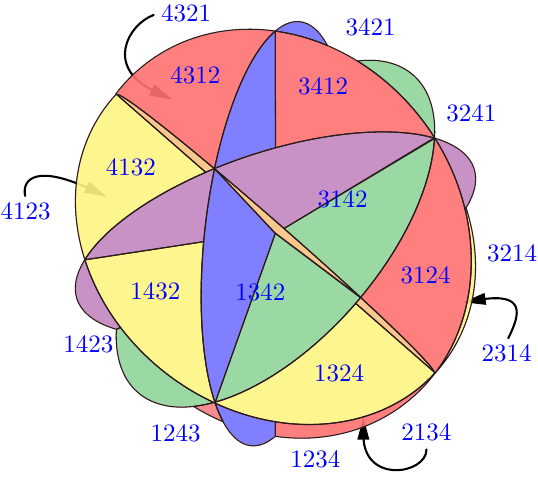} \; \includegraphics[scale=.6]{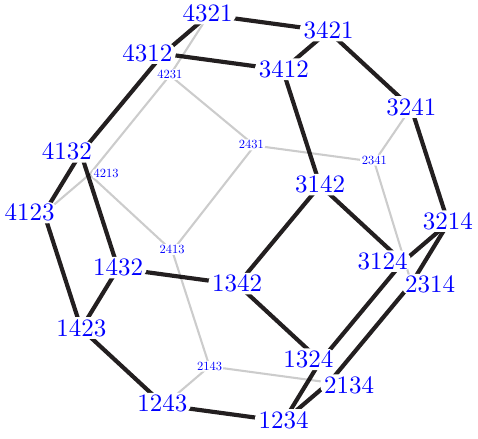}}
	\caption{The Hasse diagram of the weak order on~$\fS_4$ (left) can be seen as the dual graph of the braid fan~$\Fan_4$ (middle) or as the graph of the permutahedron~$\Perm[4]$ (right). \cite[Fig.~1]{PilaudSantos-quotientopes}}
	\label{fig:weakOrder4}
\end{figure}

%%%%%%%%

\subsection{Quotient fans and quotientopes}
\label{subsec:quotientFanQuotientopes}

\enlargethispage{.4cm}
We now consider the geometry of lattice quotients of the weak order on~$\fS_n$.
First, lattice congruences naturally yield quotient fans described in the following statement.
Although stated in the more general context of hyperplane arrangements in~\cite{Reading-HopfAlgebras} (see also~\cite{Reading-PosetRegionsChapter}), we restrict to a simple version for the braid arrangement.

\begin{theorem}[\cite{Reading-HopfAlgebras}]
\label{thm:quotientFanGluing}
Any lattice congruence~$\equiv$ of the weak order on~$\fS_n$ defines a complete fan~$\fan_\equiv$, called a \defn{quotient fan}, whose chambers are obtained by gluing together the chambers~$\polytope{C}(\sigma)$ of the braid fan~$\Fan_n$ corresponding to the permutations~$\sigma$ that belong to the same congruence class~of~$\equiv$.
\end{theorem}

As it turns out, the quotient fans of all lattice congruences of the weak order on~$\fS_n$ are polytopal.

\begin{theorem}[\cite{PilaudSantos-quotientopes}]
\label{thm:quotientopes}
For any lattice congruence~$\equiv$ of the weak order on~$\fS_n$, the quotient fan~$\fan_\equiv$ is the normal fan of a polytope~$\polytope{P}_\equiv$, called a \defn{quotientope}.
\end{theorem}

By construction, the Hasse diagram of the quotient of the weak order by~${\equiv}$ is given by the dual graph of the quotient fan~$\fan_\equiv$, or by the graph of the quotientope~$\polytope{P}_\equiv$, oriented in the direction~$\b{\gamma}$.
In this paper, we call a \defn{quotientope} any polytope whose normal fan is the quotient fan~$\fan_\equiv$, and \defn{PS-quotientopes} the quotientopes constructed in~\cite{PilaudSantos-quotientopes}.

\begin{example}[Tamari]
\label{exm:LodayAsso}
For the sylvester congruence~$\equiv_\textrm{sylv}$ of \cref{exm:sylvesterCongruence,exm:noncrossingPartitions}, the quotient fan~$\fan_\textrm{sylv}$~has
\begin{itemize}
\item a chamber $\polytope{C}(T) = \set{\b{x} \in \hyp}{\b{x}_a \le \b{x}_b \text{ if } a \text{ is a descendant of } b \text{ in } T}$ for each binary tree~$T$, 
\item a ray~$\polytope{C}(I)$ for each proper interval~$I = [i,j] \subsetneq [n]$.
\end{itemize}
Figures~\ref{fig:Tamari4}\,(middle), \ref{fig:shards3}\,(right) and~\ref{fig:shards4}\,(right) illustrate the quotient fans~$\Fan_\textrm{sylv}$ for~$n = 3$ and~$n = 4$.
The quotient fan~$\fan_\textrm{sylv}$ is the normal fan of the classical associahedron~$\Asso$ defined equivalently~as:
\begin{itemize}
\item the convex hull of the points~$\sum_{j \in [n]} \ell(T,j) \, r(T,j) \, \b{e}_j$ for all binary trees~$T$ on~$n$ nodes, where $\ell(T,j)$ and~$r(T,j)$ respectively denote the numbers of leaves in the left and right subtrees of the node~$j$ of~$T$ (labeled in inorder), see~\cite{Loday},
\item the intersection of the hyperplane~$\Hyp$ with the halfspaces~$\bigset{\b{x} \in \R^n}{\sum_{a \le i \le b} \b{x}_i \ge \binom{b-a+2}{2}}$ for all intervals~$1 \le a \le b \le n$, see~\cite{ShniderSternberg},
\item (a translate of) the Minkowski sum of~$\simplex_{[a,b]}$ for all intervals ${1 \le a \le b \le n}$, where for~${I \subseteq [n]}$, $\simplex_I \eqdef \conv\set{\b{e}_i}{i \in I}$ is the face of the standard simplex~$\simplex_{[n]}$ labeled by~$I$, see~\cite{Postnikov}.
\end{itemize}
\cref{fig:Tamari4}\,(right) shows the associahedron~$\Asso[4]$.

\begin{figure}
	\capstart
	\centerline{\includegraphics[scale=.48]{TamariLattice4} \; \includegraphics[scale=.6]{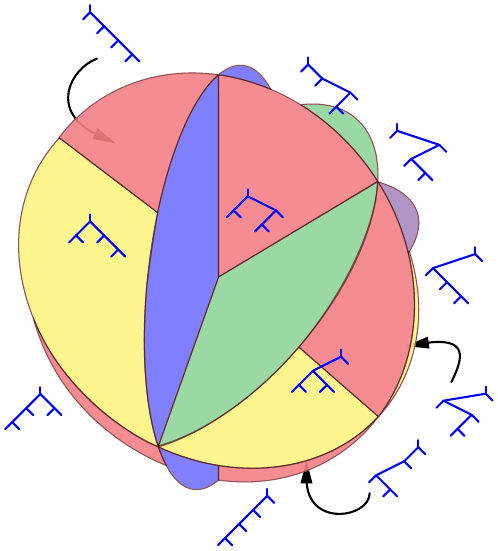} \hspace{-.3cm} \includegraphics[scale=.6]{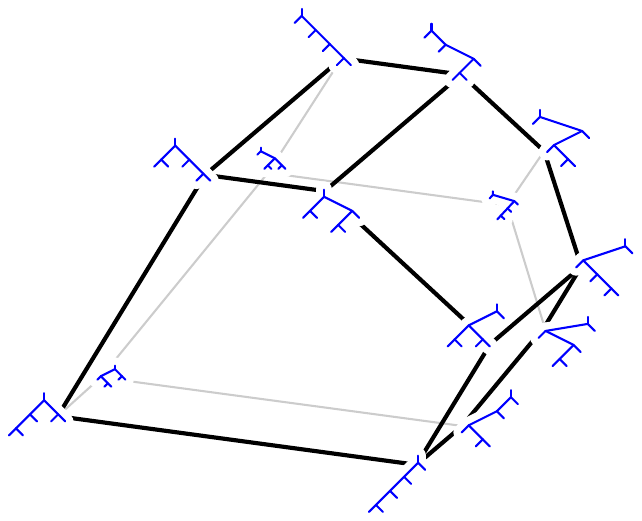}}
	\caption{The Tamari lattice (left) can be seen as the dual graph of the quotient fan~$\fan_\textrm{sylv}$ (middle) or as the graph of J.-L.~Loday's associahedron (right).}
	\label{fig:Tamari4}
\end{figure}
\end{example}

\begin{example}[Cambrian]
\label{exm:HohlwegLangeAsso}
Consider the $\arc$-Cambrian congruence of an arc~$\arc \eqdef (a, b, A, B)$ defined in \cref{exm:CambrianCongruences}.
The quotient fan is the \defn{$\arc$-Cambrian fan}~$\Fan_\arc$.
Its lineality space is generated by~$(\b{e}_i)_{i \notin [a,b]}$, and its section by~$\R^{[a,b]}$~has
\begin{itemize}
\item a chamber~$\polytope{C}(T) = \set{\b{x} \in \hyp}{\b{x}_a \le \b{x}_b \text{ if } a \text{ is a descendant of } b \text{ in } T}$ for each $\arc$-Cambrian tree~$T$ (see~\cite{LangePilaud, ChatelPilaud, PilaudPons-permutrees} or the brief description in \cref{exm:CambrianCongruences}),
\item a ray~$\polytope{C}(R)$ for each proper subset~$\varnothing \ne R \subsetneq [a,b]$ such that for all~$a \le i < j < k \le b$, if~$i,k \in R$ then~$j \in R \cup B$, and if~$i,k \notin R$ then~$j \notin R \cap B$.
\end{itemize}
The quotient fan~$\Fan_\arc$ is the normal fan of C.~Hohlweg and C.~Lange's \defn{$\arc$-associahedron}~$\Asso[\arc]$ \cite{HohlwegLange} defined equivalently as:
\begin{itemize}
\item the convex hull of the points~$\sum_{j \in [a,b]} \HL(T,j) \, \b{e}_j$ for all $\arc$-Cambrian trees~$T$, where ${\HL(T,j) = \ell(T,j) \, r(T,j)}$ (resp.~$\HL(T,j) = b - a + 2 - \ell(T,j) \, r(T,j)$), where~$\ell(T,j)$ and~$r(T,j)$ respectively denote the number of leaves in the left and right descendant (resp.~ancestor) subtrees of the node~$j$ of~$T$ if~$j \in \{a\} \cup A$ (resp.~if~$j \in B \cup \{b\}$),
\item the intersecton of the hyperplane~$\bigset{\b{x} \in \R^n}{\sum_{i \in [n]} x_i = \binom{b-a+2}{2}}$ with the halfspaces $\bigset{\b{x} \in \R^n}{\sum_{r \in R} \b{x}_r \ge \binom{|R|+1}{2}}$ for all rays~$R$ described above.
\end{itemize}
\cref{fig:shardPolytopeSums3,fig:associahedra} show the $2$- and $3$-dimensional associahedra~$\Asso[\arc]$.
\end{example}

\begin{example}[Other relevant congruences]
\label{exm:otherQuotientopes}
The quotient fans of the congruences of \cref{exm:otherCongruences} are realized by:
\begin{enumerate}
\item the parallelotope~$\sum_{i \in [n-1]} [\b{e}_i, \b{e}_{i+1}]$ for the recoil congruence,
\item the $\decoration$-permutreehedron for the $\decoration$-permutree congruence~\cite{PilaudPons-permutrees},
\item the Minkowski sum of~$\Asso$ and~$-\Asso$ for the Baxter congruence~\cite{LawReading},
\item the graphical zonotope~$\sum_{|a-b| \le p} [\b{e}_a, \b{e}_p]$ for the $p$-recoil congruence~\cite{Pilaud-brickAlgebra},
\item the brick polytope for the $p$-twist congruence~\cite{PilaudSantos-brickPolytope, Pilaud-brickAlgebra}.
\end{enumerate}
\end{example}

For an arc ideal~$\arcs \subseteq \arcs_n$, we denote by~$\Fan_\arcs$ the quotient fan~$\fan_{\equiv_\arcs}$ of the corresponding lattice congruence~$\equiv_\arcs$ via the bijection of \cref{thm:arcIdeals}.
We will use the following characterization of the rays of the quotient fan~$\Fan_\arcs$, proved for instance in~\cite[Sect.~3.1]{AlbertinPilaudRitter}.
Note that it fits with the descriptions of the rays of the quotient fans of the sylvester and Cambrian congruences given in \cref{exm:LodayAsso,exm:HohlwegLangeAsso}.

\begin{lemma}[{\cite[Sect.~3.1]{AlbertinPilaudRitter}}]
\label{lem:raysQuotientFan}
For any arc ideal~$\arcs \subseteq \arcs_n$ and any proper subset~${\varnothing \ne R \subsetneq [n]}$, the ray~$\polytope{C}(R)$ of the braid fan~$\Fan_n$ is also a ray of the quotient fan~$\Fan_\arcs$ if and only for every ${1 \le a < b \le n}$, we have~$(a, b, \varnothing, {]a,b[}) \in \arcs$ if $a, b \in R$ and~${]a,b[} \cap R = \varnothing$, and $(a, b, {]a,b[}, \varnothing) \in \arcs$ if~$a, b \notin R$ and~${]a,b[} \subseteq R$.
\end{lemma}

%%%%%%%%

\subsection{Shards}
\label{subsec:shards}

An alternative description of the quotient fan~$\Fan_\equiv$ defined in \cref{thm:quotientFanGluing} is given by its walls, each of which can be seen as the union of some preserved walls of the braid arrangement.
The conditions in the definition of lattice congruences impose strong constraints on the set of preserved walls.
Shards were introduced by N.~Reading in~\cite{Reading-posetRegions} (see also~\cite{Reading-PosetRegionsChapter, Reading-FiniteCoxeterGroupsChapter}) to understand the possible sets of preserved walls.

For any arc~$\arc \eqdef (a, b, A, B)$, the \defn{shard}~$\shard(\arc) = \shard(a, b, A, B)$ is the cone
\[
\shard(a, b, A, B) \eqdef \set{\b{x} \in \hyp}{\b{x}_a = \b{x}_b, \; \b{x}_a \ge \b{x}_{a'} \text{ for all } a' \in A, \; \b{x}_a \le \b{x}_{b'} \text{ for all } b' \in B}.
\]
We denote by~$\shards_n \eqdef \set{\shard(\arc)}{\arc \in \arcs_n}$ the set of all shards of~$\HA_n$.
Note that~$|\shards_n| = |\arcs_n| = 2^n-n-1$ is the number~$2^n-2$ of rays of~$\Fan_n$ minus the dimension~$n-1$, see \cref{lem:numberShards}.

\cref{fig:shards3,fig:shards4} illustrate the braid fans~$\Fan_n$ and their shards~$\shards_n$ when~$n = 3$ and~$n = 4$, respectively.
As the $3$-dimensional fan~$\Fan_4$ is difficult to visualize (as in \cref{fig:weakOrder4}\,(middle)), we use another classical representation in \cref{fig:shards4}\,(left): we intersect~$\Fan_4$ with a unit sphere and we stereographically project the resulting arrangement of great circles from the pole~$4321$ to the plane.
Each circle then corresponds to a hyperplane~$\b{x}_a = \b{x}_b$ with~$a < b$, separating a disk where~$\b{x}_a < \b{x}_b$ from an unbounded region where~$\b{x}_a > \b{x}_b$.
In both \cref{fig:shards3,fig:shards4}, the left picture shows the braid fan~$\Fan_n$ (where some chambers are labeled with blue permutations of~$[n]$ and some rays are labeled with red proper subsets of~$[n]$), the middle picture shows the shards~$\shards_n$ (labeled by arcs), and the right picture represents the quotient fan~$\fan_\textrm{sylv}$ of the sylvester congruence.

\begin{figure}
	\capstart
	\centerline{\includegraphics[scale=.9]{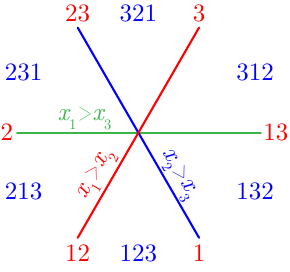} \qquad \includegraphics[scale=.9]{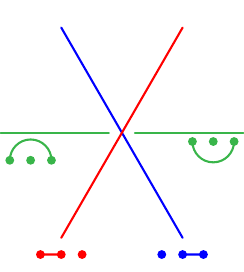} \qquad \includegraphics[scale=.9]{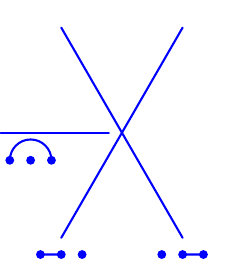}}
	\caption{The braid fan~$\Fan_3$ (left), the corresponding shards (middle), and the quotient fan of the sylvester congruence~$\equiv_\textrm{sylv}$~(right). \cite[Fig.~3]{PilaudSantos-quotientopes}}
	\label{fig:shards3}
\end{figure}

\begin{figure}
	\capstart
	\centerline{\includegraphics[scale=.45]{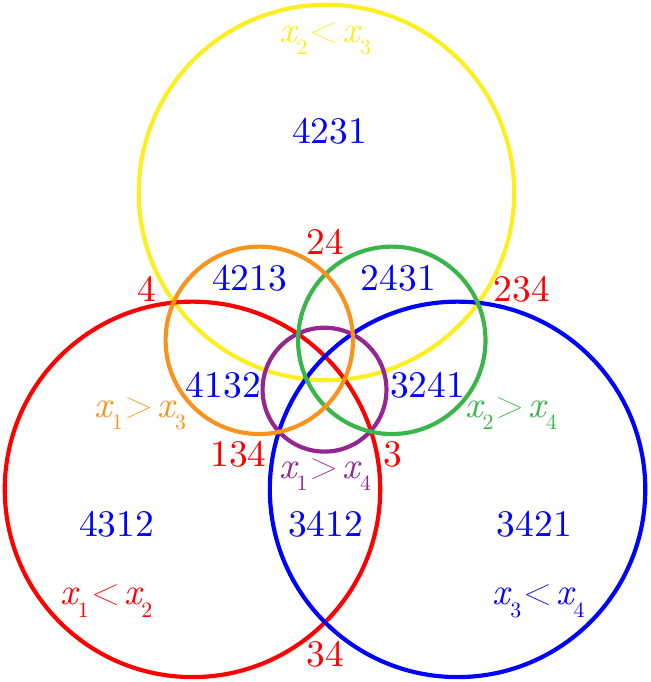} \; \includegraphics[scale=.45]{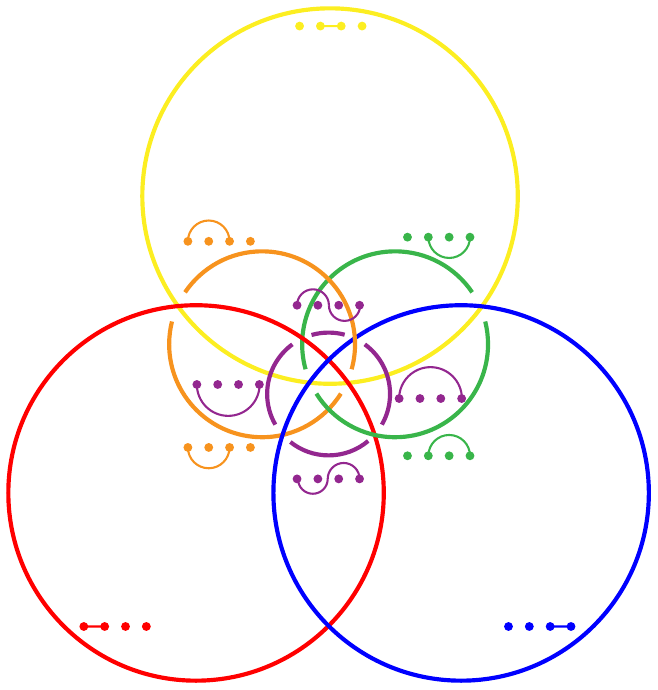} \; \includegraphics[scale=.45]{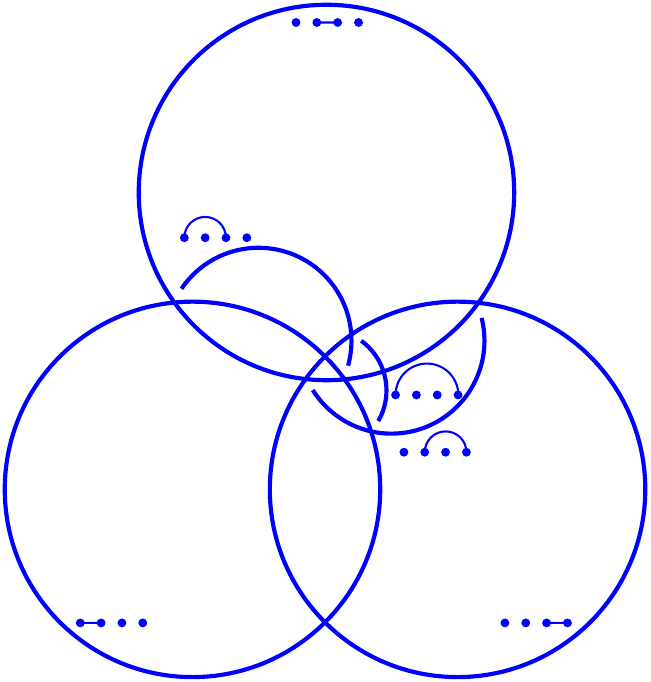}}
	\caption{A stereographic projection of the braid fan~$\Fan_4$ (left) from the pole~$4321$, the corresponding shards (middle), and the quotient fan of the sylvester congruence~$\equiv_\textrm{sylv}$~(right). \mbox{\cite[Fig.~4]{PilaudSantos-quotientopes}}}
	\label{fig:shards4}
\end{figure}

It turns out that the shards are precisely the pieces of the hyperplanes of~$\HA_n$ needed to delimit the cones of the quotient fans.
For~$\arcs \subseteq \arcs_n$, we denote by~$\shards_\arcs \eqdef \set{\shard(\arc)}{\arc \in \arcs}$ the set of shards corresponding to the arcs of~$\arcs$.

\begin{theorem}[{\cite[Sect.~10.5]{Reading-FiniteCoxeterGroupsChapter}}]
\label{thm:quotientFanShards}
For any arc ideal~$\arcs \subseteq \arcs_n$, the union of the walls of~$\Fan_\arcs$ is the union of the shards of~$\shards_\arcs$.
\end{theorem}

\begin{example}[Tamari]
\label{exm:shardsAsso}
Following \cref{exm:sylvesterCongruence,exm:noncrossingPartitions,exm:LodayAsso}, \cref{fig:shards3,fig:shards4} represent the quotient fans~$\fan_\textrm{sylv}$ corresponding to the sylvester congruences~$\equiv_\textrm{sylv}$ on~$\fS_3$ and~$\fS_4$.
It is obtained 
\begin{itemize}
\item either by gluing the chambers~$\polytope{C}(\sigma)$ of the permutations~$\sigma$ in the same sylvester class,
\item or by cutting the space with the shards of~$\shards_\textrm{sylv} = \set{\shard(a, b, {]a,b[}, \varnothing)}{1 \le a < b \le 4}$.
\end{itemize}
See also \cref{exm:LodayAsso} for a combinatorial description of the chambers and rays of this fan.
\end{example}

Finally, note that the shards and the forcing order among them can also be constructed in a purely geometrical way.
Namely, for any codimension~$2$ face~$F$ of the braid fan~$\Fan_n$, consider the subarrangement~$\HA_n^F$ of~$\HA_n$ formed by the hyperplanes containing~$F$, and call $F$-basic the two hyperplanes delimiting the region of~$\HA_n^F$ containing the fundamental chamber~$\polytope{C}(1 \dots n)$ of~$\Fan_n$.
Cut the non-$F$-basic hyperplanes by the $F$-basic hyperplanes for each codimension~$2$ face~$F$.
The resulting connected components are the (open) shards of~$\shards_n$.
The forcing relation is obtained as the transitive closure of the cutting relation: $\shard$ cuts~$\shard'$ if there is a codimension~$2$ face~$F$ contained in~$\shard$ and~$\shard'$ such that the supporting hyperplane of~$\shard$ is $F$-basic while that of~$\shard'$ is not.
Note that the cutting relation is weaker than the forcing relation: translated back on arcs, $\arc$ cuts~$\arc'$ if~$\arc$ forces~$\arc'$ and~$\arc$ and~$\arc'$ share an endpoint.
See \cite{Reading-PosetRegionsChapter} for more details.

%%%%%%%%

\subsection{Minkowski sums of associahedra}
\label{subsec:MinkowskiSumAssociahedra}

\enlargethispage{.2cm}
We conclude this preliminary section with a proof of \cref{thm:main1}, which provides alternative polytopal realizations of the quotient fans using Minkowski sums of C.~Hohlweg and C.~Lange's associahedra~\cite{HohlwegLange}.
This construction was already used for certain specific quotients, \eg for the Baxter congruence corresponding to diagonal rectangulations~\cite{LawReading} or for intersections of essential Cambrian congruences~\cite{ChatelPilaud}.
Our constructions in \cref{subsec:shardsumotopes} will be based on similar ideas.

Recall from \cref{exm:CambrianCongruences,exm:HohlwegLangeAsso} that for an arc~$\arc \in \arcs_n$, we denote by~$\arcs_\arc$ the upper ideal of the arc poset~$(\arcs_n, \prec)$ generated by~$\arc$, by~$\equiv_\arc$ the $\arc$-Cambrian congruence, by $\Fan_\arc$ the $\arc$-Cambrian fan, and by~$\Asso[\arc]$ the $\arc$-associahedron.
Our main tool is the following observation.

\begin{lemma}
\label{lem:intersectionCambrianCongruences}
For any arc ideal~$\arcs \subseteq \arcs_n$ with minimal elements~$\arc_1, \dots, \arc_p$, the lattice congruence~$\equiv_\arcs$ is the intersection of the Cambrian congruences~$\equiv_{\arc_1}, \dots, \equiv_{\arc_p}$.
\end{lemma}

\begin{proof}
The arc ideal $\arcs$ is generated by its minimal elements~$\arc_1, \dots, \arc_p$, thus it is the union of the principal ideals~$\arcs_{\arc_1}, \dots, \arcs_{\arc_p}$.
Therefore, the congruence~$\equiv_\arcs$ is the intersection of the congruences~$\equiv_{\arc_1}, \dots, \equiv_{\arc_p}$.
\end{proof}

\cref{lem:intersectionCambrianCongruences} has the following direct combinatorial and geometric consequences.
All these consequences where already observed in~\cite[arXiv version, Sect.~2.3]{ChatelPilaud} for intersections of essential Cambrian congruences (\ie with~$a = 1$ and~$b = n$), but the observation of \cref{lem:intersectionCambrianCongruences} was missing.

\begin{corollary}
For any arc ideal~$\arcs \subseteq \arcs_n$ with minimal elements~$\arc_1, \dots, \arc_p$, each congruence class of~$\equiv_\arcs$ is represented by a $p$-tuple of $\arc_1$-, \dots, $\arc_p$-Cambrian trees with a common linear extension.
\end{corollary}

\begin{proof}
By \cref{lem:intersectionCambrianCongruences}, the $\equiv_\arcs$-congruence classes are precisely the non-empty intersections of $\equiv_{\arc_1}$-, \dots, $\equiv_{\arc_p}$-congruence classes.
Each $\arc_i$-congruence class is the set of linear extensions of an $\arc_i$-Cambrian tree.
The result immediately follows.
\end{proof}

\begin{corollary}
\label{coro:commonRefinementQuotientFan}
For any arc ideal~$\arcs \subseteq \arcs_n$ with minimal elements~$\arc_1, \dots, \arc_p$, the quotient fan~$\Fan_\arcs$ is the common refinement of the quotient fans~$\Fan_{\arc_1}, \dots, \Fan_{\arc_p}$.
\end{corollary}

\begin{proof}
The union of the walls of $\Fan_\arcs$ is the union of the shards of~$\shards_\arcs$, thus the union of the shards of~$\shards_{\arc_1}, \dots, \shards_{\arc_p}$, thus the union of the walls of~$\Fan_{\arc_1}, \dots, \Fan_{\arc_p}$.
It follows that the quotient fan~$\Fan_\arcs$ is the common refinement of the quotient fans~$\Fan_{\arc_1}, \dots, \Fan_{\arc_p}$.
\end{proof}

\begin{corollary}
\label{coro:MinkowskiSumAssociahedra}
For any arc ideal~$\arcs \subseteq \arcs_n$ with minimal elements~$\arc_1, \dots, \arc_p$, the quotient fan~$\Fan_\arcs$ is the normal fan of the Minkowski sum of the associahedra~${\Asso[\arc_1], \dots, \Asso[\arc_k]}$.
\end{corollary}

\begin{proof}
For any arc~$\arc \in \arcs_n$, the quotient fan~$\Fan_\arc$ is the normal fan of the $\arc$-associahedron~$\Asso[\arc]$.
The statement thus follows from \cref{coro:commonRefinementQuotientFan} and the fact that the normal fan of a Minkowski sum is the common refinement of the the normal fans of its summands.
\end{proof}

\begin{example}[Other relevant congruences]
\enlargethispage{.5cm}
Consider the Baxter congruence~$\equiv_\textrm{Bax}$ of \cref{exm:otherCongruences}\,\eqref{item:BaxterCongruence}.
As already observed in \cref{exm:otherQuotientopes} and~\cite{LawReading}, the quotient fan~$\fan_\textrm{Bax}$ is the normal fan of the Minkowski sum of \mbox{J.-L.~Loday's} associahedron~\cite{Loday} and its opposite, see \cref{fig:BaxterAssociahedronMinkowskiSum}.
This was extended in~\cite{ChatelPilaud} to arbitrary pairs of opposite associahedra of C.~Hohlweg and C.~Lange~\cite{HohlwegLange}, and even to arbitrary intersections of essential Cambrian congruences (\ie with~$a = 1$ and~$b = n$).

\begin{figure}[h]
	\capstart
	\centerline{\includegraphics[scale=.32]{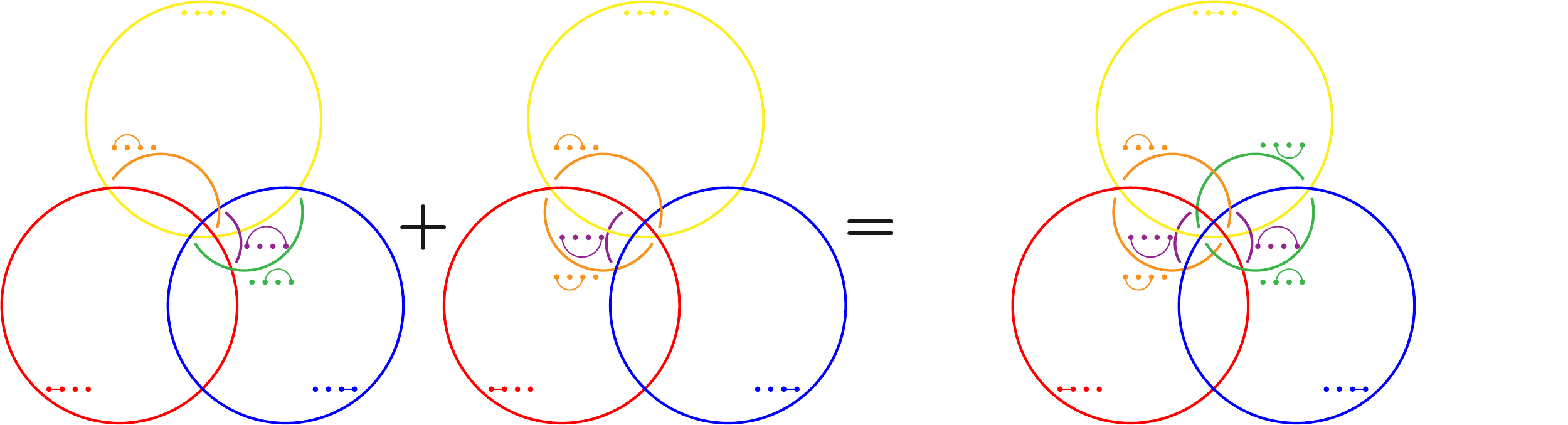}}
	\medskip
	\centerline{\includegraphics[scale=.32]{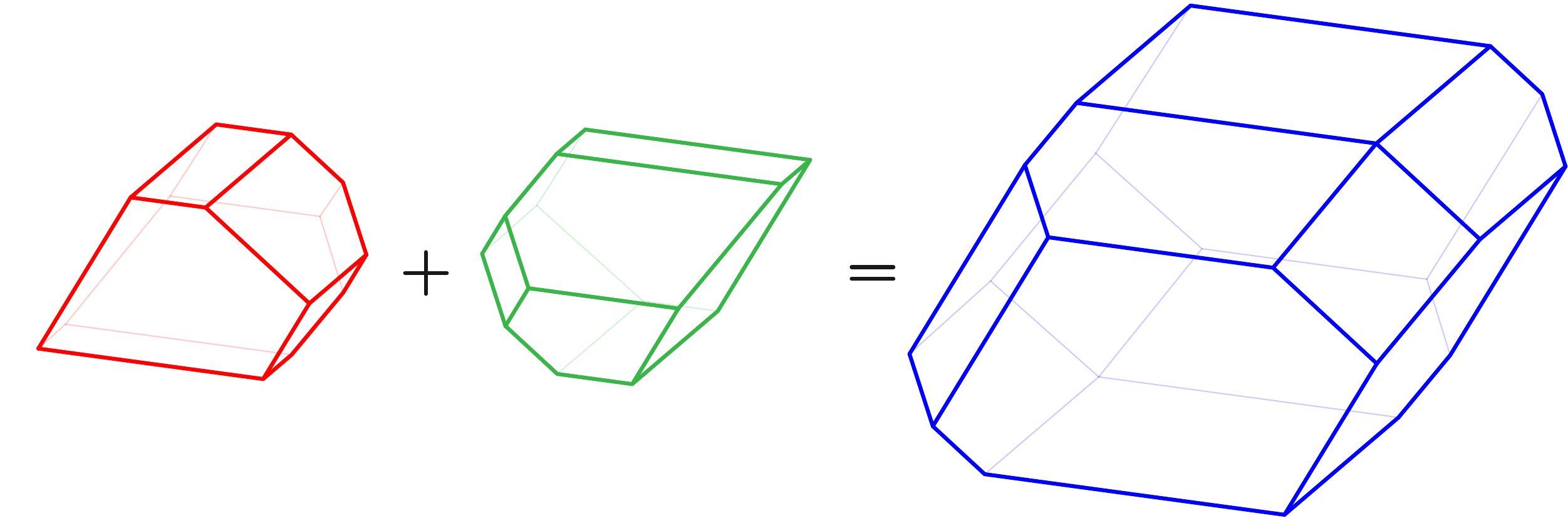}}
	\caption{The quotient fan~$\fan_\textrm{Bax}$ of the Baxter congruence is the normal fan of the Minkowski sum~$\Asso[{(1, n, [n], \varnothing)}] + \Asso[{(1, n, \varnothing, [n])}]$ of two opposite associahedra~\cite{LawReading}. \cite[Fig.~15]{ChatelPilaud}}
	\label{fig:BaxterAssociahedronMinkowskiSum}
\end{figure}
\end{example}

%%%%%%%%%%%%%%%%%%%%%%%%%%%%%%%%%%%%%%

\section{Shard polytopes}
\label{sec:shardPolytopes}

In this section, we construct alternative realizations of the quotient fans using more elementary polytopes.
Namely, to each arc~$\arc$ we associate a shard polytope~$\shardPolytope$ that will ensure the presence of the shard~$\shard(\arc)$ in the normal fan of any Minkowski sum of shard polytopes containing~$\shardPolytope$.
We introduce these shard polytopes in \cref{subsec:shardPolytopes}, study their basic geometric properties in \cref{subsec:geometricPropertiesShardPolytopes} and their normal fans in \cref{subsec:normalFansShardPolytopes}, construct quotientopes as Minkowski sums of shard polytopes in \cref{subsec:shardsumotopes}, and finally observe in \cref{subsec:MinkowskiIdentityShardPolytopes} an intriguing valuation-like Minkowski identity between shard polytopes.

%%%%%%%%

\subsection{Definition of shard polytopes}
\label{subsec:shardPolytopes}

We need the following definitions.

\begin{definition}
\label{def:alternatingMatchings}
For an arc~$\arc \eqdef (a, b, A, B)$, we define
\begin{itemize}
\item an \defn{$\arc$-alternating matching} as a (possibly empty) sequence~$M = \{a_1, b_1, \dots, a_k, b_k\}$ where ${a \le a_1 < b_1 < \dots < a_k < b_k \le b}$ and~$a_i \in \{a\} \cup A$ while~$b_i \in B \cup \{b\}$ for all~$i \in [k]$,
\item the \defn{pairs} of the $\arc$-alternating matching~$M$ as the pairs~$(a_\ell, b_\ell)$ for~$\ell \in [k]$, 
\item the \defn{characteristic vector} of the $\arc$-alternating matching~$M$ as~$\chi(M) = \sum_{i \in [k]} \b{e}_{a_i} - \b{e}_{b_i}$,
\item an \defn{$\arc$-fall} (resp.~\defn{$\arc$-rise}) as a position~$j \in [a,b[$ such that~$j \in \{a\} \cup A$ and~$j+1 \in B \cup \{b\}$ (resp.~such that~$j \in \{a\} \cup B$ and~$j+1 \in A \cup \{b\}$).
\end{itemize}
\end{definition}

Note that~$\varnothing$ and $\{a,b\}$ are always $\arc$-alternating matchings, and $a$ and~$b-1$ are always $\arc$-falls or $\arc$-rises.
Observe also that the arc~$\arc$ crosses the horizontal axis between the dots~$j$ and~$j+1$ if and only if~$j$ is an $\arc$-fall or $\arc$-rise distinct from~$a$ and~$b-1$.
The number of $\arc$-rises plus the number of $\arc$-falls is thus the number of crossings of~$\arc$ with the horizontal axis plus~$2$.

\begin{proposition}
\label{prop:shardPolytope}
The \defn{shard polytope}~$\shardPolytope$ of an arc~$\arc$ is the polytope defined equivalently as
\begin{itemize}
\item the convex hull of the characteristic vectors of all $\arc$-alternating matchings,
\item the subset of the hyperplane~$\hyp \eqdef \bigset{\b{x} \in \R^n}{\sum_{i \in [n]} \b{x}_i = 0}$ defined by
	\begin{enumerate}[$\circ$]
	\item $\b{x}_i = 0$ for any~$i \in [n] \ssm [a,b]$, $\b{x}_{a'} \ge 0$ for any~$a' \in A$, and $\b{x}_{b'} \le 0$ for any~$b' \in B$, 
	\item $\sum_{i \le f} \b{x}_i \le 1$ for any $\arc$-fall~$f$ and~$\sum_{i \le r} \b{x}_i \ge 0$ for any $\arc$-rise~$r$.
	\end{enumerate}
\end{itemize}
\end{proposition}

\begin{figure}[b]
	\capstart
	\centerline{\includegraphics[scale=1]{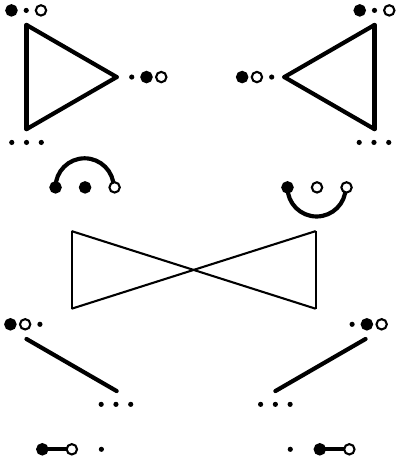}}
	\caption{Shard polytopes for all arcs with~$n = 3$.}
	\label{fig:shardPolytopes3}
\end{figure}

\begin{figure}[t]
	\capstart
	\centerline{\includegraphics[scale=1]{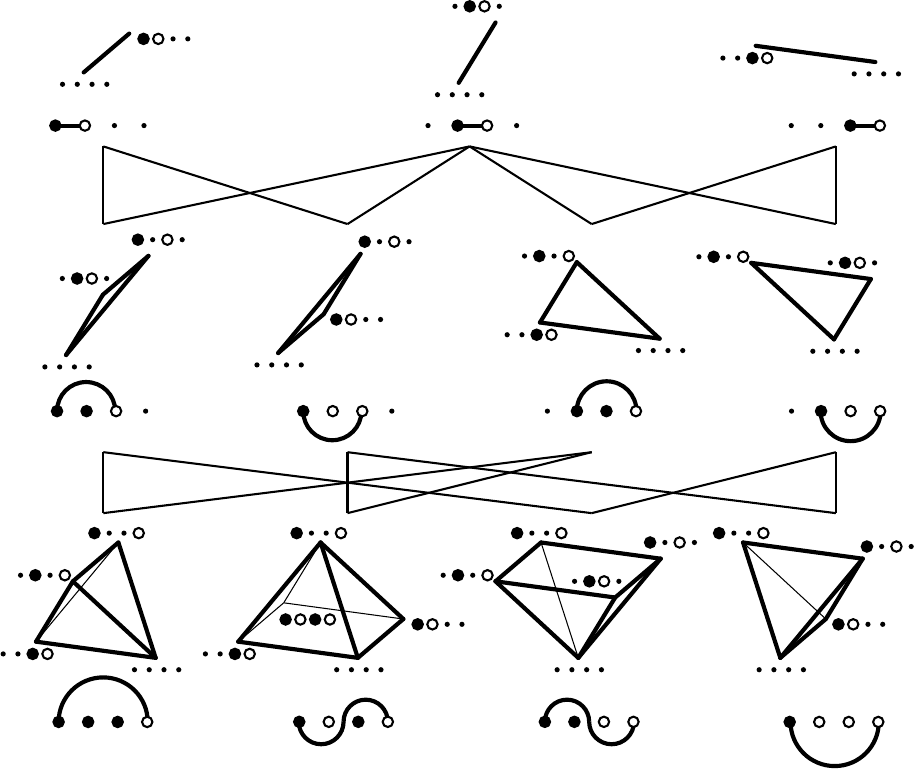}}
	\caption{Shard polytopes for all arcs with~$n = 4$.}
	\label{fig:shardPolytopes4}
\end{figure}

The shard polytopes corresponding to all arcs for~$n = 3$ and~$n = 4$ are represented in \cref{fig:shardPolytopes3,fig:shardPolytopes4}.
Both pictures are organized according to the forcing order on~$\arcs_n$.
For each~$\arc \in \arcs_n$, we represent the arc~$\arc$ below and the shard polytope~$\shardPolytope$ above.
The vertices of the shard polytopes are labeled by the corresponding $\arc$-alternating matchings, where we use solid dots~$\bullet$ for elements in~$\{a\} \cup A$ and hollow dots~$\circ$ for elements in~$B \cup \{b\}$.
The corresponding vertex coordinates are directly read replacing~$\bullet$ by~$1$ and~$\circ$ by~$-1$.
For instance, the vertex labeled~${\bullet \cdot \cdot \,\circ}$ has coordinates~$(1,0,0,-1)$.

\begin{remark}
\label{rem:inequalitiesShardPolytope}
The inequalities of \cref{prop:shardPolytope} imply that~$0 \le \sum_{i \le j} \b{x}_i \le 1$ \mbox{for any~$j \in [n]$.
Indeed,}
\begin{itemize}
\item For~$j \in [n] \ssm [a,b[$, we have~$\sum_{i \le j} \b{x}_i = 0$ since~$\sum_{i \in [n]} \b{x}_i = 0$ and~$\b{x}_i = 0$ for~$i \in [n] \ssm [a,b]$.
\item For~$j \in [a,b[$, we have~$0 \le \sum_{i \le r} \b{x}_i \le \sum_{i \le j} \b{x}_i \le \sum_{i \le f} \b{x}_i \le 1$, where 
	\begin{enumerate}[$\circ$]
	\item $r$ is the last $\arc$-rise in $[a, j[$ while $f$ is the first $\arc$-fall in $[j, b[$ if~$j \in A$,
	\item $r$ is the first $\arc$-rise in $[j, b[$ while $f$ is the last $\arc$-fall in $[a, j[$ if~$j \in B$,
	\item $r$ is the first $\arc$-rise in $[j, b[$ and $f$ is the first $\arc$-fall in $[j, b[$ if~$j = a$.
	\end{enumerate}
\end{itemize}
Considering differences of consecutive inequalities ${0 \le \sum_{i \le j} \b{x}_i \le 1}$, we obtain that~$0 \le \b{x}_{a'} \le 1$ for any~$a' \in \{a\} \cup A$ and $-1 \le \b{x}_{b'} \le 0$ for any~$b' \in B \cup \{b\}$.
The inequalities that we kept in \cref{prop:shardPolytope} will be shown to be the facet defining inequalities of~$\shardPolytope$ in \cref{prop:elemPropShardPolytope}\,\eqref{it:facetsShardPolytope}.
\end{remark}

\begin{proof}[Proof of \cref{prop:shardPolytope}]
Let~$\polytope{P}$ denote the $V$-polytope and~$\polytope{Q}$ denote the $H$-polytope defined in \cref{prop:shardPolytope}.
Since the characteristic vector of any $\arc$-alternating matching clearly satisfies the given inequalities, we have~$\polytope{P} \subseteq \polytope{Q}$.
Conversely, as the equalities and inequalities defining~$\polytope{Q}$ have coefficients in~$\{-1,0,1\}$ and satisfy the consecutive ones property, the matrix defining~$\polytope{Q}$ is unimodular, so that~$\polytope{Q}$ has integer vertices.
Consider a vertex~$\b{v}$ of~$\polytope{Q}$.
According to \cref{rem:inequalitiesShardPolytope}, the equalities and inequalities defining~$\polytope{Q}$ ensure that all the non-zero coordinates of~$\b{v}$ appear in between positions~$a$ and~$b$ and alternate between some~$1$'s at positions among~$\{a\} \cup A$ and some~$-1$'s at positions among~$B \cup \{b\}$.
Therefore, $\b{v}$ is the characteristic vector of an $\arc$-alternating matching, and we obtain that~$\polytope{Q} \subseteq \polytope{P}$.
\end{proof}

\begin{remark}
The vertex and facet descriptions of \cref{prop:shardPolytope} are nicer in the basis~$(\b{f}_{\!i})_{i \in [n-1]}$ of~$\hyp$ defined by~$\b{f}_{\!i} \eqdef \b{e}_i - \b{e}_{i+1}$.
Namely, the shard polytope~$\shardPolytope$ is given by
\begin{itemize}
\item the $0/1$-vectors~$\sum_{i \in [k]} \sum_{j \in [a_i, b_i[} \b{f}_{\!i}$ for all $\arc$-alternating matchings~${\{a_1 < b_1 < \dots < a_k < b_k\}}$,
\item the inequalities~$\b{y}_i = 0$ for any~$i \in [n] \ssm [a,b[$, $\b{y}_{a'} - \b{y}_{a'+1} \le 0$ for any~$a' \in A$ while $\b{y}_{b'} - \b{y}_{b'+1} \ge 0$ for any~$b' \in B$, and $\b{y}_f \le 0$ for any $\arc$-fall~$f$ while~$\b{y}_r \ge 0$ for any $\arc$-rise~$r$.
\end{itemize}
We stay in the basis~$(\b{e}_i)_{i \in [n]}$ of~$\R^n$ as it will be useful to fold type~$A$ in type~$B$ in \cref{part:typeB}.
\end{remark}

\begin{remark}
\label{rem:pseudoarcs}
To study certain properties of shard polytopes, it will sometimes be convenient to consider~$\shardPolytope$ for a \defn{pseudoshard}~$\arc \eqdef (a, b, A, B)$ where~$A$ and~$B$ are disjoint subsets of~$]a,b[$ whose union does not necessarily cover~$]a,b[$.
The vertex and facet descriptions of~$\shardPolytope$ are similar to \cref{prop:shardPolytope}, except that
\begin{itemize}
\item $\arc$-alternating matchings never use elements of~${]a,b[} \ssm (A \sqcup B)$,
\item we have the equalities $\b{x}_i = 0$ for all~$i \in {]a,b[} \ssm ( A \sqcup B)$, and the $\arc$-falls (resp.~$\arc$-rises) are pairs $(j,k)$ where~$j \in \{a\} \cup A$ and $k \in B \cup \{b\}$ (resp.~$j \in \{a\} \cup B$ and~$k \in A \cup \{b\}$) are consecutive in $A \sqcup B$, \ie such that~${]j,k[} \cap (A \sqcup B) = \varnothing$.
\end{itemize}
We call~$\shardPolytope$ a \defn{pseudoshard polytope}.
Note that~$\shardPolytope$ is affinely isomorphic to the shard polytope~$\shardPolytope[\pi(\arc)]$ of the arc~$\pi(\arc) \eqdef (\pi(a), \pi(b), \pi(A), \pi(B)) \in \arcs_{|\{a,b\} \cup A \cup B|}$ where $\pi$ is the order preserving bijection~$\{a,b\} \cup A \cup B \to [|\{a,b\} \cup A \cup B|]$.
\end{remark}

%%%%%%%%

\subsection{Basic geometric properties of shard polytopes}
\label{subsec:geometricPropertiesShardPolytopes}

The following statement, illustrated in \cref{fig:shardPolytopes4}, describes the vertices, edges and facets of the shard polytopes.
Throughout the paper, we use the Kronecker notation~$\delta_X = 1$ if the property~$X$ holds, and~$\delta_X = 0$ otherwise.

\begin{proposition}
\label{prop:elemPropShardPolytope}
Let~$\arc \eqdef (a, b, A, B)$ be an arc of~$\arcs_n$.
\begin{enumerate}[(i)]
\item \label{it:dimensionShardPolytope} The shard polytope~$\shardPolytope$ has dimension~$b-a$.
\item \label{it:verticesShardPolytope} The vertices of~$\shardPolytope$ are precisely all characteristic vectors of $\arc$-alternating matchings. Therefore, the number of vertices of~$\shardPolytope$ is~$w_\arc(b)$ where $v_\arc$ and~$w_\arc$ are the two functions from $[a,b]$ to~$\N$ defined by the initial conditions~$v_\arc(a) \eqdef w_\arc(a) \eqdef 1$ and the induction $v_\arc(i) \eqdef v_\arc(i-1) + \delta_{i \in \{a\} \cup A} \cdot w_\arc(i-1)$ and~$w_\arc(i) \eqdef \delta_{i \in B \cup \{b\}} \cdot v_\arc(i-1) + w_\arc(i-1)$~for~$i \in {]a,b]}$.
\item \label{it:edgesShardPolytope} The characteristic vectors of two $\arc$-alternating matchings~$M$ and~$M'$ form an edge of~$\shardPolytope$ if and only if~$|M \symdif M'| = 2$. Therefore, the edge directions of~$\shardPolytope$ are the $\binom{b-a+1}{2}$ directions~$\b{e}_i - \b{e}_j$ for~$a \le i < j \le b$.
\item \label{it:facetsShardPolytope} The facets of~$\shardPolytope$ are precisely defined by the inequalities of \cref{prop:shardPolytope}. Therefore, the number of facets of~$\shardPolytope$ is~$b-a+1$ plus the number of crossings of~$\arc$ with the horizontal axis, and $\shardPolytope$ is a simplex if and only if~$\arc$ does not cross the horizontal axis.
\end{enumerate}
\end{proposition}

\begin{proof}
For~\eqref{it:dimensionShardPolytope}, note that the shard polytope~$\shardPolytope$ lies in the space generated by the vectors~$\b{e}_i - \b{e}_j$ for~${a \le i < j \le b}$ which has dimension~$b-a$.
Conversely, $\shardPolytope$ contains the~$b-a+1$ affinely independent vertices~$0$, $\b{e}_a - \b{e}_b$, $\b{e}_a - \b{e}_{b'}$ for~$b' \in B$, and~$\b{e}_{a'} - \b{e}_b$ for~$a' \in A$.

For~\eqref{it:verticesShardPolytope}, the shard polytope~$\shardPolytope$ is the convex hull of the characteristic vectors of the $\arc$-alter\-nating matchings.
To see that they all define vertices, note that for any $\arc$-alternating matchings $M$ and~$M'$, we have~$\dotprod{2\chi(M) - \one_{\{a\} \cup A} + \one_{B \cup \{b\}}}{\chi(M')} = |M \cap M'| - |M' \ssm M|$, so that~$\chi(M)$ is the only characteristic vector of an $\arc$-alternating matching that maximizes the direction~${2\chi(M) - \one_{\{a\} \cup A} + \one_{B \cup \{b\}}}$.
The induction formula for the number of $\arc$-alternating matchings is immediate: for all~$i \in [a,b]$, the value~$v_\arc(i)$ (resp.~$w_\arc(i)$) counts the number of odd (resp.~even) subsets of~$[i]$ of the form~$M \cap [i]$ where~$M$ is an $\arc$-alternating matching.

For~\eqref{it:edgesShardPolytope}, observe first that for any three $\arc$-alternating matchings~$M, M'$ and~$M''$, we have $\dotprod{\chi(M) + \chi(M') - \one_{\{a\} \cup A} + \one_{B \cup \{b\}}}{\chi(M'')} = |M \cap M' \cap M''| - |M'' \ssm (M \cup M')|$.
When $|M \symdif M'| = 2$, we get~$\dotprod{\chi(M) + \chi(M') - \one_{\{a\} \cup A} + \one_{B \cup \{b\}}}{\chi(M'')} \le |M \cap M'|$ with equality if and only if~${M'' \in \{M, M'\}}$ (since all matchings have even cardinality).

This shows that $\shardPolytope$ has an edge joining~$\chi(M)$ to~$\chi(M')$ as soon as~$|M \symdif M'| = 2$.
Conversely, assume that~$\chi(M)$ and~$\chi(M')$ form an edge in~$\shardPolytope$.
Then the middle of~$\chi(M)$ and~$\chi(M')$ is not the middle of another pair of vertices~$\chi(M'')$ and~$\chi(M''')$ of~$\shardPolytope$. Therefore the multisets $M \cup M'$ and~$M'' \cup M'''$ are distinct for any $\arc$-alternating matchings such that~$M, M', M'', M'''$ are all distinct.
The statement then follows from the slightly more detailed combinatorial property given in \cref{lem:matchingUnionLemma} below.

For~\eqref{it:facetsShardPolytope}, we have already seen in \cref{prop:shardPolytope} that the facets of the shard polytope~$\shardPolytope$ are all defined by these inequalities.
Conversely, we prove that all these inequalities correspond to facets by showing that none of these inequalities is redundant.
Indeed, for~$a' \in A$ (resp.~${b' \in B}$, resp.~an $\arc$-fall~$f$, resp.~an $\arc$-rise~$r$), the vector~$\b{e}_a - \b{e}_{a'}$ (resp.~$\b{e}_{b'} - \b{e}_b$, resp.~$2(\b{e}_f - \b{e}_{f+1})$, resp.~${-\b{e}_r+\b{e}_{r+1}}$) satisfies all these inequalities except~$\b{x}_{a'} \ge 0$ (resp.~$\b{x}_{b'} \le 0$, resp.~${\sum_{i \le f} \b{x}_i \le 1}$, resp.~${\sum_{i \le r} \b{x}_i \ge 0}$).
The number of facets is thus $b-a-1$ plus the number of $\arc$-falls plus the number of $\arc$-rises, thus $b-a+1$ plus the number of crossings of~$\arc$ with the horizontal axis.
\end{proof}

Shard polytopes also behave very nicely with respect to their faces.

\begin{proposition}
\label{prop:facesProductShardPolytopes}
Any face of a shard polytope is affinely isomorphic to a Cartesian product of shard polytopes.
\end{proposition}

\begin{proof}
It is clearly sufficient to prove the result for facets (since an $i$-face is a facet of an $(i+1)$-face and the faces of a Cartesian product are the Cartesian products of the faces of the factors).
From the facet-defining inequalities of~$\shardPolytope$ given in \cref{prop:shardPolytope}, it is immediate to observe that:
\begin{itemize}
\item for~$i \in {]a,b[}$, the facet defined by the equality~$\b{x}_i = 0$ is the pseudoshard polytope~$\shardPolytope[\arc_{\bar i}]$ where~$\arc_{\bar i} \eqdef (a, b, A \ssm \{i\}, B \ssm \{i\})$, which is affinely equivalent to a shard polytope by \cref{rem:pseudoarcs},
\item for an $\arc$-fall~$j$ (resp.~$\arc$-rise~$j$), the facet defined by the equality~$\sum_{i \le j} \b{x}_i = 1$ (resp.~$0$) is the translate by~$\delta_{j \in A} (\b{e}_j - \b{e}_{j+1})$\julian{Notation fails if $a+1 = b$: Then $a$ is both rise and fall. We need it to be translated for the fall and not translated for the rise} of the Cartesian product of shard polytopes~${\shardPolytope[\arc_{\le j}] \times \shardPolytope[\arc_{>j}]}$, where~$\arc_{\le j} \eqdef (a, j, A \cap {]a,j[}, B \cap {]a,j[})$ and~$\arc_{> j} \eqdef (j+1, b, A \cap {]j+1, b[}, B \cap {]j+1, b[})$.
\qedhere
\end{itemize}
\end{proof}

Our next statement compares forcing among shards with faces of shard polytopes.

\begin{proposition}
\label{prop:facesShardPolytope}
Let~$\arc \eqdef (a, b, A, B)$ and~$\arc' \eqdef (a', b', A', B')$ be two arcs of~$\arcs_n$.
\begin{enumerate}[(i)]
\item \label{it:facesForcingShardPolytope1} The shard polytope~$\shardPolytope[\arc]$ is a face of the shard polytope~$\shardPolytope[\arc']$ if and only~if 
	\begin{itemize}
	\item $\arc$ forces~$\arc'$, \ie $a' \le a < b \le b'$, $A \subseteq A'$ and~$B \subseteq B'$, and
	\item $a \in \{a'\} \cup A'$ while~$b \in B' \cup \{b'\}$.
	\end{itemize}
\item \label{it:facesForcingShardPolytope2} If~$\arc$ forces~$\arc'$, then the translate of~$\shardPolytope[\arc]$ by the vector~$\translation \eqdef \delta_{a \in B'}(\b{e}_{a'} - \b{e}_a) + \delta_{b \in A'}(\b{e}_b-\b{e}_{b'})$ is a face of~$\shardPolytope[\arc']$.
\item \label{it:reconstructShardPolytopeFromForcing} The shard polytope~$\shardPolytope[\arc']$ is the convex hull of~$\b{0}$, $\b{e}_{a'} - \b{e}_{b'}$, and the vertices of the translated shard polytopes~$\shardPolytope + \translation$ over all arcs~$\arc$ forcing~$\arc'$ (or even over all arcs~$\arc$ cutting~$\arc'$).
\end{enumerate}
\end{proposition}

\begin{proof}
For~\eqref{it:facesForcingShardPolytope1}, assume first that~$\arc$ forces~$\arc'$ and $a \in \{a'\} \cup A'$ while~$b \in B' \cup \{b'\}$.
Then any \mbox{$\arc$-alter}\-nating matching is also an $\arc'$-alternating matching by definition.
Moreover, the $\arc$-alternating matchings are precisely the $\arc'$-alternating matchings whose characteristic vectors are maximal in the direction~$\one_{\{a\} \cup A} - \one_{\{b\} \cup B} - \one_{A'} + \one_{B'}$.
Therefore, $\shardPolytope[\arc]$ is a face of~$\shardPolytope[\arc']$.
Conversely, assume that~$\shardPolytope[\arc]$ is a face of~$\shardPolytope[\arc']$.
For any~$i < j$ with~$i \in \{a\} \cup A$ and~$j \in B \cup \{b\}$, the vertex~$\b{e}_i - \b{e}_j$ of~$\shardPolytope[\arc]$ must be a vertex of~$\shardPolytope[\arc']$, so that~$i \in \{a'\} \cup A'$ and~$j \in B' \cup \{b'\}$.
This shows that~$\arc$ forces~$\arc'$ and that $a \in \{a'\} \cup A'$ while~$b \in B' \cup \{b'\}$.

For~\eqref{it:facesForcingShardPolytope2}, let~$\bar a \eqdef a$ if~$a \in \{a'\} \cup A'$ and~$\bar a \eqdef a'$ otherwise, and define similarly~$\bar b$.
Associate to any $\arc$-alternating matching the $\arc'$-alternating matching~$\bar M$ obtained from~$M$ by replacing~$a$ by~$\bar a$ and~$b$ by~$\bar b$.
Note that~$\chi(\bar M) = \chi(M) + (\b{e}_{\bar a} - \b{e}_a) + (\b{e}_b - \b{e}_{\bar b})$.
Moreover, the $\arc'$-alternating matchings~$\bar M$ are precisely the $\arc'$-alternating matchings whose characteristic vectors are maximal in the direction~$\one_{\{\bar a\} \cup A} - \one_{\{\bar b\} \cup B} - \one_{A'} + \one_{B'}$.

For~\eqref{it:reconstructShardPolytopeFromForcing}, observe first that~\eqref{it:facesForcingShardPolytope2} implies that all vertices of~$\shardPolytope + \translation$ are vertices of~$\shardPolytope[\arc']$ for any arc~$\arc$ forcing~$\arc'$.
Conversely, we prove that any vertex of~$\shardPolytope[\arc']$ distinct from~$\b{0}$ and~$\b{e}_{a'}-\b{e}_{b'}$ is in fact already a vertex of one of the translated shard polytopes~$\shardPolytope[\arc_a] + \translation[\arc_a][\arc']$ for~$a \in {]a', b'[}$, where~$\arc_a \eqdef (a, b', A \cap {]a,b'[}, B \cap {]a, b'[})$.
Indeed, any vertex of~$\shardPolytope[\arc']$ distinct from~$\b{0}$ and~$\b{e}_{a'}-\b{e}_{b'}$ corresponds to a $\arc'$-alternating matching~$M'$ distinct from~$\varnothing$ and~$\{a', b'\}$.
We distinguish two cases, depending on the smallest element~$a$ of~$M'$:
\begin{itemize}
\item If~$a' < a$, then~$M'$ is an $\arc_a$-alternating matching. Since~$a \in A$, we have~$\translation[\arc_a][\arc'] = \b{0}$, so that~$\chi(M')$ is a vertex of~$\shardPolytope[\arc_a] + \translation[\arc_a][\arc']$.
\item If~$a' = a$, let~$b$ denote the second smallest element of~$M'$. Then~$M = M' \ssm \{a,b\}$ is an $\arc_b$-alternating matching. Since~$b \in B'$, we have~$\translation[\arc_b][\arc'] = \b{e}_{a'} - \b{e}_b$, so that~$\chi(M') = \b{e}_{a'} - \b{e}_b + \chi(M)$ is a vertex of~$\shardPolytope[\arc_b] + \translation[\arc_b][\arc']$.
\end{itemize}
In both cases, $\chi(M')$ is a vertex of one of the translated shard polytopes~$\shardPolytope[\arc_a] + \translation[\arc_a][\arc']$.
\end{proof}

We conclude our collection of basic geometric properties of shard polytopes with a discussion of two symmetries of shard polytopes, still illustrated in \cref{fig:shardPolytopes4}.

\begin{proposition}
\label{prop:symmetriesShardPolytope}
Define~$\bar i \eqdef n+1-i$ and consider the two involutive arc maps~$\phi, \psi : \arcs_n \to \arcs_n$ defined by~${\phi : (a, b, A, B) \mapsto (a, b, B, A)}$ (horizontal symmetry) and~${\psi : (a, b, A, B) \mapsto (\bar b, \bar a, \bar A, \bar B)}$ (vertical symmetry), and the two involutive linear maps~$\Phi, \Psi : \R^n \to \R^n$ defined by~$\Phi(\b{e}_i) = -\b{e}_i$ and~$\Psi(\b{e}_i) = \b{e}_{\bar i}$.
Then for any arc~$\arc$, we have
\[
\shardPolytope[\phi(\arc)] = \Phi(\shardPolytope) + \b{e}_a - \b{e}_b
\qquad\text{and}\qquad
\shardPolytope[\psi(\arc)] = \Psi(\shardPolytope) - \b{e}_{\bar a} + \b{e}_{\bar b}.
\]
Hence, the image of~$\arc$ by central symmetry has shard polytope~$\shardPolytope[\phi \circ \psi(\arc)] = \Phi \circ \Psi(\shardPolytope)$.
\end{proposition}

\begin{proof}
This follows from the bijections sending an $\arc$-alternating matching~$M$ to
\begin{itemize}
\item the $\phi(\arc)$-alternating matching~$\{a,b\} \symdif M$, with~$\chi(\{a,b\} \symdif M) = \Phi(\chi(M)) + \b{e}_a - \b{e}_b$,
\item the $\psi(\arc)$-alternating matching~$\{\bar a, \bar b\} \symdif \bar M$, with~$\chi(\{\bar a, \bar b\} \symdif \bar M) = \Psi(\chi(M)) - \b{e}_{\bar a} + \b{e}_{\bar b}$.
\qedhere
\end{itemize}
\end{proof}

%%%%%%%%

\subsection{Normal fans of shard polytopes}
\label{subsec:normalFansShardPolytopes}

The main goal of this section is to show the following compatibility of the normal fans of shard polytopes with the arc poset, announced in \cref{prop:main2}.

\begin{proposition}
\label{prop:shardPolytopeFan}
For any arc~$\arc$, the union of the walls of the normal fan of the shard polytope~$\shardPolytope$ contains the shard~$\shard(\arc)$ and is contained in the union of the shards~$\shard(\arc')$ for~$\arc \prec \arc'$.
\end{proposition}

We start with the following elementary and purely combinatorial lemma, illustrated in \cref{fig:matchingUnionLemma}.
This lemma is a slightly more detailed reformulation of \cref{prop:elemPropShardPolytope}\,\eqref{it:edgesShardPolytope} (and proves it).

\begin{lemma}
\label{lem:matchingUnionLemma}
For an arc~$\arc \eqdef (a, b, A, B)$ and any pair of distinct $\arc$-alternating matchings $\{M_1, M_2\}$, there exists a pair of $\arc$-alternating matchings $\{M_3, M_4\}$ disjoint from~$\{M_1, M_2\}$ such that the multisets $M_1 \cup M_2$ and $M_3 \cup M_4$ coincide\julian{I'd prefer expressing this by sums of characteristic vectors}, except if there exist two (possibly empty) $\arc$-alternating matchings $H$ and $T$ such that $\{M_1,M_2\}$ is one of the following pairs:
\begin{enumerate}
	\item $\{H < a' < b' < T, \; H < T\}$ for some $a' \in A$ and $b' \in B$,
	\item $\{H < a' < b_1 < T, \; H < a' < b_2 < T\}$ for some $a' \in A$ and $b_1 < b_2 \in B$,
	\item $\{H < a_1 < b' < T, \; H < a_2 < b' < T\}$ for some $a_1 < a_2 \in A$ and $b' \in B$,
	\item $\{H < a_1 < b_1 < a_2 < b_2 < T, \; H < a_1 < b_2 < T\}$ for some $a_1 < a_2 \in A$ and $b_1 < b_2 \in B$.
\end{enumerate}
\end{lemma}

\begin{proof}
Let $H$ (resp.~$T$) be the longest initial (resp.~final) common $\arc$-alternating matching between~$M_1$ and~$M_2$.
This part of the $\arc$-alternating matchings~$M_1$ and~$M_2$ is forced in any decomposition of~$M_1 \cup M_2$, so we can assume without loss of generality that both~$H$ and~$T$ are empty.
For a multiset~$X$ of~$[n]$, we define ${\mu(X,i) \eqdef |X \cap [i] \cap (\{a\} \cup A)| - |X \cap [i] \cap (B \cup \{b\})|}$ for~$i \in X$, and denote by~$\mu(X) \eqdef (\mu(X,i))_{i \in X}$ the sequence of these differences.
For instance, in an $\arc$-alternating matching~$M$ of semi-length~$k$, we have~$\mu(M) = (10)^k$ ($1$ for each element~$a_j$ of the $\arc$-alternating matching, and $0$ for each element~$b_j$ of the $\arc$-alternating matching).
Therefore, we have $\mu(M_1 \cup M_2, i) = 0$, $1$ or~$2$ for each~$i \in M_1 \cup M_2$.
Observe that if there is $i < j < k$ in~$M_1 \cup M_2$ such that $\mu(M_1 \cup M_2, i)$ and~$\mu(M_1 \cup M_2, k)$ are positive while $\mu(M_1 \cup M_2, j) = 0$, then both $M_1$ and $M_2$ decompose into $M_1 = L_1 \cup R_1$ and~$M_2 = L_2 \cup R_2$ where $L_1, L_2$ are on the left of~$j$ while $R_1, R_2$ are on the right of~$j$, and we obtain two distinct $\arc$-alternating matchings $M_3 = L_1 \cup R_2$ and~$M_4 = L_2 \cup R_1$ with~$M_1 \cup M_2 = L_1 \cup L_2 \cup R_1 \cup R_2 = M_3 \cup M_4$ (our assumption that~$M_1$ and~$M_2$ have empty initial and final common $\arc$-alternating matching ensures that~$\{M_1, M_2\} \cap \{M_3, M_4\} = \varnothing$).
We can thus assume that the only~$0$ of $\mu(M_1 \cup M_2)$ is at its end.
Observe moreover that $\mu(M_1 \cup M_2)$ has no consecutive repeated values since each new position of~$M_1 \cup M_2$ is either in~$\{a\} \cup A$ (in which case the entry of $\mu(M_1 \cup M_2)$ increases by $1$ or $2$) or in~$B \cup \{b\}$ (in which case the entry of $\mu(M_1 \cup M_2)$ decreases by $1$ or $2$).
We obtain the following four special cases of the statement illustrated in \cref{fig:matchingUnionLemma}\,(right), which admit a single $\arc$-alternating matching decomposition:
\begin{enumerate}
\item if $\mu(M_1 \cup M_2) = 10$, then $\{M_1, M_2\} = \{a' < b', \; \varnothing\}$ where $M_1 \cup M_2 = \{a' < b'\}$.
\item if $\mu(M_1 \cup M_2) = 210$, then $\{M_1, M_2\} = \{a' < b_1, \; a' < b_2\}$ where $M_1 \cup M_2 = \{a' < b_1 < b_2\}$.
\item if $\mu(M_1 \cup M_2) = 120$, then $\{M_1, M_2\} = \{a_1 < b', \; a_2 < b'\}$ where ${M_1 \cup M_2 = \{a_1 < a_2 < b'\}}$.
\item if $\mu(M_1 \cup M_2) = 2120$, then $\{M_1, M_2\} = \{a_1 < b_1 < a_2 < b_2, \; a_1 < b_2\}$ where $M_1 \cup M_2 = \{a_1 < b_1 < a_2 < b_2\}$.
\end{enumerate}
Assume now that we are not in these cases.
Let~$a_1 < \dots < a_k$ denote the elements of~$M_1 \cup M_2$ in~$\{a\} \cup A$ and~$b_1 < \dots < b_\ell$ denote the elements of~$M_1 \cup M_2$ in~$B \cup \{b\}$.
We distinguish again two cases, illustrated in \cref{fig:matchingUnionLemma}\,(left):
\begin{enumerate}
\addtocounter{enumi}{4}
\item if $\mu(M_1 \cup M_2)$ starts with $1$, then it starts at least with~$121$. Hence, we have~${k, \ell \ge 2}$ and $a_1 < a_2 < b_1 < b_2$. The two distinct pairs~$\{a_1 < b_1, \; a_2 < b_2\}$ and $\{a_1 < b_2, \; a_2 < b_1\}$ of $\arc$-alternating matchings covering~$\{a_1, a_2, b_1, b_2\}$ can both be completed with any decomposition of the elements of $(M_1 \cup M_2) \ssm \{a_1, a_2, b_1, b_2\}$ into two $\arc$-alternating~matchings.
\item if $\mu(M_1 \cup M_2)$ starts with $2$, then it starts at least with~$2121$. Hence, we have~$k, \ell \ge 4$ and~$a_1 = a_2 < b_1 < a_3 < b_2 < a_4 < b_3 \le b_4$. The two distinct pairs~$\{a_1 < b_4, \; a_2 < b_1 < a_3 < b_2 < a_4 < b_3\}$ and $\{a_1 < b_2 < a_4 < b_4, \; a_2 < b_1 < a_3 < b_3\}$ of $\arc$-alternating matchings covering~$\{a_1, a_2, a_3, a_4, b_1, b_2, b_3, b_4\}$ can both be completed with any decomposition of the elements of $(M_1 \cup M_2) \ssm \{a_1, a_2, a_3, a_4, b_1, b_2, b_3, b_4\}$ into two $\arc$-alternating~matchings.
\end{enumerate}
In both cases (5) and~(6), we obtained two disjoint pairs of $\arc$-alternating matchings $\{M_1, M_2\}$ and $\{M_3, M_4\}$ such that the multisets $M_1 \cup M_2$ and $M_3 \cup M_4$ coincide.
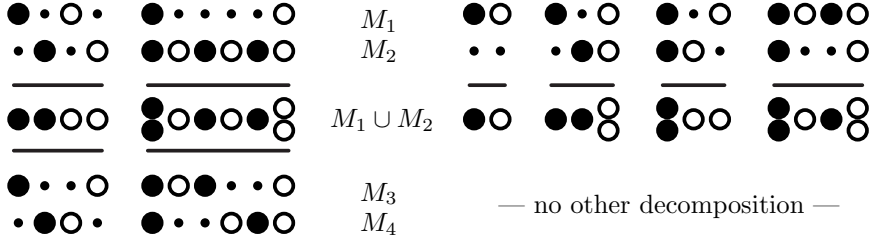
\begin{figure}
	\capstart
	\centerline{
		\begin{overpic}[scale=1.4]{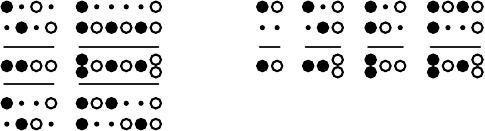}
			\put(41,24){$M_1$}
			\put(41,20.5){$M_2$}
			\put(37.5,12.5){$M_1 \cup M_2$}
			\put(41,4){$M_3$}
			\put(41,0.5){$M_4$}
			\put(57,2.5){--- no other decomposition ---}
		\end{overpic}
	}
	\caption{The union of any pair~$\{M_1,M_2\}$ of $\arc$-alternating matchings decomposes into another disjoint pair of $\arc$-alternating matchings (left), except for four special families (right).}
	\label{fig:matchingUnionLemma}
\end{figure}
\end{proof}

With the combinatorial tool of \cref{lem:matchingUnionLemma} in hand, we are ready to prove \cref{prop:shardPolytopeFan}.

\begin{proof}[Proof of \cref{prop:shardPolytopeFan}]
Fix an arc~$\arc \eqdef (a, b, A, B) \in \arcs_n$ and a vector~$\b{t} \in \R^n$.
For an $\arc$-alternating matching~$M = a_1 < b_1 < \dots < a_k < b_k$, we let~$\tau(M) \eqdef \dotprod{\b{t}}{\chi(M)} = \sum_{i = 1}^k \b{t}_{a_i} - \b{t}_{b_i}$ denote the scalar product of the vector~$\b{t}$ with the characteristic vector~$\chi(M)$ of~$M$.

Assume first that~$\b{t}$ belongs to the interior of the shard~$\shard(\arc)$.
Recall from \cref{subsec:shards} that the coordinates of~$\b{t}$ satisfy the inequalities~$\b{t}_{a'} < \b{t}_a = \b{t}_b < \b{t}_{b'}$ for all~$a < a', b' < b$ such that~$a' \in A$ and~$b' \in B$.
Since for all~$i \in [k]$, we have~$a \le a_i < b_i \le b$ and~$a_i \in \{a\} \cup A$ while~$b_i \in B \cup \{b\}$, we obtain that~$\b{t}_{a_i} \le \b{t}_{b_i}$ with equality if and only if~$a = a_i$ and~$b = b_i$.
Therefore, we obtain that~$\tau(M) \le 0$ with equality if and only if~$M = \varnothing$ or~$M = a < b$.
We conclude that~$\tau$ is maximized precisely by two $\arc$-alternating matchings, so that~$\b{t}$ belongs to the union of the walls of the normal fan of the shard polytope~$\shardPolytope$.

Assume now that $\b{t}$ belongs to the union of the walls of the normal fan of the shard polytope~$\shardPolytope$.
Let~$M_1$ and~$M_2$ be two $\arc$-alternating matchings such that~$\b{t}$ belongs to the normal cone of the edge of~$\shardPolytope$ with endpoints~$\chi(M_1)$ and~$\chi(M_2)$.
Then the midpoint between~$\chi(M_1)$ and~$\chi(M_2)$ is not the midpoint of any other pair of vertices~$\chi(M_3)$ and~$\chi(M_4)$ of~$\shardPolytope$.
It follows that~$M_1 \cup M_2 \ne M_3 \cup M_4$ for any~$\{M_3, M_4\}$ disjoint from~$\{M_1, M_2\}$.
We thus get that~$\{M_1, M_2\}$ is one of the pairs of $\arc$-alternating matchings described in \cref{lem:matchingUnionLemma}.
We now distinguish~four~cases:
\begin{enumerate}
	\item Assume $M_1 = H < a' < b' < T$ and~$M_2 = H < T$. Consider the arc~$\arc' \eqdef (a', b', A', B')$, where~$A' \eqdef A \cap {]a',b'[}$ and~$B' \eqdef B \cap {]a',b'[}$. We then have~${0 = \tau(M_1) - \tau(M_2) = \b{t}_{a'} - \b{t}_{b'}}$. Furthermore, for any~$i \in A'$, the $\arc$-alternating matching~$M_3 = H < i < b' < T$ satisfies ${0 < \tau(M_1) - \tau(M_3) = \b{t}_{a'} - \b{t}_i}$. Similarly, for any $j \in B'$, the $\arc$-alternating matching ${M_4 = H < a' < j < T}$ satisfies ${0 < \tau(M_2) - \tau(M_4) = \b{t}_j - \b{t}_{a'}}$. We get ${\b{t}_i < \b{t}_{a'} = \b{t}_{b'} < \b{t}_j}$ for any~$i \in A'$ and~$j \in B'$. Therefore~$\b{t}$ belongs to the shard~$\shard(\arc')$.
	\item Assume $M_1 = H < a' < b_1 < T$ and~$M_2 = H < a' < b_2 < T$. Consider the arc~$\arc' \eqdef (b_1, b_2, A', B')$, where~$A' \eqdef A \cap {]b_1, b_2[}$ and~$B' \eqdef B \cap {]b_1, b_2[}$. We then have ${0 = \tau(M_1) - \tau(M_2) = \b{t}_{b_2} - \b{t}_{b_1}}$. We obtain that ${\b{t}_i < \b{t}_{a'} = \b{t}_{b'} < \b{t}_j}$ for any~$i \in A'$ and ${j \in B'}$ by considering the two $\arc$-alternating matchings~${M_3 = H < a' < b_1 < i < b_2 < T}$ and $M_4 = H < a' < j < T$. Therefore~$\b{t}$ belongs to the shard~$\shard(\arc')$.
	\item Assume $M_1 = H < a_1 < b' < T$ and~$M_2 = H < a_2 < b' < T$. Consider the arc~$\arc' \eqdef (a_1, a_2, A', B')$ where~$A' \eqdef A \cap {]a_1, a_2[}$ and~$B' \eqdef B \cap {]a_1, a_2[}$. We then have ${0 = \tau(M_1) - \tau(M_2) = \b{t}_{a_1} - \b{t}_{a_2}}$. We obtain that ${\b{t}_i < \b{t}_{a'} = \b{t}_{b'} < \b{t}_j}$ for any~$i \in A'$ and ${j \in B'}$ by considering the two $\arc$-alternating matchings~$M_3 = H < i < b' < T$ and $M_4 = H < a_1 < j < a_2 < b' < T$. Therefore~$\b{t}$ belongs to the shard~$\shard(\arc')$.
	\item Assume that~$M_1 = H < a_1 < b_1 < a_2 < b_2 < T$ and~$M_2 = H < a_1 < b_2 < T$. Consider the arc~$\arc' \eqdef (b_1, a_2, A', B')$ where~$A' \eqdef A \cap {]b_1, a_2[}$ and~$B' \eqdef B \cap {]b_1, a_2[}$. We then have ${0 = \tau(M_1) - \tau(M_2) = \b{t}_{a_2} - \b{t}_{b_1}}$. We obtain that ${\b{t}_i < \b{t}_{a'} = \b{t}_{b'} < \b{t}_j}$ for any~$i \in A'$ and ${j \in B'}$ by considering the two $\arc$-alternating matchings~$M_3 = H < a_1 < b_1 < i < b_2 < T$ and $M_4 = H < a_1 < j < a_2 < b_2 < T$. Therefore~$\b{t}$ belongs to the shard~$\shard(\arc')$.
\end{enumerate}
In all four cases, we conclude that~$\b{t}$ belongs to a shard~$\shard(\arc')$ for some arc~$\arc'$ forcing~$\arc$.
\end{proof}

Using \cref{lem:matchingUnionLemma}, one can also explicitly describe the normal cones of the vertices and edges of the shard polytope~$\shardPolytope$.
We have postponed these descriptions to \cref{subsec:normalConeDescriptions} as they are a bit tedious and not essential for the rest of the paper.

%%%%%%%%

\subsection{Quotientopes from shard polytopes}
\label{subsec:shardsumotopes}

We now construct polytopal realizations of quotient fans in the same spirit as in \cref{subsec:MinkowskiSumAssociahedra}, but using shard polytopes rather than associahedra as elementary summands.
The following statement, announced in \cref{coro:main3}, immediately follows from \cref{prop:shardPolytopeFan}.

\begin{figure}[b]
	\capstart
	\centerline{\includegraphics[scale=.7]{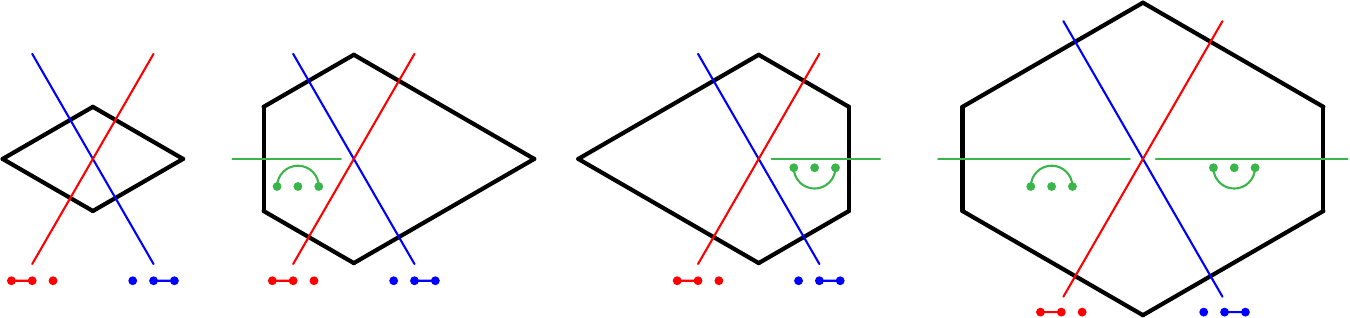}}
	\caption{The Minkowski sums~$\shardPolytope[\arcs]$ for all arc ideals~$\arcs \subseteq \arcs_3$ containing the basic arcs.}
	\label{fig:shardPolytopeSums3}
\end{figure}

\begin{corollary}
\label{coro:MinkowskiSumShardPolytopes}
For any arc ideal~$\arcs \subseteq \arcs_n$, the quotient fan~$\Fan_\arcs$ is the normal fan of the Minkowski sum~$\shardPolytope[\arcs] \eqdef \sum_{\arc \in \arcs} \shardPolytope$ of the shard polytopes~$\shardPolytope$ of all~arcs~${\arc \in \arcs}$.
\end{corollary}

The resulting polytopes are illustrated in \cref{exm:ZonoMinkowskiSum,exm:LodayAssoMinkowskiSum,exm:HohlwegLangeAssoMinkowskiSum,exm:weirdPermMinkowskiSum} below and \cref{fig:shardPolytopeSums3,fig:associahedra,fig:weirdPermutahedron}.
Observe that the quotient fan~$\Fan_\arcs$ is actually the normal fan of any Minkowski sum~$\sum_{\arc \in \arcs} \coeffSP_\arc \, \shardPolytope$ with~$\coeffSP_\arc > 0$ for any~$\arc \in \arcs$.
We stick with coefficients~$\coeffSP_\arc = 1$ here as this convention recovers the original constructions of~\cite{Loday, HohlwegLange} as described in \cref{exm:LodayAssoMinkowskiSum,exm:HohlwegLangeAssoMinkowskiSum}.
As they are not needed in the next sections, we have postponed the vertex and facet descriptions of the polytopes~$\shardPolytope[\arcs]$ to \cref{subsec:vertexFacetDescriptionsShardSumotopes}.
We will also discuss in~\cref{prop:RHStoSP} another way to derive inequality descriptions for the polytopes~$\shardPolytope[\arcs]$, passing through Minkowski decompositions of shard polytopes into sums and differences of faces of the standard simplex.
At the moment, we just want to observe here that the symmetries of the shard polytopes given in \cref{prop:symmetriesShardPolytope} translate to the following symmetries of their Minkowski sums, illustrated in \cref{fig:shardPolytopeSums3,fig:associahedra,fig:weirdPermutahedron}.

\begin{corollary}
\label{lem:symmetriesMinkowskiSumShardPolytopes}
If an arc ideal~$\arcs$ is $\phi$- or $\psi$-invariant, then the Minkowski sum~$\shardPolytope[\arcs]$ is $\Phi$- or $\Psi$-invariant up to a translation.
If~$\arcs$ is centrally symmetric, then~$\shardPolytope[\arcs] = \Phi \circ \Psi(\shardPolytope[\arcs])$.
\end{corollary}

\begin{example}
\label{exm:ZonoMinkowskiSum}
For basic arcs, the $(i, i+1, \varnothing, \varnothing)$-alternating matchings are~$\varnothing$ and~$\{i, i+1\}$, thus the shard polytope $\shardPolytope[i, i+1, \varnothing, \varnothing]$ is just the segment~$[0, \b{e}_i - \b{e}_{i+1}]$.
For the ideal of basic arcs $\arcs_\textrm{rec} \eqdef \set{(i, i+1, \varnothing, \varnothing)}{i \in [n-1]}$, we get the parallelotope~${\shardPolytope[\arcs_\textrm{rec}] = \sum_{i \in [n-1]} [0, \b{e}_i - \b{e}_{i+1}]}$.
\end{example}

\begin{example}[Tamari]
\label{exm:LodayAssoMinkowskiSum}
Consider the sylvester congruence and the Tamari lattice of \cref{exm:sylvesterCongruence,exm:noncrossingPartitions,exm:LodayAsso,exm:shardsAsso,}.
For up arcs, the~$(a, b, {]a,b[}, \varnothing)$-alternating matchings are given by~$\varnothing$ and~$\{i, b\}$ for~$a \le i < b$, thus the shard polytope $\shardPolytope[{a, b, {]a,b[}, \varnothing}]$ is the translate of the standard simplex~$\simplex_{[a,b]}$ by the vector~$-\b{e}_b$.
For the ideal of up arcs~$\arcs_\textrm{sylv} = \set{(a, b, {]a,b[}, \varnothing)}{1 \le a < b \le n}$, the Minkowski sum~$\shardPolytope[\arcs_\textrm{sylv}]$ is thus the translate by the vector~$-\sum_{i \in [n]} i \, \b{e}_i$ of J.-L.~Loday's associahedron~\cite{Loday} described in \cref{exm:LodayAsso} and illustrated in \cref{fig:associahedra}\,(left).
This Minkowski decomposition into the faces of the standard simplex corresponding to the intervals of~$[n]$ was described in~\cite{Postnikov}.

\begin{figure}
	\capstart
	\centerline{
		\begin{tabular}{c@{\quad}c@{}c@{\quad}c}
		\includegraphics[scale=.32]{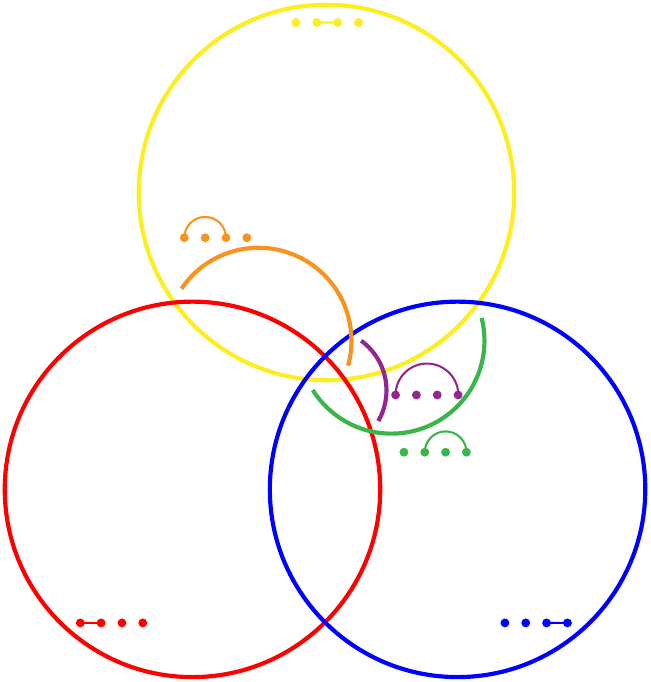} & \includegraphics[scale=.32]{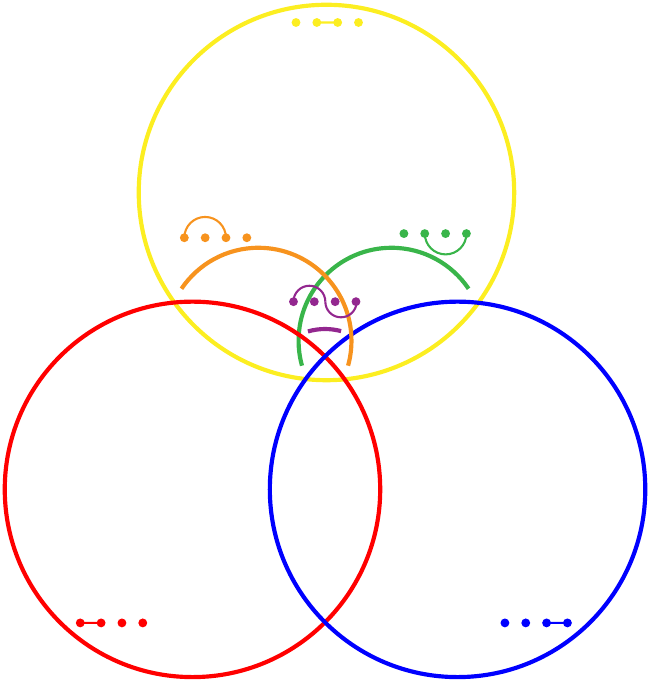} & \includegraphics[scale=.32]{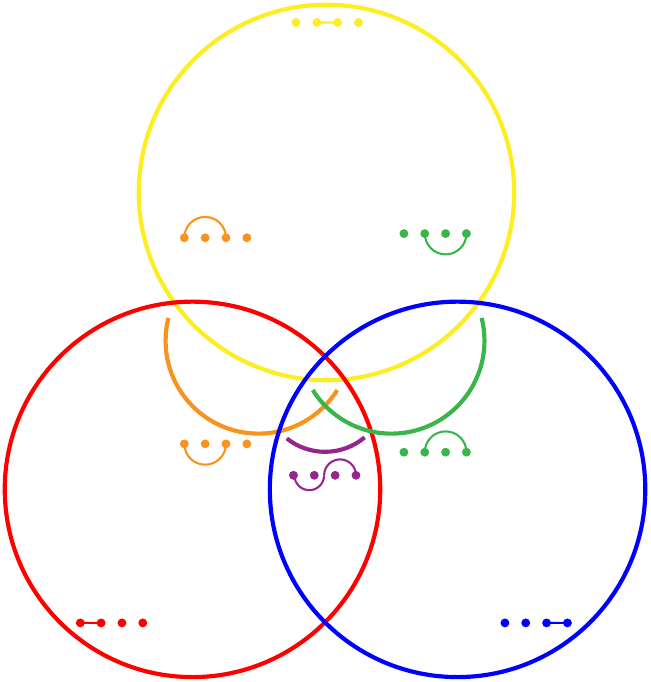} & \includegraphics[scale=.32]{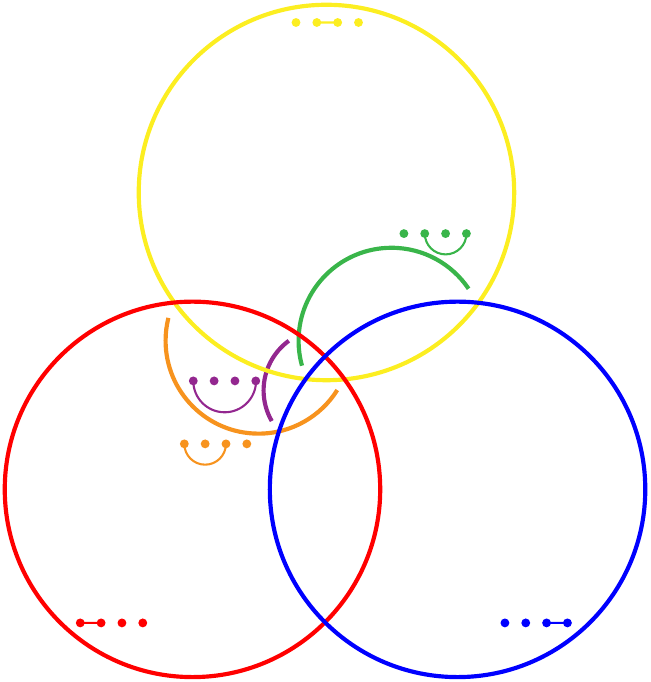} \\[.3cm]
		\includegraphics[scale=.4]{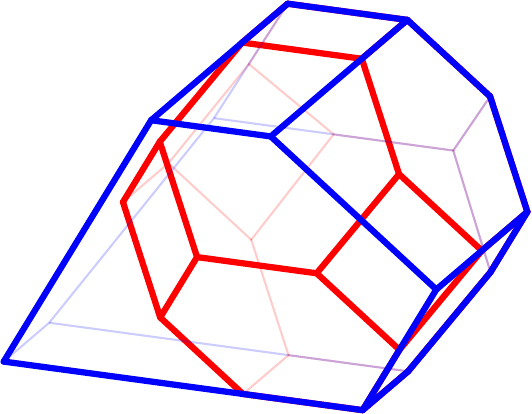} & \hspace*{.3cm}\includegraphics[scale=.4]{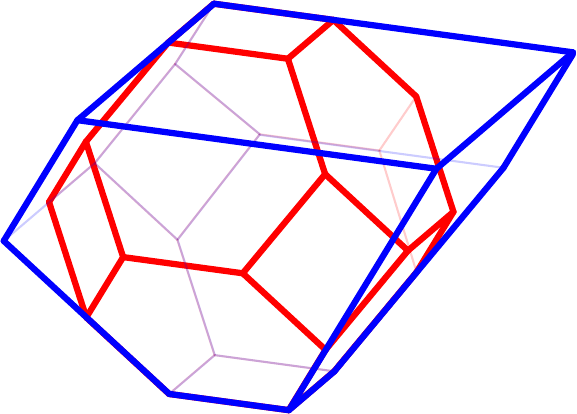} & \includegraphics[scale=.4]{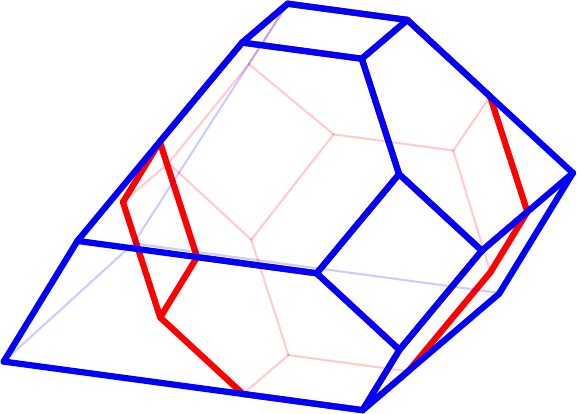}\hspace*{.3cm} & \includegraphics[scale=.4]{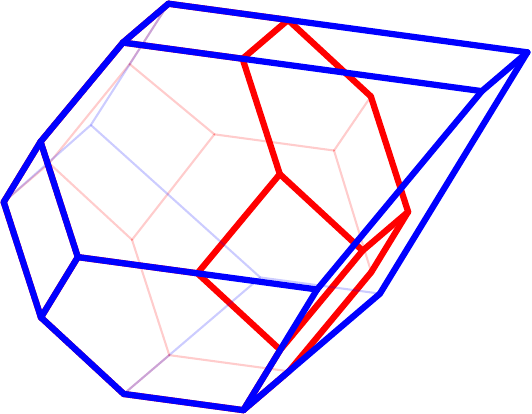}
		\end{tabular}
	}
	\caption{The associahedra (blue) obtained as Minkowski sums of shard polytopes coincide with the associahedra constructed in~\cite{HohlwegLange} by deleting inequalities in the facet description of the permutahedron~$\Perm[4]$ (red).}
	\label{fig:associahedra}
\end{figure}
\end{example}

\begin{example}[Cambrian]
\label{exm:HohlwegLangeAssoMinkowskiSum}
For the $\arc$-Cambrian congruence of \cref{exm:CambrianCongruences}, the Minkowski sum~$\shardPolytope[\arcs_\arc]$ is actually the translate by the vector~$-\sum_{i \in [a,b]} (i-a+1) \, \b{e}_i$ of C.~Hohlweg and C.~Lange's associahedron~$\Asso[\arc]$ described in \cref{exm:HohlwegLangeAsso} and  illustrated in \cref{fig:associahedra} (a formal proof of this affirmation is given in \cref{exm:HohlwegLangeAssoVertexFacetDescription}).
In fact, this Minkowski decomposition of the associahedron~$\Asso[\arc]$ already appeared in the context of brick polytopes in~\cite{PilaudSantos-brickPolytope} (see also \cref{rem:shardPolytopesAlreadyExisted}).
Alternative decompositions of the Cambrian associahedra of~\cite{HohlwegLange} as Minkowski sums and differences of faces of the standard simplex were also studied by C.~Lange in~\cite{Lange}, see \cref{coro:LangeDecomposition,coro:LangeDecompositionCambrian}.
\end{example}

\begin{example}
\label{exm:weirdPermMinkowskiSum}
For the ideal of all arcs~$\arcs_n$, the Minkowski sum of all shard polytopes gives a realization of the braid fan~$\Fan_n$.
See \cref{fig:weirdPermutahedron} for a $3$-dimensional example.
Although it is not the convex hull of all permutations of a given point as the classical permutahedron~$\Perm$, the resulting polytope has clearly a left-right and up-down symmetry given by \cref{lem:symmetriesMinkowskiSumShardPolytopes}.
\cref{coro:standardPermutahedronNotMinkowskiSumShardPolytopes} below shows that the classical permutahedron~$\Perm$ is not a Minkowski sum of dilated shard polytopes for~$n \ge 4$, and \cref{coro:permutahedronSPcoordinates} below decomposes the  classical permutahedron~$\Perm$ as a Minkowski sum and difference of dilated shard polytopes.

\begin{figure}
	\capstart
	\centerline{\raisebox{2cm}{\includegraphics[scale=.4]{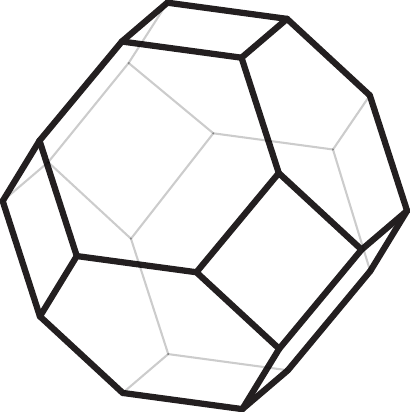}} \quad \includegraphics[scale=.4]{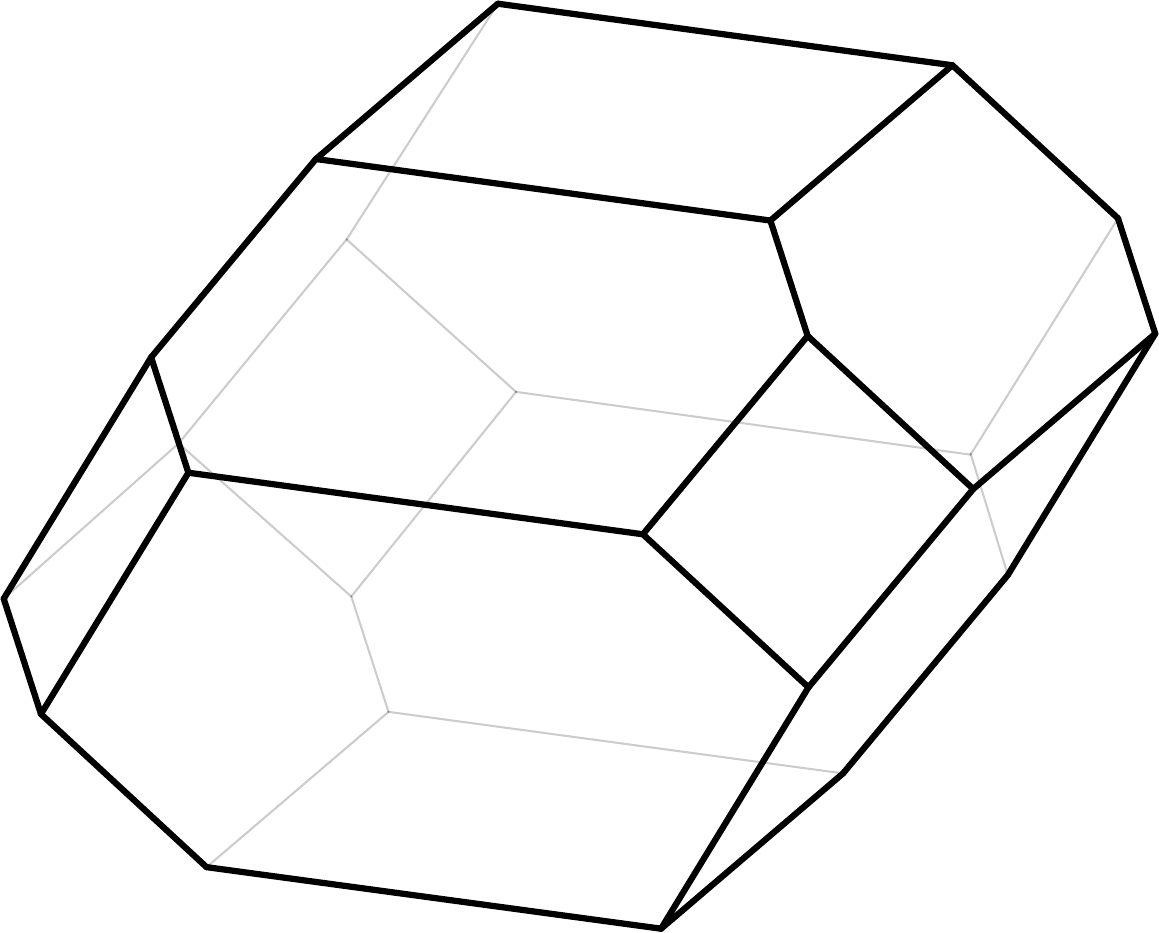} \quad \raisebox{1.6cm}{\includegraphics[scale=.4]{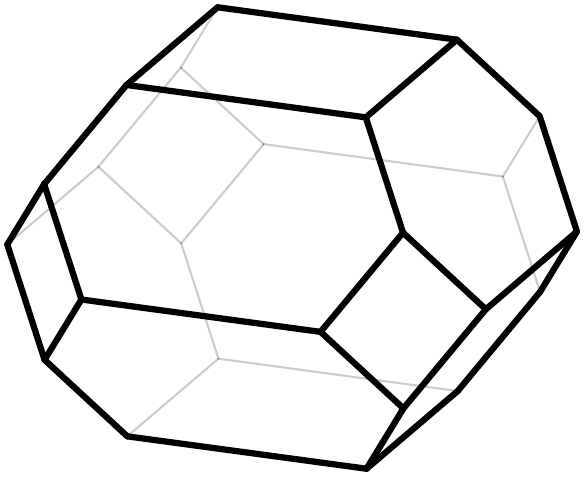}}}
	\caption{The standard permutahedron~$\Perm[4] \eqdef \conv\set{(\sigma_1, \dots, \sigma_n)}{\sigma \in \fS_4}$ (left), the Minkowski sum~$\shardPolytope[\arcs_4]$ of the shard polytopes of all arcs of~$\arcs_4$ (middle), and the Minkowski sum of the shard polytopes of all forcing minimal arcs of~$\arcs_4$ (right).}
	\label{fig:weirdPermutahedron}
\end{figure}
\end{example}

\begin{remark}
We will show in \cref{prop:quotientopesMinkowskiSumsShardPolytopes} that any quotientope constructed in~\cite{PilaudSantos-quotientopes} is a Minkowski sum of dilated shard polytopes.
\end{remark}

%%%%%%%%

\subsection{A Minkowski identity on shard polytopes}
\label{subsec:MinkowskiIdentityShardPolytopes}

In view of \cref{exm:weirdPermMinkowskiSum,fig:weirdPermutahedron}, it is natural to wonder whether the standard permutahedron~$\Perm$ can be obtained as a Minkowski sum of dilated shard polytopes.
This is the case for~$n = 3$ (see \cref{fig:submodularFunctions}), but we will prove in \cref{coro:standardPermutahedronNotMinkowskiSumShardPolytopes} that it is not the case for~$n \ge 4$.
The proof is based on the following Minkowski identity among pseudoshard polytopes, in the sense of ~\cref{rem:pseudoarcs}.
This identity is illustrated in \cref{fig:signedFormula}.

\begin{theorem}
\label{thm:inductiveMinkowskiSum}
Consider an arc~$\arc \eqdef (a, b, A, B) \in \arcs_n$ with $b-a \ge 2$, and $x \in {]a,b[}$.
Then 
\[
\shardPolytope[\arc_x^A] + \shardPolytope[\arc_x^B] = \shardPolytope[\arc_{\bar x}] + \shardPolytope[\arc_{\le x}] + \shardPolytope[\arc_{\ge x}],
\]
where
\begin{itemize}
 \item $\arc_x^A \eqdef (a, b, A \cup \{x\}, B \ssm \{x\})$,
 \item $\arc_x^B \eqdef (a, b, A \ssm \{x\}, B \cup \{x\})$,
 \item $\arc_{\bar x} \eqdef (a, b, A \ssm \{x\}, B \ssm \{x\})$ (a pseudoshard, see~\cref{rem:pseudoarcs}),
 \item $\arc_{\le x} \eqdef (a, x, A \cap {]a,x[}, B \cap {]a,x[})$, and
 \item $\arc_{\ge x} \eqdef ( x,b, A \cap {]x,b[}, B \cap {]x,b[})$.
\end{itemize}
\end{theorem}

\begin{proof}
Let $M_x^A$ and $M_x^B$ be alternating matchings of $\arc_x^A$ and $\arc_x^B$, respectively. We say that $M_x^A$ and $M_x^B$ are splittable if $M_x^A\cap [x,b]$ and $M_x^B\cap [a,x]$ are alternating matchings (i.e. have an even number of elements), respectively. 

If $M_x^A$ and $M_x^B$ are both splittable, then set 
$M_{\bar x}=\left(M_x^A\cap {[a,x[}\right)\cup  \left(M_x^B\cap {]x,b]}\right)$,  
$M_{\le x}= M_x^B\cap {[a,x]}$, and 
$M_{\ge x}=M_x^A\cap {[x,b]}$.
If $M_x^A$ is splittable and $M_x^B$ is not, then set 
$M_{\bar x}=M_x^B$,
$M_{\le x}= M_x^A\cap {[a,x[}$, and
$M_{\ge x}=M_x^A\cap {[x,b]}$.
If $M_x^A$ is not splittable and $M_x^B$ is, then set 
$M_{\bar x}=M_x^A$, $M_{\le x}= M_x^B\cap {[a,x]}$, and $M_{\ge x}=M_x^B\cap {]x,b]}$.
If neither $M_x^A$ nor $M_x^B$ are splittable, then set
$M_{\bar x}=M_x^A$, 
$M_{\le x}= \left(M_x^B\cap {[a,x]}\right)\cup \{x\}$, and  
$M_{\ge x}=\{x\}\cup \left(M_x^B\cap {]x,b]}\right)$.

In each case we have that $M_{\bar x}$, $M_{\le x}$, and $M_{\ge x}$ are alternating matchings of $\arc_{\bar x}$, $\arc_{\le x}$, and $\arc_{\ge x}$, respectively, and that
\[
\chi(M_x^A) + \chi(M_x^B) = \chi(M_{\bar x}) + \chi(M_{\le x}) + \chi(M_{\ge x}),
\]
which shows the inclusion from left to right. 
 (All the verifications are straightforward except for maybe the last one, where there is a cancelation of $\b{e}_{x}$ and $-\b{e}_{x}$ coming from $\chi(M_{\le x})$ and $\chi(M_{\ge x})$.)
The other inclusion is similar.
\end{proof}

\begin{figure}
	\capstart
	\centerline{
		\begin{tabular}{c@{}c@{}c@{}c@{}c@{}c@{}c@{}c@{}c@{}c@{}c}
		\raisebox{-1.2cm}{\includegraphics[scale=.25]{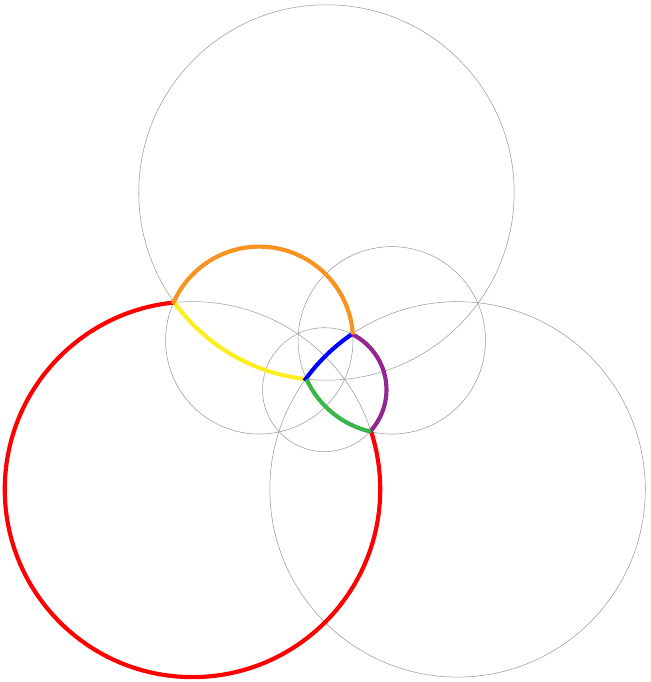}} & $\cup$ & \raisebox{-1.2cm}{\includegraphics[scale=.25]{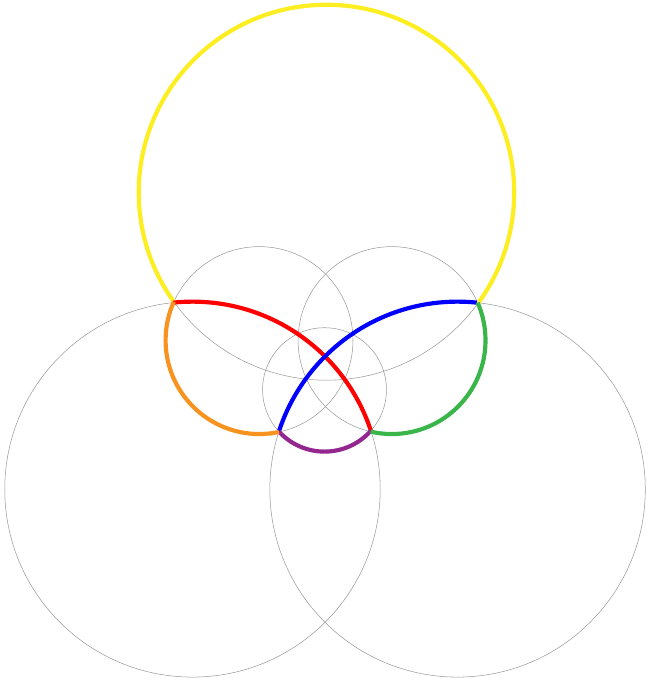}} & $=$ & \raisebox{-1.2cm}{\includegraphics[scale=.25]{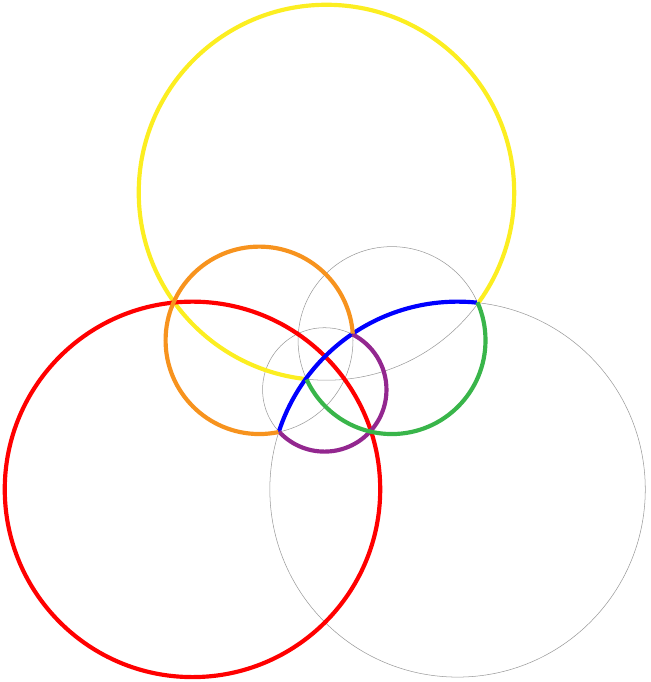}} & $=$ & \raisebox{-1.2cm}{\includegraphics[scale=.25]{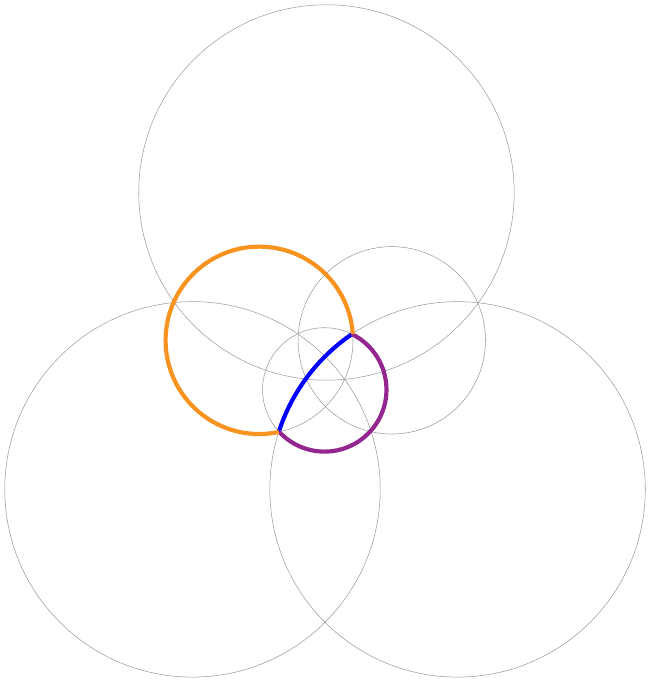}} & $\cup$ & \raisebox{-1.2cm}{\includegraphics[scale=.25]{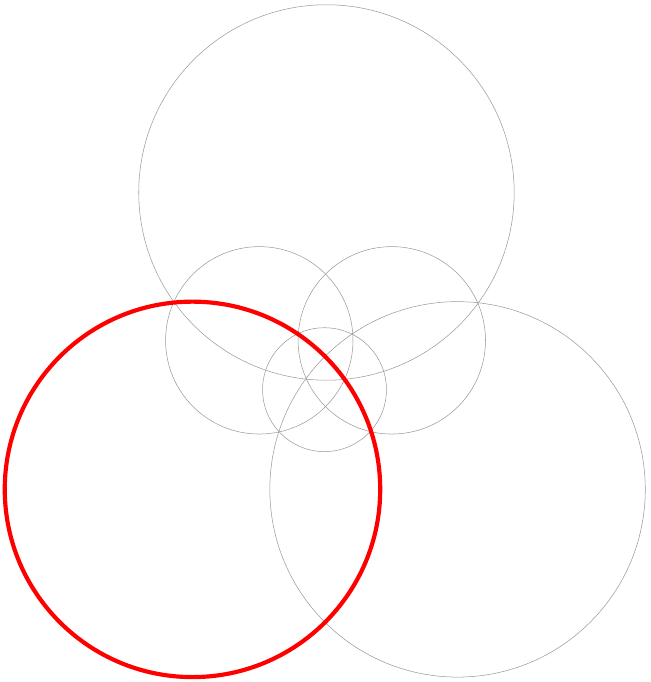}} & $\cup$ & \raisebox{-1.2cm}{\includegraphics[scale=.25]{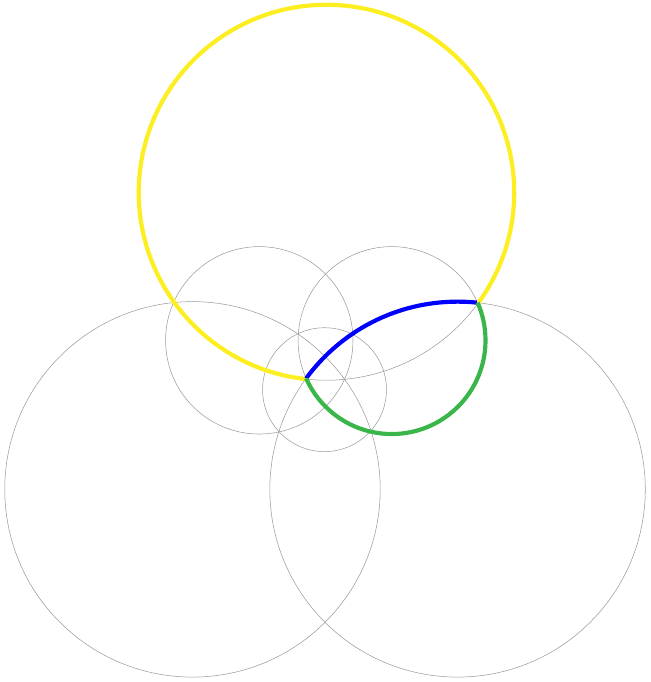}} \\[1.5cm]
		\raisebox{-.6cm}{\includegraphics[scale=.6]{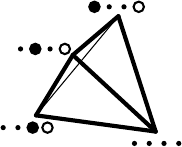}} & $+$ & \raisebox{-.6cm}{\includegraphics[scale=.6]{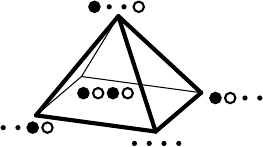}} & $=$ & \hspace{-.2cm}\raisebox{-1.8cm}{\includegraphics[scale=.6]{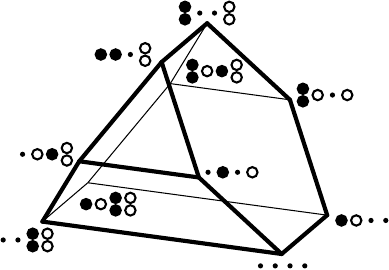}}\hspace{-.2cm} & $=$ & \raisebox{-.6cm}{\includegraphics[scale=.6]{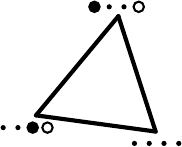}} & $+$ & \raisebox{-.6cm}{\includegraphics[scale=.6]{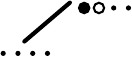}} & $+$ & \raisebox{-.6cm}{\includegraphics[scale=.6]{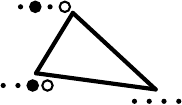}} \\[-1cm]
		\raisebox{-.25cm}{\includegraphics[scale=.6]{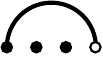}} & & \raisebox{-.25cm}{\includegraphics[scale=.6]{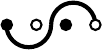}} & & & & \raisebox{-.25cm}{\includegraphics[scale=.6]{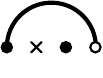}} & & \raisebox{-.25cm}{\includegraphics[scale=.6]{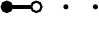}} & & \raisebox{-.25cm}{\includegraphics[scale=.6]{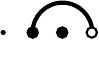}} \\[.5cm]
		\end{tabular}
	}
	\caption{An example of the Minkowski identity of \cref{thm:inductiveMinkowskiSum}.}
	\label{fig:signedFormula}
\end{figure}

Note that \cref{thm:inductiveMinkowskiSum} provides an inductive definition of shard polytopes as a Minkowski sum and difference of pseudoshard polytopes of up arcs, which are simplices as shown in~\cref{exm:LodayAssoMinkowskiSum}. As it turns out, this recovers a formula first found by F.~Ardila, C.~Benedetti and J.~Doker in the context of matroid polytopes~\cite{ArdilaBenedettiDoker}. We will revisit this formula in \cref{subsubsec:FStoSP}.

Before that, we use this formula to explain the behavior of shard polytopes reflected in Figure~\ref{fig:weirdPermutahedron}\,(right). It shows that, for $n=4$, the sum of all shard polytopes of the minimal arcs in the forcing poset (those with $a=1$ and $b=n$) is a permutahedron. The following statement shows that this is also the case for any~$n$.

\begin{corollary}
\label{coro:minimalPermutahedron}
The Minkowski sum $\sum_\arc \shardPolytope$ over all minimal arcs~$\arc$ in the forcing poset is a zonotope combinatorially equivalent to the permutahedron.
More precisely,
\begin{align*}
\sum_{A\subseteq{]1,n[}} \shardPolytope[1, n, A, {]1,n[} \ssm A] = \sum_{1 \le i < j \le n} 2^{\max(i-2, 0)} 2^{\max(n-j-1, 0)} [\b{0}, \b{e}_i-\b{e}_j].
\end{align*}
\end{corollary}

\begin{proof}
The proof is by induction on the dimension, using the pseudoshards of \cref{rem:pseudoarcs}. 
Note that, if $|S|>2$, then we can use \cref{thm:inductiveMinkowskiSum} to rewrite the sum of a pair of shard polytopes whose arcs coincide everywhere except for a point as a sum of smaller dimensional shard polytopes. 

Indeed, let $a \eqdef \min S$, $b \eqdef \max S$, and $x \eqdef \min \left( S \ssm\{a,b\} \right)$. Then 
\begin{align*}
\sum_{A \subseteq \left( S \ssm\{a,b\} \right)} \shardPolytope[a, b, A,  S \ssm (A \cup \{a,b\})] =
&   \sum_{A \subseteq \left( S \ssm\{a,b,x\} \right)} \shardPolytope[a, b, A,  S \ssm (A \cup \{a,b,x\})] \\
& + \sum_{A \subseteq \left( S \ssm\{a,b,x\} \right) } \shardPolytope[x, b, A,  S \ssm (A \cup \{a,b,x\})]\\
& + 2^{|S|-3} \shardPolytope[a, x, \varnothing, \varnothing].
\end{align*}
Now, by induction, each of these polytopes is itself a Minkowski sum of segments of the form $\shardPolytope[i, j, \varnothing, \varnothing] = [\b{0}, \b{e}_i-\b{e}_j]$, and for $S=[n]$ we recover a zonotope whose normal fan is the braid~fan~$\Fan_n$.

More precisely, applying this recursive formula to the shards corresponding to minimal arcs in the forcing order (those with $a=1$ and $b=n$), we get
\begin{align*}
\sum_{A \subseteq {]1,n[}} \shardPolytope[1, n, A, {]1,n[} \ssm A]
= & \sum_{1 \le i < j \le n} 2^{\max(i-2, 0)} 2^{\max(n-j-1, 0)} \shardPolytope[i, j, \varnothing, \varnothing]\\
= & \sum_{1 \le i < j \le n} 2^{\max(i-2, 0)} 2^{\max(n-j-1, 0)} [\b{0}, \b{e}_i-\b{e}_j].
\qedhere
\end{align*}
\end{proof}

\begin{corollary}
\label{coro:standardPermutahedronNotMinkowskiSumShardPolytopes}
For $n \ge 4$, the standard permutahedron~$\Perm$ is not a Minkowski sum of dilated shard polytopes.
\end{corollary}

\begin{proof}
Note that for each minimal arc $\arc$ in the forcing order, $\shardPolytope$ is the only shard polytope that contains this shard in its normal fan. Hence, in order to obtain the braid fan, all of them have to be included. 
 
Moreover, they all need to be included with the same weight. Indeed, in a Minkowski sum
\[
\polytope{P} = \sum_{\arc} \coeffSP_\arc \, \shardPolytope
\]
over the forcing minimal arcs, the edge of $\polytope{P}$ corresponding to the minimal shard~$\arc$ has length~$\coeffSP_\arc$. Since in the standard permutahedron all edges have the same length, all the weights $\coeffSP_\arc$ need to coincide in order to get a standard permutahedron.
 
But as seen in \cref{coro:minimalPermutahedron} the sum of all minimal arcs contains edges of length at least~$2$ when~$n \ge 4$.
And adding more shard polytopes can only increase the edge lengths.
\end{proof}

\cref{coro:standardPermutahedronNotMinkowskiSumShardPolytopes} shows that not all realizations of the quotient fans are obtained as Minkowski sums of shard polytopes.
In the next section, we will understand which ones are, and we will show that we can in fact obtain all deformed permutahedra if we allow to consider Minkowski sums and differences of dilated shard polytopes.
In particular, \cref{coro:permutahedronSPcoordinates} will describe the classical permutahedron~$\Perm$ as a signed Minkowski sum of dilated shard polytopes.

%%%%%%%%%%%%%%%%%%%%%%%%%%%%%%%%%%%%%%

\section{Minkowski geometry of shard polytopes}

In this section, we study further properties of shard polytopes with respect to Minkowski sums and differences.
First, we prove in \cref{subsec:typeCone} that shard polytopes are Minkowski indecomposable, so that they correspond to rays of the deformation cone of the permutahedron.
We then show in \cref{subsec:matroidPolytopes} that shard polytopes are matroid polytopes of certain series-parallel graphs.
We then study in \cref{subsec:virtualPolytopes} Minkowski sums and differences of shard polytopes in terms of faces of the standard simplex and \viceversa.
This enables us to prove in \cref{subsec:PSquotientopes} that any quotientope of~\cite{PilaudSantos-quotientopes} is indeed a Minkowski sum of dilated shard polytopes.
Finally, we compute in \cref{subsec:mixedVolumes} the (mixed) volumes of shard polytopes.

%%%%%%%%

\subsection{Type cones and shard polytopes}
\label{subsec:typeCone}

In \cref{coro:MinkowskiSumAssociahedra,coro:MinkowskiSumShardPolytopes} we constructed quotientopes as Minkowski sums of associahedra and of shard polytopes, respectively.
Note that associahedra can be themselves already represented as a Minkowski sum of shard polytopes.
It is natural to ask whether shard polytopes can be further decomposed into even simpler polytopes.
In this section we will see that the answer is negative, showing that shard polytopes are elementary geometric and combinatorial objects.

%%%

\subsubsection{Minkowski summands and type cone}

A \defn{weak Minkowski summand} of a polytope~$\polytope{P}$ is a polytope $\polytope{Q}$ such that there are a real $\lambda \ge 0$ and a polytope~$\polytope{R}$ such that ${\polytope{Q} + \polytope{R} = \lambda \polytope{P}}$.
The set of weak Minkowski summands of a polytope~$\polytope{P}$ has the structure of a polyhedral cone~\cite{Meyer, Shephard1963}, which is sometimes called the (closed) \defn{type cone}~\cite{McMullen-typeCone} or the \defn{deformation cone}~\cite{Postnikov} of the polytope~$\polytope{P}$.
It only depends on the normal fan~$\fan$ of~$\polytope{P}$.
When~$\fan$ is rational, it has an associated toric variety~\cite[Ch.~3]{CoxLittleSchenckToric}.
Its embeddings into projective space give rise to the~\defn{nef (numerically effective) cone}, which is equivalent to the type cone of~$\polytope{P}$~\cite[Sect.~6.3]{CoxLittleSchenckToric}.

There are several characterizations of weak Minkowski summands of~$\polytope{P}$, discussed for instance in the appendix of~\cite{PostnikovReinerWilliams}.

\begin{proposition}
The following are equivalent for two polytopes~$\polytope{P}$ and~$\polytope{Q}$:
\begin{enumerate}
\item $\polytope{Q}$ is a weak Minkowski summand of~$\polytope{P}$,
\item the normal fan of~$\polytope{Q}$ coarsens the normal fan of~$\polytope{P}$,
\item $\polytope{Q}$ can be obtained from~$\polytope{P}$ by parallelly translating its facets without moving past vertices,
\item $\polytope{Q}$ can be obtained from~$\polytope{P}$ by moving its vertices in such a way that all edge directions and orientations are preserved.
\end{enumerate}
\end{proposition}

Let $\b{G}$ be the $N \times n$ matrix whose rows are the outer normal vectors of the polytope~$\polytope{P} \subset \R^n$.
Then the rows of~$\b{G}$ also contain the outer normal vectors of the facets of any of its weak Minkowski summands, which are all of the form
\[
\polytope{P}_{\b{h}} \eqdef \set{\b{x} \in \R^n}{\b{G} \b{x} \le \b{h}}
\]
for some height vector~$\b{h} \in \R^N$.
The set of height vectors~$\b{h} \in \R^N$ for which~$\polytope{P}_{\b{h}}$ is a weak Minkowski summand of~$\polytope{P}$ is a closed polyhedral cone~\cite{McMullen-typeCone} (see for example~\cite{PadrolPaluPilaudPlamondon} for details on its description).
Its lineality subspace, which is induced by the translations of~$\R^n$, is the image of the matrix~$\b{G}$.
We can get rid of the lineality space by considering the projection onto the left kernel of~$\b{G}$.
We thus obtain a pointed polyhedral cone of dimension~$N-n$, the \defn{type cone}~$\ctypeCone(\fan)$ of the normal fan~$\fan$ of~$\polytope{P}$.
(Note that, in contrast to~\cite{PadrolPaluPilaudPlamondon}, we call the type cone the projection onto the kernel of~$\b{G}$.) 

An important property is that Minkowski sums of weak Minkowski summands translate to positive combinations in the type cone.
Thus, the rays of the type cone~$\ctypeCone(\fan)$ represent indecomposable Minkowski summands of~$\polytope{P}$ up to translation.
A polytope~$\polytope{Q} \subseteq \R^n$ is called \defn{indecomposable} if all its weak Minkowski summands are of the form $\lambda \polytope{Q} + \b{t}$ for some real~$\lambda \ge 0$ and~$\b{t} \in \R^n$.
We will use the following simple certificate of indecomposability which is a special case of an indecomposability criterion by P.~McMullen in~\cite{McMullen1987}.
We reproduce his proof adapted to our~special~case.

\begin{theorem}[{{\cite[Thm.~2]{McMullen1987}}}]
\label{thm:indecomposabilityCriterion}
If a polytope~$\polytope{P}$ has an indecomposable face $\polytope{F}$ such that every facet of $\polytope{P}$ shares at least a vertex with~$\polytope{F}$, then $\polytope{P}$ is indecomposable.
\end{theorem}

\begin{proof}
Let $\b{n}_1, \dots, \b{n}_m$ be the outer normal vectors to the facets of~$\polytope{P}$, whose $H$-representation is
\[
\polytope{P} \eqdef \set{\b{x} \in \R^n}{ \dotprod{\b{n}_i}{\b{x}} \le h_i \text{ for all } i \in [m]}.
\]
Then any of its Minkowski summands~$\polytope{Q}$ is of the form
\[
\polytope{Q} \eqdef \set{\b{x} \in \R^n}{ \dotprod{\b{n}_i}{\b{x}} \le g_i \text{ for all } i \in [m]},
\]
where we take the $g_i$ so that all the inequalities are tight.

Suppose that $\b{p}_1, \dots, \b{p}_k$ are the vertices of~$\polytope{F}$, and let $\polytope{G}$ and $\b{q}_1, \dots, \b{q}_k$ be the faces of $\polytope{Q}$ maximizing a normal vector of $\polytope{F}$ and $\b{p}_1, \dots, \b{p}_k$, respectively.
Then $\polytope{G}$ is a Minkowski summand~of~$\polytope{F}$ and hence it is of the form $\lambda \polytope{F} + \b{t}$ for some $\lambda \ge 0$ and $\b{t} \in \R^n$.
Therefore, $\b{q}_j = \lambda \b{p}_j +\b{t}$ for all~$j \in [k]$. 

If the facet of $\polytope{P}$ maximized by $\b{n}_i$ contains $\b{p}_j$, then the facet of $\polytope{Q}$ maximized by $\b{n}_i$ contains~$\b{q}_j$, and therefore 
\[
g_i = \dotprod{\b{n}_i}{\smash{\b{q}_j}} = \lambda h_i + \dotprod{\b{n_i}}{\b{t}}.
\] 
We conclude that $\polytope{Q} = \lambda \polytope{P} + \b{t}$.  
\end{proof}

%%%

\subsubsection{Shard polytopes are indecomposable deformed permutahedra}

The weak Minkowski summands of the permutahedron form a particularly interesting family, studied by A.~Postnikov under the name \defn{generalized permutahedra}~\cite{Postnikov, PostnikovReinerWilliams}.
Here, we prefer the name ``deformed permutahedra'' rather than ``generalized permutahedra'' as there are many generalizations of permutahedra.
They have a rich combinatorial, algebraic and geometric structure, and contain numerous polytopes that arise naturally in many different areas.
Note that quotientopes and shard polytopes are deformed permutahedra since their normal fan coarsens the braid fan.
The type cone of the permutahedron is usually described by the submodular inequalities~\cite{Postnikov,AguiarArdila}.
With the standard presentation in a codimension one subspace, we get the following classical description (see~\cite[Thm.~2.1]{ArdilaDoker} and~\cite[Thm.~12.3]{AguiarArdila} and the references therein for historical background).

\begin{proposition}
\label{prop:deformedPermutahedraZ}
A polytope is a deformed permutahedron if and only if it is of the form
\[
\bigset{\b{x} \in \R^n}{ \dotprod{\one}{\b{x}} = \b{h}([n]) \text{ and } \dotprod{\one_R}{\b{x}} \le \b{h}(R) \text{ for all } \varnothing \ne R \subsetneq [n]}
\]
for a submodular boolean function~$\b{h} : 2^{[n]} \to \R$, that is, $\b{h}(\varnothing) = 0$ and for all~$R, S \in 2^{[n]}$, 
\[
\b{h}(R) + \b{h}(S) \ge \b{h}(R \cup S) + \b{h}(R \cap S).
\]
The values~$\b{h}(R)$ are called the \defn{heights} of the deformed permutahedron.
\end{proposition}

\begin{example}
Type cones are high dimensional objects difficult to visualize.
We can, however, see the type cone~$\ctypeCone(\Fan_3)$ of the $2$-dimensional braid fan~$\Fan_3$ by intersecting it with a hyperplane.
The resulting polytope is a triangular bipyramid illustrated in \cref{fig:submodularFunctions}.
We have labeled the six facets of this polytope from~$1$ to~$6$ (the grey labels $3$ and~$6$ correspond to the two hidden facets of the type polytope), and the corresponding six walls of~$\Fan_3$ from~$1$ to~$6$ (bottom right corner).
We have located in the type polytope the shard polytopes of all arcs of~$\arcs_3$ together with different polytopes considered along the paper.
Note that the four shard polytopes are all vertices of the type polytope (see \cref{prop:SPindecomposable}) and form an affine basis of the space (see \cref{prop:shardPolytopeBasis}).

\begin{figure}
	\capstart
	\centerline{\includegraphics[scale=.8]{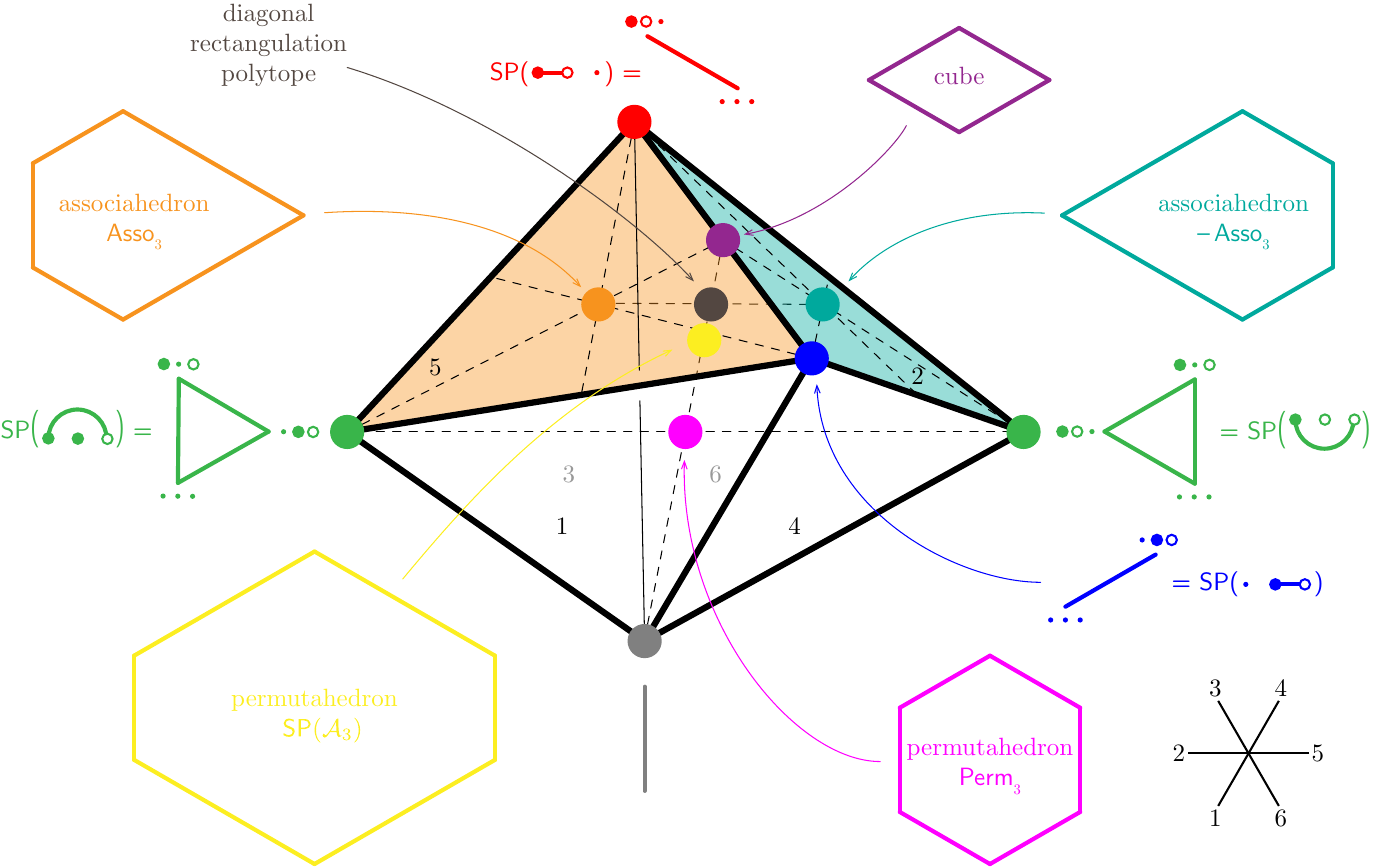}}
	\caption{The type cone of the braid fan~$\Fan_3$.}
	\label{fig:submodularFunctions}
\end{figure}
\end{example}

Indecomposable deformed permutahedra are intriguing objects.
It is very hard to characterize them, a problem first raised by J.~Edmonds in his seminal paper on polymatroids~\cite{Edmonds}.
A remarkable example is given by matroid polytopes of connected matroids~\cite{Nguyen} (see also~\cite[Sect.~7.2.]{StudenyKroupa2016}).
The same question has been extended to Coxeter deformed permutahedra and polymatroids, see \cite[Sect.~8]{ArdilaCastilloEurPostnikov} for a discussion on this topic.
As we will see in Section~\ref{subsec:matroidPolytopes}, shard polytopes are (translations of) matroid polytopes of certain connected matroids, and hence indecomposable, as stated in \cref{prop:main4}.
We provide here a direct proof with P.~McMullen's criterion.

\begin{proposition}
\label{prop:SPindecomposable}
For any arc~$\arc$, the shard polytope~$\shardPolytope$ is indecomposable.
\end{proposition}

\begin{proof}
Observe that
\begin{itemize}
\item the facet of~$\shardPolytope$ defined by the inequality~$\b{x}_{a'} \ge 0$ for some~$a' \in A$ (resp.~by~$\b{x}_{b'} \le 0$ for some~$b' \in B$) contains both vertices~$\chi(\varnothing)$ and~$\chi(\{a,b\})$ of~$\shardPolytope$,
\item the facet of~$\shardPolytope$ defined by the inequality~$\sum_{i \le r} \b{x}_i \ge 0$ for some $\arc$-rise~$r$ (resp.~by the inequality~$\sum_{i \le f} \b{x}_i \le 1$ for some $\arc$-fall~$f$) contains the vertex~$\chi(\varnothing)$ (resp.~$\chi(\{a,b\})$)~of~$\shardPolytope$.
\end{itemize}
Therefore, the edge joining~$\chi(\varnothing)$ to~$\chi(\{a,b\})$ touches all facets of~$\shardPolytope$.
The result thus follows directly from the simple criterion of \cref{thm:indecomposabilityCriterion}.
\end{proof}

Thus, shard polytopes correspond to certain rays of the submodular cone.
However, not all indecomposable deformed permutahedra are shard polytopes. Indeed, shard polytopes are precisely the rays of the submodular cone that belong to a face associated to a Cambrian fan.

\begin{theorem}
\label{thm:shardPolytopesRaysTypeCone}
For any arc~$\arc \in \arcs_n$, the shard polytopes of the arcs forcing~$\arc$ are precisely (representatives of) the rays of the type cone of the $\arc$-Cambrian fan~$\Fan_\arc$.
\end{theorem}

\begin{proof}
Recall that the $\arc$-Cambrian fan~$\Fan_\arc$ is the quotient fan that corresponds to the upper ideal~$\arcs_\arc$ of the arc poset~$(\arcs_n, \prec)$ generated by~$\arc$.
By \cref{coro:MinkowskiSumShardPolytopes}, the $\arc$-Cambrian fan~$\Fan_\arc$ refines the normal fan of the shard polytope~$\shardPolytope[\arc']$ for any arc~$\arc \prec \arc'$.
Since shard polytopes are indecomposable by \cref{prop:SPindecomposable}, this ensures that~$\shardPolytope[\arc']$ is a ray of the type cone~$\ctypeCone(\Fan_\arc)$ for each~$\arc \prec \arc'$.

It turns out that the type cone~$\ctypeCone(\Fan_\arc)$ of the $\arc$-Cambrian fan is simplicial.
This follows from works of~\cite{ArkaniHamedBaiHeYan,BazierMatteDouvilleMousavandThomasYildirim} as observed in~\cite{PadrolPaluPilaudPlamondon} (and actually extends to $\b{g}$-vector fans of any finite type cluster algebra with respect to arbitrary initial seeds, see~\cite{PadrolPaluPilaudPlamondon} and the references therein for Cambrian fans and cluster algebras).
It follows that the type cone~$\ctypeCone(\Fan_\arc)$ has as many rays as the number of rays of the $\arc$-Cambrian fan~$\Fan_\arc$ minus the dimension.
This is precisely the number of shards of~$\Fan_\arc$.

Alternatively, to see that shard polytopes give representatives of all rays of $\ctypeCone(\Fan_\arc)$, note that by the simpliciality of the type cone, they must generate one of its faces.
But $\shardPolytope[\arcs_\arc] = \sum_{\arc \prec \arc'} \shardPolytope[\arc']$ is a Minkowski sum with positive coefficients whose normal fan is the $\arc$-Cambrian fan, and hence represents a point in the relative interior of~$\ctypeCone(\Fan_\arc)$.
Therefore the face of the type cone~$\ctypeCone(\Fan_\arc)$ spanned by the shard polytopes must be the whole type cone.
\end{proof}

We get the following result as a direct consequence of the simpliciality of~$\ctypeCone(\Fan_\arc)$ and the description of its rays~\cite{PadrolPaluPilaudPlamondon}.

\begin{corollary}
\label{coro:uniqueDecompositionCambrian}
Any polytope whose normal fan is the $\arc$-Cambrian fan~$\Fan_\arc$ has a unique decomposition (up to translation) as a Minkowski sum of dilated shard polytopes~$\shardPolytope[\arc']$ for~$\arc'$ forcing~$\arc$.
\end{corollary}

\begin{remark}
\label{rem:shardPolytopesAlreadyExisted}
\cref{thm:shardPolytopesRaysTypeCone} connects shard polytopes to other interpretations of the rays of the type cone of the Cambrian fans:
\begin{itemize}
\item according to~\cite{BazierMatteDouvilleMousavandThomasYildirim}, shard polytopes are Newton polytopes of $F$-polynomials of cluster variables of acyclic type~$A$ cluster algebras~\cite{FominZelevinsky-ClusterAlgebrasI, FominZelevinsky-ClusterAlgebrasII, FominZelevinsky-ClusterAlgebrasIV},
\item according to~\cite{BrodskyStump,JahnLoweStump}, shard polytopes are brick polytope summands of certain sorting networks~\cite{PilaudSantos-brickPolytope, PilaudStump-brickPolytope}.
\end{itemize}
We skip all precise definitions here as these interpretations are not needed in the rest of this paper.
We are not aware that our vertex and facet descriptions from \cref{prop:shardPolytope} have been observed earlier for these polytopes.
\end{remark}

%%%%%%%%

\subsection{Matroid polytopes and shard polytopes}
\label{subsec:matroidPolytopes}

In this section, we show that any shard polytope is the matroid polytope of a series-parallel graph as stated in \cref{prop:main5}.
This will have strong consequences to decompose shard polytopes as Minkowski sums and differences of faces of the standard simplex in \cref{subsec:virtualPolytopes} and to compute the (mixed) volumes of shard polytopes in \cref{subsec:mixedVolumes}.

%%%

\subsubsection{Shard polytopes are matroid polytopes}

Let $\mat$ be a matroid on the ground set $[n]$ (see~\cite{Oxley} for an introduction to matroid theory).
Its \defn{matroid polytope}~$\polytope{P}_{\mat}\subset\R^n$ is the convex hull of the characteristic vectors of its bases,
\[
\polytope{P}_{\mat} \eqdef \conv \Bigset{\sum_{i \in B} \b{e}_i}{B \text{ basis of } \mat}.
\]
The following characterization gives a geometric axiomatization of matroids.

\begin{theorem}[{\cite[Thm.~4.1]{GelfandGoreskyMacPhersonSerganova1987}}]
A polytope is a matroid polytope if and only if all its vertices have $0/1$ coordinates and all its edges are translations of some vectors~$\b{e}_i-\b{e}_j$ with $i \ne j$.
\end{theorem}

As edge directions are not modified by translations, this provides directly the proof that shard polytopes are actually matroid polytopes.

\begin{corollary}
For any arc~$\arc\eqdef (a, b, A, B) \in \arcs_n$, the translated shard polytope~$\shardPolytope + \one_{B \cup \{b\}}$ is a matroid polytope.
\end{corollary}

%%%

\subsubsection{Series-parallel matroid polytopes}

We can give a precise description of these matroids, which are actually certain connected series-parallel graphic matroids.
Let us recall some terminology.
A graph is \defn{series-parallel} if it can be obtained from a single edge with distinct endpoints via the operations of series extension (replacing an edge by a path of length~$2$) and parallel extension (replacing an edge by two parallel edges with the same endpoints).
The \defn{(cycle) matroid} of a connected graph $G \eqdef (V,E)$ is the matroid on $E$ whose bases are the edge sets of spanning trees of $G$.
A matroid is \defn{graphic} if it is the cycle matroid of a graph, and \defn{series-parallel} if it is the cycle matroid of a series-parallel graph, see~\cite[Sect.~5.4]{Oxley}.
A matroid is \defn{connected} if every pair of distinct elements is contained in a common circuit.
Therefore, a graphic matroid is connected if every pair of distinct edges is contained in a common cycle, see~\cite[Sect.~6.2]{Oxley} (for a loopless graph without isolated vertices and at least three vertices, this is equivalent to the graph being $2$-connected).

\begin{definition}
\label{def:shardgraph}
For an arc $\arc \eqdef (a, b, A, B) \in \arcs_n$, let $\{a\} \cup A = \{a = a_1 < a_2 < \dots < a_{|A|+1}\}$ and $B \cup \{b\} = \{b_1 < b_2 < \dots < b_{|B|+1} = b\}$, and set $b_{0} = a-1$ for convenience.
Define the \defn{shard graph} $\Gamma_\arc$ to be the (multi-)graph with vertex set $[0, |B|+1]$ and
\begin{itemize}
\item for each $1 \le i \le |A|+1$, an edge labeled $a_i$ joining vertex~$k$ to vertex~$|B|+1$, where ${0 \le k \le |B|}$ is such that ${b_k < a_i < b_{k+1}}$,
\item for each $1 \le j \le |B|+1$, an edge labeled $b_j$ joining vertex~$j-1$ to vertex~$j$,
\item for each $k \in [n] \ssm [a,b]$, a loop labeled by~$k$ on vertex~$|B|+1$.
\end{itemize}
The \defn{shard matroid} of the arc~$\arc$ is the cycle matroid $\mat_\arc$ of $\Gamma_\arc$, whose ground set is $[n]$.
\end{definition}

This definition is illustrated in \cref{fig:seriesParallel}.
Note that the loops corresponding to~$k \in [n] \ssm [a,b]$ in the definition are only included to get a matroid with ground set~$[n]$ whose matroid polytope is embedded in~$\R^n$ like~$\shardPolytope + \one_{B \cup \{b\}}$.
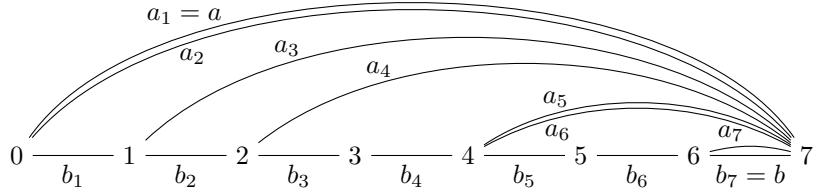
\begin{figure}
	\capstart
	\centerline{\input{seriesParallel}}
	\caption{The graph~$\Gamma_{\arc}$ for the arc~$\arc \eqdef (1, 14, \{2, 4, 6, 9, 10, 13\}, \{3, 5, 7, 8, 11, 12\})$.}
	\label{fig:seriesParallel}
\end{figure}

\begin{proposition}
\label{prop:2connected}
The graph~$\Gamma_\arc$ stripped of loops is a $2$-connected series-parallel graph.
\end{proposition}

\begin{proof}
If we do not append the loops to $\Gamma_\arc$, then it can be constructed from a single edge with a parallel extension for every element in $\{a\} \cup A$ and a series extension for every element in~$B$.
Indeed, treat the elements of $[a,b[$ in order, and for each element of $A \cup a$ do a parallel extension on $b$, and for each element of $B$ do a series extension on $b$ (the second edge of the replacing path is the one keeping the label $b$).
After the first step one obtains a double edge labeled with~$a$ and~$b$, and it is well-known that series-parallel graphs are $2$-connected when they can be obtained from a double edge via series and parallel extensions.
\end{proof}

\begin{proposition}
\label{prop:SPisMP}
The matroid polytope of the shard matroid $\mat_\arc$ is the translated shard polytope~${\translatedShardPolytope \eqdef \shardPolytope + \one_{B \cup \{b\}}}$.
\end{proposition}

\begin{proof}
To prove that $\shardPolytope + \one_{B \cup \{b\}}$ is the matroid polytope of $\mat_\arc$, we have to prove that a set $A'\cup B'$, with $A' \subseteq \{a\} \cup A$ and $B' \subseteq B \cup \{b\}$, is the support of an $\arc$-alternating matching if and only if $A' \cup (B \cup \{b\} \ssm B')$ indexes a spanning tree of $\Gamma_\arc$.
This can be proved by induction.
It is clear if $A = B = \varnothing$.
And now we can build $\arc$ by adding to $\{a,b\}$ the elements of $A \cup B$ by order.
If we add an element  $a_i \in A$, we keep the previous alternating matchings and we are allowed to add the pair $\{a_i,b\}$ to those matchings not using~$b$.
In terms of spanning trees, when we do the parallel extension the spanning trees are the previous spanning trees where we allowed to replace $b$ by $a_i$ whenever we had a spanning tree using~$b$.
If we add an element $b_j\in B$, then we keep the previous alternating matchings and we are allowed to replace $b$ by $b_j$ whenever $b$ was in the support.
In terms of spanning trees, when we do the series extension, we have to add $b_j$ to spanning trees using~$b$, and we have to add either $b$ or $b_j$ to spanning trees not using $b$.
\end{proof}

A result of H.~Q.~Nguyen~\cite[Thm.~2.1.5]{Nguyen} (see also~\cite[Sect.~7.2.]{StudenyKroupa2016}) characterizes indecomposable matroid polytopes.

\begin{theorem}[{\cite[Thm.~2.1.5]{Nguyen}}]
\label{thm:indecomposableMP}
The matroid polytope $\polytope{P}_\mat$ is indecomposable if and only if the matroid $\mat'$ obtained by removing all the loops from $\mat$ is connected.
\end{theorem}

Therefore, we obtain \cref{prop:SPindecomposable} as a corollary of \cref{prop:2connected,prop:SPisMP,thm:indecomposableMP}.

%%%%%%%%

\subsection{Virtual deformed permutahedra and shard polytopes}
\label{subsec:virtualPolytopes}

Under Minkowski addition, the set of convex polytopes in~$\R^n$ forms a commutative monoid with the cancellation property.
Its Grothendieck group is called the group of \defn{virtual polytopes}~\cite{PukhlikovKhovanskii}.
It is the group of formal differences of polytopes $\polytope{P} - \polytope{Q}$ under the equivalence relation $(\polytope{P}_1 - \polytope{Q}_1) = (\polytope{P}_2 - \polytope{Q}_2)$ whenever $\polytope{P}_1 + \polytope{P}_2 = \polytope{Q}_1 + \polytope{Q}_2$.
Note that the semigroup of polytopes is embedded into virtual polytopes via the map~$\polytope{P} \mapsto \polytope{P} - \{\b{0}\}$.

The group of virtual polytopes is extended to a real vector space $\virtualPolytopes$ via dilation. Indeed, for any real $\lambda \ge 0$, let $\lambda \polytope{P} \eqdef \set{\lambda \b{p}}{\b{p} \in \polytope{P}}$ be the dilation of $\polytope{P}$ by $\lambda$. For $\polytope{P} - \polytope{Q} \in \virtualPolytopes$ and $\lambda \in \R$, we set $\lambda (\polytope{P} - \polytope{Q}) \eqdef \lambda \polytope{P} - \lambda \polytope{Q}$ when $\lambda \ge 0$, and $\lambda (\polytope{P} - \polytope{Q}) \eqdef ((-\lambda) \polytope{Q}) - ((-\lambda) \polytope{P})$ when $\lambda < 0$. (Note in particular that $-\polytope{P}$ does not represent the reflection of $\polytope{P}$, but its group inverse.)

When $\polytope{Q}$ is a Minkowski summand of $\polytope{P}$, \ie when there is some polytope $\polytope{R}$ such that $\polytope{P} = \polytope{Q} + \polytope{R}$, then $\polytope{P} - \polytope{Q} = \polytope{R} - \{\b{0}\}$ and we say that $\polytope{R} = \set{\b{x} \in \R^n}{\b{x} + \polytope{Q} \subseteq \polytope{P}}$ is the \defn{Minkowski difference} of~$\polytope{P}$ and~$\polytope{Q}$.

Define the space $\VDP \subset \virtualPolytopes$ of \defn{virtual deformed permutahedra} as the vector subspace of virtual polytopes generated by the deformed permutahedra in~$\R^n$.
F.~Ardila, C.~Benedetti and J.~Doker proved in~\cite{ArdilaBenedettiDoker} that virtual deformed permutahedra admit the following Minkowski decompositions as sums and differences of simplices.

\begin{proposition}[\cite{ArdilaBenedettiDoker}]
\label{prop:faceSimplexBasis}
Any deformed permutahedron has a unique representation as a Minkowski sum and difference of dilated faces~$\triangle_J \eqdef \conv \set{\b{e}_j}{j \in J}$ of the standard simplex~$\triangle_{[n]}$.
In other words, the faces of the standard simplex~$(\triangle_J)_{\varnothing \ne J \subseteq [n]}$ form a linear basis of the space~$\VDP$ of virtual deformed permutahedra.
\end{proposition}

In this section, we show that shard polytopes have the same property up to translation as announced in \cref{prop:main9}.
To manipulate the quotient of the vector subspace of virtual deformed permutahedra modulo translations, we pick a representative in each translation class.
We say that a deformed permutahedron~$\polytope{P}$ is \defn{caged} if~$\b{x}_i \ge 0$ is a tight inequality for~$\polytope{P}$ for any~$i \in [n]$.
A virtual deformed permutahedron~$\polytope{P} - \polytope{Q}$ is \defn{caged} if both~$\polytope{P}$ and~$\polytope{Q}$ are caged.
We denote by~$\CVDP \subset \VDP$ the vector subspace of caged virtual deformed permutahedra.
We obtain the following analogue of \cref{prop:faceSimplexBasis}, which generalizes \cref{coro:uniqueDecompositionCambrian}.

\begin{proposition}
\label{prop:shardPolytopeBasis}
Any caged deformed permutahedron has a unique decomposition as a Minkowski sum and difference of dilated translated shard polytopes~$\translatedShardPolytope$ for~$\arc \in \arcs_n$.
In other words, the translated shard polytopes~$\big( \translatedShardPolytope \big)_{\arc \in \arcs_n}$ form a linear basis of the space~$\CVDP$ of caged virtual deformed permutahedra.
\end{proposition}

\begin{proof}
Observe first that
\begin{itemize}
\item the dimension of~$\CVDP$ is the dimension of the type cone of the braid fan~$\Fan_n$, which is~$2^n-n-1$ (because~$\Fan_n$ had~$2^n - 2$ rays and dimension~$n-1$),
\item the number of shard polytopes is~$|\arcs_n| = \sum_{1 \le a < b \le n} 2^{b-a-1} = 2^n-n-1$.
\end{itemize}
It thus suffices to prove that the shard polytopes are linearly independent.
Suppose, for the sake of contradiction, that there is a linear dependence~$\sum_{\arc \in \arcs_n} \coeffSP_\arc \, \shardPolytope$ in the type cone~$\typeCone(\Fan_n)$.
Writing~$\arcs^+ \eqdef \set{\arc \in \arcs_n}{\coeffSP_\arc > 0}$ and~$\arcs^- \eqdef \set{\arc \in \arcs_n}{\coeffSP_\arc < 0}$, we obtain an equality of Minkowski sums~$\Sigma^+ \eqdef \sum_{\arc \in \arcs^+} \coeffSP_\arc \, \shardPolytope = \sum_{\arc \in \arcs^-} (-\coeffSP_\arc) \, \shardPolytope \defeq \Sigma^-$.
Let~$\arc_\bullet$ be a forcing minimal arc in~$\arcs^+ \cup \arcs^-$.
Then~$\shardPolytope[\arc_\bullet]$ is the only polytope among the summands~$\shardPolytope$ for~$\arc \in \arcs^+ \cup \arcs^-$ whose normal fan contains the shard~$\shard(\arc_\bullet)$.
Since~$\coeffSP_{\arc_\bullet} \ne 0$, this implies that the shard~$\shard(\arc_\bullet)$ appears only in one of the normal fans of~$\Sigma^+$ and~$\Sigma^-$.
A contradiction to the equality~$\Sigma^+ = \Sigma^-$.
\end{proof}

\begin{remark}
Note that \cref{thm:inductiveMinkowskiSum,prop:shardPolytopeBasis} seem to contradict each other.
However, recall that the identity of \cref{thm:inductiveMinkowskiSum} involves pseudoshards that are not shards.
\end{remark}

\begin{remark}
Let~$\polytope{P}_\equiv$ be a quotientope for a lattice congruence~$\equiv$ of the weak order on~$\fS_n$. That is, a polytope whose normal fan is the quotient fan~$\fan_\equiv$.
By~\cref{prop:shardPolytopeBasis}, $\polytope{P}_\equiv$ can be represented as a signed Minkowski sum of dilated shard polytopes.
In this representation, only shard polytopes $\shardPolytope$ for arcs~$\arc$ belonging to the arc ideal~$\arcs_\equiv$ can appear.
Indeed, a forcing minimal arc~${\arc_\bullet \in \arcs\ssm \arcs_\equiv}$ would introduce a shard in the normal fan that does not belong to the quotient fan $\fan_\equiv$.
This means that the shard polytopes corresponding to arcs in~$\arcs_\equiv$ form a basis for the subspace of (caged) virtual polytopes spanned by (caged) quotientopes realizing the quotient fan~$\fan_\equiv$.
See also \cref{subsec:PSquotientopes} for the discussion of the particular situation of the quotientopes of~\cite{PilaudSantos-quotientopes}.
\end{remark}

It follows from \cref{prop:deformedPermutahedraZ,prop:faceSimplexBasis,prop:shardPolytopeBasis} that we have three natural parametrizations of the space of caged deformed permutahedra:
\begin{itemize}
\item by the shard polytopes as described in~\cref{prop:shardPolytopeBasis},
\item by the faces of the standard simplex as described in \cref{prop:faceSimplexBasis}, and
\item by the heights as described in~\cref{prop:deformedPermutahedraZ}.
\end{itemize}
To be consistent, we label these three families of parameters by the set~$\nonsing \eqdef \set{J \subseteq [n]}{|J| \ge 2}$.
For the parametrization in terms of shard polytopes, we thus need to relabel arcs as follows.

\begin{definition}
\label{def:coeffSP}
For a family~$\smash{\coeffSP \eqdef (\coeffSP_I)_{I \in \nonsing}}$ of real parameters, we define the caged virtual deformed permutahedron
\[
\deformedPermutahedronSP \eqdef \sum_{I \in \nonsing} \coeffSP_I \, \translatedShardPolytope[I],
\]
where~$\translatedShardPolytope[I] \eqdef \translatedShardPolytope[\arc_I] \eqdef \shardPolytope[\arc_I] + \one_{B_I \cup \{b_I\}}$ is the translated shard polytope of the arc defined by~$\arc_I \eqdef (a_I, b_I, A_I, B_I)$ where
\[
a_I \eqdef \min I,
\qquad
b_I \eqdef \max I,
\qquad
A_I \eqdef {]a_I, b_I[} \cap I
\qquad\text{and}\qquad
B_I \eqdef {]a_I, b_I[} \ssm I.
\]
\end{definition}

For the parametrization in terms of faces of the standard simplex, the caged condition is equivalent to using only faces corresponding to~$\nonsing$.

\begin{definition}
\label{def:coeffFS}
For a family~$\smash{\coeffFS \eqdef (\coeffFS_J)_{J \in \nonsing}}$ of real parameters, we define the caged virtual deformed permutahedron
\[
\deformedPermutahedronFS \eqdef \sum_{J \in \nonsing} \coeffFS_J \, \triangle_J,
\]
where~$\triangle_J \eqdef \conv \set{\b{e}_j}{j \in J}$ is a face of the standard simplex~$\triangle_{[n]}$.
\end{definition}

Finally, for the parametrization in terms of heights, we use inner normal vectors and supermodular functions (rather than outer normal vectors and submodular functions as in \cref{prop:deformedPermutahedraZ}) to fit the original presentation of~\cite{Postnikov, ArdilaBenedettiDoker}.

\begin{definition}
\label{def:coeffRHS}
For a family~$\coeffRHS \eqdef (\coeffRHS_R)_{R \in \nonsing}$ of real parameters such that~$\coeffRHS_R + \coeffRHS_S \le \coeffRHS_{R \cup S} + \coeffRHS_{R \cap S}$ (with the convention that~$\coeffRHS_R = 0$ for~$|R| \le 1$), we define the caged deformed permutahedron
\[
\deformedPermutahedronRHS \eqdef  \Bigset{\b{x} \in (\R_{\ge0})^n}{ \dotprod{\one}{\b{x}} = \coeffRHS_{[n]} \text{ and } \dotprod{\one_R}{\b{x}} \ge \coeffRHS_R \text{ for all } R \in \textstyle{\nonsing}}.
\]
\end{definition}

Note that different values of~$\coeffRHS$ can describe the same caged deformed permutahedron.
Indeed, if an inequality~$\dotprod{\one_R}{\b{x}} \ge \coeffRHS_R$ does not define a facet, then we can increase~$\coeffRHS_R$ without altering the polytope.
When writing~$\deformedPermutahedronRHS$, we always implicitly assume that all inequalities are tight.
In other words, that~$\coeffRHS$ gives the support function of~$\deformedPermutahedronRHS$ on the rays of the braid fan.

The following consequence of \cref{prop:deformedPermutahedraZ,prop:faceSimplexBasis,prop:shardPolytopeBasis} summarizes the previous discussion.

\begin{corollary}
\label{coro:threeParametrizations}
Any caged deformed permutahedron admits parametrizations of the form
\[
\deformedPermutahedronSP = \deformedPermutahedronFS = \deformedPermutahedronRHS
\]
for unique parameters~$\coeffSP$, $\coeffFS$ and~$\coeffRHS$.
\end{corollary}

\begin{remark}
\label{rem:FStoRHS}
It was proved in~\cite{Postnikov, ArdilaBenedettiDoker} that the parameters~$\coeffFS$ and~$\coeffRHS$ in \cref{coro:threeParametrizations} are related by
\[
\coeffRHS_R = \sum_{J \subseteq R} \coeffFS_J
\qquad\text{and}\qquad
\coeffFS_J = \sum_{R \subseteq J} (-1)^{|J \ssm R|} \, \coeffRHS_R,
\]
assuming that the inequalities defining~$\deformedPermutahedronRHS$ are tight.
\end{remark}

The goal of this section is to give explicit relations between the parameters~$\coeffSP$ and the other two, $\coeffFS$ and~$\coeffRHS$.
These relations are given in \cref{prop:FStoSPb,prop:SPtoFSb,prop:RHStoSP,prop:SPtoRHS}.
As applications, we will obtain the $\coeffFS$-coordinates of Cambrian associahedra~$\Asso[\arc]$ in \cref{coro:LangeDecompositionCambrian} and the $\coeffSP$-coordinates of the classical permutahedron~$\Perm$ in \cref{coro:permutahedronSPcoordinates}.

%%%

\subsubsection{From simplices to shard polytopes}
\label{subsubsec:FStoSP}

First, we decompose the shard polytopes in terms of the faces of the standard simplex, proving \cref{prop:main6}.
This representation is specially nice as all coefficients are $\pm 1$, and can be easily derived using a combinatorial formula of F.~Ardila, C.~Benedetti and J.~Doker for matroid polytopes in terms of $\beta$-invariants~\cite{ArdilaBenedettiDoker}.

Introduced by H.~Crapo~\cite{Crapo1967}, the \defn{beta invariant} of a matroid $\mat$ on the ground set~$E$ is the non-negative integer given by
\[
\beta(\mat) \eqdef (-1)^{r(\mat)} \sum_{X \subseteq E} (-1)^{|X|} \, r(X),
\]
where $r: 2^E \to \N$ is the rank function of~$\mat$.
It has the property that $\beta(\mat) = 0$ if and only if $\mat$ is disconnected, or empty, or a loop, and that $\beta(\mat) = 1$ if and only if $\mat$ is series-parallel.
The \defn{signed beta invariant} of~$\mat$ is
\[
\tilde \beta(\mat) \eqdef (-1)^{r(\mat)+1} \beta(\mat).
\]

\begin{theorem}[{\cite[Thm.~2.5]{ArdilaBenedettiDoker}}]
\label{thm:betaInvariants}
For a matroid $\mat$, we have 
\[
\polytope{P}_\mat = \sum_{K \subseteq E} \tilde \beta(M/K) \, \triangle_{E \ssm K}.
\]
\end{theorem}

\begin{lemma}
\label{lem:connectedcontraction}
Let $\arc \eqdef (a, b, A, B) \in \arcs_n$ and $K \subseteq[n]$.
The contraction $\mat_\arc/K$ is series-parallel and connected if and only if its complement is of the form $[n] \ssm K = X \cup Y$ where $X \subseteq B \cup \{a,b\}$ has at least $2$ elements, and $Y = {]i,j[} \cap A$ where $\min X \defeq i \ne j \eqdef \max X$.
\end{lemma}

\begin{proof}
All connected contractions of a series-parallel matroid, that are not empty or loops, are series-parallel.
Therefore, we need to compute the loopless non-empty connected contractions of~$\mat_\arc$.

With the notation of \cref{def:shardgraph}, let $a_j \in \{a\} \cup A$ correspond to the edge $(k,m)$, meaning that $b_k<a_j<b_{k+1}$.
If the edge labeled $a_j$ is contracted in $\Gamma_\arc$, then the contracted vertex becomes a cut vertex.
The $2$-connected blocks attached to this cut-vertex are: a loop for each $a_\ell \ne a_j$ with $b_k < a_\ell < b_{k+1}$; a (possibly empty) block with all the edges labeled with $i \le b_k$; and a block with all the edges labeled with $i \ge b_{k+1}$.
Everything but one of the loopless blocks must be contracted in order to be still $2$-connected.

If it is not empty, that is if $k \ge 1$, the block with all the edges labeled with $i \le b_k$ is isomorphic to the shard graph of $(a, b_k, A \cap {]a,b_k[}, B \cap {]a,b_k[})$ stripped of loops, and is $2$-connected.

The block with all the edges labeled with $i \ge b_{k+1}$ is either a loop if $b_{k+1} = b$, or isomorphic to the shard graph of $(b_{k+1}, b, A \cap {]b_{k+1},b[}, B \cap {]b_{k+1},b[})$ stripped of loops, in which case it is $2$-connected.

Therefore, for $K \in A \cup \{a,b\}$ the $2$-connected loopless non-empty blocks of $\mat_\arc/K$ are all isomorphic to a shard graph of an arc of the form $(x, y, A \cap {]x,y[}, B \cap {]x,y[})$ stripped of loops for some $x \ne y \in B \cup \{a,b\}$.
This reduces the study to contractions by elements in $B \cup \{b\}$.

If $b = b_m$ is contracted, then all the $a_j \in \{a\} \cup A$ with $b_{m-1} < a_j < b_m$ become loops and must be contracted too.
If $B = \varnothing$ then there are only loops left, but otherwise we recover the shard graph of  $(a, b_{m-1}, A \cap {]a,b_{m-1}[}, B \cap {]a,b_{m-1}[})$ stripped of loops, which is $2$-connected.

Finally, for any $b_i \in B$, then its contraction is the shard graph of  $(a, b, A, B \ssm \{b_i\})$ stripped of loops (on the ground set $[a,b] \ssm \{b_i\}$), which is $2$-connected.

All this combined shows the $2$-connected loopless non-empty blocks of a contraction of $\mat_\arc$ are all isomorphic to a shard graph stripped of loops of an arc of the form $(x, y, A \cap {]x,y[}, B')$ for some $x \ne y \in B \cup \{a,b\}$ and $B' \subseteq B \cap {]x,y[})$.
\end{proof}

In view of \cref{thm:betaInvariants,lem:connectedcontraction}, we define for any~$I, J \in \nonsing$
\begin{itemize}
\item $I \conn J$ if and only if $\{\min J, \max J\} \subseteq {]\min I, \max I[} \symdif I$ and ${]\min J, \max J[} \cap I \subseteq J$,
\item $\asize{J} \eqdef |J \ssm (\{\min J, \max J\} \cup I)|$.
\end{itemize}

\begin{proposition}
\label{prop:FStoSPa}
For any~$I \in \nonsing$, we have the Minkowski decomposition
\[
\translatedShardPolytope[I] = \sum_{I \conn J} (-1)^{\asize{J}} \, \triangle_J.
\]
\end{proposition}

For example, \cref{fig:signedMSFacesSimplex,fig:signedMSFacesSimplexProof} illustrate the decomposition
\[
\translatedShardPolytope[134] = \translatedShardPolytope[\raisebox{-.2cm}{\includegraphics[scale=.6]{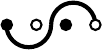}}] = \triangle_{12} + \triangle_{134} + \triangle_{234} - \triangle_{1234}.
\]

\begin{figure}
	\capstart
	\centerline{
		\begin{tabular}{c@{$=$}c@{$+$}c@{$+$}c@{$-$}c}
		\raisebox{-1cm}{\includegraphics[scale=.8]{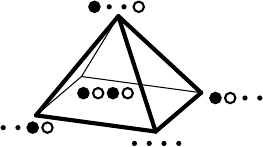}} & \raisebox{-1cm}{\includegraphics[scale=.8]{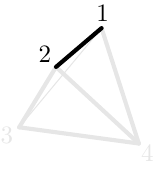}} & \raisebox{-1cm}{\includegraphics[scale=.8]{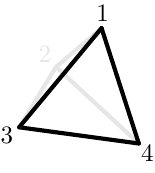}} & \raisebox{-1cm}{\includegraphics[scale=.8]{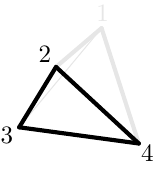}} & \raisebox{-1cm}{\includegraphics[scale=.8]{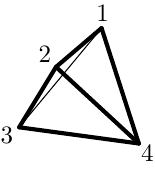}} \\[1.2cm]
		$\translatedShardPolytope[\raisebox{-.2cm}{\includegraphics[scale=.6]{signedMSFacesSimplexArc}}]$ & $\triangle_{12}$ & $\triangle_{134}$ & $\triangle_{234}$ & $\triangle_{1234}$
		\end{tabular}
	}
	\caption{An example of Minkowski decomposition into signed faces of simplices of \cref{prop:FStoSPa}. See also \cref{fig:signedMSFacesSimplexProof}.}
	\label{fig:signedMSFacesSimplex}
\end{figure}

\begin{figure}
	\capstart
	\centerline{
		\begin{tabular}{c@{}c@{}c@{}c@{}c@{}c@{}c@{}c@{}c@{}c@{}c}
		\raisebox{-1cm}{\includegraphics[scale=.8]{signedMSFacesSimplexPolytope1}} & $+$ & \raisebox{-1cm}{\includegraphics[scale=.8]{signedMSFacesSimplexPolytope5}} & $=$ & \raisebox{-2cm}{\includegraphics[scale=.8]{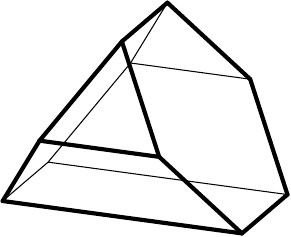}} & $=$ & \raisebox{-1cm}{\includegraphics[scale=.8]{signedMSFacesSimplexPolytope2}} & $+$ & \raisebox{-1cm}{\includegraphics[scale=.8]{signedMSFacesSimplexPolytope3}} & $+$ & \raisebox{-1cm}{\includegraphics[scale=.8]{signedMSFacesSimplexPolytope4}} \\[-.7cm]
		$\translatedShardPolytope[\raisebox{-.2cm}{\includegraphics[scale=.6]{signedMSFacesSimplexArc}}]$ & & $\triangle_{1234}$ & & & & $\triangle_{12}$ & & $\triangle_{134}$ & & $\triangle_{234}$
		\end{tabular}
	}
	\caption{Visual proof of the identity of \cref{fig:signedMSFacesSimplex}.}
	\label{fig:signedMSFacesSimplexProof}
\end{figure}

Let us rephrase \cref{prop:FStoSPa} in terms of the parameters of \cref{coro:threeParametrizations}.

\begin{proposition}
\label{prop:FStoSPb}
In \cref{coro:threeParametrizations}, the parameters~$\coeffFS$ are obtained from the parameters~$\coeffSP$ by
\[
\coeffFS_J = \sum_{I \conn J} (-1)^{\asize{J}} \, \coeffSP_I.
\]
\end{proposition}

\begin{remark}
\label{rem:FStoSPb}
To visualize this sum matricially, arrange the subsets of~$\nonsing$ by lexicographic order on the pair $(\max J - \min J, -|J|)$.
For instance, this order is given by
\(
% lower triangular
%\{1, 3\}, \{1, 2, 3\}, \{2, 3\}, \{1, 2\}
% upper triangular
\{1, 2\}, \{2, 3\}, \{1, 2, 3\}, \{1, 3\}
\)
for~$n = 3$ and
\(
% lower triangular
% \{1, 4\}, \{1, 3, 4\}, \{1, 2, 4\}, \{1, 2, 3, 4\}, \{2, 4\}, \{1, 3\}, \{2, 3, 4\}, \{1, 2, 3\}, \{3, 4\}, \{2, 3\}, \{1, 2\}
% upper triangular
\{1, 2\}, \{2, 3\}, \{3, 4\}, \{1, 2, 3\}, \{2, 3, 4\}, \{1, 3\}, \{2, 4\}, \{1, 2, 3, 4\}, \{1, 2, 4\}, \{1, 3, 4\}, \{1, 4\}
\)
for~$n = 4$.
Then the matrix~$M_\coeffFS^\coeffSP$ such that~$\coeffFS = M_\coeffFS^\coeffSP \cdot \coeffSP$ becomes upper triangular for this order.
For instance, when~$n = 3$ and~$n = 4$, the matrix~$M_\coeffFS^\coeffSP$ is given by
\[
% lower triangular
%\begin{blockarray}{cccccc}
%	13 & 123 & 23 & 12 & & \\
%	%[0d0, 0u0, .00, 00.]
%	\begin{block}{(cccc)cc}
%	 1 &  0 &  0 &  0 & 13  & \\
%	-1 &  1 &  0 &  0 & 123 & \\
%	 1 &  0 &  1 &  0 & 23  & \\
%	 1 &  0 &  0 &  1 & 12  & \\
%	\end{block}
% upper triangular
\begin{blockarray}{cccccc}
	12 & 23 & 123 & 13 & & \\
	%[00., .00, 0u0, 0d0]
	\begin{block}{(cccc)cc}
	 1 &  0 &  0 &  1 & 12  & \\
	 0 &  1 &  0 &  1 & 23  & \\
	 0 &  0 &  1 & -1 & 123 & \\
	 0 &  0 &  0 &  1 & 13  & \\
	\end{block}
\end{blockarray}
\qquad
%
% lower triangular
%\begin{blockarray}{ccccccccccccc}
%	14 & 134 & 124 & 1234 & 24 & 13 & 234 & 123 & 34 & 23 & 12 & & \\
%	%[0dd0, 0du0, 0ud0, 0uu0, .0d0, 0d0., .0u0, 0u0., ..00, .00., 00..]
%	\begin{block}{(ccccccccccc)cc}
%	 1 &  0 &  0 &  0 &  0 &  0 &  0 &  0 &  0 &  0 &  0 & 14   & \\
%	-1 &  1 &  0 &  0 &  0 &  0 &  0 &  0 &  0 &  0 &  0 & 134  & \\
%	-1 &  0 &  1 &  0 &  0 &  0 &  0 &  0 &  0 &  0 &  0 & 124  & \\
%	 1 & -1 & -1 &  1 &  0 &  0 &  0 &  0 &  0 &  0 &  0 & 1234 & \\
%	 1 &  0 &  0 &  0 &  1 &  0 &  0 &  0 &  0 &  0 &  0 & 24   & \\
%	 1 &  0 &  0 &  0 &  0 &  1 &  0 &  0 &  0 &  0 &  0 & 13   & \\
%	-1 &  1 &  0 &  0 & -1 &  0 &  1 &  0 &  0 &  0 &  0 & 234  & \\
%	-1 &  0 &  1 &  0 &  0 & -1 &  0 &  1 &  0 &  0 &  0 & 123  & \\
%	 1 &  0 &  1 &  0 &  1 &  0 &  0 &  0 &  1 &  0 &  0 & 34   & \\
%	 1 &  0 &  0 &  0 &  1 &  1 &  0 &  0 &  0 &  1 &  0 & 23   & \\
%	 1 &  1 &  0 &  0 &  0 &  1 &  0 &  0 &  0 &  0 &  1 & 12   & \\
%	\end{block}
%\end{blockarray}
% upper triangular
\begin{blockarray}{ccccccccccccc}
	12 & 23 & 34 & 123 & 234 & 13 & 24 & 1234 & 124 & 134 & 14 & & \\
	%[00.., .00., ..00, 0u0., .0u0, 0d0., .0d0, 0uu0, 0ud0, 0du0, 0dd0]
	\begin{block}{(ccccccccccc)cc}
	 1 &  0 &  0 &  0 &  0 &  1 &  0 &  0 &  0 &  1 &  1 & 12   & \\
	 0 &  1 &  0 &  0 &  0 &  1 &  1 &  0 &  0 &  0 &  1 & 23   & \\
	 0 &  0 &  1 &  0 &  0 &  0 &  1 &  0 &  1 &  0 &  1 & 34   & \\
	 0 &  0 &  0 &  1 &  0 & -1 &  0 &  0 &  1 &  0 & -1 & 123  & \\
	 0 &  0 &  0 &  0 &  1 &  0 & -1 &  0 &  0 &  1 & -1 & 234  & \\
	 0 &  0 &  0 &  0 &  0 &  1 &  0 &  0 &  0 &  0 &  1 & 13   & \\
	 0 &  0 &  0 &  0 &  0 &  0 &  1 &  0 &  0 &  0 &  1 & 24   & \\
	 0 &  0 &  0 &  0 &  0 &  0 &  0 &  1 & -1 & -1 &  1 & 1234 & \\
	 0 &  0 &  0 &  0 &  0 &  0 &  0 &  0 &  1 &  0 & -1 & 124  & \\
	 0 &  0 &  0 &  0 &  0 &  0 &  0 &  0 &  0 &  1 & -1 & 134  & \\
	 0 &  0 &  0 &  0 &  0 &  0 &  0 &  0 &  0 &  0 &  1 & 14   & \\
	\end{block}
\end{blockarray}
\]
\end{remark}

\begin{remark}
The representation of \cref{prop:FStoSPa} is unique, as the faces of a standard simplex form a base of~$\VDP$.
However, in this representation the roles of $A$ and $B$ are highly asymmetric: its size grows exponentially on the size of~$B$.
It is natural to wonder whether there is a larger generating set of~$\VDP$ that would allow to represent shard polytopes in a sparse symmetric way.
Note in particular that shard polytopes of down arcs of the form~$\arc \eqdef (a, b, \varnothing, {]a,b[})$ are represented by the exponential alternating sum~$\sum_J (-1)^{|J|} \, \triangle_J$ for all ${J \in \nonsing[{a,b}]} $, whereas $\translatedShardPolytope$ is just a translate of the reflected standard simplex~$\bigtriangledown_{[a,b]} \eqdef \conv \set{-\b{e}_j}{j \in [a,b]}$. Thus, the faces $\bigtriangledown\!_J$ of the reflected standard simplex are natural candidates.
In any case, any nontrivial Minkowski decomposition of shard polytopes in terms of $\triangle_J$ and $\bigtriangledown\!_J$ (or any other family) must involve positive and negative summands, as shard polytopes are indecomposable.
\end{remark}

Decompositions of associahedra realizing the $\arc$-Cambrian fan~$\Fan_\arc$ as signed Minkowski sums of faces of the standard simplex were thoroughly studied by C.~Lange in~\cite{Lange}.
Remember from \cref{thm:shardPolytopesRaysTypeCone} that the shard polytopes of the arcs forcing~$\arc$ are the rays of the type cone of~$\Fan_\arc$.
Therefore, by \cref{coro:uniqueDecompositionCambrian}, any associahedron realizing the $\arc$-Cambrian fan~$\Fan_\arc$ has a unique expression of the form $\sum_{\arc' \succ \arc} \coeffSP_{\arc'} \, \shardPolytope[\arc']$ up to translation.
This provides an alternative way to recover the Minkowski decompositions from~\cite{Lange}.

\begin{corollary}
\label{coro:LangeDecomposition}
For any positive coefficients~$\coeffSP_I > 0$, let $\polytope{P} = \sum_{\arc_I \succ \arc} \coeffSP_{I} \, \translatedShardPolytope[I]$ be a caged associahedron realizing the $\arc$-Cambrian fan~$\Fan_\arc$. Then 
\[
\polytope{P} = \sum_{J\subseteq [a,b]} \coeffFS_J \, \triangle_J
\qquad\text{where}\qquad
\coeffFS_J = \sum_{\substack{I \conn J \\ \arc_I \succ \arc}} (-1)^{\asize{J}} \, \coeffSP_I.
\]
\end{corollary}

In particular, for the $\arc$-associahedron~$\Asso[\arc]$ described in \cref{exm:HohlwegLangeAsso} (see \cref{exm:HohlwegLangeAssoMinkowskiSum,exm:HohlwegLangeAssoVertexFacetDescription}), all the coefficients $\coeffSP_I$ are~$1$.
This case was studied with special detail in~\cite{Lange}, who gave a combinatorial description of the coefficients in the Minkowski decomposition in terms of nested up and down interval decompositions. We obtain the following elementary description. 

\begin{corollary}
\label{coro:LangeDecompositionCambrian}
For~$\arc \eqdef (a, b, A, B)$, the $\arc$-associahedron~$\Asso[\arc]$ decomposes as
\[
\Asso[\arc] = \sum_{\varnothing \ne J \subseteq [a,b]} \coeffFS_J \, \triangle_J,
\]
where for $|J| \ge 2$ 
\[
\coeffFS_J =
\begin{cases}
(-1)^{\asize{J}[A]} & \text{ if } A \cap {]\min J, \max J[} \subset J \text{ and } \min J, \max J \in A,\\
(-1)^{\asize{J}[A]}\cdot (b - \max J + 1) & \text{ if } A \cap {]\min J, \max J[} \subset J \text{ and } \max J \notin A,\\
(-1)^{\asize{J}[A]} \cdot \min J & \text{ if } A \cap {]\min J, \max J[} \subset J \text{ and } \max J \in A,\\
(-1)^{\asize{J}[A]} \cdot (b - \max J + 1) \cdot \min J & \text{ if } A \cap {]\min J, \max J[} \subset J \text{ and } \min J, \max J \notin A,\\
0 & \text{otherwise,}
\end{cases}
\]
and for $J=\{j\}$ 
\[
\coeffFS_{\{j\}} =
\begin{cases}
j\cdot (b-j-1) & \text{ if }j\in A\cup\{a\},\\
(b-j-1) & \text{ otherwise.}
\end{cases}
\]
\end{corollary}

%%%

\subsubsection{From shard polytopes to simplices}
\label{subsubsec:shardPolytopeBasis}

We now decompose the faces of the standard simplex in terms of the shard polytopes, giving the reverse direction of \cref{prop:FStoSPa} and proving \cref{prop:main8}.

\begin{proposition}
\label{prop:SPtoFSa}
For any subset~$J \in \nonsing$, we have the Minkowski decomposition
\[
\triangle_J = \sum_{J \conn I} (-1)^{|\{\min I, \max I\} \cap \{\min J, \max J\}|} \, \translatedShardPolytope[I].
\]
\end{proposition}

For example, \cref{fig:signedMSShardPolytopes,fig:signedMSShardPolytopesProof} illustrate the decomposition
\[
\begin{array}{c@{\;}c@{\;}c@{\;}c@{\;}c@{\;}c@{\;}c@{\;}c@{\;}c}
\triangle_{124} & = & \translatedShardPolytope[124] & + & \translatedShardPolytope[1234] & - & \translatedShardPolytope[34] & - & \translatedShardPolytope[123] \\[.1cm]
& = & \translatedShardPolytope[\raisebox{-.2cm}{\includegraphics[scale=.6]{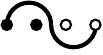}}] & + & \translatedShardPolytope[\raisebox{-.2cm}{\includegraphics[scale=.6]{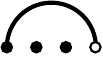}}] & - & \translatedShardPolytope[\raisebox{-.2cm}{\includegraphics[scale=.6]{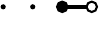}}] & - & \translatedShardPolytope[\raisebox{-.2cm}{\includegraphics[scale=.6]{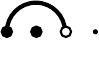}}]
\end{array}
\]

\begin{figure}
	\capstart
	\centerline{
		\begin{tabular}{c@{$=$}c@{$+$}c@{$-$}c@{$-$}c}
		\raisebox{-1cm}{\includegraphics[scale=.8]{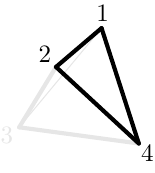}} & \raisebox{-1cm}{\includegraphics[scale=.8]{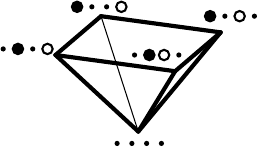}} & \raisebox{-1cm}{\includegraphics[scale=.8]{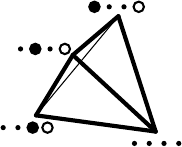}} & \hspace{.2cm}\raisebox{-.2cm}{\includegraphics[scale=.8]{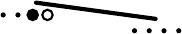}}\hspace{.2cm} & \raisebox{-1cm}{\includegraphics[scale=.8]{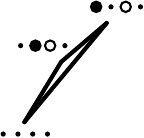}} \\[1.2cm]
		$\triangle_{124}$ & $\translatedShardPolytope[\raisebox{-.2cm}{\includegraphics[scale=.6]{signedMSShardPolytopesArc2}}]$ & $\translatedShardPolytope[\raisebox{-.2cm}{\includegraphics[scale=.6]{signedMSShardPolytopesArc3}}]$ & $\translatedShardPolytope[\raisebox{-.2cm}{\includegraphics[scale=.6]{signedMSShardPolytopesArc4}}]$ & $\translatedShardPolytope[\raisebox{-.2cm}{\includegraphics[scale=.6]{signedMSShardPolytopesArc5}}]$
		\end{tabular}
	}
	\caption{An example of Minkowski decomposition into signed faces of simplices of \cref{prop:SPtoFSa}. See also \cref{fig:signedMSShardPolytopesProof}.}
	\label{fig:signedMSShardPolytopes}
\end{figure}

\begin{figure}
	\capstart
	\centerline{
		\begin{tabular}{c@{}c@{}c@{}c@{}c@{}c@{}c@{}c@{}c@{}c@{}c}
		\raisebox{-1cm}{\includegraphics[scale=.8]{signedMSShardPolytopesPolytope1}} & $+$ & \hspace{.2cm}\raisebox{-.2cm}{\includegraphics[scale=.8]{signedMSShardPolytopesPolytope4}}\hspace{.2cm} & $+$ & \raisebox{-1cm}{\includegraphics[scale=.8]{signedMSShardPolytopesPolytope5}} & $=$ & \raisebox{-2.2cm}{\includegraphics[scale=.8]{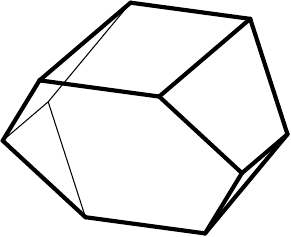}} & $=$ & \raisebox{-1cm}{\includegraphics[scale=.8]{signedMSShardPolytopesPolytope2}} & $+$ & \raisebox{-1cm}{\includegraphics[scale=.8]{signedMSShardPolytopesPolytope3}} \\[-.7cm]
		$\triangle_{124}$ & & $\translatedShardPolytope[\raisebox{-.2cm}{\includegraphics[scale=.6]{signedMSShardPolytopesArc4}}]$ & & $\translatedShardPolytope[\raisebox{-.2cm}{\includegraphics[scale=.6]{signedMSShardPolytopesArc5}}]$ & & & & $\translatedShardPolytope[\raisebox{-.2cm}{\includegraphics[scale=.6]{signedMSShardPolytopesArc2}}]$ & & $\translatedShardPolytope[\raisebox{-.2cm}{\includegraphics[scale=.6]{signedMSShardPolytopesArc3}}]$
		\end{tabular}
	}
	\caption{Visual proof of the identity of \cref{fig:signedMSShardPolytopes}.}
	\label{fig:signedMSShardPolytopesProof}
\end{figure}

Let us rephrase \cref{prop:SPtoFSa} in terms of the parameters of \cref{coro:threeParametrizations}.

\begin{proposition}
\label{prop:SPtoFSb}
In \cref{coro:threeParametrizations}, the parameters~$\coeffSP$ are obtained from the parameters~$\coeffFS$ by
\[
\coeffSP_I = \sum_{J \conn I} (-1)^{|\{\min I, \max I\} \cap \{\min J, \max J\}|} \, \coeffFS_J.
\]
\end{proposition}

\begin{remark}
\label{rem:SPtoFSb}
The matrix~$M_\coeffSP^\coeffFS$ such that~$\coeffSP = M_\coeffSP^\coeffFS \cdot \coeffFS$ is upper triangular when ordering the subsets of~$\nonsing$ by lexicographic order on the pair~$(\max J - \min J, -|J|)$.
For instance, when~$n = 3$ and~$n = 4$, the matrix~$M_\coeffSP^\coeffFS$ is given by
\[
% lower triangular
%M =
%\begin{blockarray}{cccccc}
%	13 & 123 & 23 & 12 & & \\
%	\begin{block}{(cccc)cc}
%	 1 &  0 &  0 &  0 & 13  & \\
%	 1 &  1 &  0 &  0 & 123 & \\
%	-1 &  0 &  1 &  0 & 23  & \\
%	-1 &  0 &  0 &  1 & 12  & \\
%	\end{block}
%	%[0d0, 0u0, .00, 00.]
%\end{blockarray}
% upper triangular
\begin{blockarray}{cccccc}
	12 & 23 & 123 & 13 & & \\
	\begin{block}{(cccc)cc}
	 1 &  0 &  0 & -1 & 12  & \\
	 0 &  1 &  0 & -1 & 23  & \\
	 0 &  0 &  1 &  1 & 123 & \\
	 0 &  0 &  0 &  1 & 13  & \\
	\end{block}
	%[00., .00, 0u0, 0d0]
\end{blockarray}
\qquad
%
% lower triangular
%\begin{blockarray}{ccccccccccccc}
%	14 & 134 & 124 & 1234 & 24 & 13 & 234 & 123 & 34 & 23 & 12 & & \\
%	\begin{block}{(ccccccccccc)cc}
%	 1 &  0 &  0 &  0 &  0 &  0 &  0 &  0 &  0 &  0 &  0 & 14   & \\
%	 1 &  1 &  0 &  0 &  0 &  0 &  0 &  0 &  0 &  0 &  0 & 134  & \\
%	 1 &  0 &  1 &  0 &  0 &  0 &  0 &  0 &  0 &  0 &  0 & 124  & \\
%	 1 &  1 &  1 &  1 &  0 &  0 &  0 &  0 &  0 &  0 &  0 & 1234 & \\
%	-1 &  0 &  0 &  0 &  1 &  0 &  0 &  0 &  0 &  0 &  0 & 24   & \\
%	-1 &  0 &  0 &  0 &  0 &  1 &  0 &  0 &  0 &  0 &  0 & 13   & \\
%	-1 & -1 &  0 &  0 &  1 &  0 &  1 &  0 &  0 &  0 &  0 & 234  & \\
%	-1 &  0 & -1 &  0 &  0 &  1 &  0 &  1 &  0 &  0 &  0 & 123  & \\
%	-1 &  0 & -1 &  0 & -1 &  0 &  0 &  0 &  1 &  0 &  0 & 34   & \\
%	 1 &  0 &  0 &  0 & -1 & -1 &  0 &  0 &  0 &  1 &  0 & 23   & \\
%	-1 & -1 &  0 &  0 &  0 & -1 &  0 &  0 &  0 &  0 &  1 & 12   & \\
%	\end{block}
%	%[0dd0, 0du0, 0ud0, 0uu0, .0d0, 0d0., .0u0, 0u0., ..00, .00., 00..]
%\end{blockarray}
% upper triangular
\begin{blockarray}{ccccccccccccc}
	12 & 23 & 34 & 123 & 234 & 13 & 24 & 1234 & 124 & 134 & 14 & & \\
	\begin{block}{(ccccccccccc)cc}
	 1 &  0 &  0 &  0 &  0 & -1 &  0 &  0 &  0 & -1 & -1 & 12   & \\
	 0 &  1 &  0 &  0 &  0 & -1 & -1 &  0 &  0 &  0 &  1 & 23   & \\
	 0 &  0 &  1 &  0 &  0 &  0 & -1 &  0 & -1 &  0 & -1 & 34   & \\
	 0 &  0 &  0 &  1 &  0 &  1 &  0 &  0 & -1 &  0 & -1 & 123  & \\
	 0 &  0 &  0 &  0 &  1 &  0 &  1 &  0 &  0 & -1 & -1 & 234  & \\
	 0 &  0 &  0 &  0 &  0 &  1 &  0 &  0 &  0 &  0 & -1 & 13   & \\
	 0 &  0 &  0 &  0 &  0 &  0 &  1 &  0 &  0 &  0 & -1 & 24   & \\
	 0 &  0 &  0 &  0 &  0 &  0 &  0 &  1 &  1 &  1 &  1 & 1234 & \\
	 0 &  0 &  0 &  0 &  0 &  0 &  0 &  0 &  1 &  0 &  1 & 124  & \\
	 0 &  0 &  0 &  0 &  0 &  0 &  0 &  0 &  0 &  1 &  1 & 134  & \\
	 0 &  0 &  0 &  0 &  0 &  0 &  0 &  0 &  0 &  0 &  1 & 14   & \\
	\end{block}
	%[00.., .00., ..00, 0u0., .0u0, 0d0., .0d0, 0uu0, 0ud0, 0du0, 0dd0]
\end{blockarray}
\]
\end{remark}

\begin{proof}[Proof of \cref{prop:SPtoFSa}]
By \cref{prop:FStoSPa}, we have 
\begin{align*}
\sum_{J \conn I} (-1)^{|\{\min I, \max I\} \cap \{\min J, \max J\}|} \translatedShardPolytope[\arc_I]
& = \sum_{J \conn I} (-1)^{|\{\min I, \max I\} \cap \{\min J, \max J\}|} \sum_{I \conn K} (-1)^{\asize{K}} \triangle_{K} \\
& = \sum_{J \conn K} \triangle_{K} \sum_{J \conn I \conn K} (-1)^{|\{\min I, \max I\} \cap \{\min J, \max J\}|} (-1)^{\asize{K}}.
\end{align*}

If $\min J \ne \min K$, then for each $I$ with $J \conn I \conn K$ and $\min I = \min J$, the set $I' \in \nonsing$ defined as
\[
I' \eqdef 
\begin{cases}
(I \ssm \{\min J\}) \cup \{\min K\} & \text{ if } I \cap {]\min J, \min K[} = \varnothing, \\
I \ssm \{\min J\} & \text{otherwise}
\end{cases}
\]
fulfills 
\[
(-1)^{\delta_{\max I = \max J}}  (-1)^{\asize{K}}  = (-1)^{\delta_{\max I' = \max J}} (-1)^{\asize{K}[I']},
\]
and moreover every $I'$ with $J \conn I' \conn K$ and $\min I' \ne \min J$ can be obtained this way. 

This implies that, whenever $\min J \ne \min K$, we have
\[
\sum_{J \conn I \conn K} (-1)^{|\{\min I, \max I\} \cap \{\min J, \max J\}|} (-1)^{\asize{K}}=0.
\]
And analogously whenever $\max J \ne \max K$.

Therefore, it suffices to consider the case $\{\min J, \max J\} = \{\min K,\max K\}$.
Under this constraint, $J \conn K$ if and only if $J \subseteq K$.
Moreover, any  $J \conn I \conn K$ also has the same extrema and we have $J \subseteq I \subseteq K$ and $(-1)^{\asize{K}} = (-1)^{|K \ssm I|}$.
That is, 
\begin{align*}
\sum_{J \conn K} \triangle_{K} \sum_{J \conn I \conn K} (-1)^{|\{\min I, \max I\} \cap \{\min J, \max J\}|} (-1)^{\asize{K}}
& = \sum_{\substack{J \subseteq K \\ \min J = \min K \\ \max J = \max K}} \triangle_{K} \sum_{J \subseteq I \subseteq K} (-1)^{|K \ssm I|} = \triangle_{K}, 
\end{align*}
by the inclusion-exclusion principle.
\end{proof}

Applying \cref{prop:SPtoFSb}, we can now describe the classical permutahedron as a Minkowski combination of shard polytopes.

\begin{corollary}
\label{coro:permutahedronSPcoordinates}
The classical permutahedron~$\Perm$ decomposes as
\[
\Perm = \sum_{1 \le i < j \le n} \simplex_{\{i,j\}} = \sum_{I \in \nonsing} (\min I - 2) (n - \max I -1) \, \translatedShardPolytope[I]. 
\]
\end{corollary}

When~$n = 3$, we obtain that
\[
\Perm[3] = \translatedShardPolytope[123] + \translatedShardPolytope[13] = \translatedShardPolytope[\raisebox{-.25cm}{\includegraphics[scale=.6]{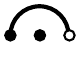}}] + \translatedShardPolytope[\raisebox{-.25cm}{\includegraphics[scale=.6]{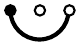}}]
\]
as illustrated in \cref{fig:submodularFunctions}.
When~$n = 4$, we obtain that
\[
\begin{array}{@{}c@{\;}c@{\;}c@{\;}c@{\;}c@{\;}c@{\;}c@{\;}c@{\;}c@{\;}c@{\;}c@{\;}c@{\;}c}
\Perm[4] & = & \translatedShardPolytope[1234] & + & \translatedShardPolytope[124] & + & \translatedShardPolytope[134] & + & \translatedShardPolytope[14] & - & \translatedShardPolytope[12] & - & \translatedShardPolytope[34] \\[.1cm]
& = & \translatedShardPolytope[\raisebox{-.4cm}{\includegraphics[scale=.6]{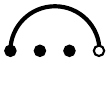}}] & + & \translatedShardPolytope[\raisebox{-.4cm}{\includegraphics[scale=.6]{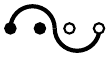}}] & + & \translatedShardPolytope[\raisebox{-.4cm}{\includegraphics[scale=.6]{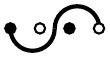}}] & + & \translatedShardPolytope[\raisebox{-.4cm}{\includegraphics[scale=.6]{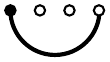}}] & - & \translatedShardPolytope[\raisebox{-.4cm}{\includegraphics[scale=.6]{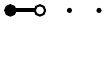}}] & - & \translatedShardPolytope[\raisebox{-.4cm}{\includegraphics[scale=.6]{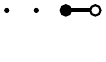}}]
\end{array}
\]
as illustrated in \cref{fig:signedMSShardPolytopesPermutahedron,fig:signedMSShardPolytopesPermutahedronProof}.
Note that for~$n \ge 4$, the coefficients of~$\shardPolytope[{[i]}]$ are always negative for all~$i \in [n-2]$, which gives another proof that~$\Perm$ is not a Minkowski sum of dilated shard polytopes as seen in \cref{coro:standardPermutahedronNotMinkowskiSumShardPolytopes}.

\begin{figure}
	\capstart
	\centerline{
		\begin{tabular}{c@{\;$=$\;}c@{$+$}c@{$+$}c@{$+$}c@{$-$}c@{$-$}c}
		\raisebox{-1.5cm}{\includegraphics[scale=.5]{permutahedron}}\hspace{.2cm} & \raisebox{-.6cm}{\includegraphics[scale=.5]{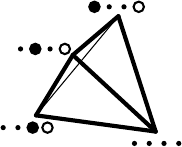}} & \raisebox{-.6cm}{\includegraphics[scale=.5]{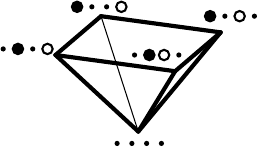}} & \hspace{-.2cm}\raisebox{-.6cm}{\includegraphics[scale=.5]{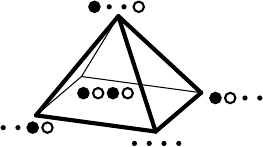}}\hspace{-.2cm} & \raisebox{-.6cm}{\includegraphics[scale=.5]{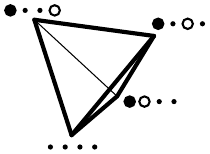}} & \raisebox{-.2cm}{\includegraphics[scale=.5]{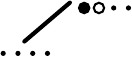}} & \raisebox{-.2cm}{\includegraphics[scale=.5]{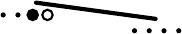}} \\
		$\Perm[4]$ & $\translatedShardPolytope[\raisebox{-.4cm}{\includegraphics[scale=.6]{signedMSShardPolytopesPermutahedronArc3}}]$ & $\translatedShardPolytope[\raisebox{-.4cm}{\includegraphics[scale=.6]{signedMSShardPolytopesPermutahedronArc4}}]$ & $\translatedShardPolytope[\raisebox{-.4cm}{\includegraphics[scale=.6]{signedMSShardPolytopesPermutahedronArc5}}]$ & $\translatedShardPolytope[\raisebox{-.4cm}{\includegraphics[scale=.6]{signedMSShardPolytopesPermutahedronArc6}}]$ & $\translatedShardPolytope[\raisebox{-.4cm}{\includegraphics[scale=.6]{signedMSShardPolytopesPermutahedronArc7}}]$ & $\translatedShardPolytope[\raisebox{-.4cm}{\includegraphics[scale=.6]{signedMSShardPolytopesPermutahedronArc8}}]$
		\end{tabular}
	}
	\caption{The Minkowski decomposition of the classical permutahedron into shard polytopes of \cref{coro:permutahedronSPcoordinates}. See also \cref{fig:signedMSShardPolytopesPermutahedronProof}.}
	\label{fig:signedMSShardPolytopesPermutahedron}
\end{figure}

\begin{figure}
	\capstart
	\centerline{
		\begin{tabular}{c@{}c@{}c@{}c@{}c@{}c@{}c@{}c@{\hspace{-.3cm}}c}
		$\translatedShardPolytope[\raisebox{-.4cm}{\includegraphics[scale=.6]{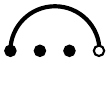}}]$ & & $\translatedShardPolytope[\raisebox{-.4cm}{\includegraphics[scale=.6]{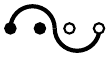}}]$ & & & & & & \raisebox{-.8cm}{$\translatedShardPolytope[\raisebox{-.4cm}{\includegraphics[scale=.6]{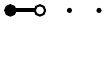}}]$} \\[-.7cm]
		\raisebox{-.6cm}{\includegraphics[scale=.5]{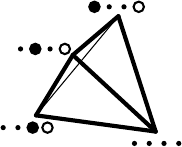}} & & \raisebox{-.6cm}{\includegraphics[scale=.5]{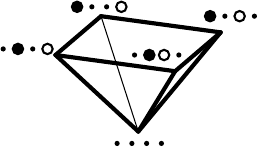}} & & & & & & \raisebox{-.5cm}{\includegraphics[scale=.5]{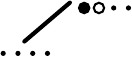}} \\[-1.7cm]
		& $+$ & & $=$ & \hspace{.2cm}\raisebox{-1.9cm}{\includegraphics[scale=.5]{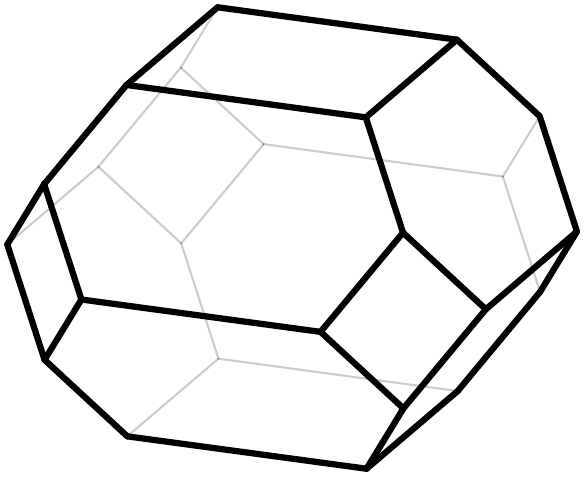}}\hspace{.2cm} & $=$ & \hspace{.2cm}\raisebox{-1.7cm}{\includegraphics[scale=.5]{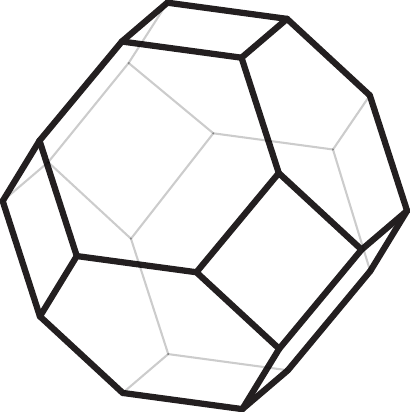}}\hspace{.2cm} & $+$ & \\[-1.7cm]
		\hspace{-.2cm}\raisebox{-.6cm}{\includegraphics[scale=.5]{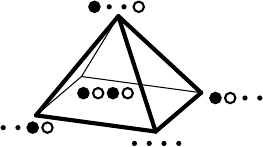}}\hspace{-.2cm} & & \raisebox{-.6cm}{\includegraphics[scale=.5]{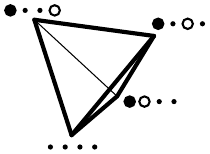}} & & & & & & \raisebox{.3cm}{\includegraphics[scale=.5]{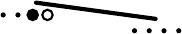}} \\[-.6cm]
		$\translatedShardPolytope[\raisebox{-.4cm}{\includegraphics[scale=.6]{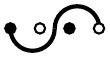}}]$ & & $\translatedShardPolytope[\raisebox{-.4cm}{\includegraphics[scale=.6]{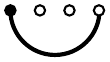}}]$ & & & & & & \raisebox{.9cm}{$\translatedShardPolytope[\raisebox{-.4cm}{\includegraphics[scale=.6]{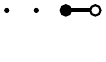}}]$}
		\end{tabular}
	}
	\caption{Visual proof of the identity of \cref{fig:signedMSShardPolytopesPermutahedron}.}
	\label{fig:signedMSShardPolytopesPermutahedronProof}
\end{figure}

%%%

\subsubsection{From heights to shard polytopes}

Composing the formulas of \cref{rem:FStoRHS,prop:SPtoFSb}, we now pass from heights to shard polytopes.

\begin{proposition}
\label{prop:RHStoSP}
In \cref{coro:threeParametrizations}, the parameters~$\coeffRHS$ are obtained from the parameters~$\coeffSP$ by
\[
\coeffRHS_R = \sum_{I \in \nonsing} \Cap(I,R) \, \coeffSP_I.
\]
where $\Cap(I,R)$ is the number of pairs~$r < s \in \big( {]\min I, \max I[} \symdif I \big) \cap R$ such that~${]r,s[} \cap I = {]r,s[} \cap R$.
\end{proposition}

For example, we have
\[
\coeffRHS_{234} = \coeffSP_{23} + \coeffSP_{34} + \coeffSP_{234} + \coeffSP_{13} + 2 \, \coeffSP_{24} + \coeffSP_{124} + \coeffSP_{134} + 2 \, \coeffSP_{14}.
\]

\begin{remark}
\label{rem:RHStoSP}
The matrix~$M_\coeffRHS^\coeffSP$ such that~$\coeffRHS = M_\coeffRHS^\coeffSP \cdot \coeffSP$ is given when~$n = 3$ and~$n = 4$ by
\[
\begin{blockarray}{cccccc}
	12 & 23 & 123 & 13 & & \\
	\begin{block}{(cccc)cc}
	 1 &  0 &  0 &  1 & 12  & \\
	 0 &  1 &  0 &  1 & 23  & \\
	 1 &  1 &  1 &  2 & 123 & \\
	 0 &  0 &  0 &  1 & 13  & \\
	\end{block}
	%[00., .00, 0u0, 0d0]
\end{blockarray}
\qquad
\begin{blockarray}{ccccccccccccc}
	12 & 23 & 34 & 123 & 234 & 13 & 24 & 1234 & 124 & 134 & 14 & & \\
	%[00.., .00., ..00, 0u0., .0u0, 0d0., .0d0, 0uu0, 0ud0, 0du0, 0dd0]
	\begin{block}{(ccccccccccc)cc}
	 1 &  0 &  0 &  0 &  0 &  1 &  0 &  0 &  0 &  1 &  1 & 12   & \\
	 0 &  1 &  0 &  0 &  0 &  1 &  1 &  0 &  0 &  0 &  1 & 23   & \\
	 0 &  0 &  1 &  0 &  0 &  0 &  1 &  0 &  1 &  0 &  1 & 34   & \\
	 1 &  1 &  0 &  1 &  0 &  2 &  1 &  0 &  1 &  1 &  2 & 123  & \\
	 0 &  1 &  1 &  0 &  1 &  1 &  2 &  0 &  1 &  1 &  2 & 234  & \\
	 0 &  0 &  0 &  0 &  0 &  1 &  0 &  0 &  0 &  0 &  1 & 13   & \\
	 0 &  0 &  0 &  0 &  0 &  0 &  1 &  0 &  0 &  0 &  1 & 24   & \\
	 1 &  1 &  1 &  1 &  1 &  2 &  2 &  1 &  2 &  2 &  3 & 1234 & \\
	 1 &  0 &  0 &  0 &  0 &  1 &  1 &  0 &  1 &  1 &  2 & 124  & \\
	 0 &  0 &  1 &  0 &  0 &  1 &  1 &  0 &  1 &  1 &  2 & 134  & \\
	 0 &  0 &  0 &  0 &  0 &  0 &  0 &  0 &  0 &  0 &  1 & 14   & \\
	\end{block}
\end{blockarray}
\]
While the matrices~$M_\coeffRHS^\coeffFS$ and $M_\coeffFS^\coeffSP$ are both upper triangular for well-chosen orders on~$\nonsing$, these orders differ so that the product~$M_\coeffRHS^\coeffSP = M_\coeffRHS^\coeffFS \cdot M_\coeffFS^\coeffSP$ is not upper triangular anymore.
Already when~$n = 4$, the matrix~$M_\coeffRHS^\coeffSP$ is not triangularizable since its minimal polynomial~$(x - 1)^7 (x^2 - 3x + 1) (x^2 - x + 1)$ does not split over~$\R$.
\end{remark}

\begin{proof}[Proof of \cref{prop:RHStoSP}]
From \cref{rem:FStoRHS,prop:FStoSPb}, we have
\[
\coeffRHS_R = \sum_{J \subseteq R} \coeffFS_J = \sum_{J \subseteq R} \sum_{I \conn J} (-1)^{\asize{J}} \, \coeffSP_I = \sum_{I \in \nonsing} \bigg( \sum_{\substack{J \subseteq R \\ I \conn J}} (-1)^{\asize{J}} \bigg) \, \coeffSP_I.
\]

Consider now a subset~$J$ such that~$J \subseteq R$ and~$I \conn J$.
Define~$r \eqdef \min J$ and~$s \eqdef \max J$.
By definition, we have~$r < s \in \big( {]\min I, \max I[} \symdif I \big) \cap R$.
Moreover, we have~${{]r,s[} \cap I \subseteq J \subseteq R}$, so that~${]r,s[} \cap I \subseteq {]r,s[} \cap R$.
Finally,
\begin{itemize}
\item if ${]r,s[} \cap I = {]r,s[} \cap R$, then the only subset~$J$ with~$J \subseteq R$ and~$I \conn J$ also satisfies ${]r,s[} \cap J = {]r,s[} \cap I$, so that~$(-1)^{\asize{J}} = 1$.
\item if ${]r,s[} \cap I \ne {]r,s[} \cap R$, then we obtain a sum over subsets of~${]r,s[} \cap (R \ssm I)$ of terms of the form~$(-1)^{\asize{J}}$ which vanishes by the inclusion-exclusion principle.
\qedhere
\end{itemize}
\end{proof}

\begin{remark}
Note that an alternative description of the parameters~$\coeffRHS$ in terms of the parameters~$\coeffSP$ can be derived from \cref{subsec:vertexFacetDescriptionsShardSumotopes}.
Indeed, we obtain in \cref{def:h,prop:h} directly the heights of the shard polytopes.
The only difference is that \cref{def:h,prop:h} are written in terms of outer normal vectors while this section is written in terms of inner normal vectors.
\end{remark}

%%%

\subsubsection{From shard polytopes to heights}
\label{subsubsec:SPtoRHS}

Composing the formulas of \cref{rem:FStoRHS,prop:SPtoFSb}, we finally pass from shard polytopes to heights, giving the reverse direction of \cref{prop:RHStoSP}.

\begin{proposition}
\label{prop:SPtoRHS}
In \cref{coro:threeParametrizations}, the parameters~$\coeffSP$ are obtained from the parameters~$\coeffRHS$ by
\[
\coeffSP_I = \coeffRHS_I - \coeffRHS_{I \symdif [1, \min I]} - \coeffRHS_{I \symdif [\max I, n]} + \coeffRHS_{I \symdif ([1, \min I] \cup [\max I, n])},
\]
with the convention that~$\coeffRHS_R = 0$ if~$|R| \le 1$.
\end{proposition}

For example, we have
\[
\coeffSP_{234} = \coeffRHS_{234} + \coeffRHS_{13} - \coeffRHS_{23} - \coeffRHS_{134}.
\]

\begin{remark}
\label{rem:SPtoRHS}
The matrix~$M_\coeffSP^\coeffRHS$ such that~$\coeffSP = M_\coeffSP^\coeffRHS \cdot \coeffRHS$ is given when~$n = 3$ and~$n = 4$ by
\[
\begin{blockarray}{cccccc}
	12 & 23 & 123 & 13 & & \\
	\begin{block}{(cccc)cc}
	 1 &  0 &  0 & -1 & 12  & \\
	 0 &  1 &  0 & -1 & 23  & \\
	-1 & -1 &  1 &  0 & 123 & \\
	 0 &  0 &  0 &  1 & 13  & \\
	\end{block}
	%[00., .00, 0u0, 0d0]
\end{blockarray}
\quad
\begin{blockarray}{ccccccccccccc}
	12 & 23 & 34 & 123 & 234 & 13 & 24 & 1234 & 124 & 134 & 14 & & \\
	\begin{block}{(ccccccccccc)cc}
	 1 &  0 &  1 &  0 &  0 &  0 &  0 &  0 &  0 & -1 &  0 & 12   & \\
	 0 &  1 &  0 &  0 &  0 & -1 & -1 &  0 &  0 &  0 &  1 & 23   & \\
	 1 &  0 &  1 &  0 &  0 &  0 &  0 &  0 & -1 &  0 &  0 & 34   & \\
	 0 & -1 &  0 &  1 &  0 &  0 &  1 &  0 & -1 &  0 &  0 & 123  & \\
	 0 & -1 &  0 &  0 &  1 &  1 &  0 &  0 &  0 & -1 &  0 & 234  & \\
	 0 &  0 &  0 &  0 &  0 &  1 &  0 &  0 &  0 &  0 & -1 & 13   & \\
	 0 &  0 &  0 &  0 &  0 &  0 &  1 &  0 &  0 &  0 & -1 & 24   & \\
	 0 &  1 &  0 & -1 & -1 &  0 &  0 &  1 &  0 &  0 &  0 & 1234 & \\
	-1 &  0 &  0 &  0 &  0 &  0 & -1 &  0 &  1 &  0 &  0 & 124  & \\
	 0 &  0 & -1 &  0 &  0 & -1 &  0 &  0 &  0 &  1 &  0 & 134  & \\
	 0 &  0 &  0 &  0 &  0 &  0 &  0 &  0 &  0 &  0 &  1 & 14   & \\
	\end{block}
	%[00.., .00., ..00, 0u0., .0u0, 0d0., .0d0, 0uu0, 0ud0, 0du0, 0dd0]
\end{blockarray}
\]
Again, $M_\coeffSP^\coeffRHS$ is not triangularizable (it is the inverse of~$M_\coeffRHS^\coeffSP$, which is not triangularizable).
\end{remark}

\begin{proof}[Proof of \cref{prop:SPtoRHS}]
For concision, let us write~$I \star J \eqdef |\{\min I, \max I\} \cap \{\min J, \max J\}|$.
From \cref{rem:FStoRHS,prop:SPtoFSb}, we have
\[
\coeffSP_I
= \sum_{J \conn I} (-1)^{I \star J} \, \coeffFS_J
= \sum_{J \conn I} (-1)^{I \star J} \sum_{R \subseteq J} (-1)^{|J \ssm R|} \, \coeffRHS_R
= \sum_{R \in \nonsing} \bigg( \sum_{\substack{J \conn I \\ R \subseteq J}} (-1)^{I \star J \, + \, |J \ssm R|} \bigg) \, \coeffRHS_R.
\]

Observe first that if~$R \cap {]\min I, \max I[} \not\subseteq I$, then there is no~$J$ such that~$J \conn I$ and~${R \subseteq J}$.
Moreover, if~$I \ssm \{\min I, \max I\} \not\subseteq R$, then grouping the subsets~$J$ with~$J \conn I$ and~$R \subseteq J$ according to $J \cap {]\min I, \max I[}$, we obtain sums over subsets of~$(I \ssm \{\min I, \max I\}) \ssm R$ of terms of the form~$(-1)^{|J \ssm R|}$ which vanish by the inclusion-exclusion principle.
We thus focus on the case~${R \cap {]\min I, \max I[} = I \cap {]\min I, \max I[}}$.

Observe now that if~$R$ contains both~$\min I$ and an element smaller than~$\min I$, then there is no~$J$ such that~$J \conn I$ and~$R \subseteq J$.
Now if~$R \cap [1, \min I]$ is neither~$\{\min I\}$ nor~${[1, \min I[}$, then we can group again the subsets~$J$ with~$J \conn I$ and~$R \subseteq J$ according to $J \cap {[1, \min I[}$ to obtain vanishing subsums by the inclusion-exclusion principle.
By symmetry, we obtain that~$R \cap [\max I, n]$ must be either~$\{\max I\}$ or~${]\max I, n]}$ for the sum not to vanish.

Finally, assume that~$R$ is one of~$I$, $I \symdif [1, \min I]$, $I \symdif [\max I, n]$, or~$I \symdif \big( [1, \min I] \cup [\max I, n]$.
Then the only subset~$J$ such that~$J \conn I$ and~$R \subseteq J$ is the set~$R$ itself.
Moreover, we have ${I \star R = 0}$ if~$I = R$, $1$ if~$R = I \symdif [1, \min I]$ or~$R = I \symdif [\max I, n]$, and~$2$ if~$R = I \symdif ([1, \min I] \cup [\max I, n])$.
\end{proof}

%%%

\subsection{PS-quotientopes from shard polytopes}
\label{subsec:PSquotientopes}

Using \cref{prop:SPtoRHS}, we can finally show that all PS-quotientopes constructed in~\cite{PilaudSantos-quotientopes} are Minkowski sums of dilated shard polytopes, as announced in \cref{prop:main10}.
Let us quickly recall the details of the construction in~\cite{PilaudSantos-quotientopes}, adapted to the notations used here.
It starts with a \defn{forcing dominant function}~$f : 2^{[n]} \to \R_{>0}$, \ie a function such that
\[
f(S) > \sum_R f(R),
\]
where the sum ranges over~$R \subseteq [n]$ such that~$R \cap {]\min S, \max S[} = S \cap {]\min S, \max S[}$.
For~${S, R \in \nonsing}$, define the \defn{contribution} of~$S$ to~$R$ as
\[
\gamma(S,R) \eqdef
\begin{cases}
1 & \text{if } |\{\min S, \max S\} \cap R| = 1 \text{ and } S \cap {]\min S, \max S[} = R \cap {]\min S, \max S[}, \\
0 & \text{otherwise.}
\end{cases}
\]
Consider a lattice congruence~$\equiv$ of the weak order on~$\fS_n$, and let~$\c{S}_\equiv \subseteq \nonsing$ be such that~$\arcs_\equiv = (\arc_S)_{S \in \c{S}_\equiv}$.
Define the \defn{height function}~$\b{h}^\equiv : 2^{[n]} \to \R_{>0}$ by
\[
\b{h}^\equiv(R) \eqdef \sum_{S \in \c{S}_\equiv} f(S) \, \gamma(S,R).
\]
Finally, recall that we can choose~$\ray(R) \eqdef |R| \one - n \one_R$ as representative of the ray of the braid fan~$\Fan_n$ corresponding to~$R$.
Then the \defn{PS-quotientope}~$\quotientope$ of~\cite{PilaudSantos-quotientopes} is defined as
\[
\quotientope \eqdef \set{\b{x} \in \R^n}{\dotprod{\one}{\b{x}} = \b{h}^\equiv([n]) \text{ and } \dotprod{\ray(R)}{\b{x}} \le \b{h}^\equiv(R)}.
\]

\begin{proposition}
\label{prop:quotientopesMinkowskiSumsShardPolytopes}
For any forcing dominant function~$f$ and any lattice congruence~$\equiv$ of the weak order on~$\fS_n$, the PS-quotientope~$\quotientope$ is a Minkowski sum of dilated shard polytopes (up to translation).
\end{proposition}

To prove \cref{prop:quotientopesMinkowskiSumsShardPolytopes}, we need the following combinatorial statement.

\begin{lemma}
\label{lem:quotientopesMinkowskiSumsShardPolytopes}
For~$I, S \in \nonsing$, the linear combination
\[
\Gamma(S,I) \eqdef \gamma(S, I \symdif [1, \min I]) + \gamma(S, I \symdif [\max I, n]) - \gamma(S, I) - \gamma(S, I \symdif ([1, \min I] \cup [\max I, n]))
\]
is
\begin{enumerate}[(i)]
\item $2$ if~$S = I$,
\item $1$ if~$S = [x, \min I] \symdif I \cup \{y\}$ or~$S = \{x\} \cup I \symdif [\max I, y]$ for~$x < \min I$ and~$y > \max I$,
\item $-1$ if~$S = [x, \min I] \symdif I$ for~$x < \min I$ or~$S = I \symdif [\max I, y]$ for~$y > \max I$,
\item $-1$ if~$S = \{x\} \cup I$ for~$x < \min I$ or~$S = I \cup \{y\}$ for~$y > \max I$,
\item $0$ otherwise.
\end{enumerate}
\end{lemma}

\begin{proof}
For convenience in this proof, define
\[
G \eqdef I \symdif [1, \min I],
\qquad
H \eqdef I \symdif [\max I, n],
\qquad\text{and}\qquad
J \eqdef I \symdif ([1, \min I] \cup [\max I, n]),
\]
so that~$\Gamma(S,I) = \gamma(S,G) + \gamma(S,H) - \gamma(S,I) + \gamma(S,J)$.
Let~$x \eqdef \min S$ and~$y \eqdef \max S$.
We distinguish cases depending on the position of~$x$ and~$y$ with respect to~$\min I$ and~$\max I$:
\begin{itemize}
\item if~$x = \min I$ and~$y = \max I$, then~${\gamma(S,G) = \gamma(S,H) = \delta_{S = I}}$ and~$\gamma(S,I) = \gamma(S,J) = 0$,
\item if~$x < \min I$ and~$y > \max I$, then~$\gamma(S,G) = \delta_{S = [x, \min I] \symdif I \cup \{y\}}$, $\gamma(S,H) = \delta_{S = \{x\} \cup I \symdif [\max I, y]}$ and~$\gamma(S,I) = \gamma(S,J) = 0$,
\item if~$x = \min I$ and~$y > \max I$, or~$x < \min I$ and~$y = \max I$, then~$\gamma(S,G) = \gamma(S,H) = 0$, $\gamma(S,I) = \delta_{S = I \cup \{x,y\}}$, and~$\gamma(S,J) = \delta_{S = I \symdif ([x, \min I] \cup [\max I, y])}$,
\item if~$x > \min I$, then~$\gamma(S,G) = \gamma(S,I)$ and $\gamma(S,H) = \gamma(S,J)$,
\item if~$y < \max I$, then~$\gamma(S,G) = \gamma(S,J)$ and $\gamma(S,H) = \gamma(S,I)$.
\qedhere
\end{itemize}
\end{proof}

\begin{proof}[Proof of \cref{prop:quotientopesMinkowskiSumsShardPolytopes}]
The inequalities defining~$\quotientope$ can be rewritten as~$\dotprod{\one_R}{\b{x}} \ge \coeffRHS_R$~where
\[
\coeffRHS_R \eqdef \frac{|R|}{n} \, \b{h}^\equiv([n]) - \frac{1}{n} \, \b{h}^\equiv(R).
\]
By \cref{prop:SPtoRHS}, we thus have for any~$I \in \nonsing$
\[
\coeffSP_I = \coeffRHS_I - \coeffRHS_{I \symdif [1, \min I]} - \coeffRHS_{I \symdif [\max I, n]} + \coeffRHS_{I \symdif ([1, \min I] \cup [\max I, n])} = \frac{1}{n} \sum_{S \in \c{S}_\equiv} f(S) \, \Gamma(S,I),
\]
where the last equality uses~$|I| - |I \symdif [1, \min I]| - |I \symdif [\max I, n]| + |I \symdif ([1, \min I] \cup [\max I, n])| = 0$.
By \cref{lem:quotientopesMinkowskiSumsShardPolytopes}, all the values~$\Gamma(S,I)$ are positive, except when~$S = [x, \min I] \symdif I$ or~$S = \{x\} \cup I$ for~$x < \min I$ or ${S = I \symdif [\max I, y]}$ or~$S = I \cup \{y\}$ for~$y > \max I$.
But as soon as~$\c{S}_\equiv$ contains one of these subsets~$S$, it also contains~$I$ since~$\c{S}_\equiv$ is an ideal.
Moreover, since~$f$ is forcing dominant, we have
\[
2f(I) > \sum_{x < \min I} \big( f([x, \min I] \symdif I) + f(\{x\} \cup I) \big) + \sum_{y > \max I} \big( f(I \symdif [\max I, y]) + f(I \cup \{y\}) \big),
\]
so that the fact that~$\Gamma(I,I) = 2$ ensures that~$\coeffSP_I > 0$.
\end{proof}

\subsection{Mixed volumes of shard polytopes}
\label{subsec:mixedVolumes}

The presentation of \cref{prop:FStoSPa} allows for the determination of (mixed) volumes of shard polytopes, as stated in \cref{thm:main11}.
An introduction to the topic is given, for example, in~\cite{SangwineYager}.
Mixed volumes are defined as coefficients in the expression of the volume of a weighted Minkowski sum.
As noted in~\cite{ArdilaBenedettiDoker}, the standard properties of mixed volumes extend to virtual polytopes.

\begin{theorem}[{\cite[Prop.~3.2]{ArdilaBenedettiDoker}}]
There exists a unique function $\Vol(\polytope{P}_1, \dots, \polytope{P}_n)$ defined on $n$-tuples of polytopes in $\R^n$, called the \defn{mixed volume} of $\polytope{P}_1, \dots, \polytope{P}_n$, such that, for any collection of~$m \ge n$ polytopes $\polytope{Q}_1, \dots, \polytope{Q}_m \in \R^n$ and any real numbers $y_1, \dots, y_m$ such that the virtual polytope~${y_1 \polytope{Q}_1 + \dots + y_m \polytope{Q}_m}$ is a convex polytope, the volume of~${y_1 \polytope{Q}_1 + \dots + y_m \polytope{Q}_m}$ is a polynomial in $y_1, \dots, y_m$ given by
\[
\Vol(y_1 \polytope{Q}_1 + \dots + y_m \polytope{Q}_m) = \sum_{(i_1, \dots, i_n) \in \binom{[m]}{n}} \Vol(\polytope{Q}_{i_1}, \dots, \polytope{Q}_{i_n}) \, y_{i_1} \dots y_{i_n},
\]
where the sum is over all ordered $n$-tuples~$(i_1, \dots, i_n)$ of~$[m]$.
\end{theorem}

Among its properties, note that $\Vol(\polytope{P}, \dots, \polytope{P}) = \Vol(\polytope{P})$ and that mixed volumes are multilinear.
Therefore, via \cref{prop:FStoSPa}, we can compute (mixed) volumes of shard polytopes using mixed volumes of simplices.
This was done by A.~Postnikov~\cite{Postnikov}.

\begin{definition}
An ordered collection of subsets $J_1, \dots, J_{n-1} \subseteq [n]$ verifies the \defn{dragon marriage condition} if it satisfies any of the following equivalent conditions:
\begin{enumerate}
\item For any distinct $i_1,\dots,i_k$, we have $|J_{i_1} \cup \cdots \cup J_{i_k}| \ge k+1$.
\item For any $j\in[n]$, there is a system of distinct representatives in $J_1, \dots, J_{n-1}$ that avoids~$j$.
\item There is a system of $2$-element representatives $\{a_i,b_i\} \in J_i$ for~$i \in [n-1]$, such that $(a_1,b_1), \dots, (a_{n-1},b_{n-1})$ are edges of a spanning tree in the complete graph~$K_n$.
\end{enumerate}
\end{definition}

\begin{lemma}[\cite{Postnikov}]
The mixed volume of a collection~$\triangle_{J_1}, \dots, \triangle_{J_{n-1}}$ of faces of the standard simplex is
\[
\Vol(\triangle_{J_1}, \dots, \triangle_{J_{n-1}}) =
\begin{cases}
	\displaystyle \frac{1}{(n-1)!} & \text{if } J_1, \dots, J_{n-1} \text{ satisfy the dragon marriage condition,} \\
	0 & \text{otherwise.}
\end{cases}
\]
\end{lemma}

\begin{theorem}
\label{thm:mixedVolumesShardPolytopes}
For any arcs~$\arc_1, \dots, \arc_{n-1} \in \arcs_n$, the mixed volume of~$\shardPolytope[\arc_1], \dots, \shardPolytope[\arc_{n-1}]$ is
\[
\Vol(\shardPolytope[\arc_1], \dots, \shardPolytope[\arc_{n-1}]) = \frac{1}{(n-1)!}\sum_{J_1, \dots, J_{n-1}} (-1)^{\asize{J_1}[A_1] + \dots +\asize{J_{n-1}}[A_{n-1}]},
\]
summing over all collections $(J_1, \dots, J_{n-1})\in \nonsing^{n-1}$ verifying $(A_i \cup \{a_i,b_i\}) \conn J_i$ for all $i \in [n-1]$, and such that $J_1, \dots ,J_{n-1}$ satisfies the dragon marriage condition.
\end{theorem}

\begin{corollary}
For any arc~$\arc\in \arcs_n$, the volume of its shard polytope~$\shardPolytope$ is
\[
\Vol(\shardPolytope) = \frac{1}{(n-1)!}\sum_{J_1,\dots, J_{n-1}} (-1)^{\asize{J_1}[A] + \dots + \asize{J_{n-1}}[A]},
\]
summing over all collections $(J_1,\dots, J_{n-1}) \in \nonsing^{n-1}$ verifying $(A \cup \{a,b\}) \conn J_i$ for all $i \in [n-1]$, and such that $J_1,\dots ,J_{n-1}$ satisfies the dragon marriage condition. 
\end{corollary}

%%%%%%%%%%%%%%%%%%%%%%%%%%%%%%%%%%%%%%
%%%%%%%%%%%%%%%%%%%%%%%%%%%%%%%%%%%%%%

\newpage
\part{Type $B$ shard polytopes}
\label{part:typeB}

%%%%%%%%%%%%%%%%%%%%%%%%%%%%%%%%%%%%%%

\section{Type~$B$ combinatorics and geometry}

In this section, we first briefly recall the classical combinatorial model for type~$B$ Coxeter groups in terms of centrally symmetric permutations, and then describe the geometry of type~$B$ lattice congruences in terms of centrally symmetric arcs.
Note that the theory of shards was developed for arbitrary hyperplane arrangements whose poset of regions is a lattice (see~\cite{Reading-PosetRegionsChapter} and \cref{subsec:concludingRemarks}), and the particular situation of type~$B$ was detailed in~\cite{Reading-latticeCongruences}.
We just present here a convenient combinatorial model in terms of centrally symmetric arcs.
We omit the proofs of the statements that are just specializations of more general statements on hyperplane arrangements, or that are proved exactly as in type~$A$.

%%%%%%%%

\subsection{Type~$B$ permutations and noncrossing arc diagrams}
\label{subsec:noncrossingArcDiagramsB}

\enlargethispage{.2cm}
We use the standard notations~$[n] \eqdef \{1, \dots, n\}$, $-[n] \eqdef \{-n, \dots, -1\}$ and~$[\pm n] \eqdef -[n] \cup [n]$.
A \defn{$B$-permutation} is a permutation of~$[\pm n]$ whose permutation table is centrally symmetric, that is~$\sigma(-i) = -\sigma(i)$ for all~$i \in [\pm n]$.
We denote by~$\fB_n$ the set of $B$-permutations of~$[\pm n]$, which forms a group under composition, called the \defn{type~$B$ Coxeter group}.
A $B$-permutation~$\sigma$ is clearly determined by the \defn{signed permutation} $\sigma_1 \dots \sigma_n$, where~$\sigma_i = \sigma(i)$ and where~$-k$ is denoted by~$\bar k$.
We use this convention for compactness in many pictures.
We refer to the monographs~\cite{Humphreys} and~\cite{BjornerBrenti} (in particular Section~8.1) for more details on the combinatorics of the type~$B$ Coxeter groups.

Following \cref{subsec:noncrossingArcDiagrams}, we define an \defn{$A$-arc} as a quadruple~$(a, b, A, B)$ consisting of two integers~$a < b \in [\pm n]$ and a partition~$A \sqcup B = {]a,b[} \ssm \{0\}$.
We use the same diagrammatic representation of $A$-arcs as in \cref{subsec:noncrossingArcDiagrams}.
We denote by~$-\arc \eqdef (-b, -a, -B, -A)$ the symmetric of~$\arc \eqdef (a, b, A, B)$, whose diagram is obtained by a central symmetry with respect to the origin.

\begin{definition}
A \defn{$B$-arc} on $[\pm n]$ is either a centrally symmetric $A$-arc on~$[\pm n]$ or a centrally symmetric and noncrossing pair of $A$-arcs on~$[\pm n]$ with disjoint endpoints.
\end{definition}

In both cases, we write a $B$-arc as~$\Barc \eqdef (-\arc, \arc)$, where~$\arc$ is the rightmost of the two $A$-arcs $-\arc$ and~$\arc$ (we repeat~$-\arc = \arc$ when the $B$-arc~$\Barc$ is just a centrally symmetric $A$-arc~$\arc$).
We call~$\arc$ the \defn{representative} $A$-arc of the $B$-arc~$\Barc$.
If~$a < b$ denote the endpoints of the representative $A$-arc~$\arc$, we thus have~$0 < b$, and as illustrated in \cref{fig:Barcs}, we say that~$\Barc$ is
\begin{itemize}
\item \defn{separated} if $a \in [n]$ (\ie $-\arc$ and~$\arc$ are strictly separated by the origin),
\item \defn{singular} if $a = -b$ (\ie $-\arc = \arc$ is centrally symmetric),
\item \defn{overlapped} if $a \in -[n] \ssm \{-b\}$ (\ie $-\arc$ and~$\arc$ overlap over the origin).
\end{itemize}

\begin{figure}[h]
	\capstart
	\centerline{\includegraphics[scale=.85]{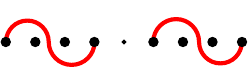} \hspace{1cm} \includegraphics[scale=.85]{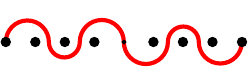} \hspace{1cm} \includegraphics[scale=.85]{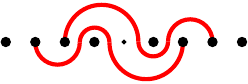}}
	\vspace{-.1cm}
	\caption{Some separated (left), singular (middle), and overlapped (right) $B$-arcs on~$[\pm 4]$.}
	\label{fig:Barcs}
\end{figure}

When~$\Barc$ is overlapped, the \defn{upper} (resp.~\defn{lower}) $A$-arc of~$\Barc$ is the $A$-arc which passes above (resp.~below) the other one (it is well-defined since the two $A$-arcs overlap and are non-crossing).

We denote by~$\Barcs_n$ the set of $B$-arcs on~$[\pm n]$.
\cref{fig:BforcingOrder} shows the $23$ $B$-arcs on~$[\pm 3]$.
More generally, $B$-arcs are enumerated as follows (an alternative argument is given in \cref{lem:numberShards}).

\begin{proposition}
\label{prop:numberBarcs}
The number of $B$-arcs on~$[\pm n]$ is $|\Barcs_n| = 3^n-1-n$.
\end{proposition}

\begin{proof}
We prove the statement by induction on~$n$.
For~$n = 1$, there is indeed a single $B$-arc (the centrally symmetric $A$-arc connecting~$-1$ to $1$).
For the induction step, we note that the set of $B$-arcs on~$[\pm n]$ is the disjoint union of the $B$-arcs on~$[\pm (n-1)]$ and those $B$-arcs on~$[\pm n]$ incident to the dots $n$ and $-n$.
We count the latter by splitting them into three categories, depending on the endpoints~$a < b$ of their representative $A$-arc~$\arc$:
\begin{itemize}
\item there are~$u_n \eqdef \sum_{k = 1}^{n-1} 2^{n-k-1} = 2^{n-1}-1$ separated ones (if $a \in [n]$, we can choose the points of~$]a,n[$ to lie above or~below~$\arc$, and the rest is determined by central~symmetry),
\item there are~$v_n \eqdef 2^{n-1}$ singular ones (if~$a = -b$, we can choose the points of~$[1,n[$ to lie above and below~$\arc$, and the rest is determined by central symmetry),
\item there are~$w_n \eqdef \sum_{k = 1}^{n-1} 2^{n-k}3^{k-1} = 2 \cdot 3^{n-1} - 2^n$ overlapped ones (if $a \in -[n] \ssm \{b\}$, we can choose the points of~$[-a,n[$ to lie above or below~$\arc$, and the points of~$[1,-a[$ to lie above both, below both, or between~$-\arc$ and~$\arc$, and the rest is determined by central symmetry).
\end{itemize}
We conclude by induction that
\[
|\Barcs_n| = |\Barcs_{n-1}| + u_n + v_n + w_n = 3^{n-1} - n + 2^{n-1} - 1 + 2^{n-1} + 2 \cdot 3^{n-1} - 2^n = 3^n - 1 - n.
\qedhere
\]
\end{proof}

We now consider noncrossing collections of $B$-arcs, which will represent $B$-permutations.

\begin{definition}
A \defn{noncrossing $B$-arc diagram} is a collection of $B$-arcs whose union forms a noncrossing $A$-arc diagram.
Equivalently, it is a centrally symmetric noncrossing arc diagram~of~$[\pm n]$.
\end{definition}

Applying the procedure of \cref{subsec:noncrossingArcDiagrams} on centrally symmetric permutations of~$[\pm n]$, we obtain maps~$\overline{\delta}{}^\textsc{b}$ and~$\underline{\delta}{}^\textsc{b}$ from $B$-permutations on~$[\pm n]$ to noncrossing $B$-arc diagrams on~$[\pm n]$ illustrated in \cref{fig:noncrossingBArcDiagrams}.
More formally, the noncrossing $B$-arc diagram~$\overline{\delta}{}^\textsc{b}(\sigma)$ associated to a $B$-permutation~$\sigma$ on~$[\pm n]$ contains:
\begin{itemize}
\item for all~$1 \le i < n$ such that~$\sigma_i < \sigma_{i+1}$, the $B$-arc~$\overline{\Barc}(\sigma, i) \eqdef (\overline{\arc}(\sigma, i), \overline{\arc}(\sigma, -i-1))$, and
\item if~$\sigma_{-1} < \sigma_1$, the singular $B$-arc~$\overline{\Barc}(\sigma, 0)$ defined as the centrally symmetric $A$-arc given by $(\sigma_{-1}, \sigma_1, \set{\sigma_j}{j < -1, \, \sigma_{-1} < \sigma_j < \sigma_1}, \set{\sigma_j}{j > 1, \, \sigma_{-1} < \sigma_j < \sigma_1})$.
\end{itemize}
The noncrossing $B$-arc diagram~$\underline{\delta}{}^\textsc{b}(\sigma)$ is defined symmetrically.
See \cref{fig:noncrossingBArcDiagrams}.
The proof of the following statement is similar to that of~\cref{thm:bijectionNoncrossingArcDiagrams} (or can be deduced from it by specializing to centrally symmetric objects on~$[\pm n]$).

\begin{figure}
	\capstart
	\centerline{\includegraphics[scale=.8]{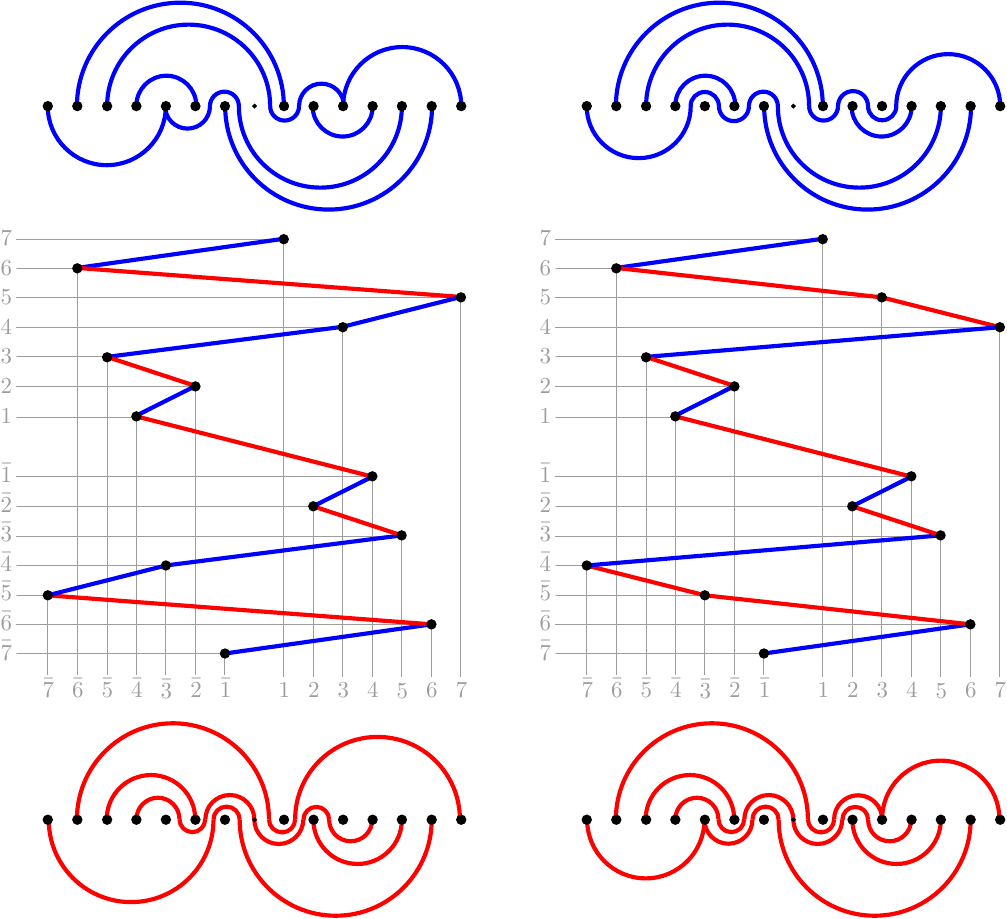}}
	\caption{The noncrossing $B$-arc diagrams~$\underline{\delta}{}^\textsc{b}(\sigma)$ (bottom) and~$\overline{\delta}{}^\textsc{b}(\sigma)$ (top) for the $B$-permutations $\sigma = \bar4 \bar2 \bar5 3 7 \bar6 1$ and $\bar4 \bar2 \bar5 7 3 \bar6 1$.}
	\label{fig:noncrossingBArcDiagrams}
\end{figure}

\begin{theorem}
\label{thm:bijectionNoncrossingArcDiagramsB}
The map~$\underline{\delta}{}^\textsc{b}$ (resp.~$\overline{\delta}{}^\textsc{b}$) is a bijection from the $B$-permutations of~$\fB_n$ to the noncrossing $B$-arc diagrams on~$\Barcs_n$.
\end{theorem}

%%%%%%%%

\subsection{Type~$B$ weak order and canonical join and meet representations}
\label{subsec:canonicalJoinRepresentationsB}

We consider the \defn{weak order} on~$\fB_n$ defined by~$\sigma \le \tau \iff \Binv(\sigma) \subseteq \Binv(\tau)$ where
\[
\Binv(\sigma) \eqdef \set{(\sigma_a, \sigma_b)}{1 \le a < b \le n \text{ and } \sigma_a > \sigma_b} \cup \set{(\sigma_{-a}, \sigma_b)}{1 \le a \le b \le n \text{ and } \sigma_{-a} > \sigma_b}
\]
is the \defn{inversion set} of the $B$-permutation~$\sigma$.
Cover relations in the weak order correspond on $B$-permutations to centrally symmetric swaps of two letters at consecutive positions: either one swap of positions~$\{\pm 1\}$, or two simultaneous swaps at positions~$\{-i-1, -i\}$ and~$\{i, i+1\}$.
On the signed permutation model, the former changes the sign of the first letter, while the latter swaps positions~$i$ and~$i+1$.
See \cref{fig:B3permutahedron} for the Hasse diagram of the weak order on~$\fB_3$ and some geometric representations recalled in \cref{subsec:CoxeterFanPermutahedronB}.

As in \cref{subsec:canonicalJoinRepresentations}, the weak order on~$\fB_n$ is a semidistributive lattice, and one can describe the canonical join and meet representations of a $B$-permutation as follows.
Consider a $B$-arc~${\Barc \eqdef (-\arc, \arc)}$ where~$\arc \eqdef (a, b, A, B)$ with~$A \eqdef \{a_1 < \dots < a_p\}$ and~$B \eqdef \{b_1 < \dots < b_q\}$.
We associate to the $B$-arc~$\Barc$ the $B$-permutation~$\underline{\lambda}(\Barc)$ defined as follows:
\begin{itemize}
\item if~$a = -b$, then~$\underline{\lambda}(\Barc) \eqdef [-n, \dots, a-1, a_1, \dots, a_p, b, a, b_1, \dots, b_q, b+1, \dots, n]$,
\item if~$a \in [n]$, then~$\underline{\lambda}(\Barc) \eqdef [-n, \dots, -b-1, -b_q, \dots, -b_1, -a, -b, -a_p, \dots, -a_1, -a+1, \dots, -1, \newline 1, \dots, a-1, a_1, \dots, a_p, b, a, b_1, \dots, b_q, b+1, \dots, n]$,
\item if~$a \in [-n] \ssm \{-b\}$ and~$-a \in A$, then~$\underline{\lambda}(\Barc) \eqdef [-n, \dots, -b-1, -b_q, \dots, -b_1, -a, -b, c_1, \dots, \newline c_r, b, a, b_1, \dots, b_q, b+1, \dots, n]$ where~$\{c_1 < \dots < c_r\} = (-A \ssm B) \cup (A \ssm -B)$,
\item if~$a \in [-n] \ssm \{-b\}$ and~$-a \notin A$, then~$\underline{\lambda}(\Barc) \eqdef [-n, \dots, -b-1, -d_s, \dots, -d_1, a_1, \dots, a_p, b, a, \newline c_1, \dots, c_r, -a, -b, -a_p, \dots, -a_1, d_1, \dots, d_s, b+1, \dots, n]$ where~$\{c_1 < \dots < c_r\} = B \cap -B$ and~$\{d_1, \dots, d_s\} = B \cap {]{-a}, b[}$.
\end{itemize}
We define similarly~$\overline{\lambda}(\Barc)$.
The following statement can be seen as a specialization of a more general statement on hyperplane arrangements~\cite[Thm.~9-7.11]{Reading-PosetRegionsChapter}, or can be proved as \cref{thm:joinMeetRepresentationsPermutations} in~\cite[Thm.~2.4]{Reading-arcDiagrams}.

\begin{theorem}
\label{thm:joinMeetRepresentationsPermutationsB}
The canonical join and meet representations of a $B$-permutation~$\sigma$ are given by $\bigJoin \set{\underline{\lambda}{}^\textsc{b}(\underline{\Barc})}{\underline{\Barc} \in \underline{\delta}{}^\textsc{b}(\sigma)}$ and~$\bigMeet \set{\overline{\lambda}{}^\textsc{b}(\overline{\Barc})}{\overline{\Barc} \in \overline{\delta}{}^\textsc{b}(\sigma)}$.
\end{theorem}

%%%%%%%%

\subsection{Type~$B$ lattice quotients}
\label{subsec:latticeQuotientsB}

\enlargethispage{.2cm}
The following statement is the analogue of \cref{thm:joinMeetRepresentationsQuotient} and is based on the semidistributivity of the weak order on~$\fB_n$.

\begin{theorem}
\label{thm:joinMeetRepresentationsQuotientB}
For any lattice congruence~$\Bequiv$ of the weak order on~$\fB_n$, the set of join-irreducibles of~$\fB_n$ uncontracted by~$\Bequiv$ corresponds to a set of $B$-arcs~$\Barcs_\Bequiv$, and the canonical join representations in the lattice quotient~$\fB_n/{\Bequiv}$ correspond to noncrossing $B$-arc diagrams using only $B$-arcs of~$\Barcs_\Bequiv$.
\end{theorem}

We now aim at an analogue of \cref{thm:arcIdeals} based on the congruence uniformity of the weak order on~$\fB_n$.
For this, we just need to describe the forcing order on $B$-arcs.

\begin{definition}
\label{def:forcingB}
A $B$-arc~$\Barc \eqdef (-\arc, \arc) \in \Barcs_n$ \defn{forces} a $B$-arc~$\Barc' \eqdef (-\arc', \arc') \in \Barcs_n$ when
\begin{enumerate}[(i)]
\item if~$\Barc$ is overlapped, then~$\Barc'$ is overlapped and the upper $A$-arc of~$\Barc$ forces the upper $A$-arc~of~$\Barc'$,
\item otherwise, some $A$-arc of~$\Barc$ forces some $A$-arc of~$\Barc'$.
\end{enumerate}
\end{definition}

We denote this relation by~$\Barc \succ \Barc'$.
Note that, for overlapped $B$-arcs, forcing is a bit subtle:
\begin{itemize}
\item an overlapped $B$-arc only forces overlapped $B$-arcs, but might be forced by any arc, (\eg \smash{\raisebox{-.22cm}{\includegraphics[scale=.8]{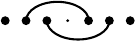}}} does not force \smash{\raisebox{-.22cm}{\includegraphics[scale=.8]{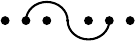}}} which forces \smash{\raisebox{-.22cm}{\includegraphics[scale=.8]{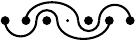}}}),
\item the forcing condition for overlapped $B$-arcs involves the vertical order between the $A$-arcs in each $B$-arc (\eg \smash{\raisebox{-.22cm}{\includegraphics[scale=.8]{BshardForcing2}}} is forced by \smash{\raisebox{-.22cm}{\includegraphics[scale=.8]{BshardForcing1}}} but not by~\smash{\raisebox{-.22cm}{\includegraphics[scale=.8]{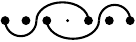}}}).
\end{itemize}
The \defn{$B$-arc poset} is the poset~$(\Barcs_n, \prec)$ of all $B$-arcs ordered by forcing.
The forcing relation and the arc poset on~$\Barcs_3$ are illustrated in \cref{fig:BforcingOrder}.
We thus obtain the following description of the lattice congruences of the weak order on~$\fB_n$.
This description can be seen as a translation of~\cite{Reading-latticeCongruences} in terms of $B$-arcs, or can be proved directly using the geometric definition of shards~\cite{Reading-PosetRegionsChapter}.

\begin{figure}
	\capstart
	\centerline{\includegraphics[scale=.58]{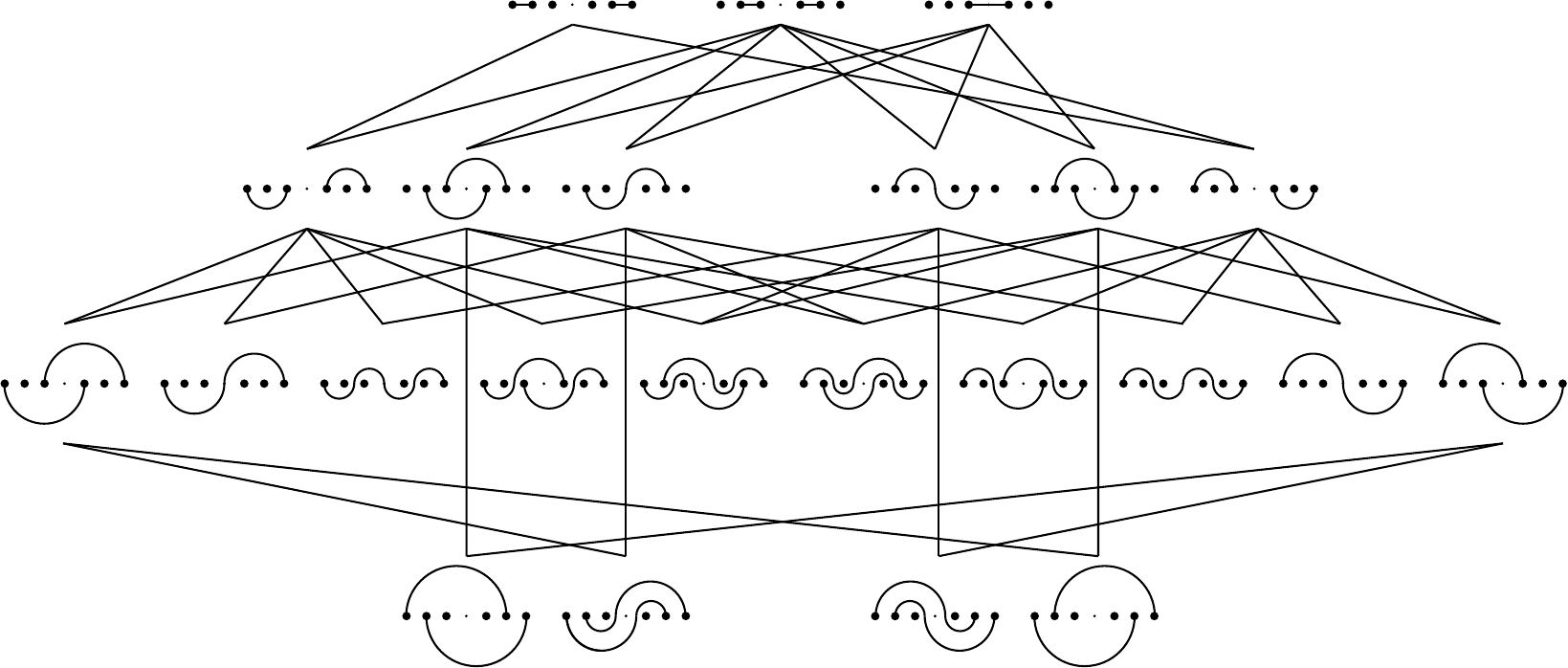}}
	\caption{The $B$-arc poset for~$n = 3$.}
	\label{fig:BforcingOrder}
\end{figure}

\begin{theorem}
\label{thm:BarcIdeals}
The map~${\Bequiv} \mapsto \Barcs_\Bequiv$ is a bijection between the lattice congruences of the weak order on~$\fB_n$ and the upper ideals of the $B$-arc poset~$(\Barcs_n, \prec)$.
\end{theorem}

A \defn{$B$-arc ideal} is an upper ideal of the $B$-arc poset~$(\Barcs_n, \prec)$, and we make no distinction between $B$-arc ideals and lattice congruences of the weak order on~$\fB_n$.

\medskip
We now observe that certain $B$-arc ideals are already familiar, as they can be obtained from $A$-arc ideals of~$[\pm n]$.
Observe first that some $B$-arc ideals can be considered as $A$-arc ideals.
%The following statements are direct consequences of \cref{def:forcingB}.

\begin{corollary}
\label{coro:BcongGivesAcong}
The $B$-arc ideal generated by a separated (resp.~singular) $B$-arc only contains separated (resp.~separated or singular) $B$-arcs, and their union is an $A$-arc ideal of~$[\pm n]$.
\end{corollary}

Conversely, any $A$-arc ideal can be converted into a $B$-arc ideal as follows.
We say that an \mbox{$A$-arc~$\arc$} on~$[\pm n]$ is \defn{centrally symmetrizable} if~$(-\arc, \arc)$ is a $B$-arc, \ie if either $\arc$ is centrally symmetric, or $a \ne -b$ and~$-\arc$ and~$\arc$ are noncrossing.

\begin{corollary}
\label{coro:AcongImpliesBcong}
For any $A$-arc ideal~$\arcs$ on~$[\pm n]$, the set~$\Barcs$ of $B$-arcs~$(-\arc, \arc)$ for all centrally symmetrizable arcs~$\arc \in \arcs$ is a $B$-arc ideal.
\end{corollary}

\begin{example}[$B$-Cambrian]
\label{exm:CambrianCongruencesB}
For a separated or singular $B$-arc~$\Barc \eqdef (-\arc, \arc)$, we denote by~$\Barcs_\Barc$ the $B$-arc ideal obtained as in \cref{coro:AcongImpliesBcong} from the $A$-arc ideal~$\arcs_\arc$ generated by~$\arc$ already considered in \cref{exm:CambrianCongruences}.
Note that when the $B$-arc~$\Barc$ is separated, the $B$-arc ideal~$\Barcs_\Barc$ is generated by~$\Barc$ as in \cref{coro:BcongGivesAcong}.
However, this is not anymore the case when the $B$-arc~$\Barc$ is singular.
Indeed, while the $A$-arc ideal~$\arcs_\arc$ is always primitive, the corresponding $B$-arc ideal~$\Barcs_\Barc$ is not primitive in that case.
For instance, the Cambrian congruence corresponding to the $B$-arc~\smash{\raisebox{-.22cm}{\includegraphics[scale=.8]{BshardForcing4}}} is generated by \smash{\raisebox{-.22cm}{\includegraphics[scale=.8]{BshardForcing4}}} and \smash{\raisebox{-.22cm}{\includegraphics[scale=.8]{BshardForcing1}}}.
The congruence~$\equiv_\Barc$ of the weak order on~$\fB_n$ with $B$-arc ideal~$\Barcs_\Barc$ is called the \defn{$\Barc$-Cambrian congruence}, and the lattice quotient~$\fB_n/{\equiv_\Barc}$ is the \defn{$\Barc$-Cambrian lattice}.
It was introduced and extensively studied by N.~Reading in~\cite{Reading-latticeCongruences, Reading-CambrianLattices}.
\end{example}

\begin{example}
\label{exm:PermutreeCongruencesB}
More generally, for any centrally symmetric decoration~$\decoration \in \Decorations^{2n}$, the type~$B$ $\decoration$-permutree congruence is defined by the $B$-arc ideal~$\Barcs_\decoration$ obtained as in \cref{coro:AcongImpliesBcong} from the $A$-arc ideal~$\arcs_{\decoration}$ of \cref{exm:otherCongruences}\,\eqref{item:permutreeCongruence}.
\end{example}

\begin{remark}
\label{rem:typeBvstypeA}
It would be tempting to believe that lattice congruences of the weak order on~$\fB_n$ are ``just centrally symmetric lattice congruences of the type~$A$ weak order'', but the type~$B$ forcing order is not ``just centrally symmetric type~$A$ forcing order''.
More precisely, while $A$-arc ideals always yield $B$-arc ideals by \cref{coro:AcongImpliesBcong}, the converse is wrong in general.
For instance, only $12$ (resp.~$1370$) of the $19$ (resp.~$8368$) type~$B$ congruences for~$n = 2$ (resp.~$n = 3$) arise from~\cref{coro:AcongImpliesBcong}.
\end{remark}

\begin{remark}
As in \cref{rem:regularQuotients}, the Hasse diagram of a lattice quotient~$\fB_n/{\Bequiv}$ is not always regular (\ie of constant degree).
From computational experiments, it seems that:
\begin{enumerate}[(i)]
\item if a forcing maximal $B$-arcs of~$\Barcs_n \ssm \Barcs_\Bequiv$ is non-singular and crosses the horizontal axis, then the Hasse diagram of a lattice quotient~$\fB_n/{\Bequiv}$ is not regular,
\label{cond:regularWeak}
\item if the forcing maximal arcs of~$\arcs_n \ssm \arcs_\equiv$ never cross the horizontal axis except possibly at the origin, then the Hasse diagram of a lattice quotient~$\fB_n/{\Bequiv}$ is regular.
\label{cond:regularStrong}
\end{enumerate}
However, none of these conditions characterizes regularity: Condition~\eqref{cond:regularWeak} is necessary but not sufficient, while Condition~\eqref{cond:regularStrong} is sufficient but not necessary.
A characterization would be interesting.
\end{remark}

%%%%%%%%

\subsection{Type~$B$ Coxeter arrangement and permutahedron}
\label{subsec:CoxeterFanPermutahedronB}

We now switch to some geometric considerations on the weak order on~$\fB_n$.
We still consider the canonical basis~$(\b{e}_i)_{i \in [n]}$ of~$\R^n$ and define~$\b{e}_{-i} \eqdef -\b{e}_i$.
Similarly, for a vector~$\b{x} = (\b{x}_1, \dots, \b{x}_n) \in \R^n$, we define~$\b{x}_{-i} \eqdef -\b{x}_i$.
The motivation for these notations will become clear in \cref{rem:projectionAtoB}.

The \defn{type~$B$ Coxeter arrangement} is the set~$\HB_n$ of hyperplanes of the form~$\set{\b{x} \in \R^n}{\b{x}_a = \b{x}_b}$ for~$a < b \in [\pm n]$.
We could also define three families of hyperplanes (either~$\b{x}_a = \b{x}_b$, or~$\b{x}_a = 0$, or~$\b{x}_a = -\b{x}_b$ for~$1 \le a < b \le n$), but it is convenient to take advantage of the notation~$\b{x}_{-i} = -\b{x}_i$.
The hyperplane arrangement~$\HB_n$ defines a fan~$\BFan_n$ called the \defn{type~$B$ Coxeter fan}.
It~has
\begin{itemize}
\item a chamber~$\polytope{C}(\sigma) \eqdef \set{\b{x} \in \R^n}{\b{x}_{\sigma_1} \le \b{x}_{\sigma_2} \le \dots \le \b{x}_{\sigma_n}}$ for each $B$-permutation~$\sigma \in \fB_n$,
\item a ray~$\polytope{C}(R) \eqdef \R_{\ge0} \ray(R)$ for each non-empty signed subset~$R$ (\ie subset of~$[\pm n]$ such that $\{-i,i\} \not\subset R$), where~$\ray(R) = \one_R = \sum_{i \in R} \b{e}_i$.
\end{itemize}
The chamber~$\polytope{C}(\sigma)$ has rays~$\polytope{C}(\sigma([k,n]))$ for~$k \in [n]$.
Note that~$\BFan_n$ has~$2^n n!$ chambers, $n 2^{n-1} n!$ walls supported by~$n^2$ hyperplanes, and~$3^n-1$ rays.
\cref{fig:B3permutahedron}\,(middle) shows the fan~$\BFan_3$ intersected with the cube~$[-1,1]^3$.
This illustrates the fact that the hyperplanes (resp.~rays) of the arrangement~$\HB_n$ are the reflection hyperplanes (resp.~correspond to the faces) of the $n$-dimensional cube~$[-1,1]^n$, or dually of the $n$-dimensional cross-polytope~$\conv\set{\pm \b{e}_i}{i \in [n]}$.
See also \cref{fig:B3shards} for the same fan~$\BFan_3$ intersected with a sphere and stereographically projected to the plane, and \cref{fig:B2shards}\,(left) for the fan~$\BFan_2$.
In these pictures, hyperplanes are labeled with inequalities of the same color corresponding to the halfspace in which the inequality appears, chambers are labeled with blue signed permutations and rays are labeled with red signed subsets.

\begin{figure}
	\capstart
	\centerline{\includegraphics[scale=.6]{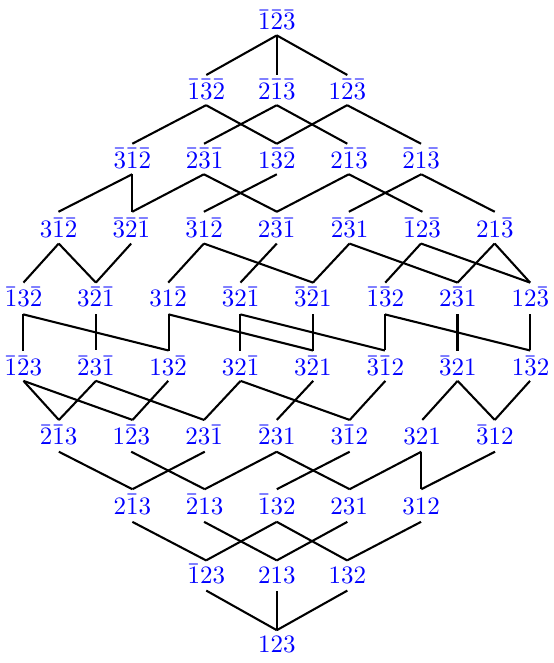} \; \includegraphics[scale=.6]{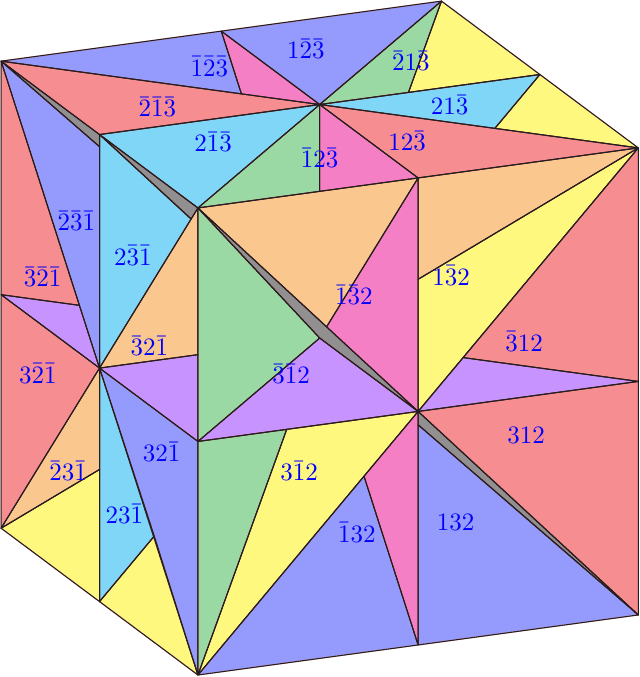} \; \includegraphics[scale=.6]{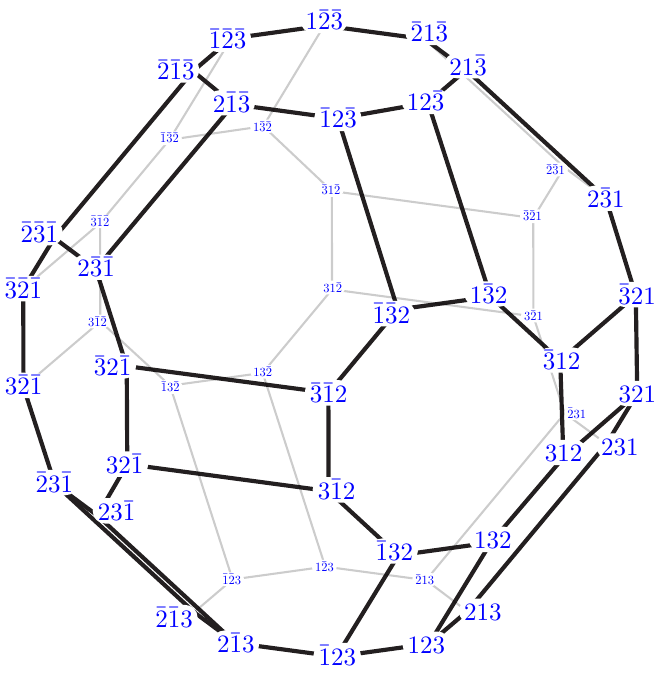}}
	\caption{The Hasse diagram of the weak order on~$\fB_3$ (left) can be seen as the dual graph of type $B_3$ Coxeter fan (middle) or as the graph of the type $B_3$ permutahedron~$\Perm[B_3]$ (right).}
	\label{fig:B3permutahedron}
\end{figure}

\begin{figure}
	\capstart
	\centerline{\includegraphics[scale=.6]{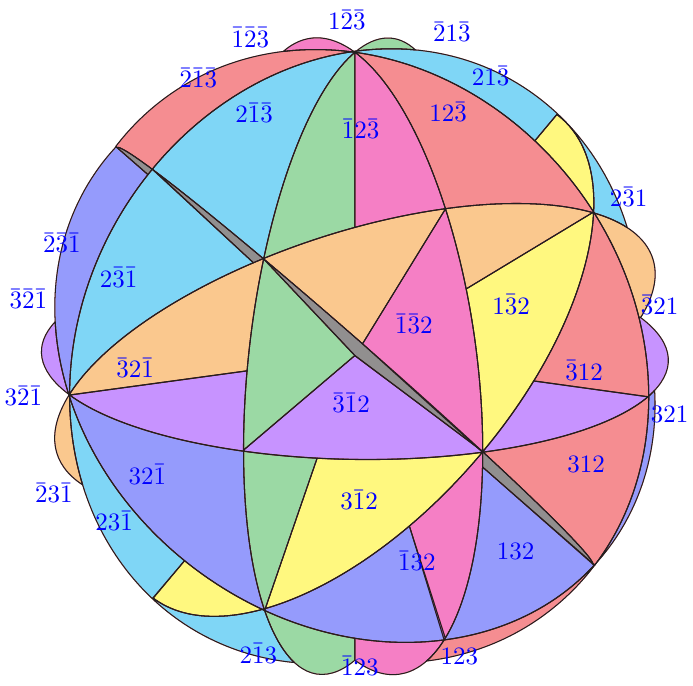} \quad \includegraphics[scale=.6]{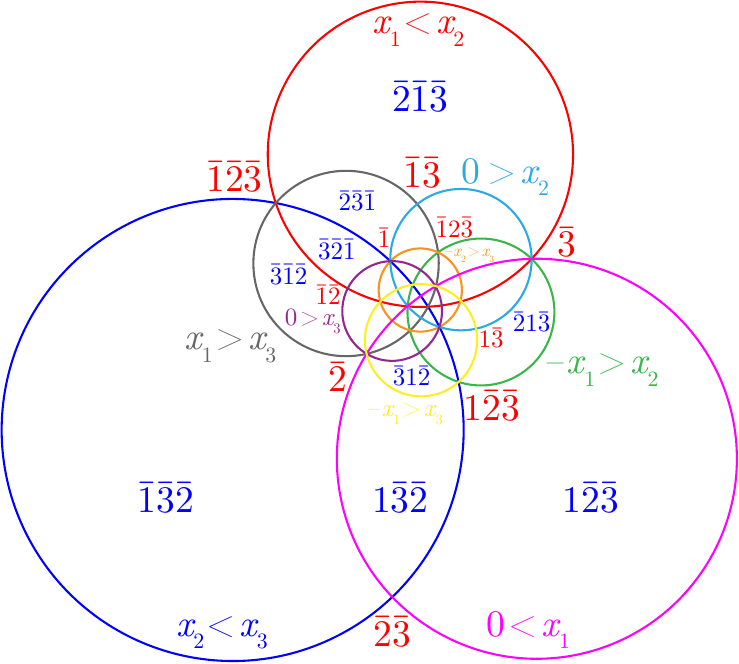}}
	\caption{The type~$B_3$ Coxeter fan~$\BFan_3$ (left) and a stereographic projection (right).}
	\label{fig:B3CoxeterArrangement}
\end{figure}

The \defn{type~$B$ permutahedron} is the polytope~$\BPerm$ defined equivalently as
\begin{itemize}
\item the convex hull of the points~$\sum_{i \in [n]} i \b{e}_{\sigma_i}$ for all $B$-permutations~$\sigma \in \fB_n$ (where~$\b{e}_{-i} = -\b{e}_i$),
\item the intersection of the halfspaces~$\bigset{\b{x} \in \R^n}{\sum_{r \in R} \b{x}_r \le \binom{n+1}{2} - \binom{|R|+1}{2}}$ for all signed subsets~$R$ of~$[\pm n]$ (where~$\b{x}_{-i} = -\b{x}_i$),
\item (a translate of) the Minkowski sum of all segments~$[\b{e}_a, \b{e}_b]$ for all $a < b \in [\pm n]$.
\end{itemize}
Note that~$\BPerm$ has~$2^n n!$ vertices, $n 2^{n-1} n!$ edges, and $3^n-1$ facets.
\cref{fig:B3permutahedron}\,(right) show the type~$B$ permutahedron~$\BPerm[3]$.
As illustrated in~\cref{fig:B3permutahedron},
\begin{itemize}
\item the normal fan of the type~$B$ permutahedron~$\BPerm$ is the type $B$ Coxeter fan~$\BFan_n$,
\item the Hasse diagram of the weak order on~$\fB_n$ can be seen geometrically as the dual graph of the type~$B$ Coxeter fan~$\BFan_n$, or the graph of the type~$B$ permutahedron~$\BPerm$, oriented in the linear direction~$-\sum_{i \in [n]} i \, \b{e}_i$.
\end{itemize}

\pagebreak

\begin{remark}
\label{rem:projectionAtoB}
There is a strong connection between the type~$A$ and $B$ Coxeter arrangements and permutahedra.
Recall from~\cite[Lem.~7.11]{Ziegler-polytopes} that, for any projection~$\rho : \R^d \to \R^e$ and any polytope~$\polytope{P} \subset \R^d$, the normal fan of~$\rho(\polytope{P})$ is isomorphic (via~$\rho^*$) to the section of the normal fan of~$\polytope{P}$ by the image of~$\rho^*$.
Let~$(\b{e}_i)_{i \in [n]}$ and~$(\b{f}_i)_{i \in [\pm n]}$ denote the canonical bases of~$\R^n$ and~$\R^{[\pm n]}$ respectively.
Consider the projection map~$\rhoB : \R^{[\pm n]} \to \R^n$ given by~${\rhoB(\b{f}_{-i}) = -\b{e}_i}$ and~${\rhoB(\b{f}_i) = \b{e}_i}$ for~$i \in [n]$.
Its dual map~${\rhoB}^* : \R^n \to \R^{[\pm n]}$ is given by~${\rhoB}^*(\b{e}_i) = \b{f}_i - \b{f}_{-i}$.
Note that the image of~${\rhoB}^*$ is the \defn{centrally symmetric space}~$\Bhyp \eqdef \set{\b{x} \in \R^{[\pm n]}}{\b{x}_{-i} = -\b{x}_i \text{ for all } i \in [\pm n]}$.~Then:
\begin{itemize}
\item the type~$B$ permutahedron~$\BPerm$ is the image of the type~$A$ permutahedron~$\Perm[{[\pm n]}]$~by~$\rhoB$,
\item the type~$B$ arrangement~$\HB_n$ is the section of the type~$A$ arrangement~$\HA_{[\pm n]}$ by~$\Bhyp$.
\end{itemize}
\end{remark}

%%%%%%%%

\subsection{Type~$B$ quotient fans and shards}
\label{subsec:shardsB}

As in \cref{subsec:quotientFanQuotientopes,subsec:shards}, we now consider the fans defined by lattice congruences of the weak order on~$\fB_n$ and their walls.
We use again the theory of shards~\cite{Reading-PosetRegionsChapter} that we specialize here explicitly in type~$B$.

Consider a $B$-arc~$\Barc \eqdef (-\arc, \arc)$ where~$\arc \eqdef (a, b, A, B)$.
The \defn{shard}~$\shard(\Barc)$ is the cone
\[
\shard(\Barc) \eqdef \set{\b{x} \in \R^n}{\b{x}_a = \b{x}_b, \; \b{x}_a \ge \b{x}_{a'} \text{ for all } a' \in A, \; \b{x}_a \le \b{x}_{b'} \text{ for all } b' \in B},
\]
where we use again the convenient convention that~$\b{x}_{-i} = -\b{x}_i$.
In other words, the shard $\shard(\Barc)$ is (the projection to~$\R^n$ of) the intersection of~$\shard(\arc)$ with the centrally symmetric space~$\Bhyp$.
\cref{fig:B2shards,fig:B3shards} illustrate the shards~$\Bshards_n$ when~$n = 2$ and~$n = 3$, labeled by the corresponding $B$-arcs.
Note that the shards corresponding to the $B$-arcs whose representative $A$-arc has endpoints~$a$ and~$b$ cover the hyperplane~$\b{x}_a = \b{x}_b$.
For instance, the shards of singular $B$-arcs with~$\arc \eqdef (a, -a, A, -A)$ cover the hyperplane~$\b{x}_a = 0$.
We denote by~$\Bshards_n \eqdef \set{\shard(\Barc)}{\Barc \in \Barcs_n}$ the set of all shards of~$\HB_n$ and by~$\Bshards_{\Barcs} \eqdef \set{\shard(\Barc)}{\Barc \in \Barcs}$ the set of shards corresponding to a set~$\Barcs$ of~$B$-arcs.
Note that ${|\Bshards_n| = |\Barcs_n| = 3^n-n-1}$, as can be seen by the combinatorial argument of \cref{prop:numberBarcs} or by the alternative geometric argument of \cref{lem:numberShards}.

\medskip
We now consider lattice congruences of the weak order on~$\fB_n$ and the geometry of their quotient fans.
The following statement is again a specialization of N.~Reading's results.
See for instance~\cite[Thm.~9-8.3]{Reading-PosetRegionsChapter}.

\begin{theorem}
\label{thm:quotientFanB}
Any lattice congruence~$\Bequiv$ of the weak order on~$\fB_n$ corresponding to a $B$-arc ideal~$\Barcs \subseteq \Barcs_n$ defines a \defn{quotient fan}~$\BFan_\Bequiv = \BFan_{\Barcs}$ whose chambers are obtained as
\begin{enumerate}[(i)]
\item either the unions of the chambers~$\polytope{C}(\sigma)$ of the Coxeter fan~$\BFan_n$ corresponding to $B$-permutations~$\sigma$ that belong to the same congruence class of~$\Bequiv$,
\item or the closures of the connected components of the complement of the union of the shards~of~$\Bshards_{\Barcs}$.
\end{enumerate}
\end{theorem}

\begin{figure}[b]
	\capstart
	\centerline{\includegraphics[scale=.9]{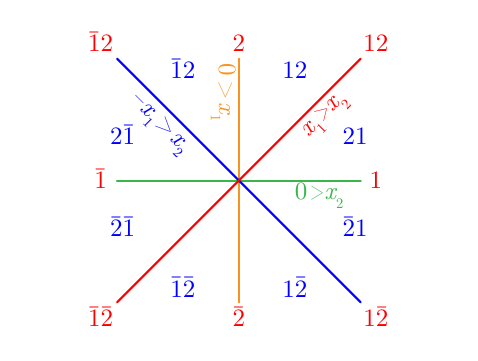} \includegraphics[scale=.9]{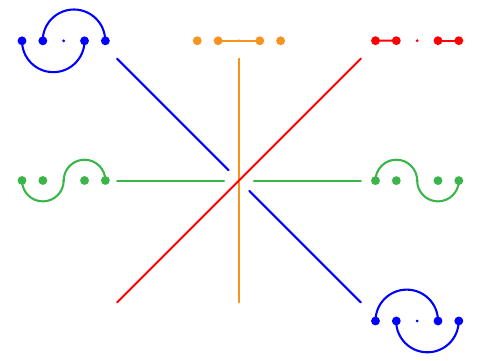}}
	\caption{The type~$B_2$ Coxeter fan~$\BFan_2$ (left) and the corresponding $B$-shards (right).}
	\label{fig:B2shards}
\end{figure}

\begin{figure}
	\capstart
	\centerline{\includegraphics[scale=.8]{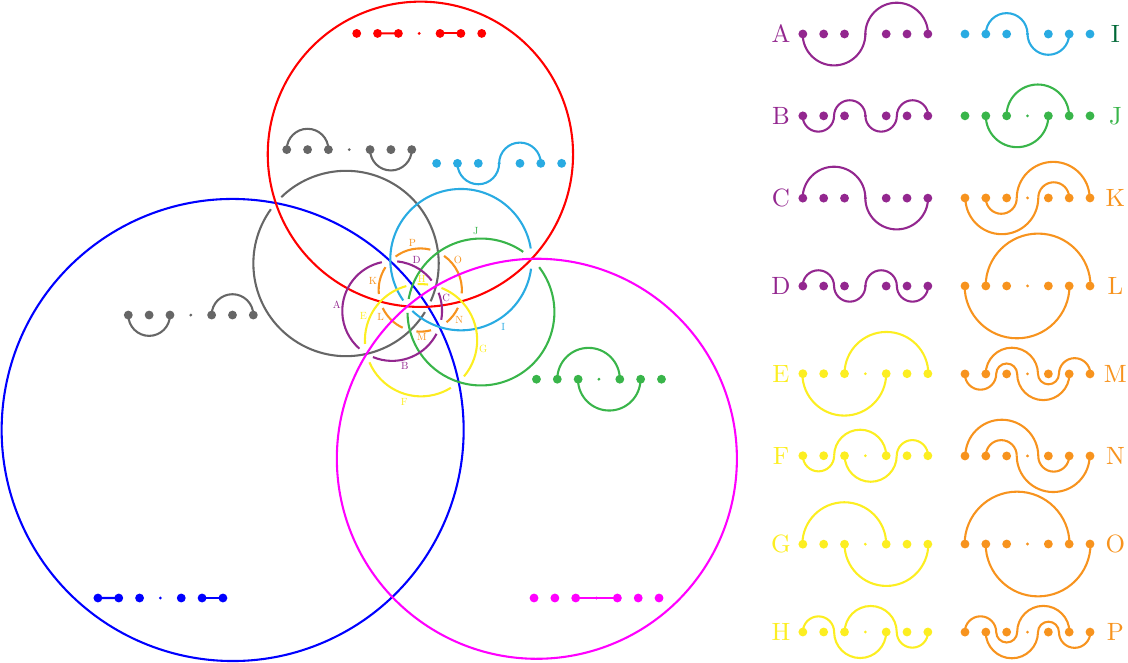}}
	\caption{The $B$-shards seen on the stereographic projection of the type~$B_3$ Coxeter fan.}
	\label{fig:B3shards}
\end{figure}

Similar to \cref{lem:raysQuotientFan}, one can describe the rays of the quotient fan~$\BFan_{\Barcs}$.

\begin{lemma}
\label{lem:BraysQuotientFan}
For any $B$-arc ideal~$\Barcs \subseteq \Barcs_n$ and any non-empty signed subset~$\varnothing \ne R \subset [\pm n]$, the ray~$\polytope{C}(R)$ of the Coxeter fan~$\BFan_n$ is also a ray of the quotient fan~$\BFan_{\Barcs}$ if and only if for any~$a < b \in [\pm n]$, the ideal~$\Barcs$ contains a $B$-arc with representative $A$-arc
\begin{itemize}
\item $(a, b, \varnothing, {]a,b[})$ if $a, b \in R$ and~${]a,b[} \cap R = \varnothing$,
\item $(a, b, {]a,b[}, \varnothing)$ if $a, b \in -R$ and~${]a,b[} \cap -R = \varnothing$,
\item and~$(a, b, {]a,b[} \cap R, {]a,b[} \cap -R)$ if~$a, b \notin R \cup -R$ and~${]a,b[} \ssm (R \cup -R) = \varnothing$.
\end{itemize}
\end{lemma}

\begin{proof}
It is easy to check from the definition of~$\shard(\Barc)$ that the $B$-arcs described in the statement are precisely those whose shards contain the ray~$\polytope{C}(R)$ in their interior.
The result thus follows from the fact that a ray is preserved if and only if all shards containing it in their interior are preserved.
See \cite[Sect.~3.1]{AlbertinPilaudRitter} for details.
\end{proof}

In contrast to type~$A$, the polytopality of the quotient fan~$\BFan_{\Barcs}$ remained open.
However, there is a natural realization of the quotient fan~$\BFan_{\Barcs}$ when the $B$-arc ideal~$\Barcs$ can be understood as an $A$-arc ideal.
Recall that an $A$-arc~$\arc$ on~$[\pm n]$ is centrally symmetrizable when~$(-\arc,\arc)$ is a $B$-arc.
Observe that if~$\arc$ is not centrally symmetrizable, then the relative interior of its shard does not meet the centrally symmetric space~$\Bhyp$.
Our next statement is the geometric counterpart of \cref{coro:AcongImpliesBcong}.

\begin{corollary}
\label{coro:AcongImpliesBcongFan}
Consider an $A$-arc ideal~$\arcs$ on~$[\pm n]$ and the $B$-arc ideal~$\Barcs$ consisting of~$(-\arc, \arc)$ for all centrally symmetrizable arcs~$\arc \in \arcs$.
Then the quotient fan~$\BFan_{\Barcs}$ is the section of the quotient fan~$\Fan_\arcs$ by the centrally symmetric space~$\Bhyp$.
Therefore, if $\Fan_\arcs$ is the normal fan of a polytope~$\polytope{P}$ then $\BFan_{\Barc}$ is the normal fan of the image of~$\polytope{P}$ by the projection~$\rhoB$ (and even of the section of~$\polytope{P}$ by~$\Bhyp$ if~$\polytope{P}$ is symmetric with respect to~$\Bhyp$).
\end{corollary}

\begin{example}[$B$-Cambrian]
\label{exm:HohlwegLangeAssoB}
Consider the $\Barc$-Cambrian congruence of a separated or singular $B$-arc~$\Barc \eqdef (-\arc, \arc)$ defined in \cref{exm:CambrianCongruencesB}.
The quotient fan~$\BFan_\Barc$, called the \defn{$\Barc$-Cambrian fan}, is the section of the $\arc$-Cambrian fan of \cref{exm:HohlwegLangeAsso} with~$\Bhyp$.
It is realized by the \defn{$\Barc$-cyclohedron}~$\Asso[\Barc]$ of~\cite{HohlwegLange}, which can be described either as the image under the projection~$\rhoB$ of the $\arc$-associahedron $\Asso[\arc]$ of \cref{exm:HohlwegLangeAsso}, or by deleting some inequalities in the facet description of the type~$B$ permutahedron~$\BPerm$.
See \cref{fig:B3associahedra} for the $3$-dimensional cyclohedra~$\Asso[\Barc]$.
Note that the polytopality of Cambrian fans is a type~$B$ incarnation of a general phenomenon: any Cambrian lattice~\cite{Reading-latticeCongruences, Reading-CambrianLattices} in any finite Coxeter group is realized by a Cambrian fan~\cite{ReadingSpeyer} and by a generalized associahedron~\cite{HohlwegLangeThomas, Stella, PilaudStump-brickPolytope, HohlwegPilaudStella}.
\end{example}

\begin{example}
More generally, consider the type~$B$ $\decoration$-permutree congruence of \cref{exm:PermutreeCongruencesB} for a centrally symmetric decoration~$\decoration \in \Decorations^{2n}$.
The quotient fan~$\BFan_\decoration$ is the section of the $\decoration$-permutree fan, and the normal fan of the image under~$\rhoB$ of the $\decoration$-permutreehedron of~\cref{exm:otherQuotientopes}.
\end{example}

%%%%%%%%%%%%%%%%%%%%%%%%%%%%%%%%%%%%%%

\section{Type~$B$ shard polytopes and quotientopes}

In this section, we construct polytopal realizations of the type~$B$ quotient fans.
Again, we use Minkowski sums of shard polytopes, defined as projections of type~$A$ shard polytopes.

%%%%%%%%

\subsection{Type~$B$ shard polytopes}

As in \cref{sec:shardPolytopes}, our construction of type~$B$ quotientopes is based on elementary polytopes associated to $B$-arcs.

\begin{definition}
\label{def:BshardPolytope}
The \defn{shard polytope}~$\shardPolytope[\Barc]$ of a $B$-arc~$\Barc \eqdef (-\arc, \arc)$ is the convex hull of the characteristic vectors of all $\arc$-alternating matchings, with the convention that~$\b{e}_{-i} = -\b{e}_i$.
\end{definition}

\begin{remark}
\label{rem:projectionShardPolytopes}
In other words, $\shardPolytope[\Barc]$ is the image of~$\shardPolytope[\arc]$ under the projection~$\rhoB : \R^{[\pm n]} \to \R^n$ of \cref{rem:projectionAtoB}.
This interpretation is essential for the proof of \cref{prop:shardPolytopeFanB} below.
\end{remark}

The shard polytopes corresponding to all $B$-arcs for~$n = 2$ and~$n = 3$ are represented in \cref{fig:B2shardPolytopes,fig:B3shardPolytopes}.
\begin{figure}[b]
	\capstart
	\centerline{\includegraphics[scale=1]{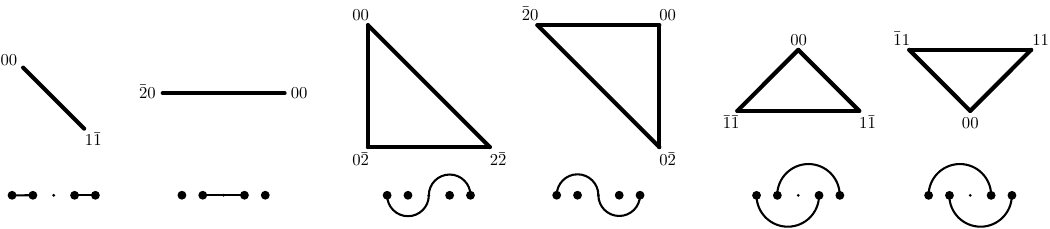}}
	\caption{Shard polytopes for all $B$-arcs with~$n = 2$.}
	\label{fig:B2shardPolytopes}
\end{figure}
\begin{figure}
	\capstart
	\centerline{\includegraphics[scale=.58]{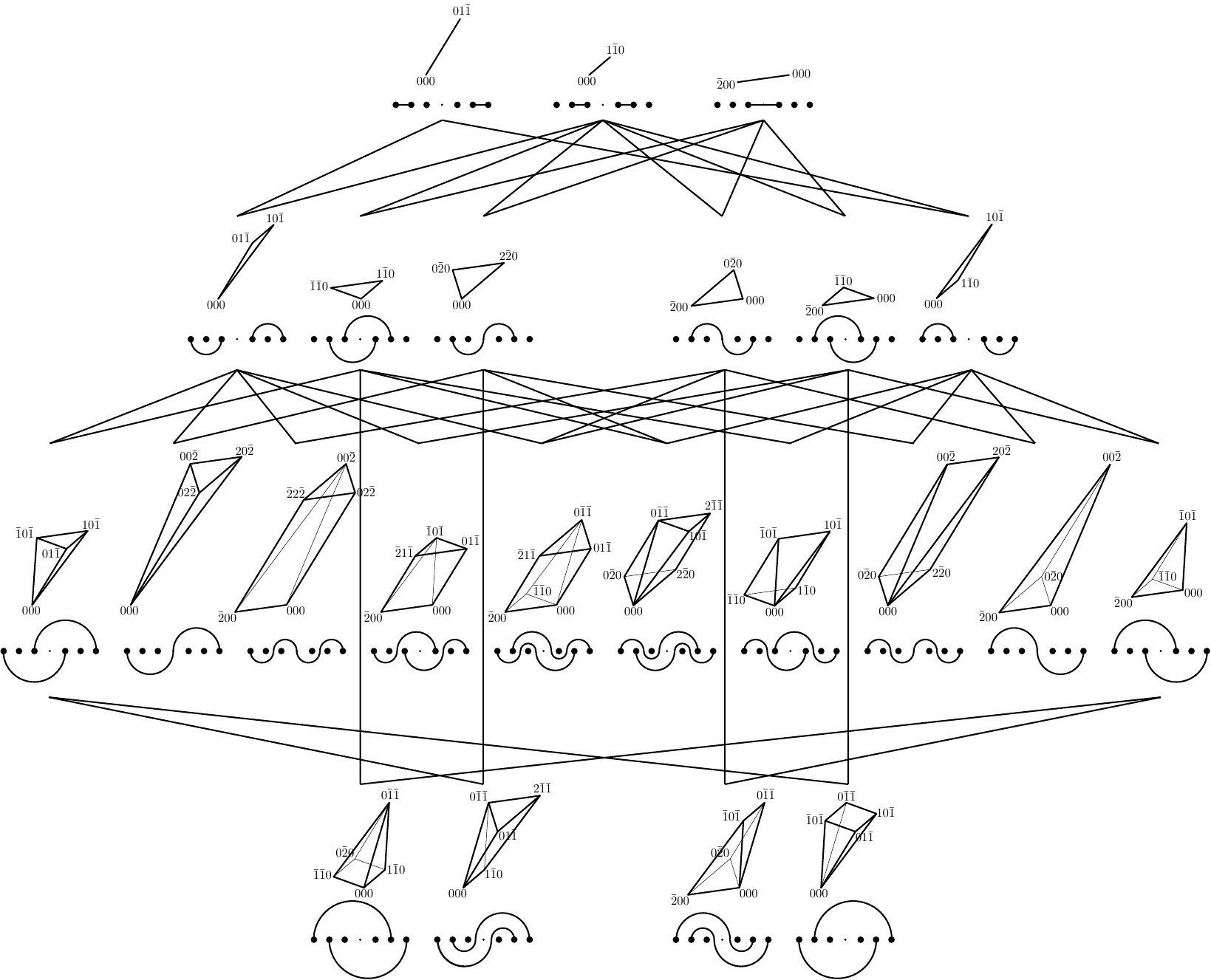}}
	\caption{Shard polytopes for all $B$-arcs with~$n = 3$.}
	\label{fig:B3shardPolytopes}
\end{figure}

\begin{remark}
For a $B$-arc~$\Barc \eqdef (-\arc, \arc)$ with $A$-arc representative~$\arc \eqdef (a, b, A, B)$, one can check that the dimension of~$\shardPolytope[\Barc]$ is~$b$ if~$\Barc$ is singular or overlapped, and~$b-a$ if~$\Barc$ is separated.
The vertex and facet descriptions are more intricate.
We just want to observe here that:
\begin{enumerate}[(i)]
\item For a separated $B$-arc~$\Barc \eqdef (-\arc, \arc)$, we just have~$\shardPolytope[\Barc] = \shardPolytope[\arc]$, so that the vertices of~$\shardPolytope[\Barc]$ are precisely the characteristic vectors of $\arc$-alternating matchings.
\item For a singular $B$-arc~$\Barc \eqdef (-\arc, \arc)$ (\ie with~$-\arc = \arc$), the vertices of~$\shardPolytope[\Barc]$ are precisely the characteristic vectors of the centrally symmetric $\arc$-alternating matchings. See \cref{lem:verticesShardPolytopeBsingular}. 
\item For an overlapped $B$-arc~$\Barc \eqdef (-\arc, \arc)$, the situation is much more intricate. In particular, in contrast to the type~$A$ situation described in \cref{prop:elemPropShardPolytope}, the vertices of the shard polytope~$\shardPolytope[\Barc]$ are not in bijection with the $\arc$-alternating matchings: some $\arc$-alternating matchings are not vertices, and some vertices correspond to multiple $\arc$-alternating matchings.
\end{enumerate}
\end{remark}

Similarly to \cref{prop:shardPolytopeFan}, we now state the main property of~$\shardPolytope[\Barc]$, whose proof is postponed to \cref{sec:proofShardPolytopeFan}.
The next two results were stated in \cref{prop:main12,coro:main13}.

\begin{proposition}
\label{prop:shardPolytopeFanB}
For any $B$-arc~$\Barc$, the union of the walls of the normal fan of the shard polytope~$\shardPolytope[\Barc]$ contains the shard~$\shard(\Barc)$ and is contained in the union of the shards~$\shard(\Barc')$ for~$\Barc \prec \Barc'$.
\end{proposition}

\begin{corollary}
\label{coro:MinkowskiSumShardPolytopesB}
For any $B$-arc ideal~$\Barcs \subseteq \Barcs_n$, the quotient fan~$\BFan_{\Barcs}$ is the normal fan of the Minkowski sum~$\shardPolytope[\Barcs] \eqdef \sum_{\Barc \in \Barcs} \shardPolytope[\Barc]$ of the shard polytopes~$\shardPolytope[\Barc]$ of all $B$-arcs~${\Barc \in \Barcs}$.
\end{corollary}

\begin{example}[$B$-Cambrian]
For the $\Barc$-Cambrian congruence of \cref{exm:CambrianCongruencesB}, the Minkowski sum~$\shardPolytope[\Barcs_\Barc]$ actually coincides with C.~Hohlweg and C.~Lange's cyclohedron~$\Asso[\Barc]$ described in \cref{exm:HohlwegLangeAssoB}.
In fact, this Minkowski decomposition of the associahedron~$\Asso[\Barc]$ already appeared in the context of brick polytopes in~\cite{PilaudStump-brickPolytope}.
\begin{figure}
	\capstart
	\centerline{
		\begin{tabular}{c@{\;}c@{\;\;}c@{\;}c}
		\includegraphics[scale=.3]{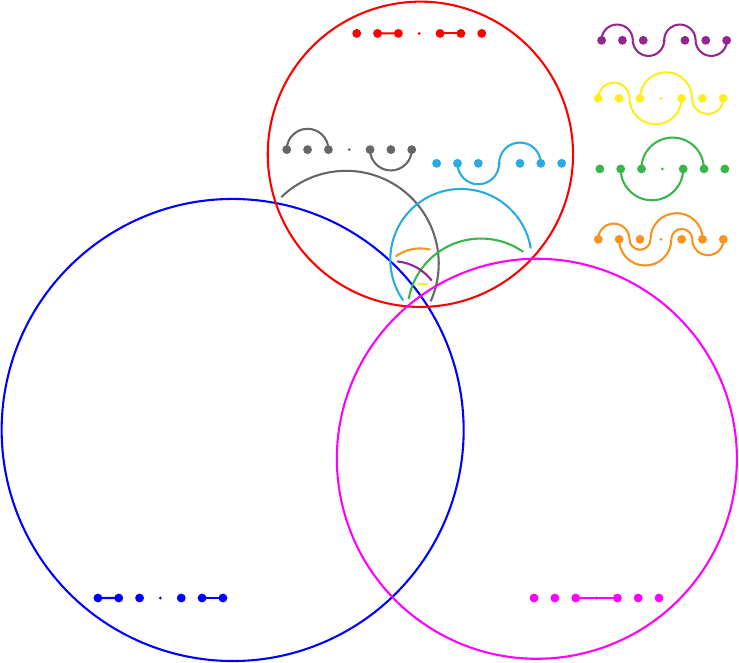} & \includegraphics[scale=.3]{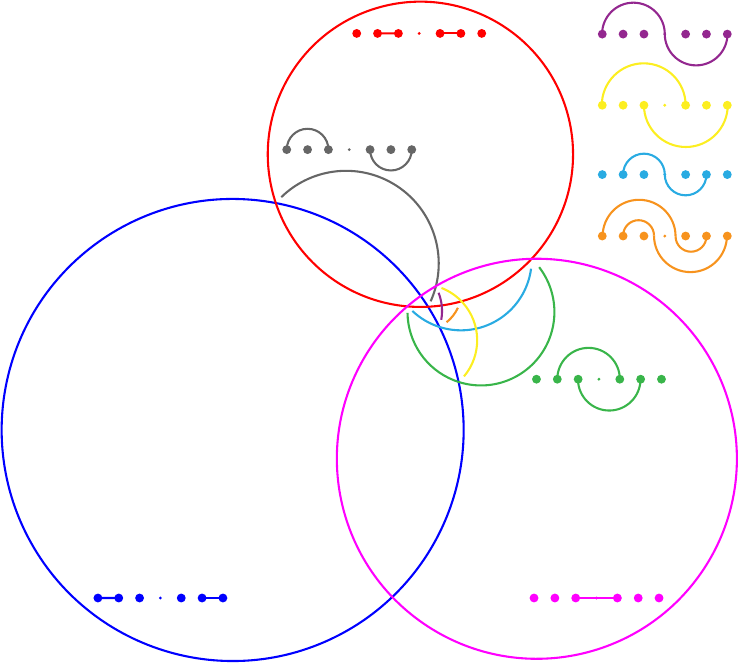} & \includegraphics[scale=.3]{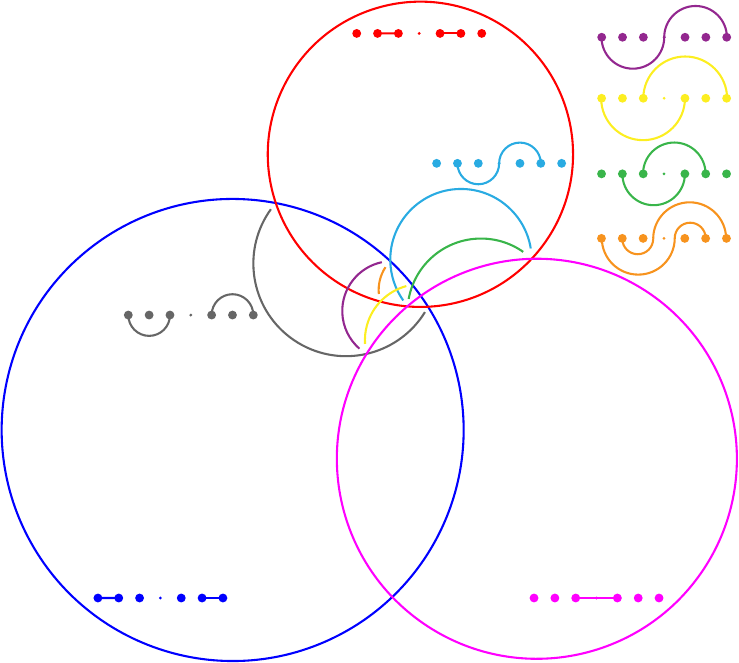} & \includegraphics[scale=.3]{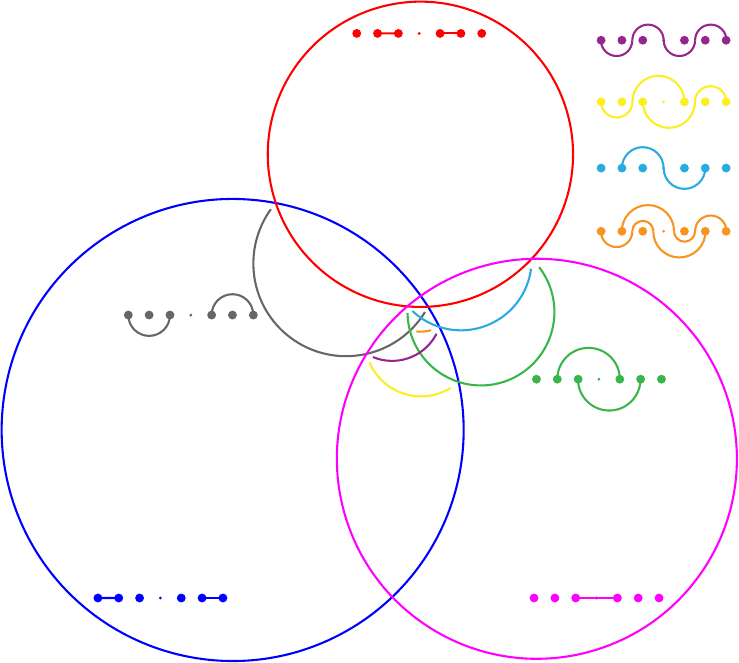} \\[.3cm]
		\input{B3asso1} & \input{B3asso2} & \input{B3asso3} & \input{B3asso4}
		\end{tabular}
	}
	\caption{Type $B_3$ associahedra obtained as Minkowski sums of shard polytopes (blue) coincide with the cyclohedra constructed in~\cite{HohlwegLange, HohlwegLangeThomas} by deleting inequalities in the facet description of the type $B$ permutahedron~$\BPerm[3]$ (red).}
	\vspace{-1cm}
	\label{fig:B3associahedra}
\end{figure}
\end{example}

\begin{example}
\label{exm:weirdPermMinkowskiSumB}
For the ideal of all $B$-arcs~$\Barcs_n$, the Minkowski sum of all shard polytopes gives a realization of the type~$B$ Coxeter fan~$\BFan_n$.
See \cref{fig:B3weirdPermutahedron} for a $3$-dimensional example.
\begin{figure}
	\capstart
	\centerline{\raisebox{2cm}{\input{B3perm}} \qquad \input{B3permWeird}}
	\caption{The standard type $B_3$ permutahedron~$\BPerm[3] \eqdef \conv\set{(\sigma_1, \dots, \sigma_n)}{\sigma \in \fB_3}$ (left) and the Minkowski sum~$\shardPolytope[\Barcs_3]$ of the shard polytopes of all $B$-arcs of~$\Barcs_3$ (right).}
	\label{fig:B3weirdPermutahedron}
\end{figure}
\end{example}

Further properties concerning the Minkowski geometry of type~$B$ shard polytopes are briefly discussed in \cref{subsec:concludingRemarks}.

%%%%%%%%

\subsection{Proof of \cref{prop:shardPolytopeFanB}}
\label{sec:proofShardPolytopeFan}

Fix a $B$-arc~$\Barc \eqdef (-\arc, \arc)$.
Recall from~\cref{rem:projectionShardPolytopes} that the type~$B$ shard polytope~$\shardPolytope[\Barc]$ can be seen as the image of the type~$A$ shard polytope~$\shardPolytope$ under the projection~$\rhoB : \R^{[\pm n]} \to \R^n$ defined in \cref{rem:projectionAtoB}.
As the image of the dual map~${{\rhoB}^* : \R^n \to \R^{[\pm n]}}$ is the centrally symmetric space~$\Bhyp$, we obtain that the normal fan of~$\shardPolytope[\Barc]$ is the section of the normal fan of~$\shardPolytope$ with~$\Bhyp$, see~\cite[Lem.~7.11]{Ziegler-polytopes}.
Therefore, the walls of the normal fan of~$\shardPolytope[\Barc]$ are the walls of the normal fan of~$\shardPolytope$ intersected with~$\Bhyp$.
This already shows the first part of \cref{prop:shardPolytopeFanB}.

\begin{lemma}
\label{lemma:ShardPolytopeFanBcontainsShard}
The union of the walls of the normal fan of~$\shardPolytope[\Barc]$ contains the shard~$\shard(\Barc)$.
\end{lemma}

\begin{proof}
By \cref{prop:shardPolytopeFan}, one of the walls of the normal fan of $\shardPolytope$ is $\shard(\arc)$, and hence $\shard(\arc) \cap \Bhyp$ is a wall of the normal fan of~$\shardPolytope[\Barc]$. But $\shard(\arc) \cap \Bhyp$ is precisely the shard~$\shard(\Barc)$.
\end{proof}

This also shows the following weak version of the second part of \cref{prop:shardPolytopeFanB}.

\begin{lemma}
\label{lem:weakShardPolytopeFanB}
The union of the walls of the normal fan of~$\shardPolytope[\Barc]$ is contained in the union of the shards~$\shard(\Barc')$ for all $B$-arcs~$\Barc' \eqdef (-\arc', \arc')$ such that~$\arc \prec \arc'$.
\end{lemma}

\begin{proof}
The walls of the normal fan of~$\shardPolytope[\Barc]$ are the walls of the normal fan of~$\shardPolytope$ intersected with~$\Bhyp$.
By \cref{prop:shardPolytopeFan}, the walls of $\shardPolytope$ are contained in the shards $\shard(\arc')$ for~$\arc \prec \arc'$.
Hence, the walls of $\shardPolytope[\Barc]$ are contained in $\shard(\arc') \cap \Bhyp$ for $\arc\prec \arc'$.
But $\shard(\arc') \cap \Bhyp$ is precisely the shard of the $B$-arc~$\shard(\Barc')$ where~$\Barc' \eqdef (-\arc',\arc')$.
\end{proof}

Our main problem here is that for two $B$-arcs~$\Barc \eqdef (-\arc, \arc)$ and~$\Barc' \eqdef (-\arc', \arc')$, the type~$A$ forcing relation $\arc \prec \arc'$ does not necessarily imply the type~$B$ forcing relation~$\Barc \prec \Barc'$.
See \cref{def:forcingB} and the discussion thereafter.
Our objective is thus to show that when~$\arc \prec \arc'$ but $\Barc \nprec \Barc'$, the centrally symmetric space~$\Bhyp$ does not intersect any wall of the normal fan $\shardPolytope$ contained in~$\shard(\arc')$.
To achieve this, we distinguish three cases, according on whether the $B$-arc~$\Barc$ is separated, singular or overlapped (see \cref{subsec:noncrossingArcDiagramsB} for this distinction).

%%%

\subsubsection{Separated $B$-arcs}

We start with the simplest case of separated $B$-arcs.

\begin{lemma}
\label{lemma:separatedShardPolytopeFanB}
For a separated $B$-arc~$\Barc$, the union of the walls of the normal fan of the shard polytope~$\shardPolytope[\Barc]$ is contained in the union of the shards~$\shard(\Barc')$ for~$\Barc \prec \Barc'$.
\end{lemma}

\begin{proof}
When~$\Barc \eqdef (-\arc, \arc)$ is separated, \cref{def:forcingB} ensures that the $B$-arcs that force~$\Barc$ are precisely the $B$-arcs~$\Barc' \eqdef (-\arc', \arc')$ where~$\arc'$ forces~$\arc$.
The statement thus immediately follows from~\cref{lem:weakShardPolytopeFanB}.
\end{proof}

%%%

\subsubsection{Singular $B$-arcs}

The situation is already slightly more subtle for singular $B$-arcs.
We start by describing the vertices of the shard polytopes of singular $B$-arcs.

\begin{lemma}
\label{lem:verticesShardPolytopeBsingular}
For a singular $B$-arc~$\Barc \eqdef (-\arc, \arc)$, the vertices of the shard polytope~$\shardPolytope[\Barc]$ are the characteristic vectors of the centrally symmetric $\arc$-alternating matchings. 
\end{lemma}

\begin{proof}
Consider an $\arc$-alternating matching~$M \eqdef \{a_1 < b_1 < \dots < a_k < b_k\}$ and its centrally symmetric image $-M \eqdef \{-b_k < -a_k < \dots < -b_1 < -a_1\}$.
Note that~$-M$ is also an $\arc$-alternating matching since~$\arc$ is centrally symmetric.
Moreover, its characteristic vector is~$\chi(-M) = \sum_{i \in [k]} \b{e}_{-b_i} - \b{e}_{-a_i} = \sum_{i \in [k]} \b{e}_{a_i} - \b{e}_{b_i} = \chi(M)$.
If~$M$ is not centrally symmetric, \ie~$M \ne -M$, then we obtain by \cref{lem:matchingUnionLemma} that there exist two other $\arc$-alternating matchings~$M_3$ and~$M_4$ such that~$\chi(M) + \chi(-M) = \chi(M_3) + \chi(M_4)$.
In other words, $\chi(M)$ is in the middle of~$\chi(M_3)$ and~$\chi(M_4)$ and is thus not extremal.

Conversely, the characteristic vectors of the centrally symmetric $\arc$-alternating matchings have $i$th coordinate either~$0$, or~$2$ if~$i \in [n] \cap (\{a\} \cup A)$, or~$-2$ if~$i \in [n] \cap (B \cup \{b\})$.
In other words, they are some vertices of the cube~$\sum_{i \in \{a\} \cup A} [0,2\b{e}_i] + \sum_{i \in B \cup \{b\}} [0,-2\b{e}_i]$.
Hence, they are all extremal.
\end{proof}

\begin{lemma}
\label{lemma:singularShardPolytopeFanB}
For a singular $B$-arc~$\Barc$, the union of the walls of the normal fan of the shard polytope~$\shardPolytope[\Barc]$ is contained in the union of the shards~$\shard(\Barc')$ for~$\Barc \prec \Barc'$.
\end{lemma}

\begin{proof}
When~$\Barc \eqdef (-\arc,\arc)$ is singular, the $B$-arcs that force~$\Barc$ are precisely the separated and singular $B$-arcs~$\Barc' \eqdef (-\arc', \arc')$ where~$\arc'$ forces~$\arc$.
According to \cref{lem:weakShardPolytopeFanB}, it thus suffices to show that no wall of the normal fan of $\shardPolytope[\Barc]$ is contained in the shard of an overlapped arc.
Said differently, that no edge of $\shardPolytope[\Barc]$ is in direction $\b{e}_i + \b{e}_j$ with $1 \le i < j \le n$.

Suppose, for the sake of contradiction, that $M$ and $M'$ are two $\arc$-alternating matchings whose characteristic vectors form an edge of $\shardPolytope[\Barc]$ in direction $\b{e}_i + \b{e}_j$ with~$1 \le i < j \le n$. We have $\chi(M') = \chi(M) + \lambda (\b{e}_i + \b{e}_j)$ for some $\lambda \ne 0$. We can assume that $\lambda > 0$ by exchanging the roles of $M$ and $M'$. Moreover, by~\cref{lem:verticesShardPolytopeBsingular}, the $k$th coordinate of any vertex of $\shardPolytope[\Barc]$ belongs to either $\{0,2\}$ or to $\{-2,0\}$, depending on $k$. Therefore, it suffices to consider the case $\lambda = 2$.

Note that, for $k \in \{i,j\}$, if $k \in \{a\} \cup A$ then $\pm k \notin M$ and $\pm k \in M'$, and if $k \in B \cup \{b\}$ then $\pm k \in M$ and $\pm k \notin M'$.
In any case, we have that
\[
|M' \cap [-n] \cap (\{a\} \cup A)| - |M' \cap [-n] \cap (B \cup \{b\})| = |M \cap [-n] \cap (\{a\} \cup A)| - |M \cap [-n] \cap (B \cup \{b\})| + 2.
\]
However, for any $\arc$-alternating matching $N$, it holds that
\[
|N \cap [-n] \cap (\{a\} \cup A)| - |N \cap [-n] \cap (B \cup \{b\})| \in \{0,1\},
\]
which gives a contradiction.
\end{proof}

%%%

\subsubsection{Overlapped $B$-arcs}

The most complicated case is when $\Barc$ is overlapped. 
We need some auxiliary results first.
We start with a simple observation on upper and lower arcs of overlapped $B$-arcs.

\begin{remark}
\label{rmk:upperlower}
If the $B$-arc $\Barc \eqdef (-\arc, \arc)$ is overlapped, then
\begin{itemize}
\item $\arc$ is upper if and only if $-a \in A$ and $-i \in A$ for all $i \in {]a,-a[} \cap B$,
\item $\arc$ is lower if and only if $-a \in B$ and $-i \in B$ for all $i \in {]a,-a[} \cap A$.
\end{itemize}
\end{remark}

\begin{lemma}
\label{lem:forcesonlyone}
Consider a $B$-arc~$\Barc \eqdef (-\arc, \arc)$, where $\arc \eqdef (a, b, A, B)$ with $-b < a < 0 < -a < b$, and an $A$-arc~$\arc' \eqdef (a', b', A', B')$ such that $\arc \prec \arc'$.
If $-\arc \nprec \arc'$ then
\begin{enumerate}[(i)]
\item $b' > -a$, or
\item there is $r \in {]a',b'[}$ such that $\{-r,r\} \subseteq B$ (if $\arc$ is lower) or $\{-r,r\} \subseteq A$ (if $\arc$ is upper).
\end{enumerate}
\end{lemma}

\begin{proof}
If $-\arc \eqdef (-b,-a,-B,-A) \nprec \arc'$ we have:
\begin{itemize}
\item either $[a',b'] \nsubseteq [-b,-a]$, which implies that $b' > -a$ (because we have $-b < a \le a' < b' \le b$ as $\arc$ is to the right of $-\arc$ and $\arc \prec \arc'$), 
\item or ${]a',b'[} \cap -B \ne {]a',b'[} \cap A' = {]a',b'[} \cap A$ and ${]a',b'[} \cap -A \ne {]a',b'[} \cap B' = {]a',b'[} \cap B$. This means that there is some $r \in {]a',b'[} \cap A$ that does not belong to $-B$ (and hence $r \in -A$ and $-r \in A$), or some $r \in {]a',b'[} \cap B$ that does not belong to $-A$ (and hence $r \in -B$ and $-r \in B$). Note that by \cref{rmk:upperlower} if $r \in{]a', b'[}$ with $\pm r \in A$ can only happen when $\arc$ is upper, and $r \in{]a', b'[}$ with $\pm r \in B$ can only happen when $\arc$ is lower.
\qedhere
\end{itemize}
\end{proof}

\begin{lemma}
\label{lem:cs}
If $\Barc \eqdef (-\arc, \arc)$ and $\Barc' \eqdef (\mp \arc', \pm \arc')$ are overlapped $B$-arcs with $\arc$ lower (resp.~upper) and $\arc'$ upper (resp.\ lower), such that $\arc \prec \arc'$, then $\arc$ is centrally symmetric in the interval $[ -\min(|a'|, |b'|), \min(|a'|, |b'|) ]$, where $\arc' \eqdef (a', b', A', B')$.
\end{lemma}

\begin{proof}
Follows directly from the characterization of upper and lower in \cref{rmk:upperlower}.
\end{proof}

\begin{lemma}
\label{lemma:overlappedShardPolytopeFanB}
For an overlapped $B$-arc~$\Barc$, the union of the walls of the normal fan of the shard polytope~$\shardPolytope[\Barc]$ is contained in the union of the shards~$\shard(\Barc')$ for~$\Barc \prec \Barc'$.
\end{lemma}

\begin{proof}
If~$\Barc \eqdef (-\arc, \arc)$ is overlapped, the $B$-arcs~$\Barc'$ that force~$\Barc$ are precisely the $B$-arcs~$\Barc' \eqdef (-\arc', \arc')$ where~$\arc'$ forces~$-\arc$ or~$\arc$, except those where $\Barc'$ is overlapped and the upper $A$-arc of~$\Barc'$ forces the lower $A$-arc of~$\Barc$. 
Therefore, we will consider :
\begin{itemize}
\item an overlapped $B$-arc $\Barc \eqdef (-\arc, \arc)$, where $\arc \eqdef (a, b, A, B)$ with $-b < a < 0 < -a < b$,
\item an $A$-arc $\arc'$ with $\arc \prec \arc'$, such that
	\begin{itemize}
	\item $\Barc' \eqdef (\mp\arc', \pm\arc')$ is overlapped, where $\arc' \eqdef (a', b', A', B')$ and the signs depend on whether $-b' < a' < 0 < -a' < b'$ or $a' < -b' < 0 < b' < -a'$,
	\item the upper $A$-arc of $\Barc'$ forces the lower $A$-arc of $\Barc$,
	\item the upper $A$-arc of $\Barc'$ does not force the upper $A$-arc of $\Barc$,
	\end{itemize}
\item and a vector $\b{t} \in \shard(\arc') \cap \Bhyp \subset \R^{[\pm n]}$, that is,
	\begin{align}
	& \b{t}_i = -\b{t}_{-i} \text{ for all } i \in [\pm n], \label{eq:sym} \\
	& \b{t}_{a'} = \b{t}_{b'}, \label{eq:abshard}\\
	& \b{t}_{i} < \b{t}_{a'} = \b{t}_{b'} \text{ for all } i \in {]a',b'[} \cap A, \label{eq:Ashard}\\
	& \b{t}_{j} > \b{t}_{a'} = \b{t}_{b'} \text{ for all } j \in {]a',b'[} \cap B. \label{eq:Bshard}
	\end{align}
\end{itemize}
We need to prove that $\b{t}$ does not lie in the interior of the normal cone of any edge of $\shardPolytope$, which would necessarily be in direction $\b{e}_{a'}-\b{e}_{b'}$.

The proof is by contradiction and consists of a detailed case analysis of this situation. Four essential big cases arise, according to the possible combinations of upper/lower and right/left arcs of the overlapped pairs. We always assume $\arc\prec\arc'$ and $-\arc\nprec\arc'$.

%%%

\paraspace{Case 1} \emph{$\arc$ is the lower right arc of $\Barc \eqdef (-\arc, \arc)$ and $\arc'$ is the upper right arc of~$\Barc' \eqdef (-\arc', \arc')$.}

\medskip
\centerline{\includegraphics[scale=1]{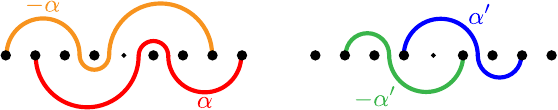}}
\medskip
  
In this case we have $-b \le -b' < a' < 0 < -a' < b' \le b$.   
By \cref{rmk:upperlower}, we have~$-a' \in A$ (because $\arc'$ is upper).
Therefore, by \eqref{eq:Ashard} we have $\b{t}_{-a'} < \b{t}_{a'} = \b{t}_{b'}$, and by \eqref{eq:sym} we have 
\begin{equation}
\label{eq:1a}
-\b{t}_{a'} < 0 < \b{t}_{a'}. 
\end{equation}
Moreover, this also means that $a' \ne a$ (because $-a \in B$ since $\arc$ is lower), and hence $a'\in B$ (because $-a' \in {]a,-a[}$ and $\arc$ is lower).

By \cref{lem:normalconeedge}, if $\b{t}$ belongs to the normal cone of an edge in direction $\b{e}_{a'} - \b{e}_{b'}$, with $a' \in B$, then there must be some $a \le s<a'$ with~$s \in \{a\} \cup A$, such that $\b{t}$ selects an alternating matching containing the pair $(s, a')$. By \cref{lem:normalconematch}, this implies that
\begin{equation}
\label{eq:1s}
\b{t}_s \ge \b{t}_{a'} > 0.
\end{equation}
Note that $s \in {[a,-a[} \cap (\{a\} \cup A)$. Since $\arc$ is lower, we conclude via \cref{rmk:upperlower} that $-s \in B$. 

We now distinguish two subcases:
\begin{itemize}
\item If $s>-b'$ then $-s \in {]a',b'[} \cap B$. In this case by \eqref{eq:Bshard} and~\eqref{eq:sym} we get 
\begin{equation}
\label{eq:1-s}
-\b{t}_s = \b{t}_{-s} > \b{t}_{a'} > 0.
\end{equation}
A contradiction with \eqref{eq:1s}.
 
\item Otherwise, $s \le -b'$. In this case, we have $b' \le -s \le -a$. By \cref{lem:forcesonlyone} there must be some $r \in {]a',b'[}$ such that $\pm r \in B$. By \cref{lem:cs}, we must have $r \in {]-a',b'[}$.
Moreover, $s<-r$, because $s \le -b'<-r$. We have thus that $-r \in {]s,a'[} \cap B$. By \cref{lem:normalconematch}, this implies that
\begin{equation}
\label{eq:1-r}
-\b{t}_{r} = \b{t}_{-r} \ge \b{t}_{a'} > 0.
\end{equation}
Moreover, $r \in {]a',b'[}\cap B$. Therefore, by~\eqref{eq:Bshard}, we have 
\begin{equation}
\label{eq:1r}
\b{t}_{r} \ge \b{t}_{a'} > 0.
\end{equation}
Together, \eqref{eq:1-r} and~\eqref{eq:1r} give a contradiction.
\end{itemize}
 
%%%

\paraspace{Case 2} \emph{$\arc$ is the lower right arc of $\Barc \eqdef (-\arc, \arc)$ and $\arc'$ is the upper left arc of $\Barc' \eqdef (\arc', -\arc')$}.

\medskip
\centerline{\includegraphics[scale=1]{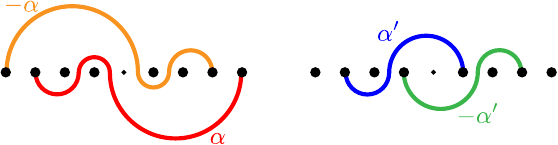}}
\medskip

In this case we have $-b < a \le a' < -b' < 0 < b' < -a' \le -a < b$.   
By \cref{rmk:upperlower}, we have $-b' \in A$ (because $\arc'$ is upper).
Therefore, by \eqref{eq:Ashard} we have $\b{t}_{-b'} < \b{t}_{a'} = \b{t}_{b'}$, and by \eqref{eq:sym} we have
\begin{equation}
\label{eq:2b}
-\b{t}_{b'} < 0 < \b{t}_{b'}.  
\end{equation}
Moreover, this implies that $b' \in B$ (because $-b' \in {]a,-a[}$ and $\arc$ is lower).
  
Since $b' < -a$, \cref{lem:forcesonlyone} guarantees the existence of an $r \in {]a',b'[}$ such that $\pm r \in B$. By \cref{lem:cs}, we must have $r \in {]a', - b'[}\subset {]a', b'[}$ and $-r \in {]b', -a'[} \subset {]b', b]}$.
  
We distinguish two subcases according to the value of $\b{t}_{-r}$:
\begin{itemize}
\item If $\b{t}_{-r} \ge \b{t}_{b'}$ then by~\eqref{eq:sym} we have
\begin{equation}
\label{eq:2-r}
-\b{t}_{r} \ge \b{t}_{b'} > 0.
\end{equation}
On the other hand, as $r \in {]a',b'[} \cap B$, then by~\eqref{eq:Bshard}
\begin{equation}
\label{eq:2r}
\b{t}_{r} \ge \b{t}_{b'} > 0.
\end{equation}
Together, \eqref{eq:2-r} and~\eqref{eq:2r} are in contradiction.
 
\item Otherwise $\b{t}_{-r} < \b{t}_{b'}$. By \cref{lem:normalconeedge}, if $\b{t}$ belongs to the normal cone of an edge in direction $\b{e}_{a'} - \b{e}_{b'}$, with $b' \in B$, then $\b{t}$ selects an alternating matching containing a pair $(s, b')$ for some $s \le a'$. By \cref{lem:normalconematch}, this implies that for all $j \in {]b',b]} \cap (B \cup \{b\})$ with $\b{t}_j < \b{t}_{b'}$ there is some $\ell \in {]b',j[} \cap A$ and $\b{t}_\ell \ge \b{t}_{b'}$. This applies in particular to $-r \in {]b', b]}$. 
There is therefore some $\ell \in {]b',-r[}\cap A$ such that
\begin{equation}
\label{eq:2l}
\b{t}_{\ell} \ge \b{t}_{b'} > 0.
\end{equation}
On the other hand, since $\ell \in {]a,-a[} \cap A$ and $\arc$ is lower, we have $-\ell \in B$. Moreover, $-\ell \in {] a',b'[}$. Therefore, by~\eqref{eq:Bshard} we have
\begin{equation}
\label{eq:2-l}
-\b{t}_{\ell} = \b{t}_{-\ell} \ge \b{t}_{b'} > 0.
\end{equation}
Together, \eqref{eq:2l} and~\eqref{eq:2-l} are in contradiction.
\end{itemize}

%%%

\paraspace{Case 3} \emph{$\arc$ is the upper right arc of $\Barc \eqdef (-\arc, \arc)$ and $\arc'$ is the lower right arc of $\Barc' \eqdef (-\arc', \arc')$}.

\medskip
\centerline{\includegraphics[scale=1]{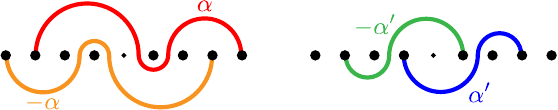}}
\medskip

In this case we have $-b \le -b' < a' < 0 < -a' < b' \le b$.   
By \cref{rmk:upperlower}, we have $-a'\in B$ (because $\arc'$ is lower).
Therefore, by \eqref{eq:Bshard} we have $\b{t}_{-a'} > \b{t}_{a'} = \b{t}_{b'}$, and by \eqref{eq:sym} we have 
\begin{equation}
\label{eq:3a} 
\b{t}_{a'} < 0 < -\b{t}_{a'}. 
\end{equation}
Moreover, this also means that $a' \ne a$ (because $-a \in A$ since $\arc$ is upper), and hence $a' \in A$ (because $-a' \in {]a,-a[}$ and $\arc$ is upper).

We claim that there is some $u \in {[a,a'[} \cap {]-b',a'[}$ such that $\pm u \in \{a\} \cup A$. Indeed, by \cref{lem:forcesonlyone} either $b' > -a$, in which case we set $u = a$ and we have $-a \in A$ since $\arc$ is upper using \cref{rmk:upperlower}; or $b' \le -a$ and then there is some $r \in {]a',b'[}$ such that $\pm r \in A$. In this case, $r \in {]-a', b'[}$ by \cref{lem:cs}, and we set $u = -r$.

We have $-u\in{]a',b'[}\cap A$. Therefore, by \eqref{eq:Ashard} we have $\b{t}_{-u} \le \b{t}_{a'} = \b{t}_{b'}$, and by \eqref{eq:sym} we conclude 
\begin{equation}
\label{eq:3-u}
-\b{t}_{u} \le \b{t}_{a'} < 0.  
\end{equation}

We distinguish two subcases according to the value of $\b{t}_{u}$:
\begin{itemize}
\item If $\b{t}_{u} \le \b{t}_{a'}$ then we have
\begin{equation}
\label{eq:3u}
\b{t}_{u} \le \b{t}_{a'} < 0.
\end{equation}
This contradicts \eqref{eq:3-u}.

\item Otherwise $\b{t}_{u} > \b{t}_{a'}$.
By \cref{lem:normalconeedge}, if $\b{t}$ belongs to the normal cone of an edge in direction $\b{e}_{a'} - \b{e}_{b'}$, with $a'\in A$, then $\b{t}$ selects an alternating matching containing a pair $(a', s)$ for some $s \ge b'$. By \cref{lem:normalconematch}, this implies that for all $i \in {[a,a'[} \cap (\{a\} \cup A)$ with $\b{t}_i> \b{t}_{a'}$ there is some $h \in {]i, a'[} \cap B$ such that $\b{t}_h \le \b{t}_{a'}$. This applies in particular to $u \in [a, a'[$. 
There is therefore some $h \in {]u, a'[} \cap B$ such that
\begin{equation}
\label{eq:3h}
\b{t}_{h} \le \b{t}_{a'} < 0.
\end{equation}
On the other hand, since $h \in {]a,-a[} \cap B$ and $\arc$ is upper, we have $-h \in A$ by \cref{rmk:upperlower}. Moreover, $-h \in {] a',b'[}$. Therefore, by~\eqref{eq:Ashard} and \eqref{eq:sym} we have
\begin{equation}
\label{eq:r-h}
-\b{t}_{h} = \b{t}_{-h} \le \b{t}_{a'} < 0.
\end{equation}
Together, \eqref{eq:3h} and \eqref{eq:r-h} are in contradiction.
\end{itemize}

%%%

\paraspace{Case 4} \emph{$\arc$ is the upper right arc of $\Barc \eqdef (-\arc, \arc)$ and $\arc'$ is the lower left arc of $\Barc' \eqdef (\arc', -\arc')$}.

\medskip
\centerline{\includegraphics[scale=1]{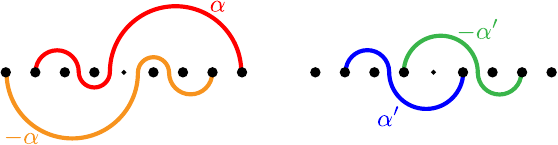}}
\medskip
 
In this case we have $-b < a \le a' < -b' < 0 < b' < -a' \le -a < b$.   
By \cref{rmk:upperlower}, we have $-b' \in B$ (because $\arc'$ is lower).
Therefore, by \eqref{eq:Bshard} we have $\b{t}_{-b'} > \b{t}_{a'} = \b{t}_{b'}$, and by \eqref{eq:sym} we have 
\begin{equation}
\label{eq:4b}
\b{t}_{b'} < 0 < -\b{t}_{b'}.
\end{equation}
Moreover, this implies that $b' \in A$ (because $-b' \in {]a,-a[}$ and $\arc$ is upper).

Since $b' < -a$,  \cref{lem:forcesonlyone} guarantees the existence of an $r \in {]a',b'[}$ such that $\pm r \in A$.
By \cref{lem:cs}, we must have $r \in {]a',- b'[}\subset {]a',b'[}$.
By~\eqref{eq:Ashard} we know that
\begin{equation}
\label{eq:4r}
\b{t}_{r} \le \b{t}_{b'} < 0.
\end{equation}
  
Moreover, by \cref{lem:normalconeedge}, if $\b{t}$ belongs to the normal cone of an edge in direction $\b{e}_{a'} - \b{e}_{b'}$, with $b'\in A$, then $\b{t}$ selects an alternating matching containing a pair $(b',v)$ for some $v\in {]b',b]} \cap (B\cup \{b\})$.
By \cref{lem:normalconematch}, this implies that
\begin{equation}
\label{eq:4v}
\b{t}_v \le \b{t}_{b'} < 0.
\end{equation}
 
We distinguish two subcases:
\begin{itemize}
\item If $v \ge -a'$, then $-r \in {]b',v[} \cap A$ because $b' < -r < -a' \le v$.  By \cref{lem:normalconematch} and~\eqref{eq:sym}, this implies that
\begin{equation}
\label{eq:4-r}
-\b{t}_r = \b{t}_{-r} \le \b{t}_{b'} < 0.
\end{equation}
This contradicts \eqref{eq:4r}.
 
\item Otherwise, $v < -a'$. Since $v \in {]a,-a[} \cap B$ and $\arc$ is upper, we get that $-v \in A$ by \cref{rmk:upperlower}. Furthermore, $-v \in {]a',-b[}$. Therefore, by~\eqref{eq:Bshard}, we get
\begin{equation}
\label{eq:4-v}
-\b{t}_{v} = \b{t}_{-r} \le \b{t}_{b'} < 0.
\end{equation}
This contradicts~\eqref{eq:4v}.
\qedhere
\end{itemize}
\end{proof}

%%%%%%%%

\newpage
\addtocontents{toc}{ \vspace{.1cm} }
\section*{Concluding remarks and further directions}
\label{subsec:concludingRemarks}

We conclude the paper with some observations and conjectures about shard polytopes and quotientopes in type~$B$ and beyond.

\para{Quotient fans, shard polytopes, and quotientopes for lattices of regions of hyperplane arrangements}
First, we quickly recall the general setting in which we hope to construct shard polytopes and quotientopes.
We refer to the recent surveys of N.~Reading~\cite{Reading-PosetRegionsChapter, Reading-FiniteCoxeterGroupsChapter} for a thourough introduction to the topic.

Consider a central hyperplane arrangement~$\HA$ defining a fan~$\fan$, and a distinguished base region~$\polytope{B}$ of~$\fan$.
The \defn{poset of regions}~$\PR(\HA, \polytope{B})$ is the poset whose elements are the regions of~$\fan$ ordered by inclusion of separating sets (the set of hyperplanes of~$\HA$ that separate the given region form the base region~$\polytope{B}$).
Its Hasse diagram is the graph of the zonotope of~$\HA$ oriented from the base region~$\polytope{B}$ to its opposite~$-\polytope{B}$.
By the work of A.~Bj\"orner, P.~Edelman and G.~Ziegler~\cite{BjornerEdelmanZiegler}, we know that~$\PR(\HA, \polytope{B})$ is always a lattice when the fan~$\fan$ is simplicial, and that the region~$\polytope{B}$ must be simplicial for~$\PR(\HA, \polytope{B})$ to be a lattice.
See also the survey of N.~Reading~\cite[Sect.~9-3]{Reading-PosetRegionsChapter} for further conditions, in particular a discussion on tight arrangements.
We assume here that~$\PR(\HA, \polytope{B})$ is a lattice.

As in type~$A$ and~$B$, the join-irreducible elements of~$\PR(\HA, \polytope{B})$ correspond to certain pieces of hyperplanes of~$\HA$ called \defn{shards}.
The shards of~$\HA$ can be defined geometrically as the pieces that remain after certain cuts.
Namely, for each codimension $2$ face~$F$ of the arrangement, consider the subarrangement~$\HA_F$ of all hyperplanes of~$\HA$ containing~$F$, and cut all non-basic hyperplanes of~$\HA_F$ by the basic hyperplanes of~$\HA_F$ (\ie the hyperplanes bounding the region of~$\HA_F$ containing the base region~$\polytope{B}$).
The shards are the pieces that remain once the cuts corresponding to all codimension~$2$ faces of~$\HA$ have been performed.
Moreover, these cuts define the forcing graph on shards.
If the hyperplane~$\polytope{H}$ cuts the hyperplane~$\polytope{H}'$, and we have two shards $\shard \subseteq \polytope{H}$ and $\shard' \subseteq \polytope{H}'$ whose intersection has codimension~$2$, then we have~$\shard \to \shard'$ in the forcing graph.
Note that this graph has no oriented cycle if and only if the lattice~$\PR(\HA, \polytope{B})$ is congruence uniform.
Finally, the lattice of congruences of the poset of regions~$\PR(\HA, \polytope{B})$ is isomorphic to the inclusion poset of upper ideals of the forcing relation.
See \cite[Sect.~9-7]{Reading-PosetRegionsChapter}.

Consider now a lattice congruence~$\equiv$ of the poset of regions~$\PR(\HA, \polytope{B})$, and let~$\shards_\equiv$ denote the corresponding shard ideal.
As in type~$A$ and~$B$, N.~Reading proved in~\cite{Reading-HopfAlgebras} that~$\equiv$ defines a \defn{quotient fan}~$\Fan_\equiv$ whose chambers are obtained as
\begin{itemize}
\item either the unions of the chambers of the fan~$\BFan_n$ in the same congruence class of~$\equiv$,
\item or the closures of the connected components of the complement of the union of the shards~of~$\shards_\equiv$.
\end{itemize}
It is then tempting to conjecture the following statement.

\begin{conjecture}
\label{conj:quotientopesHyperplaneArrangements}
For any hyperplane arrangement~$\HA$ defining a fan~$\fan$ and any base region~$\polytope{B}$ of~$\fan$ such that the poset of regions~$\PR(\HA, \polytope{B})$ is a congruence uniform lattice, and for any lattice congruence~$\equiv$ of the poset of regions~$\PR(\HA, \polytope{B})$, the quotient fan~$\Fan_\equiv$ is the normal fan of a polytope.
\end{conjecture}

We hope that this conjecture might be approached using shard polytopes in the sense of \cref{prop:shardPolytopeFan,prop:shardPolytopeFanB}, as already mentioned in \cref{prop:main16,prop:main17}.

\begin{definition}
\label{def:shardPolytopesHyperplaneArrangements}
A polytope~$\polytope{P}$ is a \defn{weak shard polytope} for a shard~$\shard$ if the union of the walls of the normal fan of~$\polytope{P}$ contains the shard~$\shard$ and is contained in the union of the shards~$\shard'$ forcing~$\shard$.
\end{definition}

Note that we do not require weak shard polytopes to be Minkowski indecomposable.
For instance, both the $\arc$-associahedron~$\Asso$ and the shard polytope~$\shardPolytope$ are weak shard polytopes for the shard~$\shard_\arc$ of an arc~$\arc \in \arcs_n$, but only the latter is Minkowski indecomposable.

\begin{proposition}
\label{prop:quotientopesHyperplaneArrangements}
Consider a hyperplane arrangement~$\HA$ with a base region~$\polytope{B}$ such that the poset of regions~$\PR(\HA, \polytope{B})$ is a lattice, and a lattice congruence~$\equiv$ of~$\PR(\HA, \polytope{B})$ with shard ideal~$\shards_\equiv$.
If each shard~$\shard$ of~$\shards_\equiv$ admits a weak shard polytope~$\polytope{P}_\shard$, then the quotient fan~$\Fan_\equiv$ is the normal fan of the Minkowski sum~$\sum_{\shard \in \shards_\equiv} \polytope{P}_\shard$.
\end{proposition}

Therefore, \cref{conj:quotientopesHyperplaneArrangements} is implied by the following stronger conjecture, which is motivated by \cref{prop:shardPolytopeFan} in type~$A$ and \cref{prop:shardPolytopeFanB} in type~$B$.

\begin{conjecture}
\label{conj:existenceShardPolytopes}
For any hyperplane arrangement~$\HA$ and any base region~$\polytope{B}$ such that the poset of regions~$\PR(\HA, \polytope{B})$ is a congruence uniform lattice, any shard admits a weak shard polytope.
\end{conjecture}

We now focus on a simplicial arrangement~$\HA$.
Observe first that one easily knows the number of shards of~$\HA$.

\begin{lemma}
\label{lem:numberShards}
In a simplicial arrangement, the number of shards is the number of rays minus the dimension.
\end{lemma}

\begin{proof}
In a simplicial arrangement, there is a natural bijection between the join-irreducible regions and the rays not in the base region.
Namely, each join-irreducible region maps to the unique ray not contained in its unique descent hyperplane, and conversely each ray not in the base region maps to the smallest region containing it.
\end{proof}

For instance, we have observed along the paper that the number of shards in the Coxeter arrangements of type~$A$ and~$B$ are given by~$|\arcs_n| = 2^n-n-1$ and~$|\Barcs_n| = 3^n-n-1$.
The attentive reader will have recognized that the number of shards in a simplicial arrangement~$\HA$ is the dimension of the type cone of~$\HA$.
This yields the following generalization of \cref{prop:shardPolytopeBasis}.

\begin{proposition}
\label{prop:basisTypeConeHyperplaneArrangements}
Consider a simplicial arrangement~$\HA$ and a base region~$\polytope{B}$ such that the poset of regions~$\PR(\HA,B)$ is a congruence uniform lattice, and any shard~$\shard$ admits a weak shard polytope~$\polytope{P}_\shard$.
Then the collection of weak shard polytopes~$(\polytope{P}_\shard)_{\shard \in \shards}$ is a linear basis of the vector subspace of virtual polytopes generated by the deformations of the zonotope of~$\HA$ up to translations.
\end{proposition}

\begin{proof}
The proof is identical to that of \cref{prop:shardPolytopeBasis}.
Namely, the number of shards is the dimension of the type cone of~$\HA$.
To see that~$(\polytope{P}_\shard)_{\shard \in \shards}$ is linearly independent, observe that the forcing minimal non-zero coefficient in a linear dependence is carried by a single shard, since the forcing relation on shards is acyclic.
\end{proof}

For instance, the Cambrian associahedra~$\big( \Asso[\arc] \big)_{\arc \in \arcs_n}$ and the shard polytopes~$\big( \shardPolytope \big)_{\arc \in \arcs_n}$ form two linear bases of the virtual deformed permutahedra, as mentioned in \cref{prop:main9,prop:shardPolytopeBasis}.
Similarly, this yields the following relevant basis in type~$B$, answering a question raised in~\cite{ArdilaCastilloEurPostnikov} as mentioned in \cref{prop:main15}. The problem of finding an explicit natural basis remains open for the type cones of Coxeter permutahedra beyond types~$A$ and~$B$. 

\begin{corollary}
\label{coro:typeBShardPolytopesBasis}
Any type~$B$ deformed permutahedron has a unique decomposition as a Minkowski sum and difference of dilated type~$B$ shard polytopes (up to translation).
In other words, for any~$n \in \N$, the type~$B$ shard polytopes~$\big( \shardPolytope[\Barc] \big)_{\Barc \in \Barcs_n}$ form a linear basis of the vector subspace of virtual polytopes generated by type~$B$ deformed permutahedra.
\end{corollary}

\para{Dimension~$2$ shard polytopes}
We now discuss the special situation of rank~$2$, which includes in particular the case of dihedral groups~$I_2(n)$.
Of course, there is nothing deep here since any rank~$2$ complete fan is polytopal.
However, we want to use shard polytopes to construct these polytopal realizations.

\begin{proposition}
\label{prop:rank2ShardPolytopes}
Any shard~$\shard$ of any rank~$2$ hyperplane arrangement~$\HA$ with respect to any base region~$\polytope{B}$ admits a shard polytope which is a segment if~$\shard$ is a basic shard, and a triangle otherwise.
\end{proposition}

\begin{proof}
Let~$\b{n}$ denote a normal vector to the shard~$\shard$.
If~$\shard$ is basic, then it consists of a complete line, so that the segment~$[\b{e}, \b{n}]$ is a shard polytope for~$\shard$.
Otherwise, let~$\shard_1$ and~$\shard_2$ denote the two basic shards and~$\b{n}_1$ and~$\b{n}_2$ denote their normal vectors.
Since~$\{\b{n}_1, \b{n}_2\}$ is a basis of the plane, we can write~$\b{n} = \alpha_1 \b{n}_1 + \alpha_2 \b{n}_2$.
Consider now the two triangles~$T_1 \eqdef \conv \{ \b{0}, \alpha_1 \b{n}_1, \b{n} \}$ and~$T_2 \eqdef \conv \{\b{0}, \alpha_2 \b{n}_2, \b{n} \}$.
The walls of their normal fans are all contained in the lines orthogonal to~$\b{n}_1$, $\b{n}_2$, and~$\b{n}$, and one of them contains the shard~$\shard$.
We therefore conclude that either~$T_1$ or~$T_2$ is a shard polytope for~$\shard$.
\end{proof}

This provides the following polytopal realizations of rank~$2$ quotient fans.

\begin{corollary}
Any quotient fan of a rank~$2$ hyperplane arrangement is the normal fan of a Minkowski sum of segments and triangles.
\end{corollary}

\begin{remark}
\cref{prop:rank2ShardPolytopes} was the motivation for \cref{prop:facesShardPolytope}.
We tried to mimic this idea for hyperplane arrangements of arbitrary rank as follows.

Consider a shard~$\shard$ and a normal vector~$\b{n}$ to~$\shard$.
Each facet~$\polytope{F}$ of~$\shard$ corresponds to two shards~$\shard_\polytope{F}$ and~$\shard_\polytope{F}'$ cutting~$\shard$.
Choose normal vectors~$\b{n}_\polytope{F}$ and~$\b{n}_\polytope{F}'$ to these shards such that~$\b{n} = \b{n}_\polytope{F} + \b{n}_\polytope{F}'$.
Assume that we constructed by induction shard polytopes~$\polytope{P}_\polytope{F}$ and~$\polytope{P}_\polytope{F}'$ for the shards~$\shard_\polytope{F}$ and~$\shard_\polytope{F}'$ for all facets~$\polytope{F}$ of~$\shard$.
Up to translation and scaling, we can assume that~$\polytope{P}_\polytope{F}$ contains the edge~$\b{e}_\polytope{F} \eqdef [\b{0}, \b{n}_\polytope{F}]$ and~$\polytope{P}_\polytope{F}'$ contains the edge~$\b{e}_\polytope{F}' \eqdef [\b{0}, \b{n}_\polytope{F}']$.
To construct a shard polytope for~$\shard$, first place an edge~$\b{e} \eqdef [\b{0}, \b{n}]$.
Then for each facet~$\polytope{F}$ of~$\shard$, translate the shard polytopes~$\polytope{P}_\polytope{F}$ and~$\polytope{P}_\polytope{F}'$ so that the edges~$\b{e}$, $\b{e}_\polytope{F}$, and~$\b{e}_\polytope{F}'$ form a triangle as in \cref{prop:rank2ShardPolytopes} (we either translate~$\polytope{P}_\polytope{F}$ by~$\b{n}-\b{n}_\polytope{F}$ or~$\polytope{P}_\polytope{F}'$ by~$\b{n}-\b{n}_\polytope{F}'$, depending on whether~$\shard$ is above~$\shard_\polytope{F}$ and below~$\shard_\polytope{F}'$ or the opposite).
We then take the convex hull of the translated shard polytopes~$\polytope{P}_\polytope{F}$ and~$\polytope{P}_\polytope{F}'$ for all facets~$\polytope{F}$ of~$\shard$.

While this natural construction works in type~$A$ by \cref{prop:facesShardPolytope}\,\eqref{it:facesForcingShardPolytope2}, it unfortunately already fails in type~$B_4$ for the $B$-arcs
\[
\raisebox{-1.1cm}{\includegraphics[scale=1]{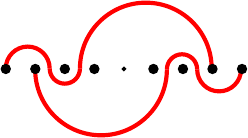}} % 0uduuu0.
\qquad\text{and}\qquad
\raisebox{-1.1cm}{\includegraphics[scale=1]{constructionFails1}} \; . % 0duddd0.
\]
\end{remark}

\para{Sortable and unsortable shards}
We now want to underline an important subtlety that drastically impacts the hunt for shard polytopes beyond rank~$2$ or type~$A$.
In type~$A$, any shard of~$\Fan_n$ belongs to the union of the walls of the $\arc$-Cambrian fan for at least one forcing minimal~$\arc \in \arcs_n$.
In other words, any join-irreducible permutation is $c$-sortable for some Coxeter element~$c$.
Nevertheless, this property is lost beyond rank~$2$ or type~$A$.
For instance, the shards labeled~L and~O in \cref{fig:B3shards} appear in none of the Cambrian fans of \cref{fig:B3associahedra}, so that the corresponding elements~$s_2 s_3 s_2 s_1$ and~$s_1 s_2 s_3 s_2 s_1$ are not $c$-sortable for any type~$B$ Coxeter element~$c$.
In an arbitrary Coxeter group, we say that a shard is \defn{sortable} if it belongs to at least one Cambrian fan (or equivalently, its corresponding join-irreducible is $c$-sortable for at least one Coxeter element~$c$), and \defn{unsortable} otherwise.
One can check that any sortable shard forces at most one shard per hyperplane.

\para{Indecoposability of type~$B$ shard polytopes}
We have seen in \cref{prop:SPindecomposable} that the shard polytope~$\shardPolytope$ of any~$A$-arc~$\arc$ is indecomposable.
We expect the same property to hold in type~$B$, as already mentioned in \cref{conj:main14}.

\begin{conjecture}
\label{conj:BshardPolytopesIndecomposable}
For any $B$-arc~$\Barc$, the shard polytope~$\shardPolytope[\Barc]$ is indecomposable.
\end{conjecture}

This conjecture was verified by computer experiments up to~$n = 4$.
Note that the simple indecomposability criterion of \cref{thm:indecomposabilityCriterion} that we used in the proof of \cref{prop:SPindecomposable} fails in type~$B$.
Namely, consider the $B$-arcs
\[
\Barc \eqdef \raisebox{-1.1cm}{\includegraphics[scale=1]{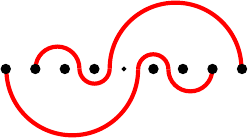}} % 0dddud0.
\qquad\text{and}\qquad
\Barc' \eqdef \phi(\Barc) = \raisebox{-1.1cm}{\includegraphics[scale=1]{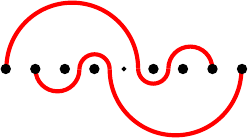}} \; . % 0uuudu0.
\]
The shard polytope~$\shardPolytope[\Barc]$ (resp.~$\shardPolytope[\Barc']$) is indecomposable, but the edge joining the characteristic vectors of alternating matchings~$\chi(\varnothing) = \b{0}$ to~$\chi(\{-3, 4\}) = - \b{e}_3 - \b{e}_4$ does not meet the facet defined by the inequality~$\b{x}_2 = 1$ (resp.~$\b{x}_2 = -1$).
The arcs $\Barc$ and~$\Barc'$ are the unique shards for which the criterion fails in type~$B_n$ for~$n \le 4$, and they are both unsortable.
We have not found a sortable shard failing the simple criterion of \cref{thm:indecomposabilityCriterion}.
A more general criterion is given in~\cite[Thm.~2]{McMullen1987}, but applying it would require a much better understanding of the faces of the type $B$ shard polytopes.
Note that \cref{conj:BshardPolytopesIndecomposable} would imply the type~$B$ analogue of \cref{thm:shardPolytopesRaysTypeCone} by~\cite{PadrolPaluPilaudPlamondon}.

\para{Newton polytopes of $F$-polynomials}
We have seen in \cref{rem:shardPolytopesAlreadyExisted} that type~$A$ shard polytopes are Newton polytopes of $F$-polynomials of cluster variables with respect to acyclic initial seeds.
\cref{conj:BshardPolytopesIndecomposable} would also imply by~\cite{BazierMatteDouvilleMousavandThomasYildirim} the analogue property for sortable shards.
In fact, we believe that this holds for arbitrary finite Weyl groups.

\begin{conjecture}
\label{conj:NewtonPolytopesFPolynomials}
In any finite type cluster algebra, Newton polytopes of $F$-polynomials are indecomposable shard polytopes.
\end{conjecture}

Note that, for simply-laced types, it was proved in~\cite{BazierMatteDouvilleMousavandThomasYildirim} that Newton polytopes of $F$-polynomials are indecomposable Minkowski summands of the associahedron of~\cite{HohlwegLangeThomas}.
Note by the way that the simple indecomposability criterion of \cref{thm:indecomposabilityCriterion} seems to hold for all Newton polytopes of $F$-polynomials (checked computationaly in types~$A_4$, $A_5$, $B_3$, $B_4$, $D_4$, $D_5$).
It however does not directly imply that Newton polytopes of $F$-polynomials behave properly with respect to shards.

A tempting approach to \cref{conj:NewtonPolytopesFPolynomials} is to use scattering diagrams~\cite{GrossHackingKeelKontsevich, Reading-scatteringFans} which reveals the connection between the cluster algebra and the geometry of its $\b{g}$-vector fan.
The philosophy is that the $F$-polynomial of a cluster variable~$x$ is obtained by applying to the $\b{g}$-vector~$\b{g}_x$ of~$x$ the scattering functions along any path from a $\b{g}$-vector cone containing~$\b{g}_x$ to the positive orthan.
The main observation here is that each $2$-face of a $W$-associahedron is either a $W'$-associahedron or a square.
Since all $W'$-associahedra are either pentagons or hexagons, we can choose a path that only crosses forced shards.
The details are however a bit more subtle.

Note that \cref{conj:NewtonPolytopesFPolynomials} would show the existence of shard polytopes for all sortable shards in finite Weyl groups.
It would nevertheless not tell anything about unsortable shards.
However, one can hope that scattering techniques could still provide shard polytopes for unsortable shards, even if it would go beyond the scope of consistent scattering diagrams.

\para{Supersolvable arrangements}
Coxeter arrangements of types $A_{n-1}$, $B_n$ and $I_2(m)$ are examples of supersolvable hyperplane arrangements (they are actually the only irreducible reflection arrangements that are supersolvable~\cite[Thm.~5.1]{BarceloIhrig1999}).
It is natural to ask whether our constructions for shard polytopes and quotientopes extend to the whole family, generalizing~\cref{prop:shardPolytopeFan,prop:shardPolytopeFanB,prop:rank2ShardPolytopes} to all supersolvable arrangements (at least in the semidistributive case).

A hyperplane arrangement is called \defn{supersolvable} if its intersection lattice is supersolvable in the sense of R.~Stanley~\cite{Stanley1972}.
The following alternative inductive geometric characterization was given by A.~Bj\"orner, P.~Edelman and G.~Ziegler in~\cite[Thm.~4.3]{BjornerEdelmanZiegler}.
See \cref{fig:supersolvable} for illustration of the decomposition in types~$A_3$ and~$B_3$.

\begin{definition}
Every hyperplane arrangement of rank at most~$2$ is supersolvable.
A hyperplane arrangement~$\HA$ of rank~$r \ge 3$ is supersolvable if and only if it can be written as $\HA = \HA_0 \sqcup \HA_1$,~where
\begin{enumerate}[(i)]
\item $\HA_0$ is a supersolvable arrangement of rank~$r-1$.
\item For any distinct $\polytope{H}', \polytope{H}'' \in \HA_1$, there is a unique $\polytope{H} \in \HA_0$ such that $\polytope{H}' \cap \polytope{H}'' \subseteq \polytope{H}$.
\end{enumerate}
\end{definition}

\begin{figure}
	\capstart
	\centerline{\includegraphics[scale=.5]{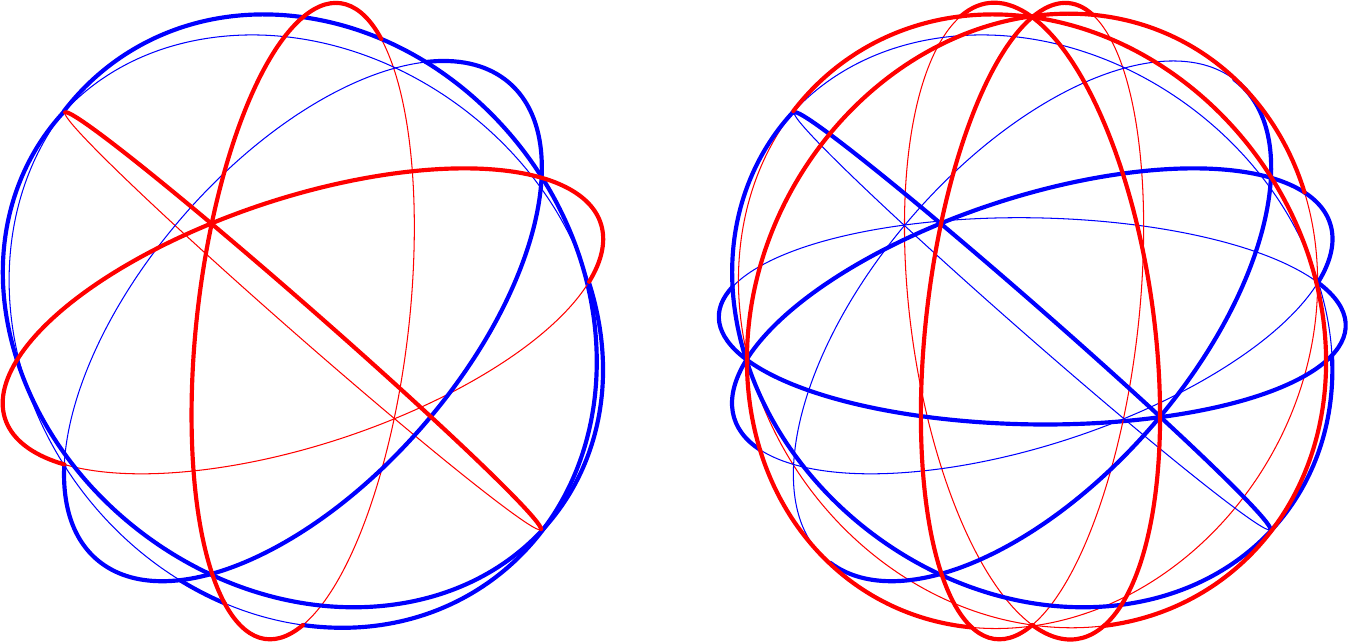}}
	\caption{The decomposition~$\HA = \HA_0 \sqcup \HA_1$ (where~$\HA_0$ is red and $\HA_1$ is blue) for the Coxeter arrangements of type~$A_3$ (left) and~$B_3$ (right).}
	\label{fig:supersolvable}
\end{figure}

By definition, for a fixed region~$R$ of $\HA_0$, the set of regions of $\HA$ contained in $R$ form a path in the adjacency graph of~$\HA$. In~\cite{BjornerEdelmanZiegler}, a \defn{canonical base region} of a supersolvable arrangement is defined inductively as follows. Any region of an arrangement of rank~$2$ is a canonical base region. If $\HA = \HA_0 \sqcup \HA_1$ is of rank at least~$3$, and $\polytope{B}_0$ is a canonical base region of~$\HA_0$, then both extremities of the path of regions of $\HA$ contained in $\polytope{B}_0$ are canonical base regions of~$\HA$.

As observed in~\cite{BjornerEdelmanZiegler}, the poset of regions~$\PR(\HA, \polytope{B})$ of a supersolvable arrangement~$\HA$ with respect to a canonical base region~$\polytope{B}$ is always a lattice, and N.~Reading showed in~\cite{Reading-posetRegions} that it is actually always a congruence normal lattice. This is a weaker condition than congruence uniformity. In fact, a finite lattice is congruence uniform if and only if it is both congruence normal and semidistributive, see for example~\cite[Thm.~3-2.41]{AdarichevaNation}.

Besides our examples in types $A_{n-1}$, $B_n$ and $I_2(m)$, we provide more evidence towards a proof of \cref{conj:existenceShardPolytopes} for supersolvable arrangements (with a canonical base region) by constructing shard polytopes for a particular shard in each hyperplane. Thanks to this, an analogue of the valuation-like formula from \cref{thm:inductiveMinkowskiSum} would be the only missing ingredient for a generalized construction for shard polytopes in supersolvable arrangements.

To define the \defn{canonical shard} of a hyperplane~$\polytope{H} \in \HA$, note that in a  supersolvable arrangement ordered with respect to a canonical base region~$\polytope{B}$, the hyperplanes in $\HA_1$ do not cut shards in the hyperplanes in~$\HA_0$.
It suffices hence to do an inductive definition.
If $\HA$ is of rank~$1$, then there is a single hyperplane and a single shard, which is the canonical shard. If $\polytope{H} \in \HA_1$, then its canonical shard is the shard that contains $\polytope{H} \cap \polytope{B}_0$, where $\polytope{B}_0$ is the canonical base region of $\HA_0$ containing~$\polytope{B}$.

\begin{proposition}
\label{prop:canonicalShardSupersolvable}
Let~$\HA$ be a supersolvable arrangement with a canonical base region~$\polytope{B}$.
Then every canonical shard of~$\PR(\HA, \polytope{B})$ admits a shard polytope.
\end{proposition}

\begin{proof}
We consider the case of rank greater than~$2$, as arrangements of rank~$2$ are covered by~\cref{prop:rank2ShardPolytopes}. By induction, it suffices to consider hyperplanes in~$\HA_1$. The linear order of the regions contained in $\polytope{B}_0$ induces a linear order~$\polytope{H}_1 \prec \polytope{H}_2 \prec \cdots \prec \polytope{H}_n$ on the hyperplanes of $\HA_1$, with $\polytope{H}_1$ bounding~$\polytope{B}$. For a fixed $\polytope{H}_k \in \HA_1$, we say that $\polytope{H}_i, \polytope{H}_j \in \HA_1$ are \defn{$\polytope{H}_k$-parallel} if $\polytope{H}_i \cap \polytope{H}_k = \polytope{H}_j \cap \polytope{H}_k$. Then, $\polytope{H}_i \in \HA_1$ cuts~$\polytope{H}_k$ if and only if $i < k$ and $\polytope{H}_i$ is the smallest in its $\polytope{H}_k$-parallelism class. And $\polytope{H} \in \HA_0$ cuts $\polytope{H}_k \in \HA_1$ if $\polytope{H}_i \cap \polytope{H}_k \subseteq \polytope{H}$ for some $\polytope{H}_i \in \HA_1$ with $i < k$. This defines the shard forcing order.
 
By construction, $\HA_0$ has a one-dimensional linearity.
Let $\b{u}$ be a non-zero vector in this linearity oriented towards the base region.
For $i \in [n]$, let $\b{v}_i$ be a normal vector to $\polytope{H}_i$ such that $\dotprod{\b{v}_i}{\b{u}} = 1$.
We claim that the pyramid $\polytope{P}_k \eqdef \conv \big( \{\b{0}\} \cup \set{\b{v}_i}{i \in [k]} \big)$ is a shard polytope for the canonical shard in~$\polytope{H}_k$.
 
Indeed, consider first the unbounded cone generated by the~$\b{v}_i$'s.
Its polar is the polyhedral cone $\polytope{C}_k \eqdef \bigcap_{i \le k} \polytope{H}_i^-$ if we consider all $\polytope{H}_i \in \HA_1$ oriented towards the base region.
One of its facets is $\polytope{H}_k \cap \bigcap_{i < k} \polytope{H}_i^-$, which is the canonical shard in $\polytope{H}_k$.
And its other facets are contained in the canonical shards of the $\polytope{H}_i$ with~$i < k$.
The normal fan of~$\polytope{P}_k$ contains the cone~$\polytope{C}_k$ (it is the one corresponding to the vertex at the origin).
The remaining cones of the fan are generated by the facets of~$\polytope{C}_k$ together with~$\b{u}$.
Their facets belong to the hyperplanes $\polytope{H}_{ij} \in \HA_0$ such that $\polytope{H}_i \cap \polytope{H}_j \subseteq \polytope{H}_{ij}$ for hyperplanes $\polytope{H}_i, \polytope{H}_j \in \HA_1$ with $\polytope{H}_i$ cutting $\polytope{H}_j$ in the forcing order.
As the $\polytope{H}_i$ do not cut the hyperplanes in $\HA_0$, these pieces of $\polytope{H}_{ij}$ are subsets of the union of all the shards of $\polytope{H}_{ij}$ that meet the intersection of the canonical shards of $\polytope{H}_i$ and $\polytope{H}_j$.
\end{proof}

%%%%%%%%%%%%%%%%%%%%%%%%%%%%%%%%%%%%%%
%%%%%%%%%%%%%%%%%%%%%%%%%%%%%%%%%%%%%%

\newpage
\addtocontents{toc}{ \vspace{.1cm} }
\appendix
\section{Detailed descriptions of Minkowski sums of shard polytopes}

This appendix is devoted to detailed vertex and facet descriptions of the Minkowski sums of shard polytopes considered in \cref{coro:MinkowskiSumShardPolytopes}.
These descriptions first require us to understand the normal cones of the vertices of shard polytopes.

%%%%%%%%

\subsection{Normal cones of vertices and edges of shard polytopes}
\label{subsec:normalConeDescriptions}

The next two lemmas provide detailed descriptions of the normal cones of the vertices and of the edges of the shard polytopes.
Recall that we say that a pair~$(i,j)$ (with~$i < j$) is contained in an alternating matching~$M$ when~$i, j \in M$ and $k \notin M$ for all~$i < k < j$.

\begin{lemma}
\label{lem:normalconematch}
Consider an arc~$\arc \eqdef (a, b, A, B) \in \arcs_n$, and let~$i \in \{a\} \cup A$ and~$j \in B \cup \{b\}$ with~$i < j$.
A vector~$\b{t} \in \R^n$ belongs to the normal cone of the shard polytope~$\shardPolytope$ corresponding to an \mbox{$\arc$-alternating} matching $M$ containing the pair $(i,j)$ if and only if the following conditions~hold:
\begin{itemize}
\item $\b{t}_i \ge \b{t}_j$,
\item for all $i < k < j$, if $k\in A$ then $\b{t}_i \ge \b{t}_k$ and if $k \in B$ then $\b{t}_k \ge \b{t}_j$,
\item for all $i < j' < i' < j$ with  $i' \in A$  and $j' \in B$ we have $\b{t}_{i'} \le \b{t}_{j'}$,
\item for all $a \le i' < i$ with $i' \in \{a\} \cup A$ and $\b{t}_{i'} > \b{t}_i$ there is $i' < h < i$ with $h \in B$ and $\b{t}_i \ge \b{t}_h$,
\item for all $j < j' \le b$ with $j' \in B \cup \{b\}$ and $\b{t}_{j'} < \b{t}_j$ there is $j < \ell < j'$ with $\ell \in A$ and $\b{t}_\ell \ge \b{t}_j$.
\end{itemize}
\end{lemma}

\begin{proof}
Let $M'$ and $M''$ be the alternating matchings maximized by $\b{t}$ in~$[a,i[$ and~$]j,b]$ respectively, and let $M \eqdef \{M',i,j,M''\} = \{a_1 < b_1 < \dots < a_m = i < b_m = j <\dots, a_r, b_r\}$.
We claim that the vector~$\b{t}$ belongs to the normal cone of~$\shardPolytope$ corresponding to~$M$.
 
Indeed, if this was not the case, there would be an edge of $\shardPolytope$ connecting the characteristic vectors of two $\arc$-alternating matchings~$M$ and~$M'$ such that $\tau(M') > \tau(M)$ (recall that we denote~$\tau(M) \eqdef \dotprod{\b{t}}{\chi(M)}$) and $(i,j)$ is not an pair of~$M'$ (either $\{i,j\} \nsubseteq M'$ or there is $k \in M$ with $i < k < j$).
By \cref{prop:elemPropShardPolytope}\,\eqref{it:edgesShardPolytope}, we have~$|M \symdif M'| = 2$.
We distinguish the following cases:
\begin{itemize}
\item If $M \symdif M' = \{i,i'\}$ with $i' \in \{a\} \cup A$, then we must have $b_{m-1} < i' < b_m = j$. Then $\tau(M')-\tau(M) = \b{t}_{i'}-\b{t}_i$. If $i < i' < j$ then $\b{t}_{i'} \le \b{t}_i$. If $i' < i$, then either $\b{t}_{i'} \le \b{t}_i$ or there is $h \in B$ such that $\b{t}_{i'} > \b{t}_i \ge \b{t}_h$ which would contradict the maximality of $M'$ in~$[a,i[$, as $M' \cup \{i',h\}$ increases the value in direction~$\b{t}$.
  
\item An analogous reasoning works for the case $M \symdif M'=\{j,j'\}$ with $j'\in B\cup\{b\}$.

\item If $M \symdif M' = \{i,j\}$, then $\tau(M')-\tau(M) = \b{t}_{j}-\b{t}_i \le 0$.

\item If $M \symdif M' = \{i',j'\}$ with $i' \in \{a\} \cup A \ssm \{i\}$ and $j' \in B \cup \{b\} \ssm \{j\}$, then $i < j' < i' < j$. We have $\tau(M')-\tau(M) = \b{t}_{i'}-\b{t}_{j'} \ge 0$.
  
\item If $M \symdif M' = \{i,j'\}$ with $j' \in B \cup \{b\} \ssm \{j\}$, then $j' = b_{m-1}$ and $\tau(M')-\tau(M) = \b{t}_{b_{m-1}}-\b{t}_{i}$. Note that $\b{t}_{b_{m-1}} \le \b{t}_{a_{m-1}}$ (by the maximality of $M'$). If $\b{t}_{a_{m-1}} \le \b{t}_i$ then $\tau(M')-\tau(M) \le 0$. Otherwise, there is $a_{m-1} < h < i$ with $h \in B$ such that $\b{t}_i \ge \b{t}_h$. We must have $\b{t}_h \ge \b{t}_{b_{m-1}}$ (by the maximality of $M'$), and hence $\b{t}_i \ge \b{t}_{b_{m-1}} \le 0$.
    
\item  An analogous reasoning works for the case $M \symdif M'=\{i',j\}$ with $i'\in \{a\} \cup A\ssm\{i\}$.
\end{itemize}

The same case analysis also shows that all the conditions are necessary.
Indeed, assume that  $M =\{a_1 < b_1 < \dots < a_m = i < b_m = j <\dots, a_r, b_r\}$ is an alternating matching such that~$\chi(M)$ maximizes~$\b{t}$.
We distinguish the following cases:
\begin{itemize}
\item If $\b{t}_i< \b{t}_j$ then $\tau(M \ssm \{i,j\}) > \tau(M)$.
\item If there is $i < k < j$ with $k \in A$ and $\b{t}_i < \b{t}_k$, then $\tau(M \cup \{k\} \ssm \{i\}) > \tau(M)$, and if there is $k \in B$ with $\b{t}_k < \b{t}_j$, then $\tau(M \cup \{k\} \ssm \{j\}) > \tau(M)$.
\item If there are $i < j' < i' < j$ with $i' \in A$ and $j' \in B$ and $\b{t}_{i'} > \b{t}_{j'}$, then ${\tau(M \cup \{i',j'\}) > \tau(M)}$.
\item If there is some $a \le i' < i$ with $i' \in \{a\} \cup A$ and $\b{t}_{i'} > \b{t}_i$ such that there is no $i' < h < i$ with $h \in B$ such that $\b{t}_i \ge \b{t}_h$, then $i' \notin M$, and $\tau(M \cup \{i'\} \ssm \{i\}) > \tau(M)$.
\item If there is some $j < j' \le b$ with $j' \in B \cup \{b\}$ and $\b{t}_{j'} < \b{t}_j$ such that there is no $j < \ell < j'$ with $\ell \in A$ such that $\b{t}_\ell \ge \b{t}_j$, then $\tau(M \cup \{j'\} \ssm \{j\}) > \tau(M)$.
\qedhere
\end{itemize}    
\end{proof}

Note that, together with the following observation, this allows for a precise characterization of the normal cones of the edges of $\shardPolytope$, and hence could be used to give an alternative proof for \cref{prop:shardPolytopeFan}.

\begin{lemma}
\label{lem:normalconeedge}
Let $\arc \eqdef (a, b, A, B) \in \arcs_n$ and $1 \le i < j \le n$.
The vector~$\b{t} \in \R^n$ is in the normal cone of an edge in direction~$\b{e}_i - \b{e}_j$ if and only if $\b{t}_i = \b{t}_j$ and
\begin{itemize}
 \item if $i \in \{a\} \cup A$ and $j \in B \cup \{b\}$, $\b{t}$ belongs to the normal cone of an alternating matching containing the pair~$(i,j)$,
 \item if $i \in B \cup \{b\}$ and $j \in \{a\} \cup A$, there are $a'<i<j<b'$ with~${a' \in \{a\} \cup A}$~and~${b' \in B \cup \{b\}}$ such that $\b{t}$ belongs to the normal cone of an alternating matching containing the pair~$(a',b')$ (and hence also of an alternating matching containing the pairs~$(a',i)$ and $(j,b')$),
 \item if $i, j \in A \cup \{a\}$, there is some $i<j<b'$ with $b' \in B \cup \{b\}$ such that $\b{t}$ belongs to the normal cone of an alternating matching containing the pair~$(i,b')$ (and hence also of an alternating matching containing the pair~$(j,b')$),
 \item if $i, j \in B \cup \{b\}$, there is some $a'<i<j$ with $a' \in \{a\} \cup A$ such that $\b{t}$ belongs to the normal cone of an alternating matching containing the pair~$(a',j)$ (and hence also of an alternating matching containing the pair~$(a',i)$).
\end{itemize}
\end{lemma}

%%%%%%%%

\subsection{Vertex and facet descriptions of Minkowski sums of shard polytopes}
\label{subsec:vertexFacetDescriptionsShardSumotopes}

We have seen in \cref{prop:elemPropShardPolytope}\,(\ref{it:verticesShardPolytope}--\ref{it:facetsShardPolytope}) that \cref{prop:shardPolytope} gives the vertex and facet descriptions of a single shard polytope.
In this section we provide explicit vertex and facet descriptions of the Minkowski sums of shard polytopes considered in \cref{coro:MinkowskiSumShardPolytopes}.
The vertex description requires the following definition motivated by \cref{lem:normalconematch}.

\begin{definition}
\label{def:vertexMaximizingDirection}
Fix a vector~$\b{t} \in \R^n$.
For any arc~$\arc \eqdef (a, b, A, B) \in \arcs_n$, define~${\b{v}(\b{t}, \arc) \in \{-1, 0, 1\}^n}$ as the point with $j$th coordinate
\begin{enumerate}[$\quad\bullet$]
\item $1$ if~$j \in \{a\} \cup A$ and
	\begin{enumerate}[$\circ$]
	\item for all $a \le i < j$ with $i \in \{a\} \cup A$ and $\b{t}_i > \b{t}_j$, there is~$i < h < j$ with~$h \in B$~and~$\b{t}_h < \b{t}_j$,
	\item there is~$j < k \le b$ with $k \in B \cup \{b\}$ and $\b{t}_j > \b{t}_k$, and for all~$j < \ell < k$, if~$\ell \in A$ then~${\b{t}_j > \b{t}_\ell}$.
	\end{enumerate}
\item $-1$ if~$j \in B \cup \{b\}$ and
	\begin{enumerate}[$\circ$]
	\item \mbox{there is~$a \le i < j$ with $i \in \{a\} \cup A$ and $\b{t}_i > \b{t}_j$, and for all~$i < h < j$, if~$h \in B$~then~$\b{t}_h > \b{t}_j$},
	\item for all~$j < k \le b$ with $k \in B \cup \{b\}$ and $\b{t}_j > \b{t}_k$, there is~$j < \ell < k$ with~$\ell \in A$ and $\b{t}_j < \b{t}_\ell$.
	\end{enumerate}
\end{enumerate}
For any arc ideal~$\arcs$, define~$\b{v}(\b{t}, \arcs) \eqdef \sum_{\arc \in \arcs} \b{v}(\b{t}, \arc)$.
\end{definition}

\begin{proposition}
\label{prop:vertexMaximizingDirection}
For any vector~$\b{t} \in \R^n$ and any arc~$\arc \in \arcs_n$ (resp.~arc ideal~$\arcs \subseteq \arcs_n$), the point~$\b{v}(\b{t}, \arc)$ (resp.~$\b{v}(\b{t}, \arcs)$) is a vertex of the shard polytope~$\shardPolytope$ (resp.~quotientope~$\shardPolytope[\arcs]$) maximizing the direction~$\b{t}$.
\end{proposition}

\begin{proof}
For the shard polytope~$\shardPolytope$, it is a direct consequence of \cref{lem:normalconematch}.
The result for the quotientope~$\shardPolytope[\arcs]$ follows by standard properties of Minkowski sums recalled in \cref{subsec:fansPolytopes}.
\end{proof}

Applying \cref{prop:vertexMaximizingDirection} to a vector~$\b{t} \in \R^n$ with distinct coordinates yields the vertex description of the quotientope~$\shardPolytope[\arcs]$ of \cref{coro:MinkowskiSumShardPolytopes}.

\begin{corollary}
For any arc ideal~$\arcs \subseteq \arcs_n$, the vertices of the quotientope~$\shardPolytope[\arcs]$ are given by~$\b{v}(\pi^{-1}, \arcs)$, for a set of permutations~$\pi$ representing the congruence classes of~$\equiv_\arcs$ (\eg the minimal permutation of each congruence class).
\end{corollary}

\begin{proof}
Recall that for any permutation~$\pi$, the direction~$\pi^{-1}$ belongs to the cone~$\polytope{C}(\pi)$ of the braid fan.
Therefore, the vertex~$\b{v}(\pi^{-1}, \arcs)$ of~$\shardPolytope[\arcs]$ maximizing the direction~$\pi^{-1}$ corresponds to the $\equiv_\arcs$-congruence class of~$\pi$.
The set of all vertices of~$\shardPolytope[\arcs]$ is thus obtained by choosing a collection of representatives of the congruence classes of~$\equiv_\arcs$.
\end{proof}

We now give the facet description of the quotientope~$\shardPolytope[\arcs]$.
Note that we already obtained a description in terms of inequalities in \cref{prop:RHStoSP}, passing through decompositions of shard polytopes as Minkowski sums and differences of faces of the standard simplex.
Here, we simply apply \cref{prop:vertexMaximizingDirection} to the vectors~$\one_R$ for~$\varnothing \ne R \subsetneq [n]$, which motivates the following definition.

\begin{definition}
\label{def:h}
Fix a proper subset~$\varnothing \ne R \subsetneq [n]$.
For any arc~$\arc \eqdef (a, b, A, B) \in \arcs_n$, define~$h(R, \arc)$ as the number of pairs~$1 \le r < s \le n$ with~$r \in (\{a\} \cup A) \cap R$ and~$s \in (B \cup \{b\}) \ssm R$, and such that $\ell \notin B \symdif R$ for any~$r < \ell < s$.
For any arc ideal~$\arcs$, define~$h(R, \arcs) \eqdef \sum_{\arc \in \arcs} h(R, \arc)$.
\end{definition}

\begin{proposition}
\label{prop:h}
For any proper subset~$\varnothing \ne R \subsetneq [n]$, the value~$h(R, \arc)$ (resp.~$h(R, \arcs)$) is the maximum of the scalar product~$\dotprod{\one_R}{\b{x}}$ over the shard polytope~$\shardPolytope$ (resp.~quotientope~$\shardPolytope[\arcs]$).
\end{proposition}

\begin{proof}
For any~$\varnothing \ne R \subsetneq [n]$ and any arc~$\arc$, the maximum of the scalar product~$\dotprod{\one_R}{\b{x}}$ over~$\shardPolytope$ is given by~$\dotprod{\one_R}{\b{v}(\one_R, \arc)} = \sum_{r \in R} \b{v}(\one_R, \arc)_r$.
The description of \cref{def:vertexMaximizingDirection} ensures that for~$r \in R$, we have either~$\b{v}(\one_R, \arc)_r = 0$, or~$\b{v}(\one_R, \arc)_r = 1$ if $r \in \{a\} \cup A$ and there exists~$r < s \le b$ such that~$s \in (B \cup \{b\}) \ssm R$ and~$\ell \notin R \ssm B$ for all~$r < \ell < s$.
Considering for each such~$r$ the leftmost such~$s$, each non-vanishing coordinate of~$\b{v}(\one_R, \arc)$ in~$R$ can thus be associated to a pair~$1 \le r < s \le n$, where~$r \in (\{a\} \cup A) \cap R$ while~$s \in (B \cup \{b\}) \ssm R$ and such~$\ell \notin B \symdif R$ for all~$r < \ell < s$.
This yields the result for the shard polytope~$\shardPolytope$.
The result for the quotientope~$\shardPolytope[\arcs]$ follows by standard properties of Minkowski sums recalled in \cref{subsec:fansPolytopes}.
\end{proof}

\begin{corollary}
\label{prop:inequalitiesShardsumotopesOuterNormals}
For any arc ideal~$\arcs \subseteq \arcs_n$, the quotientope~$\shardPolytope[\arcs]$ is given by 
\[
\shardPolytope[\arcs]=\bigset{\b{x} \in \R^n}{\dotprod{\one}{\b{x}} = 0 \text{ and } \dotprod{\one_R}{\b{x}} \le h(R, \arcs) \text{ for all } \varnothing \ne R \subsetneq [n]}.
\]
\end{corollary}

The facet description of the quotientope~$\shardPolytope[\arcs]$ immediately follows from this statement and the description of the rays of the quotient fan given in \cref{lem:raysQuotientFan}.
To fit with the facet description of \cref{exm:LodayAsso,exm:HohlwegLangeAsso}, it is convenient to translate the outer normal description of \cref{prop:inequalitiesShardsumotopesOuterNormals} into an inner normal description (this is a transparent operation since the quotientope~$\shardPolytope[\arcs]$ lives in the hyperplane~$\dotprod{\one}{\b{x}} = 0$).

\begin{corollary}
\label{prop:facetsShardsumotopesInnerNormals}
For any arc ideal~$\arcs \subseteq \arcs_n$, the facets of the quotientope~$\shardPolytope[\arcs]$ are given by the inequalities $\dotprod{\one_R}{\b{x}} \ge -h([n] \ssm R, \arcs)$, for all proper subsets~$\varnothing \ne R \subsetneq [n]$ satisfying the conditions of \cref{lem:raysQuotientFan}.
\end{corollary}

\begin{example}
Following up \cref{exm:ZonoMinkowskiSum}, for a basic arc, we have
\begin{itemize}
\item $\b{v}(\b{t}, (i, i+1, \varnothing, \varnothing)) = \b{e}_i - \b{e}_{i+1}$ if ${\b{t}_i > \b{t}_{i+1}}$ (\ie if $\b{t}$ has a descent at~$i$), and~$0$ otherwise,
\item $h(R, (i, i+1, \varnothing, \varnothing)) = 1$ if~$i \in R$ and~$i+1 \notin R$, and~$0$ otherwise.
\end{itemize}
For the ideal of basic arcs~$\arcs_\textrm{rec} \eqdef \set{(i, i+1, \varnothing, \varnothing)}{i \in [n-1]}$, the parallelotope~${\shardPolytope[\arcs_\textrm{rec}]}$ has 
\begin{itemize}
\item a vertex $\b{v}(\b{t}, \arcs_\textrm{rec}) = \sum_{\b{t}_i > \b{t}_{i+1}} \b{e}_i - \b{e}_{i+1}$ for each recoil class,
\item two facets defined by~$0 \le \dotprod{\one_{[i]}}{\b{x}} \le 1$ for each~$i \in [n-1]$.
\end{itemize}
\end{example}

\begin{example}[Tamari]
\label{exm:LodayShardPolytope}
Following up \cref{exm:sylvesterCongruence,exm:noncrossingPartitions,exm:LodayAsso,exm:shardsAsso,exm:LodayAssoMinkowskiSum}, for an up arc, we have
\begin{itemize}
\item $\b{v}(\b{t}, (a, b, {]a,b[}, \varnothing)) = \b{e}_j - \b{e}_b$ where ${j \in [a,b]}$ maximal such that~${\b{t}_j = \max_{i \in [a,b]} \b{t}_j}$,
\item $h(R, (a, b, {]a,b[}, \varnothing)) = 1$ if~${[a,b[} \cap R \ne \varnothing$ and~$b \notin R$, and~$0$ otherwise.
\end{itemize}
For the ideal of up arcs~$\arcs_\textrm{sylv} = \set{(a, b, {]a,b[}, \varnothing)}{1 \le a < b \le n}$, the associahedron~$\shardPolytope[\arcs_\textrm{sylv}]$ has
\begin{itemize}
\item a vertex~$\b{v}(\b{t}, \arcs_\textrm{sylv})$ for each sylvester class, with coordinates $\b{v}(\b{t}, \arcs_\textrm{sylv})_j = (j-i) (k-j) - j$, where~$i = \max \big( \{0\} \cup \set{h \in {[1,j[}}{\b{t}_h > \b{t}_j} \big)$ and ${k = \min \big( \set{\ell \in {]j,n]}}{\b{t}_j < \b{t}_\ell}} \cup \{n+1\} \big)$,
\item a facet defined by~$\dotprod{\one_{[i,j]}}{\b{x}} \ge -(i-1)(j-i+1)$ for each interval $1 \le i \le j \le n$.
\end{itemize}
These descriptions show that~$\shardPolytope[\arcs_{\textrm{sylv}}]$ is the translate by the vector~$(-1, \dots, -n)$ of J.-L.~Loday's associahedron~$\Asso$ constructed in~\cite{ShniderSternberg,Loday} and described in \cref{exm:LodayAsso}.
Indeed, note that:
\begin{enumerate}[(i)]
\item Our formula for the vertices~$\b{v}(\b{t}, \arcs_\textrm{sylv})$ of~$\shardPolytope[\arcs_{\textrm{sylv}}]$ shows that they correspond to sylvester classes: $\b{t}$ and~$\b{t'}$ belong to the same chamber of the Tamari quotient fan if and only if the permutations~$\pi$ and~$\pi'$ such that~${\b{t}_{\pi(1)} < \dots < \b{t}_{\pi(n)}}$ and~${t'_{\pi'(1)} < \dots < t'_{\pi'(n)}}$ differ by repeated applications of the rewriting rule~$U i k V j W \equiv_\textrm{sylv} U k i V j W$ for~${i < j < k}$.
\item Sylvester classes are known to be labeled by binary trees: the direction~$\b{t}$ belongs to a binary tree~$T$ if and only if $\b{t}_i < \b{t}_j$ when $i$ is a descendant of $j$ in~$T$ (using the infix labeling of~$T$). In our formula for~$\b{v}(\b{t}, \arcs_\textrm{sylv})_j = (j-i) (k-j) - j$, we can then reinterpret~$(j-i)$ and~$(k-j)$ as the numbers of leaves in the left and right subtrees of~$j$ in~$T$, thus recovering J.-L.~Loday's vertex coordinates~\cite{Loday}, up to a global translation of vector~$(-1, \dots, -n)$.
\item Applying this translation, we also recover that the facet inequalities of J.-L.~Loday's associahedron are given by~$\dotprod{\one_{[i,j]}}{\b{x}} \ge -(i-1)(j-i+1) + \sum_{k = i}^j k = \binom{j-i+2}{2}$ for all intervals~$1 \le i \le j \le n$.
\item As already mentioned in \cref{exm:LodayAssoMinkowskiSum}, we recover the description of~$\Asso$ as the Minkowski sum of the faces of the standard simplex corresponding to intervals~\cite{Postnikov}.
\end{enumerate}
\end{example}

\begin{example}[Cambrian]
\label{exm:HohlwegLangeAssoVertexFacetDescription}
Following up \cref{exm:CambrianCongruences,exm:HohlwegLangeAsso,exm:HohlwegLangeAssoMinkowskiSum}, consider the $\arc$-Cambrian congruence of an arc ${\arc \eqdef (a, b, A, B)}$ and the corresponding arc ideal~$\arcs_\arc$ generated by~$\arc$.
The associahedron~$\shardPolytope[\arcs]$~has
\begin{itemize}
\item a vertex~$\b{v}(\b{t}, \arcs_\arc)$ for each $\arc$-Cambrian class, with $j$th coordinate given by
	\begin{enumerate}[$\circ$]
	\item if~$j \in \{a\} \cup A$, then
	\[
	\hspace{2cm} \b{v}(\b{t}, \arcs_\textrm{sylv})_j = (|{]i,j[} \cap A| + 1) (|{]j,k[} \cap A| + 1) - (j-a+1)
	\]
	where
	\begin{align*}
	\hspace{2cm} i & = \max \big( \{a-1\} \cup \set{h \in {[1,j[} \cap (\{a\} \cup A)}{\b{t}_h > \b{t}_j \text{ and } \b{t}_g > \b{t}_j \text{ for all } g \in {]h,j[} \cap B} \! \big), \\
	k & = \min \big( \! \set{\ell \in {]j,n]} \cap (A \cup \{b\})}{\b{t}_j < \b{t}_\ell \text{ and } \b{t}_j < \b{t}_m \text{ for all } m \in {]j,\ell[} \cap B} \cup \{b+1\} \big),
	\end{align*}
	\item if~$j \in B \cup \{b\}$, then
	\[
	\hspace{2cm} \b{v}(\b{t}, \arcs_\textrm{sylv})_j = b-a+2 - (|{]i,j[} \cap B| + 1) (|{]j,k[} \cap B| + 1) - (j-a+1)
	\]
	where
	\begin{align*}
	\hspace{2cm} i & = \max \big( \{a-1\} \cup \set{h \in {[1,j[} \cap (\{a\} \cup B)}{\b{t}_h < \b{t}_j \text{ and } \b{t}_g < \b{t}_j \text{ for all } g \in {]h,j[} \cap A} \! \big), \\
	k & = \min \big( \! \set{\ell \in {]j,n]} \cap (B \cup \{b\})}{\b{t}_j > \b{t}_\ell \text{ and } \b{t}_j > \b{t}_m \text{ for all } m \in {]j,\ell[} \cap A} \cup \{b+1\} \big).
	\end{align*}
	\end{enumerate}
\item a facet defined by~$\dotprod{\one_R}{\b{x}} \ge \binom{|R|+1}{2} - \sum_{r \in R} (r-a+1)$ for any subset~$\varnothing \ne R \subsetneq [a,b]$ described in \cref{exm:HohlwegLangeAsso}. Indeed, for any such~$R$ and any~$\arc' \in \arcs_\arc$ with endpoints~$a' < b'$, we have~$h([n] \ssm R, \arc') = 0$ if~$a' \in R$ or~$b' \notin R$, while~$h([n] \ssm R, \arc') = 1$ if~$a' \notin R$ and~$b' \in R$ since we count only the pair~$(r,s)$ given by~$r \eqdef \max \big( {[a',b']} \cap (\{a\} \cup A) \ssm R \big)$ and~$s \eqdef \min \big( {[a',b']} \cap (B \cup \{b\}) \cap R \big)$. Therefore, $h([n] \ssm R, \arcs_\arc)$ counts the number of pairs~$a \le a' < b' \le b$ such that~$a' \notin R$ and~$b' \in R$, which is given by~$\sum_{i \in [k]} r_i-a+1-i$ for~$R = \{r_1 < \dots < r_k\}$.
\end{itemize}
These descriptions show that~$\shardPolytope[\arcs_\arc]$ is the translate by the vector~$-\sum_{i \in [a,b]} (i-a+1) \, \b{e}_i$ of C.~Hohlweg and C.~Lange's associahedron~$\Asso[\arc]$ constructed in~\cite{HohlwegLange} and described in \cref{exm:HohlwegLangeAsso}.
Indeed, the connection between the facet descriptions given here and in \cref{exm:HohlwegLangeAsso} is immediate, and the reader familiar with the combinatorics of permutrees will recognize here the vertex description given in \cref{exm:HohlwegLangeAsso}.
\cref{fig:shardPolytopeSums3,fig:associahedra} provide illustrations when~$n = 3$ and~$n = 4$.
As mentioned in \cref{exm:HohlwegLangeAssoMinkowskiSum}, this Minkowski sum of Cambrian associahedra in terms of shard polytopes already appeared in the context of brick polytopes.
\end{example}

\begin{example}
Following up \cref{exm:weirdPermMinkowskiSum}, for the ideal of all arcs~$\arcs_n$, the polytope~$\shardPolytope[\arcs_n]$~has
\begin{itemize}
\item a vertex~$\b{v}(\b{t}, \arcs_n) = \Big[ 2^{n-1} \cdot \Big( \sum_{\substack{i < j \\ \b{t}_i > \b{t}_j}} 2^{i-j} - \sum_{\substack{j < k \\ \b{t}_j > \b{t}_k}} 2^{j-k} \Big) \Big]_{j \in [n]}$ for each permutation (note that the indices that appear in the sums are the inversions with~$j$),
\item a facet defined by~$\dotprod{\one_R}{\b{x}} \le \sum_{\substack{i < j \\ i \in R, \; j \notin R}} 2^{n+i-j}$ for each proper subset~$\varnothing \ne R \subsetneq [n]$.
\end{itemize}
As already mentioned in \cref{exm:weirdPermMinkowskiSum}, the resulting quotientope is not the permutahedron~$\Perm$.
Even when we translate its barycenter to the origin, its vertex set is not invariant by coordinate permutations.
See \cref{fig:shardPolytopeSums3,fig:weirdPermutahedron} for illustrations when~$n = 3$ and~$n = 4$.
\end{example}

%%%%%%%%%%%%%%%%%%%%%%%%%%%%%%%%%%%%%%

\newpage
\addtocontents{toc}{ \vspace{.2cm} }
\bibliographystyle{alpha}
\bibliography{shardsumotopes}
\label{sec:biblio}

\end{document}

%% file: figures/seriesParallel.tex
\begin{tikzpicture}[xscale=1.5]

%% nodes
\node (1) at (1, 0) {$0$};
\node (2) at (2, 0) {$1$};
\node (3) at (3, 0) {$2$};
\node (4) at (4, 0) {$3$};
\node (5) at (5, 0) {$4$};
\node (6) at (6, 0) {$5$};
\node (7) at (7, 0) {$6$};
\node (8) at (8, 0) {$7$};

%% a-edges
\path (1) edge [bend left=65] node [near start, above=1pt] {$a_1 = a$} (8);
\path (1) edge [bend left=60] node [near start, below=-1pt] {$a_2$} (8);
\path (2) edge [bend left=55] node [near start, above=-1pt] {$a_3$} (8);
\path (3) edge [bend left=50] node [near start, above=-1pt] {$a_4$} (8);
\path (5) edge [bend left=45] node [near start, above=-1pt] {$a_5$} (8);
\path (5) edge [bend left=40] node [near start, below=-1pt] {$a_6$} (8);
\path (7) edge [bend left=20] node [near start, above=-1pt] {$a_7$} (8);

%% b-edges
\draw (1) -- (2) node [midway, below] {$b_1$};
\draw (2) -- (3) node [midway, below] {$b_2$};
\draw (3) -- (4) node [midway, below] {$b_3$};
\draw (4) -- (5) node [midway, below] {$b_4$};
\draw (5) -- (6) node [midway, below] {$b_5$};
\draw (6) -- (7) node [midway, below] {$b_6$};
\draw (7) -- (8) node [midway, below] {$b_7 = b$};

\end{tikzpicture}

%% file: figures/B3asso1.tex
% sage code
% sage: cong = B_Shards(3, B_shards = repr_to_B_shards(['0d0...', '.0u0..', '.0ud0.']), restriction='complement_generators')
% sage: Img = cong.quotientope_from_Minkowski_sums().projection().tikz([-723,-409,-556], 94.3, scale=.3)
% sage: f=open('B3Asso1.tex','w'); f.write(Img); f.close()
% then add the picuture of the permutahedron in the middle.

\begin{tikzpicture}%
	[x={(0.487424cm, -0.236521cm)},
	y={(0.872822cm, 0.104998cm)},
	z={(0.024475cm, 0.965936cm)},
	scale=.3,
	back/.style={very thin, opacity=0.5},
	edgePerm/.style={color=red!95!black, very thick, cap=round},
	edgeAsso/.style={color=blue!95!black, very thick, cap=round},
	facetPerm/.style={fill=red!95!black,fill opacity=0},
	facetAsso/.style={fill=blue!95!black,fill opacity=0},
	vertexPerm/.style={inner sep=0pt, circle, anchor=base},
	vertexAsso/.style={inner sep=0pt, circle, anchor=base}]

%% Coordinate of permutahedron vertices:
\coordinate (-4, -4, -4) at (-4, -4, -4);
\coordinate (-4, -4, -2) at (-4, -4, -2);
\coordinate (-4, -3, -5) at (-4, -3, -5);
\coordinate (-4, -3, -1) at (-4, -3, -1);
\coordinate (-4, -1, -5) at (-4, -1, -5);
\coordinate (-4, -1, -1) at (-4, -1, -1);
\coordinate (-4, 0, -4) at (-4, 0, -4);
\coordinate (-4, 0, -2) at (-4, 0, -2);
\coordinate (-3, -5, -4) at (-3, -5, -4);
\coordinate (-3, -5, -2) at (-3, -5, -2);
\coordinate (-3, -3, -6) at (-3, -3, -6);
\coordinate (-3, -3, 0) at (-3, -3, 0);
\coordinate (-3, -1, -6) at (-3, -1, -6);
\coordinate (-3, -1, 0) at (-3, -1, 0);
\coordinate (-3, 1, -4) at (-3, 1, -4);
\coordinate (-3, 1, -2) at (-3, 1, -2);
\coordinate (-2, -5, -5) at (-2, -5, -5);
\coordinate (-2, -5, -1) at (-2, -5, -1);
\coordinate (-2, -4, -6) at (-2, -4, -6);
\coordinate (-2, -4, 0) at (-2, -4, 0);
\coordinate (-2, 0, -6) at (-2, 0, -6);
\coordinate (-2, 0, 0) at (-2, 0, 0);
\coordinate (-2, 1, -5) at (-2, 1, -5);
\coordinate (-2, 1, -1) at (-2, 1, -1);
\coordinate (0, -5, -5) at (0, -5, -5);
\coordinate (0, -5, -1) at (0, -5, -1);
\coordinate (0, -4, -6) at (0, -4, -6);
\coordinate (0, -4, 0) at (0, -4, 0);
\coordinate (0, 0, -6) at (0, 0, -6);
\coordinate (0, 0, 0) at (0, 0, 0);
\coordinate (0, 1, -5) at (0, 1, -5);
\coordinate (0, 1, -1) at (0, 1, -1);
\coordinate (1, -5, -4) at (1, -5, -4);
\coordinate (1, -5, -2) at (1, -5, -2);
\coordinate (1, -3, -6) at (1, -3, -6);
\coordinate (1, -3, 0) at (1, -3, 0);
\coordinate (1, -1, -6) at (1, -1, -6);
\coordinate (1, -1, 0) at (1, -1, 0);
\coordinate (1, 1, -4) at (1, 1, -4);
\coordinate (1, 1, -2) at (1, 1, -2);
\coordinate (2, -4, -4) at (2, -4, -4);
\coordinate (2, -4, -2) at (2, -4, -2);
\coordinate (2, -3, -5) at (2, -3, -5);
\coordinate (2, -3, -1) at (2, -3, -1);
\coordinate (2, -1, -5) at (2, -1, -5);
\coordinate (2, -1, -1) at (2, -1, -1);
\coordinate (2, 0, -4) at (2, 0, -4);
\coordinate (2, 0, -2) at (2, 0, -2);

%% Coordinate of associahedron vertices:
\coordinate (-4, -8, 0) at (-4, -8, 0);
\coordinate (10, -4, -6) at (10, -4, -6);
\coordinate (10, -10, 0) at (10, -10, 0);
\coordinate (5, 1, -6) at (5, 1, -6);
\coordinate (0, 1, -1) at (0, 1, -1);
\coordinate (-4, -3, -5) at (-4, -3, -5);
\coordinate (0, 0, 0) at (0, 0, 0);
\coordinate (-1, 1, -6) at (-1, 1, -6);
\coordinate (-2, 1, -1) at (-2, 1, -1);
\coordinate (-4, -2, 0) at (-4, -2, 0);
\coordinate (-4, -1, -5) at (-4, -1, -5);
\coordinate (-2, 0, 0) at (-2, 0, 0);
\coordinate (-4, 0, -4) at (-4, 0, -4);
\coordinate (-4, 0, -2) at (-4, 0, -2);
\coordinate (-2, -4, -6) at (-2, -4, -6);
\coordinate (-2, -10, 0) at (-2, -10, 0);
\coordinate (-3, -3, -6) at (-3, -3, -6);
\coordinate (-3, 1, -2) at (-3, 1, -2);
\coordinate (-3, 1, -4) at (-3, 1, -4);
\coordinate (-3, -1, -6) at (-3, -1, -6);

%% Drawing associahedron edges in the back
\draw[edgeAsso,back] (-4, -8, 0) -- (-4, -3, -5);
\draw[edgeAsso,back] (5, 1, -6) -- (-1, 1, -6);
\draw[edgeAsso,back] (0, 1, -1) -- (-2, 1, -1);
\draw[edgeAsso,back] (-4, -3, -5) -- (-4, -1, -5);
\draw[edgeAsso,back] (-4, -3, -5) -- (-3, -3, -6);
\draw[edgeAsso,back] (-1, 1, -6) -- (-3, 1, -4);
\draw[edgeAsso,back] (-1, 1, -6) -- (-3, -1, -6);
\draw[edgeAsso,back] (-2, 1, -1) -- (-2, 0, 0);
\draw[edgeAsso,back] (-2, 1, -1) -- (-3, 1, -2);
\draw[edgeAsso,back] (-4, -2, 0) -- (-4, 0, -2);
\draw[edgeAsso,back] (-4, -1, -5) -- (-4, 0, -4);
\draw[edgeAsso,back] (-4, -1, -5) -- (-3, -1, -6);
\draw[edgeAsso,back] (-4, 0, -4) -- (-4, 0, -2);
\draw[edgeAsso,back] (-4, 0, -4) -- (-3, 1, -4);
\draw[edgeAsso,back] (-4, 0, -2) -- (-3, 1, -2);
\draw[edgeAsso,back] (-2, -4, -6) -- (-3, -3, -6);
\draw[edgeAsso,back] (-3, -3, -6) -- (-3, -1, -6);
\draw[edgeAsso,back] (-3, 1, -2) -- (-3, 1, -4);

%% Drawing associahedron vertices in the back
\node[vertexAsso,back] at (-1, 1, -6)     {};
\node[vertexAsso,back] at (-3, -3, -6)     {};
\node[vertexAsso,back] at (-3, -1, -6)     {};
\node[vertexAsso,back] at (-2, 1, -1)     {};
\node[vertexAsso,back] at (-3, 1, -2)     {};
\node[vertexAsso,back] at (-3, 1, -4)     {};
\node[vertexAsso,back] at (-4, -3, -5)     {};
\node[vertexAsso,back] at (-4, -1, -5)     {};
\node[vertexAsso,back] at (-4, 0, -4)     {};
\node[vertexAsso,back] at (-4, 0, -2)     {};

%% Drawing permutahedron edges in the back
\draw[edgePerm,back] (-4, -4, -4) -- (-4, -4, -2);
\draw[edgePerm,back] (-4, -4, -4) -- (-4, -3, -5);
\draw[edgePerm,back] (-4, -4, -4) -- (-3, -5, -4);
\draw[edgePerm,back] (-4, -4, -2) -- (-4, -3, -1);
\draw[edgePerm,back] (-4, -4, -2) -- (-3, -5, -2);
\draw[edgePerm,back] (-4, -3, -5) -- (-4, -1, -5);
\draw[edgePerm,back] (-4, -3, -5) -- (-3, -3, -6);
\draw[edgePerm,back] (-4, -3, -1) -- (-4, -1, -1);
\draw[edgePerm,back] (-4, -3, -1) -- (-3, -3, 0);
\draw[edgePerm,back] (-4, -1, -5) -- (-4, 0, -4);
\draw[edgePerm,back] (-4, -1, -5) -- (-3, -1, -6);
\draw[edgePerm,back] (-4, -1, -1) -- (-4, 0, -2);
\draw[edgePerm,back] (-4, -1, -1) -- (-3, -1, 0);
\draw[edgePerm,back] (-4, 0, -4) -- (-4, 0, -2);
\draw[edgePerm,back] (-4, 0, -4) -- (-3, 1, -4);
\draw[edgePerm,back] (-4, 0, -2) -- (-3, 1, -2);
\draw[edgePerm,back] (-3, -3, -6) -- (-3, -1, -6);
\draw[edgePerm,back] (-3, -3, -6) -- (-2, -4, -6);
\draw[edgePerm,back] (-3, -1, -6) -- (-2, 0, -6);
\draw[edgePerm,back] (-3, 1, -4) -- (-3, 1, -2);
\draw[edgePerm,back] (-3, 1, -4) -- (-2, 1, -5);
\draw[edgePerm,back] (-3, 1, -2) -- (-2, 1, -1);
\draw[edgePerm,back] (-2, 0, -6) -- (-2, 1, -5);
\draw[edgePerm,back] (-2, 0, -6) -- (0, 0, -6);
\draw[edgePerm,back] (-2, 0, 0) -- (-2, 1, -1);
\draw[edgePerm,back] (-2, 1, -5) -- (0, 1, -5);
\draw[edgePerm,back] (-2, 1, -1) -- (0, 1, -1);

%% Drawing permutahedron vertices in the back
\node[vertexPerm,back] at (-4, -4, -2)     {};
\node[vertexPerm,back] at (-4, -3, -1)     {};
\node[vertexPerm,back] at (-4, -3, -5)     {};
\node[vertexPerm,back] at (-4, -1, -5)     {};
\node[vertexPerm,back] at (-3, -3, -6)     {};
\node[vertexPerm,back] at (-3, -1, -6)     {};
\node[vertexPerm,back] at (-4, -4, -4)     {};
\node[vertexPerm,back] at (-4, -1, -1)     {};
\node[vertexPerm,back] at (-4, 0, -4)     {};
\node[vertexPerm,back] at (-3, 1, -4)     {};
\node[vertexPerm,back] at (-2, 0, -6)     {};
\node[vertexPerm,back] at (-2, 1, -5)     {};
\node[vertexPerm,back] at (-4, 0, -2)     {};
\node[vertexPerm,back] at (-3, 1, -2)     {};
\node[vertexPerm,back] at (-2, 1, -1)     {};

%% Drawing permutahedron facets
\fill[facetPerm] (2, 0, -2) -- (2, -1, -1) -- (2, -3, -1) -- (2, -4, -2) -- (2, -4, -4) -- (2, -3, -5) -- (2, -1, -5) -- (2, 0, -4) -- cycle {};
\fill[facetPerm] (2, 0, -2) -- (1, 1, -2) -- (0, 1, -1) -- (0, 0, 0) -- (1, -1, 0) -- (2, -1, -1) -- cycle {};
\fill[facetPerm] (2, 0, -2) -- (1, 1, -2) -- (1, 1, -4) -- (2, 0, -4) -- cycle {};
\fill[facetPerm] (2, -1, -1) -- (1, -1, 0) -- (1, -3, 0) -- (2, -3, -1) -- cycle {};
\fill[facetPerm] (0, -4, -6) -- (-2, -4, -6) -- (-2, -5, -5) -- (0, -5, -5) -- cycle {};
\fill[facetPerm] (0, -4, 0) -- (-2, -4, 0) -- (-2, -5, -1) -- (0, -5, -1) -- cycle {};
\fill[facetPerm] (2, -3, -5) -- (1, -3, -6) -- (0, -4, -6) -- (0, -5, -5) -- (1, -5, -4) -- (2, -4, -4) -- cycle {};
\fill[facetPerm] (2, -3, -1) -- (1, -3, 0) -- (0, -4, 0) -- (0, -5, -1) -- (1, -5, -2) -- (2, -4, -2) -- cycle {};
\fill[facetPerm] (1, -5, -2) -- (0, -5, -1) -- (-2, -5, -1) -- (-3, -5, -2) -- (-3, -5, -4) -- (-2, -5, -5) -- (0, -5, -5) -- (1, -5, -4) -- cycle {};
\fill[facetPerm] (2, 0, -4) -- (1, 1, -4) -- (0, 1, -5) -- (0, 0, -6) -- (1, -1, -6) -- (2, -1, -5) -- cycle {};
\fill[facetPerm] (2, -4, -2) -- (1, -5, -2) -- (1, -5, -4) -- (2, -4, -4) -- cycle {};
\fill[facetPerm] (2, -1, -5) -- (1, -1, -6) -- (1, -3, -6) -- (2, -3, -5) -- cycle {};
\fill[facetPerm] (1, -1, 0) -- (0, 0, 0) -- (-2, 0, 0) -- (-3, -1, 0) -- (-3, -3, 0) -- (-2, -4, 0) -- (0, -4, 0) -- (1, -3, 0) -- cycle {};

%% Drawing permutahedron edges in the front
\draw[edgePerm] (-3, -5, -4) -- (-3, -5, -2);
\draw[edgePerm] (-3, -5, -4) -- (-2, -5, -5);
\draw[edgePerm] (-3, -5, -2) -- (-2, -5, -1);
\draw[edgePerm] (-3, -3, 0) -- (-3, -1, 0);
\draw[edgePerm] (-3, -3, 0) -- (-2, -4, 0);
\draw[edgePerm] (-3, -1, 0) -- (-2, 0, 0);
\draw[edgePerm] (-2, -5, -5) -- (-2, -4, -6);
\draw[edgePerm] (-2, -5, -5) -- (0, -5, -5);
\draw[edgePerm] (-2, -5, -1) -- (-2, -4, 0);
\draw[edgePerm] (-2, -5, -1) -- (0, -5, -1);
\draw[edgePerm] (-2, -4, -6) -- (0, -4, -6);
\draw[edgePerm] (-2, -4, 0) -- (0, -4, 0);
\draw[edgePerm] (-2, 0, 0) -- (0, 0, 0);
\draw[edgePerm] (0, -5, -5) -- (0, -4, -6);
\draw[edgePerm] (0, -5, -5) -- (1, -5, -4);
\draw[edgePerm] (0, -5, -1) -- (0, -4, 0);
\draw[edgePerm] (0, -5, -1) -- (1, -5, -2);
\draw[edgePerm] (0, -4, -6) -- (1, -3, -6);
\draw[edgePerm] (0, -4, 0) -- (1, -3, 0);
\draw[edgePerm] (0, 0, -6) -- (0, 1, -5);
\draw[edgePerm] (0, 0, -6) -- (1, -1, -6);
\draw[edgePerm] (0, 0, 0) -- (0, 1, -1);
\draw[edgePerm] (0, 0, 0) -- (1, -1, 0);
\draw[edgePerm] (0, 1, -5) -- (1, 1, -4);
\draw[edgePerm] (0, 1, -1) -- (1, 1, -2);
\draw[edgePerm] (1, -5, -4) -- (1, -5, -2);
\draw[edgePerm] (1, -5, -4) -- (2, -4, -4);
\draw[edgePerm] (1, -5, -2) -- (2, -4, -2);
\draw[edgePerm] (1, -3, -6) -- (1, -1, -6);
\draw[edgePerm] (1, -3, -6) -- (2, -3, -5);
\draw[edgePerm] (1, -3, 0) -- (1, -1, 0);
\draw[edgePerm] (1, -3, 0) -- (2, -3, -1);
\draw[edgePerm] (1, -1, -6) -- (2, -1, -5);
\draw[edgePerm] (1, -1, 0) -- (2, -1, -1);
\draw[edgePerm] (1, 1, -4) -- (1, 1, -2);
\draw[edgePerm] (1, 1, -4) -- (2, 0, -4);
\draw[edgePerm] (1, 1, -2) -- (2, 0, -2);
\draw[edgePerm] (2, -4, -4) -- (2, -4, -2);
\draw[edgePerm] (2, -4, -4) -- (2, -3, -5);
\draw[edgePerm] (2, -4, -2) -- (2, -3, -1);
\draw[edgePerm] (2, -3, -5) -- (2, -1, -5);
\draw[edgePerm] (2, -3, -1) -- (2, -1, -1);
\draw[edgePerm] (2, -1, -5) -- (2, 0, -4);
\draw[edgePerm] (2, -1, -1) -- (2, 0, -2);
\draw[edgePerm] (2, 0, -4) -- (2, 0, -2);

%% Drawing permutahedron vertices in the front
\node[vertexPerm] at (-3, -5, -4)     {};
\node[vertexPerm] at (-3, -5, -2)     {};
\node[vertexPerm] at (-3, -3, 0)     {};
\node[vertexPerm] at (-3, -1, 0)     {};
\node[vertexPerm] at (-2, -5, -5)     {};
\node[vertexPerm] at (-2, -5, -1)     {};
\node[vertexPerm] at (-2, -4, -6)     {};
\node[vertexPerm] at (-2, -4, 0)     {};
\node[vertexPerm] at (-2, 0, 0)     {};
\node[vertexPerm] at (0, -5, -5)     {};
\node[vertexPerm] at (0, -5, -1)     {};
\node[vertexPerm] at (0, -4, -6)     {};
\node[vertexPerm] at (0, -4, 0)     {};
\node[vertexPerm] at (0, 0, -6)     {};
\node[vertexPerm] at (0, 0, 0)     {};
\node[vertexPerm] at (0, 1, -5)     {};
\node[vertexPerm] at (0, 1, -1)     {};
\node[vertexPerm] at (1, -5, -4)     {};
\node[vertexPerm] at (1, -5, -2)     {};
\node[vertexPerm] at (1, -3, -6)     {};
\node[vertexPerm] at (1, -3, 0)     {};
\node[vertexPerm] at (1, -1, -6)     {};
\node[vertexPerm] at (1, -1, 0)     {};
\node[vertexPerm] at (1, 1, -4)     {};
\node[vertexPerm] at (1, 1, -2)     {};
\node[vertexPerm] at (2, -4, -4)     {};
\node[vertexPerm] at (2, -4, -2)     {};
\node[vertexPerm] at (2, -3, -5)     {};
\node[vertexPerm] at (2, -3, -1)     {};
\node[vertexPerm] at (2, -1, -5)     {};
\node[vertexPerm] at (2, -1, -1)     {};
\node[vertexPerm] at (2, 0, -4)     {};
\node[vertexPerm] at (2, 0, -2)     {};

%% Drawing associahedron facets
\fill[facetAsso] (0, 0, 0) -- (10, -10, 0) -- (10, -4, -6) -- (5, 1, -6) -- (0, 1, -1) -- cycle {};
\fill[facetAsso] (-2, -10, 0) -- (10, -10, 0) -- (10, -4, -6) -- (-2, -4, -6) -- cycle {};
\fill[facetAsso] (0, 0, 0) -- (10, -10, 0) -- (-2, -10, 0) -- (-4, -8, 0) -- (-4, -2, 0) -- (-2, 0, 0) -- cycle {};

%% Drawing associahedron edges in the front
\draw[edgeAsso] (-4, -8, 0) -- (-4, -2, 0);
\draw[edgeAsso] (-4, -8, 0) -- (-2, -10, 0);
\draw[edgeAsso] (10, -4, -6) -- (10, -10, 0);
\draw[edgeAsso] (10, -4, -6) -- (5, 1, -6);
\draw[edgeAsso] (10, -4, -6) -- (-2, -4, -6);
\draw[edgeAsso] (10, -10, 0) -- (0, 0, 0);
\draw[edgeAsso] (10, -10, 0) -- (-2, -10, 0);
\draw[edgeAsso] (5, 1, -6) -- (0, 1, -1);
\draw[edgeAsso] (0, 1, -1) -- (0, 0, 0);
\draw[edgeAsso] (0, 0, 0) -- (-2, 0, 0);
\draw[edgeAsso] (-4, -2, 0) -- (-2, 0, 0);
\draw[edgeAsso] (-2, -4, -6) -- (-2, -10, 0);

%% Drawing associahedron vertices in the front
\node[vertexAsso] at (-4, -8, 0)     {};
\node[vertexAsso] at (10, -4, -6)     {};
\node[vertexAsso] at (10, -10, 0)     {};
\node[vertexAsso] at (5, 1, -6)     {};
\node[vertexAsso] at (0, 1, -1)     {};
\node[vertexAsso] at (0, 0, 0)     {};
\node[vertexAsso] at (-4, -2, 0)     {};
\node[vertexAsso] at (-2, 0, 0)     {};
\node[vertexAsso] at (-2, -4, -6)     {};
\node[vertexAsso] at (-2, -10, 0)     {};

\end{tikzpicture}

%% file: figures/B3asso2.tex
% sage code
% sage: cong = B_Shards(3, B_shards = repr_to_B_shards(['0d0...', '.0d0..', '.0du0.']), restriction='complement_generators')
% sage: Img = cong.quotientope_from_Minkowski_sums().projection().tikz([-723,-409,-556], 94.3, scale=.3)
% sage: f=open('B3Asso1.tex','w'); f.write(Img); f.close()
% then add the picuture of the permutahedron in the middle.

\begin{tikzpicture}%
	[x={(0.487424cm, -0.236521cm)},
	y={(0.872822cm, 0.104998cm)},
	z={(0.024475cm, 0.965936cm)},
	scale=0.3,
	back/.style={very thin, opacity=0.5},
	edgePerm/.style={color=red!95!black, very thick, cap=round},
	edgeAsso/.style={color=blue!95!black, very thick, cap=round},
	facetPerm/.style={fill=red!95!black,fill opacity=0},
	facetAsso/.style={fill=blue!95!black,fill opacity=0},
	vertexPerm/.style={inner sep=0pt, circle, anchor=base},
	vertexAsso/.style={inner sep=0pt, circle, anchor=base}]
	
%% Coordinate of permutahedron vertices:
\coordinate (-4, -4, -4) at (-4, -4, -4);
\coordinate (-4, -4, -2) at (-4, -4, -2);
\coordinate (-4, -3, -5) at (-4, -3, -5);
\coordinate (-4, -3, -1) at (-4, -3, -1);
\coordinate (-4, -1, -5) at (-4, -1, -5);
\coordinate (-4, -1, -1) at (-4, -1, -1);
\coordinate (-4, 0, -4) at (-4, 0, -4);
\coordinate (-4, 0, -2) at (-4, 0, -2);
\coordinate (-3, -5, -4) at (-3, -5, -4);
\coordinate (-3, -5, -2) at (-3, -5, -2);
\coordinate (-3, -3, -6) at (-3, -3, -6);
\coordinate (-3, -3, 0) at (-3, -3, 0);
\coordinate (-3, -1, -6) at (-3, -1, -6);
\coordinate (-3, -1, 0) at (-3, -1, 0);
\coordinate (-3, 1, -4) at (-3, 1, -4);
\coordinate (-3, 1, -2) at (-3, 1, -2);
\coordinate (-2, -5, -5) at (-2, -5, -5);
\coordinate (-2, -5, -1) at (-2, -5, -1);
\coordinate (-2, -4, -6) at (-2, -4, -6);
\coordinate (-2, -4, 0) at (-2, -4, 0);
\coordinate (-2, 0, -6) at (-2, 0, -6);
\coordinate (-2, 0, 0) at (-2, 0, 0);
\coordinate (-2, 1, -5) at (-2, 1, -5);
\coordinate (-2, 1, -1) at (-2, 1, -1);
\coordinate (0, -5, -5) at (0, -5, -5);
\coordinate (0, -5, -1) at (0, -5, -1);
\coordinate (0, -4, -6) at (0, -4, -6);
\coordinate (0, -4, 0) at (0, -4, 0);
\coordinate (0, 0, -6) at (0, 0, -6);
\coordinate (0, 0, 0) at (0, 0, 0);
\coordinate (0, 1, -5) at (0, 1, -5);
\coordinate (0, 1, -1) at (0, 1, -1);
\coordinate (1, -5, -4) at (1, -5, -4);
\coordinate (1, -5, -2) at (1, -5, -2);
\coordinate (1, -3, -6) at (1, -3, -6);
\coordinate (1, -3, 0) at (1, -3, 0);
\coordinate (1, -1, -6) at (1, -1, -6);
\coordinate (1, -1, 0) at (1, -1, 0);
\coordinate (1, 1, -4) at (1, 1, -4);
\coordinate (1, 1, -2) at (1, 1, -2);
\coordinate (2, -4, -4) at (2, -4, -4);
\coordinate (2, -4, -2) at (2, -4, -2);
\coordinate (2, -3, -5) at (2, -3, -5);
\coordinate (2, -3, -1) at (2, -3, -1);
\coordinate (2, -1, -5) at (2, -1, -5);
\coordinate (2, -1, -1) at (2, -1, -1);
\coordinate (2, 0, -4) at (2, 0, -4);
\coordinate (2, 0, -2) at (2, 0, -2);

%% Coordinate of associahedron vertices:
\coordinate (-12, 0, 0) at (-12, 0, 0);
\coordinate (-12, 1, -1) at (-12, 1, -1);
\coordinate (2, 0, -2) at (2, 0, -2);
\coordinate (2, 0, -4) at (2, 0, -4);
\coordinate (2, -1, -5) at (2, -1, -5);
\coordinate (2, -2, 0) at (2, -2, 0);
\coordinate (2, -3, -5) at (2, -3, -5);
\coordinate (2, -8, 0) at (2, -8, 0);
\coordinate (1, 1, -2) at (1, 1, -2);
\coordinate (1, 1, -4) at (1, 1, -4);
\coordinate (-7, 1, -6) at (-7, 1, -6);
\coordinate (1, -1, -6) at (1, -1, -6);
\coordinate (1, -3, -6) at (1, -3, -6);
\coordinate (0, 1, -1) at (0, 1, -1);
\coordinate (0, 0, 0) at (0, 0, 0);
\coordinate (-2, -10, 0) at (-2, -10, 0);
\coordinate (0, -4, -6) at (0, -4, -6);
\coordinate (0, -10, 0) at (0, -10, 0);
\coordinate (-1, 1, -6) at (-1, 1, -6);
\coordinate (-2, -4, -6) at (-2, -4, -6);

%% Drawing associahedron edges in the back
\draw[edgeAsso,back] (-12, 0, 0) -- (-12, 1, -1);
\draw[edgeAsso,back] (-12, 1, -1) -- (-7, 1, -6);
\draw[edgeAsso,back] (-12, 1, -1) -- (0, 1, -1);
\draw[edgeAsso,back] (-7, 1, -6) -- (-1, 1, -6);
\draw[edgeAsso,back] (-7, 1, -6) -- (-2, -4, -6);

%% Drawing associahedron vertices in the back
\node[vertexAsso,back] at (-12, 1, -1)     {};
\node[vertexAsso,back] at (-7, 1, -6)     {};

%% Drawing permutahedron edges in the back
\draw[edgePerm,back] (-4, -4, -4) -- (-4, -4, -2);
\draw[edgePerm,back] (-4, -4, -4) -- (-4, -3, -5);
\draw[edgePerm,back] (-4, -4, -4) -- (-3, -5, -4);
\draw[edgePerm,back] (-4, -4, -2) -- (-4, -3, -1);
\draw[edgePerm,back] (-4, -4, -2) -- (-3, -5, -2);
\draw[edgePerm,back] (-4, -3, -5) -- (-4, -1, -5);
\draw[edgePerm,back] (-4, -3, -5) -- (-3, -3, -6);
\draw[edgePerm,back] (-4, -3, -1) -- (-4, -1, -1);
\draw[edgePerm,back] (-4, -3, -1) -- (-3, -3, 0);
\draw[edgePerm,back] (-4, -1, -5) -- (-4, 0, -4);
\draw[edgePerm,back] (-4, -1, -5) -- (-3, -1, -6);
\draw[edgePerm,back] (-4, -1, -1) -- (-4, 0, -2);
\draw[edgePerm,back] (-4, -1, -1) -- (-3, -1, 0);
\draw[edgePerm,back] (-4, 0, -4) -- (-4, 0, -2);
\draw[edgePerm,back] (-4, 0, -4) -- (-3, 1, -4);
\draw[edgePerm,back] (-4, 0, -2) -- (-3, 1, -2);
\draw[edgePerm,back] (-3, -3, -6) -- (-3, -1, -6);
\draw[edgePerm,back] (-3, -3, -6) -- (-2, -4, -6);
\draw[edgePerm,back] (-3, -1, -6) -- (-2, 0, -6);
\draw[edgePerm,back] (-3, 1, -4) -- (-3, 1, -2);
\draw[edgePerm,back] (-3, 1, -4) -- (-2, 1, -5);
\draw[edgePerm,back] (-3, 1, -2) -- (-2, 1, -1);
\draw[edgePerm,back] (-2, 0, -6) -- (-2, 1, -5);
\draw[edgePerm,back] (-2, 0, -6) -- (0, 0, -6);
\draw[edgePerm,back] (-2, 0, 0) -- (-2, 1, -1);
\draw[edgePerm,back] (-2, 1, -5) -- (0, 1, -5);
\draw[edgePerm,back] (-2, 1, -1) -- (0, 1, -1);

%% Drawing permutahedron vertices in the back
\node[vertexPerm,back] at (-4, -4, -2)     {};
\node[vertexPerm,back] at (-4, -3, -1)     {};
\node[vertexPerm,back] at (-4, -3, -5)     {};
\node[vertexPerm,back] at (-4, -1, -5)     {};
\node[vertexPerm,back] at (-3, -3, -6)     {};
\node[vertexPerm,back] at (-3, -1, -6)     {};
\node[vertexPerm,back] at (-4, -4, -4)     {};
\node[vertexPerm,back] at (-4, -1, -1)     {};
\node[vertexPerm,back] at (-4, 0, -4)     {};
\node[vertexPerm,back] at (-3, 1, -4)     {};
\node[vertexPerm,back] at (-2, 0, -6)     {};
\node[vertexPerm,back] at (-2, 1, -5)     {};
\node[vertexPerm,back] at (-4, 0, -2)     {};
\node[vertexPerm,back] at (-3, 1, -2)     {};
\node[vertexPerm,back] at (-2, 1, -1)     {};

%% Drawing permutahedron facets
\fill[facetPerm] (2, 0, -2) -- (2, -1, -1) -- (2, -3, -1) -- (2, -4, -2) -- (2, -4, -4) -- (2, -3, -5) -- (2, -1, -5) -- (2, 0, -4) -- cycle {};
\fill[facetPerm] (2, 0, -2) -- (1, 1, -2) -- (0, 1, -1) -- (0, 0, 0) -- (1, -1, 0) -- (2, -1, -1) -- cycle {};
\fill[facetPerm] (2, 0, -2) -- (1, 1, -2) -- (1, 1, -4) -- (2, 0, -4) -- cycle {};
\fill[facetPerm] (2, -1, -1) -- (1, -1, 0) -- (1, -3, 0) -- (2, -3, -1) -- cycle {};
\fill[facetPerm] (0, -4, -6) -- (-2, -4, -6) -- (-2, -5, -5) -- (0, -5, -5) -- cycle {};
\fill[facetPerm] (0, -4, 0) -- (-2, -4, 0) -- (-2, -5, -1) -- (0, -5, -1) -- cycle {};
\fill[facetPerm] (2, -3, -5) -- (1, -3, -6) -- (0, -4, -6) -- (0, -5, -5) -- (1, -5, -4) -- (2, -4, -4) -- cycle {};
\fill[facetPerm] (2, -3, -1) -- (1, -3, 0) -- (0, -4, 0) -- (0, -5, -1) -- (1, -5, -2) -- (2, -4, -2) -- cycle {};
\fill[facetPerm] (1, -5, -2) -- (0, -5, -1) -- (-2, -5, -1) -- (-3, -5, -2) -- (-3, -5, -4) -- (-2, -5, -5) -- (0, -5, -5) -- (1, -5, -4) -- cycle {};
\fill[facetPerm] (2, 0, -4) -- (1, 1, -4) -- (0, 1, -5) -- (0, 0, -6) -- (1, -1, -6) -- (2, -1, -5) -- cycle {};
\fill[facetPerm] (2, -4, -2) -- (1, -5, -2) -- (1, -5, -4) -- (2, -4, -4) -- cycle {};
\fill[facetPerm] (2, -1, -5) -- (1, -1, -6) -- (1, -3, -6) -- (2, -3, -5) -- cycle {};
\fill[facetPerm] (1, -1, 0) -- (0, 0, 0) -- (-2, 0, 0) -- (-3, -1, 0) -- (-3, -3, 0) -- (-2, -4, 0) -- (0, -4, 0) -- (1, -3, 0) -- cycle {};

%% Drawing permutahedron edges in the front
\draw[edgePerm] (-3, -5, -4) -- (-3, -5, -2);
\draw[edgePerm] (-3, -5, -4) -- (-2, -5, -5);
\draw[edgePerm] (-3, -5, -2) -- (-2, -5, -1);
\draw[edgePerm] (-3, -3, 0) -- (-3, -1, 0);
\draw[edgePerm] (-3, -3, 0) -- (-2, -4, 0);
\draw[edgePerm] (-3, -1, 0) -- (-2, 0, 0);
\draw[edgePerm] (-2, -5, -5) -- (-2, -4, -6);
\draw[edgePerm] (-2, -5, -5) -- (0, -5, -5);
\draw[edgePerm] (-2, -5, -1) -- (-2, -4, 0);
\draw[edgePerm] (-2, -5, -1) -- (0, -5, -1);
\draw[edgePerm] (-2, -4, -6) -- (0, -4, -6);
\draw[edgePerm] (-2, -4, 0) -- (0, -4, 0);
\draw[edgePerm] (-2, 0, 0) -- (0, 0, 0);
\draw[edgePerm] (0, -5, -5) -- (0, -4, -6);
\draw[edgePerm] (0, -5, -5) -- (1, -5, -4);
\draw[edgePerm] (0, -5, -1) -- (0, -4, 0);
\draw[edgePerm] (0, -5, -1) -- (1, -5, -2);
\draw[edgePerm] (0, -4, -6) -- (1, -3, -6);
\draw[edgePerm] (0, -4, 0) -- (1, -3, 0);
\draw[edgePerm] (0, 0, -6) -- (0, 1, -5);
\draw[edgePerm] (0, 0, -6) -- (1, -1, -6);
\draw[edgePerm] (0, 0, 0) -- (0, 1, -1);
\draw[edgePerm] (0, 0, 0) -- (1, -1, 0);
\draw[edgePerm] (0, 1, -5) -- (1, 1, -4);
\draw[edgePerm] (0, 1, -1) -- (1, 1, -2);
\draw[edgePerm] (1, -5, -4) -- (1, -5, -2);
\draw[edgePerm] (1, -5, -4) -- (2, -4, -4);
\draw[edgePerm] (1, -5, -2) -- (2, -4, -2);
\draw[edgePerm] (1, -3, -6) -- (1, -1, -6);
\draw[edgePerm] (1, -3, -6) -- (2, -3, -5);
\draw[edgePerm] (1, -3, 0) -- (1, -1, 0);
\draw[edgePerm] (1, -3, 0) -- (2, -3, -1);
\draw[edgePerm] (1, -1, -6) -- (2, -1, -5);
\draw[edgePerm] (1, -1, 0) -- (2, -1, -1);
\draw[edgePerm] (1, 1, -4) -- (1, 1, -2);
\draw[edgePerm] (1, 1, -4) -- (2, 0, -4);
\draw[edgePerm] (1, 1, -2) -- (2, 0, -2);
\draw[edgePerm] (2, -4, -4) -- (2, -4, -2);
\draw[edgePerm] (2, -4, -4) -- (2, -3, -5);
\draw[edgePerm] (2, -4, -2) -- (2, -3, -1);
\draw[edgePerm] (2, -3, -5) -- (2, -1, -5);
\draw[edgePerm] (2, -3, -1) -- (2, -1, -1);
\draw[edgePerm] (2, -1, -5) -- (2, 0, -4);
\draw[edgePerm] (2, -1, -1) -- (2, 0, -2);
\draw[edgePerm] (2, 0, -4) -- (2, 0, -2);

%% Drawing permutahedron vertices in the front
\node[vertexPerm] at (-3, -5, -4)     {};
\node[vertexPerm] at (-3, -5, -2)     {};
\node[vertexPerm] at (-3, -3, 0)     {};
\node[vertexPerm] at (-3, -1, 0)     {};
\node[vertexPerm] at (-2, -5, -5)     {};
\node[vertexPerm] at (-2, -5, -1)     {};
\node[vertexPerm] at (-2, -4, -6)     {};
\node[vertexPerm] at (-2, -4, 0)     {};
\node[vertexPerm] at (-2, 0, 0)     {};
\node[vertexPerm] at (0, -5, -5)     {};
\node[vertexPerm] at (0, -5, -1)     {};
\node[vertexPerm] at (0, -4, -6)     {};
\node[vertexPerm] at (0, -4, 0)     {};
\node[vertexPerm] at (0, 0, -6)     {};
\node[vertexPerm] at (0, 0, 0)     {};
\node[vertexPerm] at (0, 1, -5)     {};
\node[vertexPerm] at (0, 1, -1)     {};
\node[vertexPerm] at (1, -5, -4)     {};
\node[vertexPerm] at (1, -5, -2)     {};
\node[vertexPerm] at (1, -3, -6)     {};
\node[vertexPerm] at (1, -3, 0)     {};
\node[vertexPerm] at (1, -1, -6)     {};
\node[vertexPerm] at (1, -1, 0)     {};
\node[vertexPerm] at (1, 1, -4)     {};
\node[vertexPerm] at (1, 1, -2)     {};
\node[vertexPerm] at (2, -4, -4)     {};
\node[vertexPerm] at (2, -4, -2)     {};
\node[vertexPerm] at (2, -3, -5)     {};
\node[vertexPerm] at (2, -3, -1)     {};
\node[vertexPerm] at (2, -1, -5)     {};
\node[vertexPerm] at (2, -1, -1)     {};
\node[vertexPerm] at (2, 0, -4)     {};
\node[vertexPerm] at (2, 0, -2)     {};

%% Drawing associahedron facets
\fill[facetAsso] (2, -8, 0) -- (2, -2, 0) -- (2, 0, -2) -- (2, 0, -4) -- (2, -1, -5) -- (2, -3, -5) -- cycle {};
\fill[facetAsso] (0, 0, 0) -- (2, -2, 0) -- (2, 0, -2) -- (1, 1, -2) -- (0, 1, -1) -- cycle {};
\fill[facetAsso] (1, 1, -4) -- (2, 0, -4) -- (2, 0, -2) -- (1, 1, -2) -- cycle {};
\fill[facetAsso] (0, -4, -6) -- (-2, -4, -6) -- (-2, -10, 0) -- (0, -10, 0) -- cycle {};
\fill[facetAsso] (-1, 1, -6) -- (1, 1, -4) -- (2, 0, -4) -- (2, -1, -5) -- (1, -1, -6) -- cycle {};
\fill[facetAsso] (0, -10, 0) -- (2, -8, 0) -- (2, -2, 0) -- (0, 0, 0) -- (-12, 0, 0) -- (-2, -10, 0) -- cycle {};
\fill[facetAsso] (0, -10, 0) -- (2, -8, 0) -- (2, -3, -5) -- (1, -3, -6) -- (0, -4, -6) -- cycle {};
\fill[facetAsso] (1, -3, -6) -- (2, -3, -5) -- (2, -1, -5) -- (1, -1, -6) -- cycle {};

%% Drawing associahedron edges in the front
\draw[edgeAsso] (-12, 0, 0) -- (0, 0, 0);
\draw[edgeAsso] (-12, 0, 0) -- (-2, -10, 0);
\draw[edgeAsso] (2, 0, -2) -- (2, 0, -4);
\draw[edgeAsso] (2, 0, -2) -- (2, -2, 0);
\draw[edgeAsso] (2, 0, -2) -- (1, 1, -2);
\draw[edgeAsso] (2, 0, -4) -- (2, -1, -5);
\draw[edgeAsso] (2, 0, -4) -- (1, 1, -4);
\draw[edgeAsso] (2, -1, -5) -- (2, -3, -5);
\draw[edgeAsso] (2, -1, -5) -- (1, -1, -6);
\draw[edgeAsso] (2, -2, 0) -- (2, -8, 0);
\draw[edgeAsso] (2, -2, 0) -- (0, 0, 0);
\draw[edgeAsso] (2, -3, -5) -- (2, -8, 0);
\draw[edgeAsso] (2, -3, -5) -- (1, -3, -6);
\draw[edgeAsso] (2, -8, 0) -- (0, -10, 0);
\draw[edgeAsso] (1, 1, -2) -- (1, 1, -4);
\draw[edgeAsso] (1, 1, -2) -- (0, 1, -1);
\draw[edgeAsso] (1, 1, -4) -- (-1, 1, -6);
\draw[edgeAsso] (1, -1, -6) -- (1, -3, -6);
\draw[edgeAsso] (1, -1, -6) -- (-1, 1, -6);
\draw[edgeAsso] (1, -3, -6) -- (0, -4, -6);
\draw[edgeAsso] (0, 1, -1) -- (0, 0, 0);
\draw[edgeAsso] (-2, -10, 0) -- (0, -10, 0);
\draw[edgeAsso] (-2, -10, 0) -- (-2, -4, -6);
\draw[edgeAsso] (0, -4, -6) -- (0, -10, 0);
\draw[edgeAsso] (0, -4, -6) -- (-2, -4, -6);

%% Drawing associahedron vertices in the front
\node[vertexAsso] at (-12, 0, 0)     {};
\node[vertexAsso] at (2, 0, -2)     {};
\node[vertexAsso] at (2, 0, -4)     {};
\node[vertexAsso] at (2, -1, -5)     {};
\node[vertexAsso] at (2, -2, 0)     {};
\node[vertexAsso] at (2, -3, -5)     {};
\node[vertexAsso] at (2, -8, 0)     {};
\node[vertexAsso] at (1, 1, -2)     {};
\node[vertexAsso] at (1, 1, -4)     {};
\node[vertexAsso] at (1, -1, -6)     {};
\node[vertexAsso] at (1, -3, -6)     {};
\node[vertexAsso] at (0, 1, -1)     {};
\node[vertexAsso] at (0, 0, 0)     {};
\node[vertexAsso] at (-2, -10, 0)     {};
\node[vertexAsso] at (0, -4, -6)     {};
\node[vertexAsso] at (0, -10, 0)     {};
\node[vertexAsso] at (-1, 1, -6)     {};
\node[vertexAsso] at (-2, -4, -6)     {};

\end{tikzpicture}

%% file: figures/B3asso3.tex
% sage code
% sage: cong = B_Shards(3, B_shards = repr_to_B_shards(['0u0...', '.0u0..', '.0ud0.']), restriction='complement_generators')
% sage: Img = cong.quotientope_from_Minkowski_sums().projection().tikz([-723,-409,-556], 94.3, scale=.3)
% sage: f=open('B3Asso1.tex','w'); f.write(Img); f.close()
% then add the picuture of the permutahedron in the middle.

\begin{tikzpicture}%
	[x={(0.487424cm, -0.236521cm)},
	y={(0.872822cm, 0.104998cm)},
	z={(0.024475cm, 0.965936cm)},
	scale=0.3,
	back/.style={very thin, opacity=0.5},
	edgePerm/.style={color=red!95!black, very thick, cap=round},
	edgeAsso/.style={color=blue!95!black, very thick, cap=round},
	facetPerm/.style={fill=red!95!black,fill opacity=0},
	facetAsso/.style={fill=blue!95!black,fill opacity=0},
	vertexPerm/.style={inner sep=0pt, circle, anchor=base},
	vertexAsso/.style={inner sep=0pt, circle, anchor=base}]

%% Coordinate of permutahedron vertices:
\coordinate (-4, -4, -4) at (-4, -4, -4);
\coordinate (-4, -4, -2) at (-4, -4, -2);
\coordinate (-4, -3, -5) at (-4, -3, -5);
\coordinate (-4, -3, -1) at (-4, -3, -1);
\coordinate (-4, -1, -5) at (-4, -1, -5);
\coordinate (-4, -1, -1) at (-4, -1, -1);
\coordinate (-4, 0, -4) at (-4, 0, -4);
\coordinate (-4, 0, -2) at (-4, 0, -2);
\coordinate (-3, -5, -4) at (-3, -5, -4);
\coordinate (-3, -5, -2) at (-3, -5, -2);
\coordinate (-3, -3, -6) at (-3, -3, -6);
\coordinate (-3, -3, 0) at (-3, -3, 0);
\coordinate (-3, -1, -6) at (-3, -1, -6);
\coordinate (-3, -1, 0) at (-3, -1, 0);
\coordinate (-3, 1, -4) at (-3, 1, -4);
\coordinate (-3, 1, -2) at (-3, 1, -2);
\coordinate (-2, -5, -5) at (-2, -5, -5);
\coordinate (-2, -5, -1) at (-2, -5, -1);
\coordinate (-2, -4, -6) at (-2, -4, -6);
\coordinate (-2, -4, 0) at (-2, -4, 0);
\coordinate (-2, 0, -6) at (-2, 0, -6);
\coordinate (-2, 0, 0) at (-2, 0, 0);
\coordinate (-2, 1, -5) at (-2, 1, -5);
\coordinate (-2, 1, -1) at (-2, 1, -1);
\coordinate (0, -5, -5) at (0, -5, -5);
\coordinate (0, -5, -1) at (0, -5, -1);
\coordinate (0, -4, -6) at (0, -4, -6);
\coordinate (0, -4, 0) at (0, -4, 0);
\coordinate (0, 0, -6) at (0, 0, -6);
\coordinate (0, 0, 0) at (0, 0, 0);
\coordinate (0, 1, -5) at (0, 1, -5);
\coordinate (0, 1, -1) at (0, 1, -1);
\coordinate (1, -5, -4) at (1, -5, -4);
\coordinate (1, -5, -2) at (1, -5, -2);
\coordinate (1, -3, -6) at (1, -3, -6);
\coordinate (1, -3, 0) at (1, -3, 0);
\coordinate (1, -1, -6) at (1, -1, -6);
\coordinate (1, -1, 0) at (1, -1, 0);
\coordinate (1, 1, -4) at (1, 1, -4);
\coordinate (1, 1, -2) at (1, 1, -2);
\coordinate (2, -4, -4) at (2, -4, -4);
\coordinate (2, -4, -2) at (2, -4, -2);
\coordinate (2, -3, -5) at (2, -3, -5);
\coordinate (2, -3, -1) at (2, -3, -1);
\coordinate (2, -1, -5) at (2, -1, -5);
\coordinate (2, -1, -1) at (2, -1, -1);
\coordinate (2, 0, -4) at (2, 0, -4);
\coordinate (2, 0, -2) at (2, 0, -2);

%% Coordinate of associahedron vertices:
\coordinate (-4, -4, -4) at (-4, -4, -4);
\coordinate (-4, -4, -2) at (-4, -4, -2);
\coordinate (10, -4, -6) at (10, -4, -6);
\coordinate (-4, -3, -1) at (-4, -3, -1);
\coordinate (-4, -2, -6) at (-4, -2, -6);
\coordinate (10, -5, -5) at (10, -5, -5);
\coordinate (5, -5, 0) at (5, -5, 0);
\coordinate (0, 6, -6) at (0, 6, -6);
\coordinate (-4, -1, -1) at (-4, -1, -1);
\coordinate (0, 0, 0) at (0, 0, 0);
\coordinate (-1, -5, 0) at (-1, -5, 0);
\coordinate (-2, 6, -6) at (-2, 6, -6);
\coordinate (-2, 0, 0) at (-2, 0, 0);
\coordinate (-4, 4, -6) at (-4, 4, -6);
\coordinate (-3, -5, -4) at (-3, -5, -4);
\coordinate (-3, -5, -2) at (-3, -5, -2);
\coordinate (-2, -4, -6) at (-2, -4, -6);
\coordinate (-2, -5, -5) at (-2, -5, -5);
\coordinate (-3, -1, 0) at (-3, -1, 0);
\coordinate (-3, -3, 0) at (-3, -3, 0);

%% Drawing associahedron edges in the back
\draw[edgeAsso,back] (-4, -4, -4) -- (-4, -4, -2);
\draw[edgeAsso,back] (-4, -4, -4) -- (-4, -2, -6);
\draw[edgeAsso,back] (-4, -4, -4) -- (-3, -5, -4);
\draw[edgeAsso,back] (-4, -4, -2) -- (-4, -3, -1);
\draw[edgeAsso,back] (-4, -4, -2) -- (-3, -5, -2);
\draw[edgeAsso,back] (-4, -3, -1) -- (-4, -1, -1);
\draw[edgeAsso,back] (-4, -3, -1) -- (-3, -3, 0);
\draw[edgeAsso,back] (-4, -2, -6) -- (-4, 4, -6);
\draw[edgeAsso,back] (-4, -2, -6) -- (-2, -4, -6);
\draw[edgeAsso,back] (0, 6, -6) -- (-2, 6, -6);
\draw[edgeAsso,back] (-4, -1, -1) -- (-4, 4, -6);
\draw[edgeAsso,back] (-4, -1, -1) -- (-3, -1, 0);
\draw[edgeAsso,back] (-2, 6, -6) -- (-2, 0, 0);
\draw[edgeAsso,back] (-2, 6, -6) -- (-4, 4, -6);

%% Drawing associahedron vertices in the back
\node[vertexAsso,back] at (-4, -2, -6)     {};
\node[vertexAsso,back] at (-2, 6, -6)     {};
\node[vertexAsso,back] at (-4, 4, -6)     {};
\node[vertexAsso,back] at (-4, -4, -4)     {};
\node[vertexAsso,back] at (-4, -1, -1)     {};
\node[vertexAsso,back] at (-4, -4, -2)     {};
\node[vertexAsso,back] at (-4, -3, -1)     {};

%% Drawing permutahedron edges in the back
\draw[edgePerm,back] (-4, -4, -4) -- (-4, -4, -2);
\draw[edgePerm,back] (-4, -4, -4) -- (-4, -3, -5);
\draw[edgePerm,back] (-4, -4, -4) -- (-3, -5, -4);
\draw[edgePerm,back] (-4, -4, -2) -- (-4, -3, -1);
\draw[edgePerm,back] (-4, -4, -2) -- (-3, -5, -2);
\draw[edgePerm,back] (-4, -3, -5) -- (-4, -1, -5);
\draw[edgePerm,back] (-4, -3, -5) -- (-3, -3, -6);
\draw[edgePerm,back] (-4, -3, -1) -- (-4, -1, -1);
\draw[edgePerm,back] (-4, -3, -1) -- (-3, -3, 0);
\draw[edgePerm,back] (-4, -1, -5) -- (-4, 0, -4);
\draw[edgePerm,back] (-4, -1, -5) -- (-3, -1, -6);
\draw[edgePerm,back] (-4, -1, -1) -- (-4, 0, -2);
\draw[edgePerm,back] (-4, -1, -1) -- (-3, -1, 0);
\draw[edgePerm,back] (-4, 0, -4) -- (-4, 0, -2);
\draw[edgePerm,back] (-4, 0, -4) -- (-3, 1, -4);
\draw[edgePerm,back] (-4, 0, -2) -- (-3, 1, -2);
\draw[edgePerm,back] (-3, -3, -6) -- (-3, -1, -6);
\draw[edgePerm,back] (-3, -3, -6) -- (-2, -4, -6);
\draw[edgePerm,back] (-3, -1, -6) -- (-2, 0, -6);
\draw[edgePerm,back] (-3, 1, -4) -- (-3, 1, -2);
\draw[edgePerm,back] (-3, 1, -4) -- (-2, 1, -5);
\draw[edgePerm,back] (-3, 1, -2) -- (-2, 1, -1);
\draw[edgePerm,back] (-2, 0, -6) -- (-2, 1, -5);
\draw[edgePerm,back] (-2, 0, -6) -- (0, 0, -6);
\draw[edgePerm,back] (-2, 0, 0) -- (-2, 1, -1);
\draw[edgePerm,back] (-2, 1, -5) -- (0, 1, -5);
\draw[edgePerm,back] (-2, 1, -1) -- (0, 1, -1);

%% Drawing permutahedron vertices in the back
\node[vertexPerm,back] at (-4, -4, -2)     {};
\node[vertexPerm,back] at (-4, -3, -1)     {};
\node[vertexPerm,back] at (-4, -3, -5)     {};
\node[vertexPerm,back] at (-4, -1, -5)     {};
\node[vertexPerm,back] at (-3, -3, -6)     {};
\node[vertexPerm,back] at (-3, -1, -6)     {};
\node[vertexPerm,back] at (-4, -4, -4)     {};
\node[vertexPerm,back] at (-4, -1, -1)     {};
\node[vertexPerm,back] at (-4, 0, -4)     {};
\node[vertexPerm,back] at (-3, 1, -4)     {};
\node[vertexPerm,back] at (-2, 0, -6)     {};
\node[vertexPerm,back] at (-2, 1, -5)     {};
\node[vertexPerm,back] at (-4, 0, -2)     {};
\node[vertexPerm,back] at (-3, 1, -2)     {};
\node[vertexPerm,back] at (-2, 1, -1)     {};

%% Drawing permutahedron facets
\fill[facetPerm] (2, 0, -2) -- (2, -1, -1) -- (2, -3, -1) -- (2, -4, -2) -- (2, -4, -4) -- (2, -3, -5) -- (2, -1, -5) -- (2, 0, -4) -- cycle {};
\fill[facetPerm] (2, 0, -2) -- (1, 1, -2) -- (0, 1, -1) -- (0, 0, 0) -- (1, -1, 0) -- (2, -1, -1) -- cycle {};
\fill[facetPerm] (2, 0, -2) -- (1, 1, -2) -- (1, 1, -4) -- (2, 0, -4) -- cycle {};
\fill[facetPerm] (2, -1, -1) -- (1, -1, 0) -- (1, -3, 0) -- (2, -3, -1) -- cycle {};
\fill[facetPerm] (0, -4, -6) -- (-2, -4, -6) -- (-2, -5, -5) -- (0, -5, -5) -- cycle {};
\fill[facetPerm] (0, -4, 0) -- (-2, -4, 0) -- (-2, -5, -1) -- (0, -5, -1) -- cycle {};
\fill[facetPerm] (2, -3, -5) -- (1, -3, -6) -- (0, -4, -6) -- (0, -5, -5) -- (1, -5, -4) -- (2, -4, -4) -- cycle {};
\fill[facetPerm] (2, -3, -1) -- (1, -3, 0) -- (0, -4, 0) -- (0, -5, -1) -- (1, -5, -2) -- (2, -4, -2) -- cycle {};
\fill[facetPerm] (1, -5, -2) -- (0, -5, -1) -- (-2, -5, -1) -- (-3, -5, -2) -- (-3, -5, -4) -- (-2, -5, -5) -- (0, -5, -5) -- (1, -5, -4) -- cycle {};
\fill[facetPerm] (2, 0, -4) -- (1, 1, -4) -- (0, 1, -5) -- (0, 0, -6) -- (1, -1, -6) -- (2, -1, -5) -- cycle {};
\fill[facetPerm] (2, -4, -2) -- (1, -5, -2) -- (1, -5, -4) -- (2, -4, -4) -- cycle {};
\fill[facetPerm] (2, -1, -5) -- (1, -1, -6) -- (1, -3, -6) -- (2, -3, -5) -- cycle {};
\fill[facetPerm] (1, -1, 0) -- (0, 0, 0) -- (-2, 0, 0) -- (-3, -1, 0) -- (-3, -3, 0) -- (-2, -4, 0) -- (0, -4, 0) -- (1, -3, 0) -- cycle {};

%% Drawing permutahedron edges in the front
\draw[edgePerm] (-3, -5, -4) -- (-3, -5, -2);
\draw[edgePerm] (-3, -5, -4) -- (-2, -5, -5);
\draw[edgePerm] (-3, -5, -2) -- (-2, -5, -1);
\draw[edgePerm] (-3, -3, 0) -- (-3, -1, 0);
\draw[edgePerm] (-3, -3, 0) -- (-2, -4, 0);
\draw[edgePerm] (-3, -1, 0) -- (-2, 0, 0);
\draw[edgePerm] (-2, -5, -5) -- (-2, -4, -6);
\draw[edgePerm] (-2, -5, -5) -- (0, -5, -5);
\draw[edgePerm] (-2, -5, -1) -- (-2, -4, 0);
\draw[edgePerm] (-2, -5, -1) -- (0, -5, -1);
\draw[edgePerm] (-2, -4, -6) -- (0, -4, -6);
\draw[edgePerm] (-2, -4, 0) -- (0, -4, 0);
\draw[edgePerm] (-2, 0, 0) -- (0, 0, 0);
\draw[edgePerm] (0, -5, -5) -- (0, -4, -6);
\draw[edgePerm] (0, -5, -5) -- (1, -5, -4);
\draw[edgePerm] (0, -5, -1) -- (0, -4, 0);
\draw[edgePerm] (0, -5, -1) -- (1, -5, -2);
\draw[edgePerm] (0, -4, -6) -- (1, -3, -6);
\draw[edgePerm] (0, -4, 0) -- (1, -3, 0);
\draw[edgePerm] (0, 0, -6) -- (0, 1, -5);
\draw[edgePerm] (0, 0, -6) -- (1, -1, -6);
\draw[edgePerm] (0, 0, 0) -- (0, 1, -1);
\draw[edgePerm] (0, 0, 0) -- (1, -1, 0);
\draw[edgePerm] (0, 1, -5) -- (1, 1, -4);
\draw[edgePerm] (0, 1, -1) -- (1, 1, -2);
\draw[edgePerm] (1, -5, -4) -- (1, -5, -2);
\draw[edgePerm] (1, -5, -4) -- (2, -4, -4);
\draw[edgePerm] (1, -5, -2) -- (2, -4, -2);
\draw[edgePerm] (1, -3, -6) -- (1, -1, -6);
\draw[edgePerm] (1, -3, -6) -- (2, -3, -5);
\draw[edgePerm] (1, -3, 0) -- (1, -1, 0);
\draw[edgePerm] (1, -3, 0) -- (2, -3, -1);
\draw[edgePerm] (1, -1, -6) -- (2, -1, -5);
\draw[edgePerm] (1, -1, 0) -- (2, -1, -1);
\draw[edgePerm] (1, 1, -4) -- (1, 1, -2);
\draw[edgePerm] (1, 1, -4) -- (2, 0, -4);
\draw[edgePerm] (1, 1, -2) -- (2, 0, -2);
\draw[edgePerm] (2, -4, -4) -- (2, -4, -2);
\draw[edgePerm] (2, -4, -4) -- (2, -3, -5);
\draw[edgePerm] (2, -4, -2) -- (2, -3, -1);
\draw[edgePerm] (2, -3, -5) -- (2, -1, -5);
\draw[edgePerm] (2, -3, -1) -- (2, -1, -1);
\draw[edgePerm] (2, -1, -5) -- (2, 0, -4);
\draw[edgePerm] (2, -1, -1) -- (2, 0, -2);
\draw[edgePerm] (2, 0, -4) -- (2, 0, -2);

%% Drawing permutahedron vertices in the front
\node[vertexPerm] at (-3, -5, -4)     {};
\node[vertexPerm] at (-3, -5, -2)     {};
\node[vertexPerm] at (-3, -3, 0)     {};
\node[vertexPerm] at (-3, -1, 0)     {};
\node[vertexPerm] at (-2, -5, -5)     {};
\node[vertexPerm] at (-2, -5, -1)     {};
\node[vertexPerm] at (-2, -4, -6)     {};
\node[vertexPerm] at (-2, -4, 0)     {};
\node[vertexPerm] at (-2, 0, 0)     {};
\node[vertexPerm] at (0, -5, -5)     {};
\node[vertexPerm] at (0, -5, -1)     {};
\node[vertexPerm] at (0, -4, -6)     {};
\node[vertexPerm] at (0, -4, 0)     {};
\node[vertexPerm] at (0, 0, -6)     {};
\node[vertexPerm] at (0, 0, 0)     {};
\node[vertexPerm] at (0, 1, -5)     {};
\node[vertexPerm] at (0, 1, -1)     {};
\node[vertexPerm] at (1, -5, -4)     {};
\node[vertexPerm] at (1, -5, -2)     {};
\node[vertexPerm] at (1, -3, -6)     {};
\node[vertexPerm] at (1, -3, 0)     {};
\node[vertexPerm] at (1, -1, -6)     {};
\node[vertexPerm] at (1, -1, 0)     {};
\node[vertexPerm] at (1, 1, -4)     {};
\node[vertexPerm] at (1, 1, -2)     {};
\node[vertexPerm] at (2, -4, -4)     {};
\node[vertexPerm] at (2, -4, -2)     {};
\node[vertexPerm] at (2, -3, -5)     {};
\node[vertexPerm] at (2, -3, -1)     {};
\node[vertexPerm] at (2, -1, -5)     {};
\node[vertexPerm] at (2, -1, -1)     {};
\node[vertexPerm] at (2, 0, -4)     {};
\node[vertexPerm] at (2, 0, -2)     {};

%% Drawing associahedron facets
\fill[facetAsso] (0, 0, 0) -- (5, -5, 0) -- (10, -5, -5) -- (10, -4, -6) -- (0, 6, -6) -- cycle {};
\fill[facetAsso] (-2, -5, -5) -- (10, -5, -5) -- (10, -4, -6) -- (-2, -4, -6) -- cycle {};
\fill[facetAsso] (-3, -5, -4) -- (-2, -5, -5) -- (10, -5, -5) -- (5, -5, 0) -- (-1, -5, 0) -- (-3, -5, -2) -- cycle {};
\fill[facetAsso] (-3, -3, 0) -- (-1, -5, 0) -- (5, -5, 0) -- (0, 0, 0) -- (-2, 0, 0) -- (-3, -1, 0) -- cycle {};

%% Drawing associahedron edges in the front
\draw[edgeAsso] (10, -4, -6) -- (10, -5, -5);
\draw[edgeAsso] (10, -4, -6) -- (0, 6, -6);
\draw[edgeAsso] (10, -4, -6) -- (-2, -4, -6);
\draw[edgeAsso] (10, -5, -5) -- (5, -5, 0);
\draw[edgeAsso] (10, -5, -5) -- (-2, -5, -5);
\draw[edgeAsso] (5, -5, 0) -- (0, 0, 0);
\draw[edgeAsso] (5, -5, 0) -- (-1, -5, 0);
\draw[edgeAsso] (0, 6, -6) -- (0, 0, 0);
\draw[edgeAsso] (0, 0, 0) -- (-2, 0, 0);
\draw[edgeAsso] (-1, -5, 0) -- (-3, -5, -2);
\draw[edgeAsso] (-1, -5, 0) -- (-3, -3, 0);
\draw[edgeAsso] (-2, 0, 0) -- (-3, -1, 0);
\draw[edgeAsso] (-3, -5, -4) -- (-3, -5, -2);
\draw[edgeAsso] (-3, -5, -4) -- (-2, -5, -5);
\draw[edgeAsso] (-2, -4, -6) -- (-2, -5, -5);
\draw[edgeAsso] (-3, -1, 0) -- (-3, -3, 0);

%% Drawing associahedron vertices in the front
\node[vertexAsso] at (10, -4, -6)     {};
\node[vertexAsso] at (10, -5, -5)     {};
\node[vertexAsso] at (5, -5, 0)     {};
\node[vertexAsso] at (0, 6, -6)     {};
\node[vertexAsso] at (0, 0, 0)     {};
\node[vertexAsso] at (-1, -5, 0)     {};
\node[vertexAsso] at (-2, 0, 0)     {};
\node[vertexAsso] at (-3, -5, -4)     {};
\node[vertexAsso] at (-3, -5, -2)     {};
\node[vertexAsso] at (-2, -4, -6)     {};
\node[vertexAsso] at (-2, -5, -5)     {};
\node[vertexAsso] at (-3, -1, 0)     {};
\node[vertexAsso] at (-3, -3, 0)     {};

\end{tikzpicture}

%% file: figures/B3asso4.tex
% sage code
% sage: cong = B_Shards(3, B_shards = repr_to_B_shards(['0u0...', '.0d0..', '.0du0.']), restriction='complement_generators')
% sage: Img = cong.quotientope_from_Minkowski_sums().projection().tikz([-723,-409,-556], 94.3, scale=.3)
% sage: f=open('B3Asso1.tex','w'); f.write(Img); f.close()
% then add the picuture of the permutahedron in the middle.

\begin{tikzpicture}%
	[x={(0.487424cm, -0.236521cm)},
	y={(0.872822cm, 0.104998cm)},
	z={(0.024475cm, 0.965936cm)},
	scale=0.3,
	back/.style={very thin, opacity=0.5},
	edgePerm/.style={color=red!95!black, very thick, cap=round},
	edgeAsso/.style={color=blue!95!black, very thick, cap=round},
	facetPerm/.style={fill=red!95!black,fill opacity=0},
	facetAsso/.style={fill=blue!95!black,fill opacity=0},
	vertexPerm/.style={inner sep=0pt, circle, anchor=base},
	vertexAsso/.style={inner sep=0pt, circle, anchor=base}]

%% Coordinate of permutahedron vertices:
\coordinate (-4, -4, -4) at (-4, -4, -4);
\coordinate (-4, -4, -2) at (-4, -4, -2);
\coordinate (-4, -3, -5) at (-4, -3, -5);
\coordinate (-4, -3, -1) at (-4, -3, -1);
\coordinate (-4, -1, -5) at (-4, -1, -5);
\coordinate (-4, -1, -1) at (-4, -1, -1);
\coordinate (-4, 0, -4) at (-4, 0, -4);
\coordinate (-4, 0, -2) at (-4, 0, -2);
\coordinate (-3, -5, -4) at (-3, -5, -4);
\coordinate (-3, -5, -2) at (-3, -5, -2);
\coordinate (-3, -3, -6) at (-3, -3, -6);
\coordinate (-3, -3, 0) at (-3, -3, 0);
\coordinate (-3, -1, -6) at (-3, -1, -6);
\coordinate (-3, -1, 0) at (-3, -1, 0);
\coordinate (-3, 1, -4) at (-3, 1, -4);
\coordinate (-3, 1, -2) at (-3, 1, -2);
\coordinate (-2, -5, -5) at (-2, -5, -5);
\coordinate (-2, -5, -1) at (-2, -5, -1);
\coordinate (-2, -4, -6) at (-2, -4, -6);
\coordinate (-2, -4, 0) at (-2, -4, 0);
\coordinate (-2, 0, -6) at (-2, 0, -6);
\coordinate (-2, 0, 0) at (-2, 0, 0);
\coordinate (-2, 1, -5) at (-2, 1, -5);
\coordinate (-2, 1, -1) at (-2, 1, -1);
\coordinate (0, -5, -5) at (0, -5, -5);
\coordinate (0, -5, -1) at (0, -5, -1);
\coordinate (0, -4, -6) at (0, -4, -6);
\coordinate (0, -4, 0) at (0, -4, 0);
\coordinate (0, 0, -6) at (0, 0, -6);
\coordinate (0, 0, 0) at (0, 0, 0);
\coordinate (0, 1, -5) at (0, 1, -5);
\coordinate (0, 1, -1) at (0, 1, -1);
\coordinate (1, -5, -4) at (1, -5, -4);
\coordinate (1, -5, -2) at (1, -5, -2);
\coordinate (1, -3, -6) at (1, -3, -6);
\coordinate (1, -3, 0) at (1, -3, 0);
\coordinate (1, -1, -6) at (1, -1, -6);
\coordinate (1, -1, 0) at (1, -1, 0);
\coordinate (1, 1, -4) at (1, 1, -4);
\coordinate (1, 1, -2) at (1, 1, -2);
\coordinate (2, -4, -4) at (2, -4, -4);
\coordinate (2, -4, -2) at (2, -4, -2);
\coordinate (2, -3, -5) at (2, -3, -5);
\coordinate (2, -3, -1) at (2, -3, -1);
\coordinate (2, -1, -5) at (2, -1, -5);
\coordinate (2, -1, -1) at (2, -1, -1);
\coordinate (2, 0, -4) at (2, 0, -4);
\coordinate (2, 0, -2) at (2, 0, -2);

%% Coordinate of associahedron vertices:
\coordinate (-12, 0, 0) at (-12, 0, 0);
\coordinate (2, 4, -6) at (2, 4, -6);
\coordinate (2, -1, -1) at (2, -1, -1);
\coordinate (-12, 6, -6) at (-12, 6, -6);
\coordinate (2, -2, -6) at (2, -2, -6);
\coordinate (2, -3, -1) at (2, -3, -1);
\coordinate (2, -4, -2) at (2, -4, -2);
\coordinate (2, -4, -4) at (2, -4, -4);
\coordinate (1, -1, 0) at (1, -1, 0);
\coordinate (1, -3, 0) at (1, -3, 0);
\coordinate (-7, -5, 0) at (-7, -5, 0);
\coordinate (1, -5, -2) at (1, -5, -2);
\coordinate (1, -5, -4) at (1, -5, -4);
\coordinate (0, 6, -6) at (0, 6, -6);
\coordinate (0, 0, 0) at (0, 0, 0);
\coordinate (0, -4, -6) at (0, -4, -6);
\coordinate (0, -5, -5) at (0, -5, -5);
\coordinate (-1, -5, 0) at (-1, -5, 0);
\coordinate (-2, -4, -6) at (-2, -4, -6);
\coordinate (-2, -5, -5) at (-2, -5, -5);

%% Drawing associahedron edges in the back
\draw[edgeAsso,back] (-12, 0, 0) -- (-12, 6, -6);
\draw[edgeAsso,back] (-12, 6, -6) -- (0, 6, -6);
\draw[edgeAsso,back] (-12, 6, -6) -- (-2, -4, -6);

%% Drawing associahedron vertices in the back
\node[vertexAsso,back] at (-12, 6, -6)     {};

%% Drawing permutahedron edges in the back
\draw[edgePerm,back] (-4, -4, -4) -- (-4, -4, -2);
\draw[edgePerm,back] (-4, -4, -4) -- (-4, -3, -5);
\draw[edgePerm,back] (-4, -4, -4) -- (-3, -5, -4);
\draw[edgePerm,back] (-4, -4, -2) -- (-4, -3, -1);
\draw[edgePerm,back] (-4, -4, -2) -- (-3, -5, -2);
\draw[edgePerm,back] (-4, -3, -5) -- (-4, -1, -5);
\draw[edgePerm,back] (-4, -3, -5) -- (-3, -3, -6);
\draw[edgePerm,back] (-4, -3, -1) -- (-4, -1, -1);
\draw[edgePerm,back] (-4, -3, -1) -- (-3, -3, 0);
\draw[edgePerm,back] (-4, -1, -5) -- (-4, 0, -4);
\draw[edgePerm,back] (-4, -1, -5) -- (-3, -1, -6);
\draw[edgePerm,back] (-4, -1, -1) -- (-4, 0, -2);
\draw[edgePerm,back] (-4, -1, -1) -- (-3, -1, 0);
\draw[edgePerm,back] (-4, 0, -4) -- (-4, 0, -2);
\draw[edgePerm,back] (-4, 0, -4) -- (-3, 1, -4);
\draw[edgePerm,back] (-4, 0, -2) -- (-3, 1, -2);
\draw[edgePerm,back] (-3, -3, -6) -- (-3, -1, -6);
\draw[edgePerm,back] (-3, -3, -6) -- (-2, -4, -6);
\draw[edgePerm,back] (-3, -1, -6) -- (-2, 0, -6);
\draw[edgePerm,back] (-3, 1, -4) -- (-3, 1, -2);
\draw[edgePerm,back] (-3, 1, -4) -- (-2, 1, -5);
\draw[edgePerm,back] (-3, 1, -2) -- (-2, 1, -1);
\draw[edgePerm,back] (-2, 0, -6) -- (-2, 1, -5);
\draw[edgePerm,back] (-2, 0, -6) -- (0, 0, -6);
\draw[edgePerm,back] (-2, 0, 0) -- (-2, 1, -1);
\draw[edgePerm,back] (-2, 1, -5) -- (0, 1, -5);
\draw[edgePerm,back] (-2, 1, -1) -- (0, 1, -1);

%% Drawing permutahedron vertices in the back
\node[vertexPerm,back] at (-4, -4, -2)     {};
\node[vertexPerm,back] at (-4, -3, -1)     {};
\node[vertexPerm,back] at (-4, -3, -5)     {};
\node[vertexPerm,back] at (-4, -1, -5)     {};
\node[vertexPerm,back] at (-3, -3, -6)     {};
\node[vertexPerm,back] at (-3, -1, -6)     {};
\node[vertexPerm,back] at (-4, -4, -4)     {};
\node[vertexPerm,back] at (-4, -1, -1)     {};
\node[vertexPerm,back] at (-4, 0, -4)     {};
\node[vertexPerm,back] at (-3, 1, -4)     {};
\node[vertexPerm,back] at (-2, 0, -6)     {};
\node[vertexPerm,back] at (-2, 1, -5)     {};
\node[vertexPerm,back] at (-4, 0, -2)     {};
\node[vertexPerm,back] at (-3, 1, -2)     {};
\node[vertexPerm,back] at (-2, 1, -1)     {};

%% Drawing permutahedron facets
\fill[facetPerm] (2, 0, -2) -- (2, -1, -1) -- (2, -3, -1) -- (2, -4, -2) -- (2, -4, -4) -- (2, -3, -5) -- (2, -1, -5) -- (2, 0, -4) -- cycle {};
\fill[facetPerm] (2, 0, -2) -- (1, 1, -2) -- (0, 1, -1) -- (0, 0, 0) -- (1, -1, 0) -- (2, -1, -1) -- cycle {};
\fill[facetPerm] (2, 0, -2) -- (1, 1, -2) -- (1, 1, -4) -- (2, 0, -4) -- cycle {};
\fill[facetPerm] (2, -1, -1) -- (1, -1, 0) -- (1, -3, 0) -- (2, -3, -1) -- cycle {};
\fill[facetPerm] (0, -4, -6) -- (-2, -4, -6) -- (-2, -5, -5) -- (0, -5, -5) -- cycle {};
\fill[facetPerm] (0, -4, 0) -- (-2, -4, 0) -- (-2, -5, -1) -- (0, -5, -1) -- cycle {};
\fill[facetPerm] (2, -3, -5) -- (1, -3, -6) -- (0, -4, -6) -- (0, -5, -5) -- (1, -5, -4) -- (2, -4, -4) -- cycle {};
\fill[facetPerm] (2, -3, -1) -- (1, -3, 0) -- (0, -4, 0) -- (0, -5, -1) -- (1, -5, -2) -- (2, -4, -2) -- cycle {};
\fill[facetPerm] (1, -5, -2) -- (0, -5, -1) -- (-2, -5, -1) -- (-3, -5, -2) -- (-3, -5, -4) -- (-2, -5, -5) -- (0, -5, -5) -- (1, -5, -4) -- cycle {};
\fill[facetPerm] (2, 0, -4) -- (1, 1, -4) -- (0, 1, -5) -- (0, 0, -6) -- (1, -1, -6) -- (2, -1, -5) -- cycle {};
\fill[facetPerm] (2, -4, -2) -- (1, -5, -2) -- (1, -5, -4) -- (2, -4, -4) -- cycle {};
\fill[facetPerm] (2, -1, -5) -- (1, -1, -6) -- (1, -3, -6) -- (2, -3, -5) -- cycle {};
\fill[facetPerm] (1, -1, 0) -- (0, 0, 0) -- (-2, 0, 0) -- (-3, -1, 0) -- (-3, -3, 0) -- (-2, -4, 0) -- (0, -4, 0) -- (1, -3, 0) -- cycle {};

%% Drawing permutahedron edges in the front
\draw[edgePerm] (-3, -5, -4) -- (-3, -5, -2);
\draw[edgePerm] (-3, -5, -4) -- (-2, -5, -5);
\draw[edgePerm] (-3, -5, -2) -- (-2, -5, -1);
\draw[edgePerm] (-3, -3, 0) -- (-3, -1, 0);
\draw[edgePerm] (-3, -3, 0) -- (-2, -4, 0);
\draw[edgePerm] (-3, -1, 0) -- (-2, 0, 0);
\draw[edgePerm] (-2, -5, -5) -- (-2, -4, -6);
\draw[edgePerm] (-2, -5, -5) -- (0, -5, -5);
\draw[edgePerm] (-2, -5, -1) -- (-2, -4, 0);
\draw[edgePerm] (-2, -5, -1) -- (0, -5, -1);
\draw[edgePerm] (-2, -4, -6) -- (0, -4, -6);
\draw[edgePerm] (-2, -4, 0) -- (0, -4, 0);
\draw[edgePerm] (-2, 0, 0) -- (0, 0, 0);
\draw[edgePerm] (0, -5, -5) -- (0, -4, -6);
\draw[edgePerm] (0, -5, -5) -- (1, -5, -4);
\draw[edgePerm] (0, -5, -1) -- (0, -4, 0);
\draw[edgePerm] (0, -5, -1) -- (1, -5, -2);
\draw[edgePerm] (0, -4, -6) -- (1, -3, -6);
\draw[edgePerm] (0, -4, 0) -- (1, -3, 0);
\draw[edgePerm] (0, 0, -6) -- (0, 1, -5);
\draw[edgePerm] (0, 0, -6) -- (1, -1, -6);
\draw[edgePerm] (0, 0, 0) -- (0, 1, -1);
\draw[edgePerm] (0, 0, 0) -- (1, -1, 0);
\draw[edgePerm] (0, 1, -5) -- (1, 1, -4);
\draw[edgePerm] (0, 1, -1) -- (1, 1, -2);
\draw[edgePerm] (1, -5, -4) -- (1, -5, -2);
\draw[edgePerm] (1, -5, -4) -- (2, -4, -4);
\draw[edgePerm] (1, -5, -2) -- (2, -4, -2);
\draw[edgePerm] (1, -3, -6) -- (1, -1, -6);
\draw[edgePerm] (1, -3, -6) -- (2, -3, -5);
\draw[edgePerm] (1, -3, 0) -- (1, -1, 0);
\draw[edgePerm] (1, -3, 0) -- (2, -3, -1);
\draw[edgePerm] (1, -1, -6) -- (2, -1, -5);
\draw[edgePerm] (1, -1, 0) -- (2, -1, -1);
\draw[edgePerm] (1, 1, -4) -- (1, 1, -2);
\draw[edgePerm] (1, 1, -4) -- (2, 0, -4);
\draw[edgePerm] (1, 1, -2) -- (2, 0, -2);
\draw[edgePerm] (2, -4, -4) -- (2, -4, -2);
\draw[edgePerm] (2, -4, -4) -- (2, -3, -5);
\draw[edgePerm] (2, -4, -2) -- (2, -3, -1);
\draw[edgePerm] (2, -3, -5) -- (2, -1, -5);
\draw[edgePerm] (2, -3, -1) -- (2, -1, -1);
\draw[edgePerm] (2, -1, -5) -- (2, 0, -4);
\draw[edgePerm] (2, -1, -1) -- (2, 0, -2);
\draw[edgePerm] (2, 0, -4) -- (2, 0, -2);

%% Drawing permutahedron vertices in the front
\node[vertexPerm] at (-3, -5, -4)     {};
\node[vertexPerm] at (-3, -5, -2)     {};
\node[vertexPerm] at (-3, -3, 0)     {};
\node[vertexPerm] at (-3, -1, 0)     {};
\node[vertexPerm] at (-2, -5, -5)     {};
\node[vertexPerm] at (-2, -5, -1)     {};
\node[vertexPerm] at (-2, -4, -6)     {};
\node[vertexPerm] at (-2, -4, 0)     {};
\node[vertexPerm] at (-2, 0, 0)     {};
\node[vertexPerm] at (0, -5, -5)     {};
\node[vertexPerm] at (0, -5, -1)     {};
\node[vertexPerm] at (0, -4, -6)     {};
\node[vertexPerm] at (0, -4, 0)     {};
\node[vertexPerm] at (0, 0, -6)     {};
\node[vertexPerm] at (0, 0, 0)     {};
\node[vertexPerm] at (0, 1, -5)     {};
\node[vertexPerm] at (0, 1, -1)     {};
\node[vertexPerm] at (1, -5, -4)     {};
\node[vertexPerm] at (1, -5, -2)     {};
\node[vertexPerm] at (1, -3, -6)     {};
\node[vertexPerm] at (1, -3, 0)     {};
\node[vertexPerm] at (1, -1, -6)     {};
\node[vertexPerm] at (1, -1, 0)     {};
\node[vertexPerm] at (1, 1, -4)     {};
\node[vertexPerm] at (1, 1, -2)     {};
\node[vertexPerm] at (2, -4, -4)     {};
\node[vertexPerm] at (2, -4, -2)     {};
\node[vertexPerm] at (2, -3, -5)     {};
\node[vertexPerm] at (2, -3, -1)     {};
\node[vertexPerm] at (2, -1, -5)     {};
\node[vertexPerm] at (2, -1, -1)     {};
\node[vertexPerm] at (2, 0, -4)     {};
\node[vertexPerm] at (2, 0, -2)     {};

%% Drawing associahedron facets
\fill[facetAsso] (2, -4, -4) -- (2, -2, -6) -- (2, 4, -6) -- (2, -1, -1) -- (2, -3, -1) -- (2, -4, -2) -- cycle {};
\fill[facetAsso] (0, 0, 0) -- (1, -1, 0) -- (2, -1, -1) -- (2, 4, -6) -- (0, 6, -6) -- cycle {};
\fill[facetAsso] (1, -5, -2) -- (1, -5, -4) -- (0, -5, -5) -- (-2, -5, -5) -- (-7, -5, 0) -- (-1, -5, 0) -- cycle {};
\fill[facetAsso] (-2, -5, -5) -- (0, -5, -5) -- (0, -4, -6) -- (-2, -4, -6) -- cycle {};
\fill[facetAsso] (-12, 0, 0) -- (-7, -5, 0) -- (-1, -5, 0) -- (1, -3, 0) -- (1, -1, 0) -- (0, 0, 0) -- cycle {};
\fill[facetAsso] (-1, -5, 0) -- (1, -3, 0) -- (2, -3, -1) -- (2, -4, -2) -- (1, -5, -2) -- cycle {};
\fill[facetAsso] (1, -3, 0) -- (2, -3, -1) -- (2, -1, -1) -- (1, -1, 0) -- cycle {};
\fill[facetAsso] (0, -5, -5) -- (1, -5, -4) -- (2, -4, -4) -- (2, -2, -6) -- (0, -4, -6) -- cycle {};
\fill[facetAsso] (1, -5, -4) -- (2, -4, -4) -- (2, -4, -2) -- (1, -5, -2) -- cycle {};

%% Drawing associahedron edges in the front
\draw[edgeAsso] (-12, 0, 0) -- (-7, -5, 0);
\draw[edgeAsso] (-12, 0, 0) -- (0, 0, 0);
\draw[edgeAsso] (2, 4, -6) -- (2, -1, -1);
\draw[edgeAsso] (2, 4, -6) -- (2, -2, -6);
\draw[edgeAsso] (2, 4, -6) -- (0, 6, -6);
\draw[edgeAsso] (2, -1, -1) -- (2, -3, -1);
\draw[edgeAsso] (2, -1, -1) -- (1, -1, 0);
\draw[edgeAsso] (2, -2, -6) -- (2, -4, -4);
\draw[edgeAsso] (2, -2, -6) -- (0, -4, -6);
\draw[edgeAsso] (2, -3, -1) -- (2, -4, -2);
\draw[edgeAsso] (2, -3, -1) -- (1, -3, 0);
\draw[edgeAsso] (2, -4, -2) -- (2, -4, -4);
\draw[edgeAsso] (2, -4, -2) -- (1, -5, -2);
\draw[edgeAsso] (2, -4, -4) -- (1, -5, -4);
\draw[edgeAsso] (1, -1, 0) -- (1, -3, 0);
\draw[edgeAsso] (1, -1, 0) -- (0, 0, 0);
\draw[edgeAsso] (1, -3, 0) -- (-1, -5, 0);
\draw[edgeAsso] (-7, -5, 0) -- (-1, -5, 0);
\draw[edgeAsso] (-7, -5, 0) -- (-2, -5, -5);
\draw[edgeAsso] (1, -5, -2) -- (1, -5, -4);
\draw[edgeAsso] (1, -5, -2) -- (-1, -5, 0);
\draw[edgeAsso] (1, -5, -4) -- (0, -5, -5);
\draw[edgeAsso] (0, 6, -6) -- (0, 0, 0);
\draw[edgeAsso] (0, -4, -6) -- (0, -5, -5);
\draw[edgeAsso] (0, -4, -6) -- (-2, -4, -6);
\draw[edgeAsso] (0, -5, -5) -- (-2, -5, -5);
\draw[edgeAsso] (-2, -4, -6) -- (-2, -5, -5);

%% Drawing associahedron vertices in the front
\node[vertexAsso] at (-12, 0, 0)     {};
\node[vertexAsso] at (2, 4, -6)     {};
\node[vertexAsso] at (2, -1, -1)     {};
\node[vertexAsso] at (2, -2, -6)     {};
\node[vertexAsso] at (2, -3, -1)     {};
\node[vertexAsso] at (2, -4, -2)     {};
\node[vertexAsso] at (2, -4, -4)     {};
\node[vertexAsso] at (1, -1, 0)     {};
\node[vertexAsso] at (1, -3, 0)     {};
\node[vertexAsso] at (-7, -5, 0)     {};
\node[vertexAsso] at (1, -5, -2)     {};
\node[vertexAsso] at (1, -5, -4)     {};
\node[vertexAsso] at (0, 6, -6)     {};
\node[vertexAsso] at (0, 0, 0)     {};
\node[vertexAsso] at (0, -4, -6)     {};
\node[vertexAsso] at (0, -5, -5)     {};
\node[vertexAsso] at (-1, -5, 0)     {};
\node[vertexAsso] at (-2, -4, -6)     {};
\node[vertexAsso] at (-2, -5, -5)     {};

\end{tikzpicture}

%% file: figures/B3perm.tex
% sage code
% sage: bperm = Polyhedron(list(SignedPermutations(3))) - vector([1,2,3])
% sage: Img = bperm.projection().tikz([-723,-409,-556], 94.3, scale=.5)
% sage: f=open('B3Perm.tex','w'); f.write(Img); f.close()

\begin{tikzpicture}%
	[x={(0.487424cm, -0.236521cm)},
	y={(0.872822cm, 0.104998cm)},
	z={(0.024475cm, 0.965936cm)},
	scale=.5,
	back/.style={very thin, opacity=0.5},
	edge/.style={color=black, very thick, cap=round},
	facet/.style={fill=blue!95!black, fill opacity=0},
	vertex/.style={inner sep=0pt, circle, anchor=base}]

%% Coordinate of vertices:
\coordinate (-4, -4, -4) at (-4, -4, -4);
\coordinate (-4, -4, -2) at (-4, -4, -2);
\coordinate (-4, -3, -5) at (-4, -3, -5);
\coordinate (-4, -3, -1) at (-4, -3, -1);
\coordinate (-4, -1, -5) at (-4, -1, -5);
\coordinate (-4, -1, -1) at (-4, -1, -1);
\coordinate (-4, 0, -4) at (-4, 0, -4);
\coordinate (-4, 0, -2) at (-4, 0, -2);
\coordinate (-3, -5, -4) at (-3, -5, -4);
\coordinate (-3, -5, -2) at (-3, -5, -2);
\coordinate (-3, -3, -6) at (-3, -3, -6);
\coordinate (-3, -3, 0) at (-3, -3, 0);
\coordinate (-3, -1, -6) at (-3, -1, -6);
\coordinate (-3, -1, 0) at (-3, -1, 0);
\coordinate (-3, 1, -4) at (-3, 1, -4);
\coordinate (-3, 1, -2) at (-3, 1, -2);
\coordinate (-2, -5, -5) at (-2, -5, -5);
\coordinate (-2, -5, -1) at (-2, -5, -1);
\coordinate (-2, -4, -6) at (-2, -4, -6);
\coordinate (-2, -4, 0) at (-2, -4, 0);
\coordinate (-2, 0, -6) at (-2, 0, -6);
\coordinate (-2, 0, 0) at (-2, 0, 0);
\coordinate (-2, 1, -5) at (-2, 1, -5);
\coordinate (-2, 1, -1) at (-2, 1, -1);
\coordinate (0, -5, -5) at (0, -5, -5);
\coordinate (0, -5, -1) at (0, -5, -1);
\coordinate (0, -4, -6) at (0, -4, -6);
\coordinate (0, -4, 0) at (0, -4, 0);
\coordinate (0, 0, -6) at (0, 0, -6);
\coordinate (0, 0, 0) at (0, 0, 0);
\coordinate (0, 1, -5) at (0, 1, -5);
\coordinate (0, 1, -1) at (0, 1, -1);
\coordinate (1, -5, -4) at (1, -5, -4);
\coordinate (1, -5, -2) at (1, -5, -2);
\coordinate (1, -3, -6) at (1, -3, -6);
\coordinate (1, -3, 0) at (1, -3, 0);
\coordinate (1, -1, -6) at (1, -1, -6);
\coordinate (1, -1, 0) at (1, -1, 0);
\coordinate (1, 1, -4) at (1, 1, -4);
\coordinate (1, 1, -2) at (1, 1, -2);
\coordinate (2, -4, -4) at (2, -4, -4);
\coordinate (2, -4, -2) at (2, -4, -2);
\coordinate (2, -3, -5) at (2, -3, -5);
\coordinate (2, -3, -1) at (2, -3, -1);
\coordinate (2, -1, -5) at (2, -1, -5);
\coordinate (2, -1, -1) at (2, -1, -1);
\coordinate (2, 0, -4) at (2, 0, -4);
\coordinate (2, 0, -2) at (2, 0, -2);

%% Drawing edges in the back
\draw[edge,back] (-4, -4, -4) -- (-4, -4, -2);
\draw[edge,back] (-4, -4, -4) -- (-4, -3, -5);
\draw[edge,back] (-4, -4, -4) -- (-3, -5, -4);
\draw[edge,back] (-4, -4, -2) -- (-4, -3, -1);
\draw[edge,back] (-4, -4, -2) -- (-3, -5, -2);
\draw[edge,back] (-4, -3, -5) -- (-4, -1, -5);
\draw[edge,back] (-4, -3, -5) -- (-3, -3, -6);
\draw[edge,back] (-4, -3, -1) -- (-4, -1, -1);
\draw[edge,back] (-4, -3, -1) -- (-3, -3, 0);
\draw[edge,back] (-4, -1, -5) -- (-4, 0, -4);
\draw[edge,back] (-4, -1, -5) -- (-3, -1, -6);
\draw[edge,back] (-4, -1, -1) -- (-4, 0, -2);
\draw[edge,back] (-4, -1, -1) -- (-3, -1, 0);
\draw[edge,back] (-4, 0, -4) -- (-4, 0, -2);
\draw[edge,back] (-4, 0, -4) -- (-3, 1, -4);
\draw[edge,back] (-4, 0, -2) -- (-3, 1, -2);
\draw[edge,back] (-3, -3, -6) -- (-3, -1, -6);
\draw[edge,back] (-3, -3, -6) -- (-2, -4, -6);
\draw[edge,back] (-3, -1, -6) -- (-2, 0, -6);
\draw[edge,back] (-3, 1, -4) -- (-3, 1, -2);
\draw[edge,back] (-3, 1, -4) -- (-2, 1, -5);
\draw[edge,back] (-3, 1, -2) -- (-2, 1, -1);
\draw[edge,back] (-2, 0, -6) -- (-2, 1, -5);
\draw[edge,back] (-2, 0, -6) -- (0, 0, -6);
\draw[edge,back] (-2, 0, 0) -- (-2, 1, -1);
\draw[edge,back] (-2, 1, -5) -- (0, 1, -5);
\draw[edge,back] (-2, 1, -1) -- (0, 1, -1);

%% Drawing vertices in the back
\node[vertex,back] at (-4, -4, -2)     {};
\node[vertex,back] at (-4, -3, -1)     {};
\node[vertex,back] at (-4, -3, -5)     {};
\node[vertex,back] at (-4, -1, -5)     {};
\node[vertex,back] at (-3, -3, -6)     {};
\node[vertex,back] at (-3, -1, -6)     {};
\node[vertex,back] at (-4, -4, -4)     {};
\node[vertex,back] at (-4, -1, -1)     {};
\node[vertex,back] at (-4, 0, -4)     {};
\node[vertex,back] at (-3, 1, -4)     {};
\node[vertex,back] at (-2, 0, -6)     {};
\node[vertex,back] at (-2, 1, -5)     {};
\node[vertex,back] at (-4, 0, -2)     {};
\node[vertex,back] at (-3, 1, -2)     {};
\node[vertex,back] at (-2, 1, -1)     {};

%% Drawing facets
\fill[facet] (2, 0, -2) -- (2, -1, -1) -- (2, -3, -1) -- (2, -4, -2) -- (2, -4, -4) -- (2, -3, -5) -- (2, -1, -5) -- (2, 0, -4) -- cycle {};
\fill[facet] (2, 0, -2) -- (1, 1, -2) -- (0, 1, -1) -- (0, 0, 0) -- (1, -1, 0) -- (2, -1, -1) -- cycle {};
\fill[facet] (2, 0, -2) -- (1, 1, -2) -- (1, 1, -4) -- (2, 0, -4) -- cycle {};
\fill[facet] (2, -1, -1) -- (1, -1, 0) -- (1, -3, 0) -- (2, -3, -1) -- cycle {};
\fill[facet] (0, -4, -6) -- (-2, -4, -6) -- (-2, -5, -5) -- (0, -5, -5) -- cycle {};
\fill[facet] (0, -4, 0) -- (-2, -4, 0) -- (-2, -5, -1) -- (0, -5, -1) -- cycle {};
\fill[facet] (2, -3, -5) -- (1, -3, -6) -- (0, -4, -6) -- (0, -5, -5) -- (1, -5, -4) -- (2, -4, -4) -- cycle {};
\fill[facet] (2, -3, -1) -- (1, -3, 0) -- (0, -4, 0) -- (0, -5, -1) -- (1, -5, -2) -- (2, -4, -2) -- cycle {};
\fill[facet] (1, -5, -2) -- (0, -5, -1) -- (-2, -5, -1) -- (-3, -5, -2) -- (-3, -5, -4) -- (-2, -5, -5) -- (0, -5, -5) -- (1, -5, -4) -- cycle {};
\fill[facet] (2, 0, -4) -- (1, 1, -4) -- (0, 1, -5) -- (0, 0, -6) -- (1, -1, -6) -- (2, -1, -5) -- cycle {};
\fill[facet] (2, -4, -2) -- (1, -5, -2) -- (1, -5, -4) -- (2, -4, -4) -- cycle {};
\fill[facet] (2, -1, -5) -- (1, -1, -6) -- (1, -3, -6) -- (2, -3, -5) -- cycle {};
\fill[facet] (1, -1, 0) -- (0, 0, 0) -- (-2, 0, 0) -- (-3, -1, 0) -- (-3, -3, 0) -- (-2, -4, 0) -- (0, -4, 0) -- (1, -3, 0) -- cycle {};

%% Drawing edges in the front
\draw[edge] (-3, -5, -4) -- (-3, -5, -2);
\draw[edge] (-3, -5, -4) -- (-2, -5, -5);
\draw[edge] (-3, -5, -2) -- (-2, -5, -1);
\draw[edge] (-3, -3, 0) -- (-3, -1, 0);
\draw[edge] (-3, -3, 0) -- (-2, -4, 0);
\draw[edge] (-3, -1, 0) -- (-2, 0, 0);
\draw[edge] (-2, -5, -5) -- (-2, -4, -6);
\draw[edge] (-2, -5, -5) -- (0, -5, -5);
\draw[edge] (-2, -5, -1) -- (-2, -4, 0);
\draw[edge] (-2, -5, -1) -- (0, -5, -1);
\draw[edge] (-2, -4, -6) -- (0, -4, -6);
\draw[edge] (-2, -4, 0) -- (0, -4, 0);
\draw[edge] (-2, 0, 0) -- (0, 0, 0);
\draw[edge] (0, -5, -5) -- (0, -4, -6);
\draw[edge] (0, -5, -5) -- (1, -5, -4);
\draw[edge] (0, -5, -1) -- (0, -4, 0);
\draw[edge] (0, -5, -1) -- (1, -5, -2);
\draw[edge] (0, -4, -6) -- (1, -3, -6);
\draw[edge] (0, -4, 0) -- (1, -3, 0);
\draw[edge] (0, 0, -6) -- (0, 1, -5);
\draw[edge] (0, 0, -6) -- (1, -1, -6);
\draw[edge] (0, 0, 0) -- (0, 1, -1);
\draw[edge] (0, 0, 0) -- (1, -1, 0);
\draw[edge] (0, 1, -5) -- (1, 1, -4);
\draw[edge] (0, 1, -1) -- (1, 1, -2);
\draw[edge] (1, -5, -4) -- (1, -5, -2);
\draw[edge] (1, -5, -4) -- (2, -4, -4);
\draw[edge] (1, -5, -2) -- (2, -4, -2);
\draw[edge] (1, -3, -6) -- (1, -1, -6);
\draw[edge] (1, -3, -6) -- (2, -3, -5);
\draw[edge] (1, -3, 0) -- (1, -1, 0);
\draw[edge] (1, -3, 0) -- (2, -3, -1);
\draw[edge] (1, -1, -6) -- (2, -1, -5);
\draw[edge] (1, -1, 0) -- (2, -1, -1);
\draw[edge] (1, 1, -4) -- (1, 1, -2);
\draw[edge] (1, 1, -4) -- (2, 0, -4);
\draw[edge] (1, 1, -2) -- (2, 0, -2);
\draw[edge] (2, -4, -4) -- (2, -4, -2);
\draw[edge] (2, -4, -4) -- (2, -3, -5);
\draw[edge] (2, -4, -2) -- (2, -3, -1);
\draw[edge] (2, -3, -5) -- (2, -1, -5);
\draw[edge] (2, -3, -1) -- (2, -1, -1);
\draw[edge] (2, -1, -5) -- (2, 0, -4);
\draw[edge] (2, -1, -1) -- (2, 0, -2);
\draw[edge] (2, 0, -4) -- (2, 0, -2);

%% Drawing vertices in the front
\node[vertex] at (-3, -5, -4)     {};
\node[vertex] at (-3, -5, -2)     {};
\node[vertex] at (-3, -3, 0)     {};
\node[vertex] at (-3, -1, 0)     {};
\node[vertex] at (-2, -5, -5)     {};
\node[vertex] at (-2, -5, -1)     {};
\node[vertex] at (-2, -4, -6)     {};
\node[vertex] at (-2, -4, 0)     {};
\node[vertex] at (-2, 0, 0)     {};
\node[vertex] at (0, -5, -5)     {};
\node[vertex] at (0, -5, -1)     {};
\node[vertex] at (0, -4, -6)     {};
\node[vertex] at (0, -4, 0)     {};
\node[vertex] at (0, 0, -6)     {};
\node[vertex] at (0, 0, 0)     {};
\node[vertex] at (0, 1, -5)     {};
\node[vertex] at (0, 1, -1)     {};
\node[vertex] at (1, -5, -4)     {};
\node[vertex] at (1, -5, -2)     {};
\node[vertex] at (1, -3, -6)     {};
\node[vertex] at (1, -3, 0)     {};
\node[vertex] at (1, -1, -6)     {};
\node[vertex] at (1, -1, 0)     {};
\node[vertex] at (1, 1, -4)     {};
\node[vertex] at (1, 1, -2)     {};
\node[vertex] at (2, -4, -4)     {};
\node[vertex] at (2, -4, -2)     {};
\node[vertex] at (2, -3, -5)     {};
\node[vertex] at (2, -3, -1)     {};
\node[vertex] at (2, -1, -5)     {};
\node[vertex] at (2, -1, -1)     {};
\node[vertex] at (2, 0, -4)     {};
\node[vertex] at (2, 0, -2)     {};

\end{tikzpicture}

%% file: figures/B3permWeird.tex
% sage code
% sage: Img = B_Shards(3, restriction='all').quotientope_from_Minkowski_sums().projection().tikz([-723,-409,-556],94.3, scale=.3)
% sage: f=open('B3PermWeird.tex','w'); f.write(Img); f.close()

\begin{tikzpicture}%
	[x={(0.487424cm, -0.236521cm)},
	y={(0.872822cm, 0.104998cm)},
	z={(0.024475cm, 0.965936cm)},
	scale=0.300000,
	back/.style={very thin, opacity=0.5},
	edge/.style={color=black, very thick, cap=round},
	facet/.style={fill=blue!95!black, fill opacity=0},
	vertex/.style={inner sep=0pt, circle, anchor=base}]

%% Coordinate of vertices:
\coordinate (-23, -10, -5) at (-23, -10, -5);
\coordinate (-23, -10, -3) at (-23, -10, -3);
\coordinate (-23, -9, -2) at (-23, -9, -2);
\coordinate (-23, -3, -2) at (-23, -3, -2);
\coordinate (-23, 0, -15) at (-23, 0, -15);
\coordinate (-23, 6, -15) at (-23, 6, -15);
\coordinate (-23, 7, -14) at (-23, 7, -14);
\coordinate (-23, 7, -12) at (-23, 7, -12);
\coordinate (18, -2, -16) at (18, -2, -16);
\coordinate (18, -2, -18) at (18, -2, -18);
\coordinate (18, -3, -19) at (18, -3, -19);
\coordinate (-21, -9, 0) at (-21, -9, 0);
\coordinate (18, -9, -19) at (18, -9, -19);
\coordinate (-21, -3, 0) at (-21, -3, 0);
\coordinate (18, -12, -6) at (18, -12, -6);
\coordinate (18, -18, -6) at (18, -18, -6);
\coordinate (18, -19, -7) at (18, -19, -7);
\coordinate (18, -19, -9) at (18, -19, -9);
\coordinate (16, -3, -21) at (16, -3, -21);
\coordinate (16, -9, -21) at (16, -9, -21);
\coordinate (-19, 11, -14) at (-19, 11, -14);
\coordinate (-19, 11, -12) at (-19, 11, -12);
\coordinate (-18, 0, 0) at (-18, 0, 0);
\coordinate (-18, 11, -11) at (-18, 11, -11);
\coordinate (-17, 0, -21) at (-17, 0, -21);
\coordinate (-17, 6, -21) at (-17, 6, -21);
\coordinate (14, -23, -7) at (14, -23, -7);
\coordinate (14, -23, -9) at (14, -23, -9);
\coordinate (13, -12, -21) at (13, -12, -21);
\coordinate (13, -23, -10) at (13, -23, -10);
\coordinate (12, -12, 0) at (12, -12, 0);
\coordinate (12, -18, 0) at (12, -18, 0);
\coordinate (-13, 10, -21) at (-13, 10, -21);
\coordinate (-13, 11, -20) at (-13, 11, -20);
\coordinate (8, -22, 0) at (8, -22, 0);
\coordinate (8, -23, -1) at (8, -23, -1);
\coordinate (-10, -23, -5) at (-10, -23, -5);
\coordinate (-10, -23, -3) at (-10, -23, -3);
\coordinate (5, 11, -16) at (5, 11, -16);
\coordinate (5, 11, -18) at (5, 11, -18);
\coordinate (-8, -23, -1) at (-8, -23, -1);
\coordinate (-8, -22, 0) at (-8, -22, 0);
\coordinate (3, 11, -20) at (3, 11, -20);
\coordinate (3, 10, -21) at (3, 10, -21);
\coordinate (-5, -23, -10) at (-5, -23, -10);
\coordinate (-5, -12, -21) at (-5, -12, -21);
\coordinate (0, 11, -11) at (0, 11, -11);
\coordinate (0, 0, 0) at (0, 0, 0);

%% Drawing edges in the back
\draw[edge,back] (-23, -10, -5) -- (-23, -10, -3);
\draw[edge,back] (-23, -10, -5) -- (-23, 0, -15);
\draw[edge,back] (-23, -10, -5) -- (-10, -23, -5);
\draw[edge,back] (-23, -10, -3) -- (-23, -9, -2);
\draw[edge,back] (-23, -10, -3) -- (-10, -23, -3);
\draw[edge,back] (-23, -9, -2) -- (-23, -3, -2);
\draw[edge,back] (-23, -9, -2) -- (-21, -9, 0);
\draw[edge,back] (-23, -3, -2) -- (-23, 7, -12);
\draw[edge,back] (-23, -3, -2) -- (-21, -3, 0);
\draw[edge,back] (-23, 0, -15) -- (-23, 6, -15);
\draw[edge,back] (-23, 0, -15) -- (-17, 0, -21);
\draw[edge,back] (-23, 6, -15) -- (-23, 7, -14);
\draw[edge,back] (-23, 6, -15) -- (-17, 6, -21);
\draw[edge,back] (-23, 7, -14) -- (-23, 7, -12);
\draw[edge,back] (-23, 7, -14) -- (-19, 11, -14);
\draw[edge,back] (-23, 7, -12) -- (-19, 11, -12);
\draw[edge,back] (-19, 11, -14) -- (-19, 11, -12);
\draw[edge,back] (-19, 11, -14) -- (-13, 11, -20);
\draw[edge,back] (-19, 11, -12) -- (-18, 11, -11);
\draw[edge,back] (-18, 0, 0) -- (-18, 11, -11);
\draw[edge,back] (-18, 11, -11) -- (0, 11, -11);
\draw[edge,back] (-17, 0, -21) -- (-17, 6, -21);
\draw[edge,back] (-17, 0, -21) -- (-5, -12, -21);
\draw[edge,back] (-17, 6, -21) -- (-13, 10, -21);
\draw[edge,back] (-13, 10, -21) -- (-13, 11, -20);
\draw[edge,back] (-13, 10, -21) -- (3, 10, -21);
\draw[edge,back] (-13, 11, -20) -- (3, 11, -20);

%% Drawing vertices in the back
\node[vertex,back] at (-18, 11, -11)     {};
\node[vertex,back] at (-23, -9, -2)     {};
\node[vertex,back] at (-23, -3, -2)     {};
\node[vertex,back] at (-23, -10, -5)     {};
\node[vertex,back] at (-23, -10, -3)     {};
\node[vertex,back] at (-23, 0, -15)     {};
\node[vertex,back] at (-23, 6, -15)     {};
\node[vertex,back] at (-23, 7, -14)     {};
\node[vertex,back] at (-23, 7, -12)     {};
\node[vertex,back] at (-17, 0, -21)     {};
\node[vertex,back] at (-19, 11, -12)     {};
\node[vertex,back] at (-19, 11, -14)     {};
\node[vertex,back] at (-13, 11, -20)     {};
\node[vertex,back] at (-17, 6, -21)     {};
\node[vertex,back] at (-13, 10, -21)     {};

%% Drawing facets
\fill[facet] (18, -19, -9) -- (18, -9, -19) -- (18, -3, -19) -- (18, -2, -18) -- (18, -2, -16) -- (18, -12, -6) -- (18, -18, -6) -- (18, -19, -7) -- cycle {};
\fill[facet] (0, 0, 0) -- (12, -12, 0) -- (18, -12, -6) -- (18, -2, -16) -- (5, 11, -16) -- (0, 11, -11) -- cycle {};
\fill[facet] (5, 11, -18) -- (18, -2, -18) -- (18, -2, -16) -- (5, 11, -16) -- cycle {};
\fill[facet] (8, -22, 0) -- (12, -18, 0) -- (12, -12, 0) -- (0, 0, 0) -- (-18, 0, 0) -- (-21, -3, 0) -- (-21, -9, 0) -- (-8, -22, 0) -- cycle {};
\fill[facet] (-8, -22, 0) -- (8, -22, 0) -- (8, -23, -1) -- (-8, -23, -1) -- cycle {};
\fill[facet] (8, -23, -1) -- (14, -23, -7) -- (18, -19, -7) -- (18, -18, -6) -- (12, -18, 0) -- (8, -22, 0) -- cycle {};
\fill[facet] (14, -23, -9) -- (18, -19, -9) -- (18, -19, -7) -- (14, -23, -7) -- cycle {};
\fill[facet] (13, -23, -10) -- (14, -23, -9) -- (18, -19, -9) -- (18, -9, -19) -- (16, -9, -21) -- (13, -12, -21) -- cycle {};
\fill[facet] (-10, -23, -3) -- (-10, -23, -5) -- (-5, -23, -10) -- (13, -23, -10) -- (14, -23, -9) -- (14, -23, -7) -- (8, -23, -1) -- (-8, -23, -1) -- cycle {};
\fill[facet] (-5, -12, -21) -- (13, -12, -21) -- (13, -23, -10) -- (-5, -23, -10) -- cycle {};
\fill[facet] (3, 10, -21) -- (16, -3, -21) -- (18, -3, -19) -- (18, -2, -18) -- (5, 11, -18) -- (3, 11, -20) -- cycle {};
\fill[facet] (12, -18, 0) -- (18, -18, -6) -- (18, -12, -6) -- (12, -12, 0) -- cycle {};
\fill[facet] (16, -9, -21) -- (18, -9, -19) -- (18, -3, -19) -- (16, -3, -21) -- cycle {};

%% Drawing edges in the front
\draw[edge] (18, -2, -16) -- (18, -2, -18);
\draw[edge] (18, -2, -16) -- (18, -12, -6);
\draw[edge] (18, -2, -16) -- (5, 11, -16);
\draw[edge] (18, -2, -18) -- (18, -3, -19);
\draw[edge] (18, -2, -18) -- (5, 11, -18);
\draw[edge] (18, -3, -19) -- (18, -9, -19);
\draw[edge] (18, -3, -19) -- (16, -3, -21);
\draw[edge] (-21, -9, 0) -- (-21, -3, 0);
\draw[edge] (-21, -9, 0) -- (-8, -22, 0);
\draw[edge] (18, -9, -19) -- (18, -19, -9);
\draw[edge] (18, -9, -19) -- (16, -9, -21);
\draw[edge] (-21, -3, 0) -- (-18, 0, 0);
\draw[edge] (18, -12, -6) -- (18, -18, -6);
\draw[edge] (18, -12, -6) -- (12, -12, 0);
\draw[edge] (18, -18, -6) -- (18, -19, -7);
\draw[edge] (18, -18, -6) -- (12, -18, 0);
\draw[edge] (18, -19, -7) -- (18, -19, -9);
\draw[edge] (18, -19, -7) -- (14, -23, -7);
\draw[edge] (18, -19, -9) -- (14, -23, -9);
\draw[edge] (16, -3, -21) -- (16, -9, -21);
\draw[edge] (16, -3, -21) -- (3, 10, -21);
\draw[edge] (16, -9, -21) -- (13, -12, -21);
\draw[edge] (-18, 0, 0) -- (0, 0, 0);
\draw[edge] (14, -23, -7) -- (14, -23, -9);
\draw[edge] (14, -23, -7) -- (8, -23, -1);
\draw[edge] (14, -23, -9) -- (13, -23, -10);
\draw[edge] (13, -12, -21) -- (13, -23, -10);
\draw[edge] (13, -12, -21) -- (-5, -12, -21);
\draw[edge] (13, -23, -10) -- (-5, -23, -10);
\draw[edge] (12, -12, 0) -- (12, -18, 0);
\draw[edge] (12, -12, 0) -- (0, 0, 0);
\draw[edge] (12, -18, 0) -- (8, -22, 0);
\draw[edge] (8, -22, 0) -- (8, -23, -1);
\draw[edge] (8, -22, 0) -- (-8, -22, 0);
\draw[edge] (8, -23, -1) -- (-8, -23, -1);
\draw[edge] (-10, -23, -5) -- (-10, -23, -3);
\draw[edge] (-10, -23, -5) -- (-5, -23, -10);
\draw[edge] (-10, -23, -3) -- (-8, -23, -1);
\draw[edge] (5, 11, -16) -- (5, 11, -18);
\draw[edge] (5, 11, -16) -- (0, 11, -11);
\draw[edge] (5, 11, -18) -- (3, 11, -20);
\draw[edge] (-8, -23, -1) -- (-8, -22, 0);
\draw[edge] (3, 11, -20) -- (3, 10, -21);
\draw[edge] (-5, -23, -10) -- (-5, -12, -21);
\draw[edge] (0, 11, -11) -- (0, 0, 0);

%% Drawing vertices in the front
\node[vertex] at (18, -2, -16)     {};
\node[vertex] at (18, -2, -18)     {};
\node[vertex] at (18, -3, -19)     {};
\node[vertex] at (-21, -9, 0)     {};
\node[vertex] at (18, -9, -19)     {};
\node[vertex] at (-21, -3, 0)     {};
\node[vertex] at (18, -12, -6)     {};
\node[vertex] at (18, -18, -6)     {};
\node[vertex] at (18, -19, -7)     {};
\node[vertex] at (18, -19, -9)     {};
\node[vertex] at (16, -3, -21)     {};
\node[vertex] at (16, -9, -21)     {};
\node[vertex] at (-18, 0, 0)     {};
\node[vertex] at (14, -23, -7)     {};
\node[vertex] at (14, -23, -9)     {};
\node[vertex] at (13, -12, -21)     {};
\node[vertex] at (13, -23, -10)     {};
\node[vertex] at (12, -12, 0)     {};
\node[vertex] at (12, -18, 0)     {};
\node[vertex] at (8, -22, 0)     {};
\node[vertex] at (8, -23, -1)     {};
\node[vertex] at (-10, -23, -5)     {};
\node[vertex] at (-10, -23, -3)     {};
\node[vertex] at (5, 11, -16)     {};
\node[vertex] at (5, 11, -18)     {};
\node[vertex] at (-8, -23, -1)     {};
\node[vertex] at (-8, -22, 0)     {};
\node[vertex] at (3, 11, -20)     {};
\node[vertex] at (3, 10, -21)     {};
\node[vertex] at (-5, -23, -10)     {};
\node[vertex] at (-5, -12, -21)     {};
\node[vertex] at (0, 11, -11)     {};
\node[vertex] at (0, 0, 0)     {};

\end{tikzpicture}